\documentclass[leqno,11pt]{amsbook}
\pdfoutput1

\usepackage{amssymb,color,hyperref,mathrsfs,stmaryrd,relsize}
\usepackage{amsmath}

\usepackage[usenames,dvipsnames]{xcolor}

\usepackage[curve,matrix,arrow]{xy}

\numberwithin{equation}{chapter}
\newtheorem{theorem}{Theorem}[chapter]
\newtheorem{lemma}[theorem]{Lemma}

\newtheorem{prop}[theorem]{Proposition}
\newtheorem{cor}[theorem]{Corollary}
\newtheorem{Th}{Theorem}

\newtheorem{corollary}[theorem]{Corollary}

\theoremstyle{definition}

\newtheorem{remark}[theorem]{Remark}

\newtheorem{notation}[theorem]{Notation}
\newtheorem{definition}[theorem]{Definition}
\newtheorem{Rmk}{Remark}
\newtheorem{Defn}{Definition}

\newcommand{\bN}{\mathbb{N}}

\newcommand{\F}{\mathcal{F}}
\newcommand{\E}{\mathcal{E}}

\newcommand{\C}{\mathcal{C}}
\renewcommand{\L}{\mathcal{L}}
\renewcommand{\P}{\mathcal{P}}
\newcommand{\tL}{\widetilde{\mathcal{L}}}
\newcommand{\tDelta}{\tilde{\Delta}}

\newcommand{\tS}{\widetilde{S}}

\newcommand{\tPi}{\widetilde{\Pi}}
\newcommand{\tD}{\widetilde{\mathbf{D}}}
\newcommand{\M}{\mathcal{M}}
\newcommand{\N}{\mathcal{N}}
\newcommand{\K}{\mathcal{K}}
\newcommand{\X}{\mathcal{X}}
\newcommand{\Y}{\mathcal{Y}}
\newcommand{\D}{\mathbf{D}}
\renewcommand{\H}{\mathcal{H}}
\newcommand{\wL}{\widetilde{\mathcal{L}}}

\newcommand{\Hom}{\operatorname{Hom}}
\newcommand{\Iso}{\operatorname{Iso}}
\newcommand{\Aut}{\operatorname{Aut}}
\newcommand{\Inn}{\operatorname{Inn}}
\newcommand{\Out}{\operatorname{Out}}
\newcommand{\Syl}{\operatorname{Syl}}

\newcommand{\id}{\operatorname{id}}

\newcommand{\One}{\operatorname{\mathbf{1}}}
\newcommand{\W}{\mathbf{W}}
\newcommand{\fN}{\mathfrak{N}}
\newcommand{\hyp}{\mathfrak{hyp}}
\newcommand{\foc}{\mathfrak{foc}}
\newcommand{\fC}{\mathfrak{C}}
\newcommand{\subn}{\unlhd\!\unlhd\;}
\newcommand{\Comp}{\operatorname{Comp}}

\newcommand{\tF}{\widetilde{\mathcal{F}}}

\def \<{\langle }
\def \>{\rangle }

\renewcommand{\phi}{\varphi}
\makeatletter
\def\blfootnote{\xdef\@thefnmark{}\@footnotetext}
\makeatother

\title[Commuting partial normal subgroups and regular localities]{Commuting partial normal subgroups and regular localities}

\author[E.~Henke]{Ellen Henke}
\address{Institut f{\"u}r Algebra, Fakult{\"a}t Mathematik, Technische Universit{\"a}t Dresden, 01062 Dresden, Germany}
\email{ellen.henke@tu-dresden.de}
\begin{document}

\maketitle

\tableofcontents

\addtocontents{toc}{\setcounter{tocdepth}{-1}}
\chapter*{Abstract}
\addtocontents{toc}{\setcounter{tocdepth}{0}}

In this paper, important concepts from finite group theory are translated to localities, in particular to linking localities. Here localities  are group-like structures associated to fusion systems which were introduced by Chermak. Linking localities (by Chermak also called proper localities) are special kinds of localities which correspond to linking systems. Thus they contain  the algebraic information that is needed to study $p$-completed classifying spaces of fusion systems as generalizations of $p$-completed classifying spaces of finite groups. 

Because of the group-like nature of localities, there is a natural notion of partial normal subgroups. Given a locality $\mathcal{L}$ and a partial normal subgroup $\mathcal{N}$ of $\mathcal{L}$, we show that there is a largest partial normal subgroup $\mathcal{N}^\perp$ of $\mathcal{L}$ which, in a certain sense, commutes elementwise with $\mathcal{N}$ and thus morally plays the role of a ``centralizer'' of $\mathcal{N}$ in $\mathcal{L}$. This leads to a nice notion of the generalized Fitting subgroup $F^*(\mathcal{L})$ of a linking locality $\mathcal{L}$. Building on these results we define and study special kinds of linking localities called regular localities. It turns out that there is a theory of components of regular localities akin to the theory of components of finite groups. The main concepts we introduce and work with in the present paper (in particular $\mathcal{N}^\perp$ in the special case of linking localities, $F^*(\mathcal{L})$, regular localities and components of regular localities) were already introduced and studied in a preprint by Chermak. However, we give a different and self-contained approach to the subject where we reprove Chermak's theorems and also show several new results.

\blfootnote{The author would like to thank the Isaac Newton Institute for Mathematical Sciences, Cambridge, for support und hospitality during the programme Groups, representations and applications, where work on this paper was undertaken and supported by EPSRC grant no EP/R014604/1.}

\chapter*{Introduction}

\section{Fusion systems and linking localities} A saturated fusion system is a category $\F$ together a finite $p$-group $S$ such that the objects of $\F$ are all the subgroups of $S$ and morphisms are injective group homomorphisms subject to certain axioms. Any finite group $G$ with a Sylow $p$-subgroup $S$ leads to a saturated fusion system $\F_S(G)$ whose morphisms are the conjugation maps by elements of $G$. Under a different name, saturated fusion systems were first introduced by Puig in the 1990s (cf. \cite{Puig:2006a}). Later on, for the purposes of homotopy theory, Broto, Levi and Oliver \cite{BLO2} introduced centric linking systems associated to saturated fusion systems. The longstanding conjecture that there is a unique centric linking system associated to every saturated fusion system was proved by Chermak \cite{Chermak:2013}. In this context, Chermak introduced group-like structures called localities, which were then studied further by Chermak in three preprints \cite{Chermak:2015,ChermakII,ChermakIII}, by the author of the present paper \cite{Henke:2015a,Henke:2015,Henke:2020} and in a joint preprint by both authors \cite{Chermak/Henke}. 

\smallskip

Even though Oliver \cite{Oliver:2013} gave a proof of the existence and uniqueness of centric linking systems which does not involve localities, we feel that it is nevertheless useful to investigate localities further to gain a deeper understanding of the theory of fusion systems. Chermak started in three preprints \cite{Chermak:2015,ChermakII,ChermakIII} to develop a theory of localities akin to the local theory of groups, and in the joint preprint \cite{Chermak/Henke} by Chermak and the author of the present paper we relate these concepts to corresponding concepts in fusion systems. We believe that the theory of localities and its close connection to the theory of fusion systems will make simplifications possible in Aschbacher's program to revisit the proof of the classification of finite simple groups via classification theorems for fusion systems. A proper understanding of localities might also motivate a more comprehensive extension theory of saturated fusion systems. The results from Chermak's preprint \cite{ChermakIII} are needed in \cite{Chermak/Henke} and we expect them also to be of major importance in all future endeavors. Unfortunately, some of Chermak's arguments are rather complicated and difficult to read, and some proofs seem to be  incorrect in the details. While we believe that there would be direct ways to fill in the gaps in Chermak's proofs of the main results in \cite{ChermakIII}, we have nevertheless chosen to give a different and self-contained approach to the subject in the present paper. In particular, we simplify and generalize many of Chermak's arguments, and we prove some new results which seem natural and useful to include. Some of the similarities and differences to Chermak's approach will be summarized in Section~\ref{SS:ChermaksApproach}. Further pointers are given throughout the text. With clean proofs of the results presented in this paper in place, it is possible to obtain new proofs of many theorems about fusion systems stated in \cite{Aschbacher:2011}. In \cite{Chermak/Henke,Henke:NEX,Henke:NK} such new proofs as well as some new results about fusion systems are obtained as a byproduct of the correspondences between localities and fusion systems shown in \cite{Chermak/Henke}.

\smallskip

We will now explain the main results shown in this paper. Throughout let $p$ be a prime. The reader is referred to \cite{Chermak:2015} (or to the summary in  Chapter~\ref{S:PartialGroupsLocalities}) for an introduction to partial groups and localities. Roughly speaking, a partial group is a set $\L$ together with a ``product'' $\Pi\colon \D\rightarrow \L$ defined on a set $\D$ of words in $\L$ and with an ``inversion map'' $\L\rightarrow\L,f\mapsto f^{-1}$ subject to certain group-like axioms. A locality is a triple $(\L,\Delta,S)$ where $\L$ is a finite partial group, $S$ is a maximal $p$-subgroup of $\L$ and $\Delta$ is a set of subgroups of $S$ which in a certain sense determines the domain of the product on $\L$. There is a natural notion of conjugation and thus of normalizers and centralizers in partial groups. It follows from the precise definition of a locality that, for any locality $(\L,\Delta,S)$ and any $P\in\Delta$, the normalizer $N_\L(P)$ forms a finite group. Moreover, every locality $(\L,\Delta,S)$ leads naturally to a fusion system $\F_S(\L)$. We will say that $(\L,\Delta,S)$ is a locality over $\F$ to indicate that $\F=\F_S(\L)$. The following definition is central.

\begin{Defn}
\begin{itemize}
\item A finite group $G$ is said to be of \emph{characteristic $p$} if $C_G(O_p(G))\leq O_p(G)$. 
\item A locality $(\L,\Delta,S)$ is called a \emph{linking locality} if $\F_S(\L)^{cr}\subseteq\Delta$ and $N_\L(P)$ is of characteristic $p$ for every $P\in\Delta$. 
\end{itemize}
\end{Defn}

For a fusion system $\F$, we use a slightly non-standard definition of $\F^{cr}$ here, which ensures that the fusion system of a linking locality is saturated (cf. Definition~\ref{D:Fcr} and Theorem~\ref{T:Saturation}). Linking localities over a fusion system $\F$ as defined above correspond to linking systems associated to $\F$ in the sense of \cite{Henke:2015}. Chermak \cite{ChermakII,ChermakIII} uses the term \emph{proper locality} instead of linking locality, and this terminology is also used in \cite{Chermak/Henke}. 

\smallskip

Because localities are so group-like, there is a natural notion of partial normal subgroups of $\L$ (cf. Definition~\ref{D:PartialSubgroup}). We will write $\N\unlhd\L$ to indicate that $\N$ is a partial normal subgroup of $\L$.

\section{Commuting partial normal subgroups} In a group, the centralizer of any subset is a subgroup and the centralizer of any normal subgroup is again normal. The situation is more complicated in localities, where centralizers of subsets are naturally defined, but rarely turn out to be partial subgroups. If $(\L,\Delta,S)$ is a linking locality and $\N\unlhd\L$, then it seems that the centralizer $C_\L(\N)$ is just not quite the right object to consider in general. However, it turns out that there is a partial normal subgroup $\N^\perp$ of $\L$ that morally should be thought of as a replacement for the centralizer of $\N$. To describe our results in more detail, we will use the following definition.

\begin{Defn}\label{D:mainCommute}
Let $\X$ and $\Y$ be subsets of a partial group $\L$ with partial product $\Pi\colon\D\rightarrow \L$. For $(x,y)\in\D$ set $xy:=\Pi(x,y)$.

The subset $\X$ is said to \emph{commute (elementwise) with} $\Y$ if, for all $x\in\X$ and $y\in\Y$, the following implication holds:
\[(x,y)\in\D\Longrightarrow (y,x)\in\D\mbox{ and }xy=yx.\]
We say that $\X$ \emph{fixes $\Y$ under conjugation} if, for all $x\in\X$ and $y\in\Y$, the following implication holds:
\[y\in\D(x)\Longrightarrow y^x=y.\]
Here for any $x\in\L$, the set $\D(x):=\{y\in\L\colon (x^{-1},y,x)\in\D\}$ is the set of elements of $\L$ for which conjugation by $x\in\L$ is defined, and we set $y^x:=\Pi(x^{-1},y,x)$ for all $y\in\D(x)$.\\
The \emph{centralizer} of $\X$ in $\L$ is defined by   
\[C_\L(\X):=\{y\in\L\colon x\in\D(y)\mbox{ and }x^y=x\mbox{ for all }x\in\X\}.\]
\end{Defn}

Both the condition that $\X$ commutes with $\Y$ and the condition that $\X$ fixes $\Y$ under conjugation should be thought of as a ``weak centralizing condition''. However, these conditions do not imply in general that $\Y\subseteq C_\L(\X)$. Indeed, it can happen that, for all $x\in\X$ and $y\in\Y$, we have $(x,y)\not\in\D$ (or similarly $y\not\in\D(x)$). This leads to examples where $\X$ commutes with $\Y$, but $\Y$ does not commute with $\X$ and similarly to examples where $\X$ fixes $\Y$ under conjugation, but $\Y$ does not fix $\X$ under conjugation. Also, while it is true that $\Y$ fixes $\X$ under conjugation if $\X$ commutes with $\Y$, the converse does not hold in general. In contrast, we show in Lemma~\ref{L:CentralizerHelp} that $\Y\subseteq C_\L(\X)$ if and only if $\X\subseteq C_\L(\Y)$; moreover, if so, then $\X$ commutes with $\Y$ and vice versa; similarly $\X$ fixes then $\Y$ under conjugation and vice versa.

\smallskip

The above observations indicate that the condition that $\X$ commutes with $\Y$ or that $\X$ fixes $\Y$ under conjugation is only meaningful if one puts further assumptions on $\X$ and $\Y$. Indeed we obtain interesting results if we consider partial normal subgroups of localities. Our most important findings are summarized in the following theorem.

\begin{Th}\label{T:mainNperp}
Let $(\L,\Delta,S)$ be a locality with partial product $\Pi\colon\D\rightarrow\L$. Let $\M$ and $\N$ be partial normal subgroups of $\L$.
\begin{itemize}
\item[(a)] There is a partial normal subgroup $\N^\perp$ of $\L$, which is (with respect to inclusion) the largest partial normal subgroup of $\L$ that commutes with $\N$. We have $\N\subseteq (\N^\perp)^\perp$ and $\N^\perp\cap S\leq C_S(\N)$.
\item[(b)] The following conditions are equivalent:
\begin{itemize}
\item [(i)] $\M$ commutes with $\N$;
\item [(ii)] $\N$ commutes with $\M$;
\item[(iii)] $\M$ fixes $\N$ under conjugation.
\end{itemize}
\item [(c)] $(\M\N)^\perp=\M^\perp\cap\N^\perp$.
\item [(d)] Suppose that $\M$ and $\N$ commute, let $m\in\M$ and $n\in\N$ with $(m,n)\in\D$. Then  
\[S_m\cap S_n=S_{(m,n)}=S_{(n,m)}.\] 
If $(\L,\Delta,S)$ is a linking locality, then also $S_{mn}=S_{(m,n)}$.
\item [(e)] If $(\L,\Delta,S)$ is a linking locality, then properties (i)-(iii) from part (b) are also equivalent to the following properties:
\begin{itemize}
\item [(iv)] $\M\cap S\subseteq \N^\perp$ and $\N\cap S\subseteq C_S(\M)$.
\item [(v)] $\M\cap S\subseteq \N^\perp$ and $\N\cap S\subseteq\M^\perp$.
\end{itemize}
\item[(f)] Suppose $(\L,\Delta,S)$ is a linking locality. If $\F^q\subseteq \Delta$ or if $\delta(\F)\subseteq\Delta$ (with $\delta(\F)$ as introduced below), then \[\N^\perp\cap S=C_S(\N).\]
\end{itemize}
\end{Th}

We remark that, if $\M$ and $\N$ are partial normal subgroups of a locality $(\L,\Delta,S)$, then it is shown in \cite[Theorem~1]{Henke:2015a} that $\M\N\unlhd\L$. Hence, statement (c) of Theorem~\ref{T:mainNperp} makes sense. In the case that $(\L,\Delta,S)$ is a linking locality, a concrete description of $\N^\perp$ and of $\N^\perp\cap S$ together with a fairly general sufficient condition for $\N^\perp\cap S=C_S(\N)$ is given in Corollary~\ref{C:NperpCapS}. 

%\smallskip

%The idea to study commuting partial normal subgroups of localities and to characterize $\N^\perp$ as in Theorem~\ref{T:mainNperp}(a) appears to be new. A partial normal subgroup denoted by $\N^\perp$ was however already defined and studied by Chermak \cite[Definition~5.5]{ChermakIII} for linking localities (which  Chermak calls proper localities). We show in Corollary~\ref{C:NperpCapS} that Chermak's notion coincides with ours in this special case. Indeed, if $(\L,\Delta,S)$ is a linking locality, then Corollary~\ref{C:NperpCapS} gives a concrete description of $\N^\perp$ and of $\N^\perp\cap S$; this leads also to a fairly general sufficient condition for $\N^\perp\cap S=C_S(\N)$. 

\smallskip

With the definition of $\N^\perp$ in place, we are now able to define and study the generalized Fitting subgroup of a linking locality. Recall that, for a finite group $G$, the generalized Fitting subgroup $F^*(G)$ is the smallest normal subgroup $N$ of $G$ with $F(G)\leq N$ and $C_G(N)\leq N$. Here $F(G)$ is the Fitting subgroup, which is the product of the subgroups $O_q(G)$, where $q$ runs over all primes dividing $|G|$. Since for a linking locality $(\L,\Delta,S)$ every non-trivial partial normal subgroup of $\L$ intersects non-trivially with the $p$-group $S$ (which plays the role of a ``Sylow subgroup'' of $\L$), one should think of $O_p(\L)$ as the ``Fitting subgroup'' of $\L$. Similarly, $\N^\perp$ should be thought of as a replacement for the centralizer of a partial normal subgroup $\N$ of $\L$.  This motivates the following definition.

\begin{Defn}
Let $(\L,\Delta,S)$ be a linking locality and $\N\unlhd\L$. 
\begin{itemize}
\item Define $\N$ to be  \emph{centric radical} if $\N^\perp\subseteq\N$ and $O_p(\L)\subseteq\N$. 
\item The \emph{generalized Fitting subgroup} $F^*(\L)$ is defined as the intersection of all centric radical partial normal subgroups.
\end{itemize} 
\end{Defn}

The following theorem will be proved towards the end of Chapter~\ref{S:Nperp}. It generalizes \cite[Lemma~6.7, Corollary~6.9]{ChermakIII}.

\begin{Th}\label{T:GeneralizedFitting}
If $(\L,\Delta,S)$ is a linking locality, then the intersection of any two centric radical partial normal subgroups of $\L$ is again centric radical in $\L$. In particular, $F^*(\L)$ is a centric radical partial normal subgroup of $\L$. 
\end{Th}

We will use the generalized Fitting subgroup now to define a class of particularly nice linking localities.

\section{Regular localities} Fix a saturated fusion system $\F$ over a finite $p$-group $S$. The set $\F^s$ of subgroups of $S$ which are \emph{subcentric} in $\F$ is defined in \cite[Definition~1]{Henke:2015}. Whenever $(\L,\Delta,S)$ is a linking locality over $\F$, then $\Delta\subseteq\F^s$. Moreover, by \cite[Theorem~A]{Henke:2015}, there exists a linking locality $(\L^s,\F^s,S)$ over $\F$ which is unique up to a suitable notion of isomorphism. We call such a linking locality a \emph{subcentric locality} over $\F$. While subcentric localities have many nice properties, in some situations it is better to work with another class of linking localities which we will introduce now following Chermak's ideas \cite{ChermakIII}. The previously stated results on $\N^\perp$ and $F^*(\L)$ are important in this context. 

\smallskip

If $(\L^s,\F^s,S)$ is a subcentric locality attached to $\F$ set 
\[\delta(\F):=\{P\leq S\colon P\cap F^*(\L^s)\in\F^s\}.\]
It can easily be shown that $\delta(\F)$ does not depend on the choice of $\L^s$. Moreover,
we have
\[\delta(\F)=\{P\leq S\colon P\cap F^*(\L)\in\F^s\}\]
for any linking locality $(\L,\Delta,S)$ over $\F$ (cf. Lemma~\ref{L:deltaFDefinition}). 

\begin{Rmk}
It is shown in \cite[Theorem~E(c)]{Chermak/Henke} that $\F_{S\cap F^*(\L^s)}(F^*(\L^s))=F^*(\F)$ for every subcentric locality $(\L^s,\F^s,S)$ over $\F$. Here $F^*(\F)$ is the generalized Fitting subsystem of $\F$ which can be defined purely in terms of the fusion system $\F$. Therefore, one can indeed characterize $\delta(\F)$ without any reference to localities. Namely, if $F^*(\F)$ is a subsystem of $\F$ over $T^*\leq S$, then 
\[\delta(\F):=\{P\leq S\colon P\cap T^*\in\F^s\}.\]
Using Lemma~\ref{L:RegularConjAutomorphisms}(f), one can conclude from this that $\delta(\F)$ is the set of all subgroups of $S$ containing an element of $F^*(\F)^s$. 
\end{Rmk}

In Lemma~\ref{L:deltaFBasic} it is proved that $\delta(\F)$ is closed under taking $\F$-conjugates and overgroups in $S$ and that $\F^{cr}\subseteq\delta(\F)\subseteq\F^s$. By \cite[Theorem~A]{Henke:2015}, this implies that there exists a linking locality $(\L,\delta(\F),S)$ over $\F$ which is unique up to a suitable notion of isomorphism.

\begin{Defn}
A linking locality $(\L,\Delta,S)$ over $\F$ is called \emph{regular} if $\Delta=\delta(\F)$.
\end{Defn}

We develop a theory of regular localities in Chapters~\ref{S:Regular} and \ref{S:Components}. The reader is referred to these chapters for details, but we present a few highlights here. If $(\L,\Delta,S)$ is a locality and $\H$ is a partial subgroup of $\L$, then $\F_{S\cap\H}(\H)$ denotes the fusion system over $S\cap \H$ which is generated by the maps between subgroups of $S\cap \H$ that are induced by conjugation with an element of $\H$. The perhaps most notable property of regular localities is that partial normal and partial subnormal subgroups of regular localities are again regular localities. More precisely, for partial normal subgroups, we show the following theorem.  The reader is also referred to Theorem~\ref{T:RegularPartialNormal}, which implies Theorem~\ref{T:mainRegularPartialNormal} and gives some additional information about partial normal subgroups of regular localities.

\begin{Th}[{Chermak \cite[Theorem~C]{ChermakIII}}]\label{T:mainRegularPartialNormal}
Let $(\L,\Delta,S)$ be a regular locality, let $\N\unlhd\L$, $T:=\N\cap S$ and $\E:=\F_T(\N)$. Then $\E$ is saturated, $(\N,\delta(\E),T)$ is a regular locality over $\E$ and 
\[\delta(\E)=\{P\leq T\colon PC_S(\N)\in\Delta\}.\]
Moreover $C_\L(\N)=\N^\perp=O^p(C_\L(S\cap\N))C_S(\N)\unlhd\L$. 
\end{Th}

As a consequence, one can prove by induction on the length of a subnormal series that, for every partial subnormal subgroup $\H$ of a regular locality $(\L,\Delta,S)$, the fusion system $\F_{S\cap\H}(\H)$ is saturated and $(\H,\delta(\F_{S\cap\H}(\H)),S\cap \H)$ is a regular locality over $\F_{S\cap\H}(\H)$. This makes it possible to define and study components of regular localities. 

\begin{Defn}
\begin{itemize}
 \item A partial group $\L$ is called \emph{simple} if it has precisely two partial normal subgroups, namely $\{\One\}$ and $\L$.
 \item For any locality $(\L,\Delta,S)$, the \emph{$p$-residual} $O^p(\L)$ is defined to be the intersection of all partial normal subgroups $\N$ of $\L$ with $S\N=\L$.
 \item If $(\L,\Delta,S)$ is a linking locality, then define $\L$ to be \emph{quasisimple} if $\L=O^p(\L)$ and $\L/Z(\L)$ is simple.
 \item If $(\L,\Delta,S)$ is a regular locality, then a partial subnormal subgroup $\K$ of $\L$ is called a \emph{component} of $\L$ if $\K$ is quasisimple. The set of all components of $\L$ is denoted by $\Comp(\L)$.
\end{itemize}
\end{Defn}

We develop a theory of components of regular localities in Chapter~\ref{S:Components}. A summary of the most important properties is given in the next theorem. If $\fC=\{\K_1,\dots,\K_n\}$ is a set of $n$ distinct partial subgroups of a partial group $\L$ such that $\prod_{i=1}^n\K_i=\prod_{i=1}^n\K_{i\sigma}$ for every permutation $\sigma\in\Sigma_n$, then we put $\prod_{\K\in\fC}\K:=\prod_{i=1}^n\K_i$ and say that $\prod_{\K\in\fC}\K$ is well-defined. Set moreover $\prod_{\K\in\emptyset}\K:=\{\One\}$. A notion of internal central products of partial groups is defined in Chapter~\ref{S:CentralProduct}.

\begin{Th}\label{T:mainComponents}
Let $(\L,\Delta,S)$ be a regular locality. Then the following hold.
\begin{itemize}
 \item [(a)] If $\fC\subseteq\Comp(\L)$, then $\prod_{\K\in\fC}\K$ is well-defined and a partial normal subgroup of $F^*(\L)$; moreover $\prod_{\K\in\fC}\K$ is a central product of the elements of $\fC$. In particular, \[E(\L):=\prod_{\K\in\Comp(\L)}\K\] is well-defined and a central product of the components of $\L$.
 \item [(b)] $E(\L)\unlhd\L$ and $E(\L)\subseteq O_p(\L)^\perp$.
 \item [(c)] The generalized Fitting subgroup $F^*(\L)$ is a central product of $O_p(\L)$ and the components of $\L$. In particular, $F^*(\L)=E(\L)O_p(\L)=O_p(\L)E(\L)$.
 \item [(d)] If $\H\subn\L$ and $\K\in\Comp(\L)$ with $\K\not\subseteq\H$, then $\K\subseteq C_\L(\H)$ and $\K\H=\H\K\subn\L$ is a central product of $\K$ and $\H$. 
 \item [(e)] The following conditions are equivalent:
\begin{itemize}
 \item [(i)] $\F$ is constrained;
 \item [(ii)] $\L$ is a group of characteristic $p$;
 \item [(iii)] $F^*(\L)=O_p(\L)$;
 \item [(iv)] $\Comp(\L)=\emptyset$. 
\end{itemize}
\end{itemize}
\end{Th}

The partial normal subgroup $E(\L)$ defined above is called the \emph{layer} of $\L$.

\section{E-balance} An important ingredient in the proof of the classification of finite simple groups is the L-balance Theorem of Gorenstein and Walter \cite{Gorenstein/Walter}. Aschbacher \cite[Theorem~7]{Aschbacher:2011} proved a version of this theorem for fusion systems. More precisely, he showed  that, if $\F$ is a saturated fusion system, then  $E(N_\F(X))\subseteq E(\F)$ for every fully $\F$-normalized subgroup $X\leq S$. 

\smallskip
 
It is not quite so easy to find the right setting to formulate an E-balance theorem for linking localities. As explained before, regular localities have the advantage that there is a nice theory of partial subnormal subgroups and components. However, it is difficult to consider normalizers of $p$-subgroups in regular localities. More precisely, if we work with a regular locality $(\L_\delta,\delta(\F),S)$ over $\F$ and $X\leq S$ is fully $\F$-normalized, then we do not see any canonical way to construct a regular locality (or indeed any linking locality) over $N_\F(X)$ which lies inside of $\L_\delta$. Therefore, it seems  necessary to work with a  subcentric locality $(\L,\Delta,S)$ over $\F$. Given a fully $\F$-normalized subgroup $X\leq S$, there is then a construction of a partial group $\bN_\L(X)$ contained in $N_\L(X)$ such that $(\bN_\L(X),N_\F(X)^s,N_S(X))$ is a subcentric locality over $N_\F(X)$ (cf. Section~\ref{SS:Impartial}). Unfortunately, subcentric localities have the disadvantage that partial normal and partial subnormal subgroups are not necessarily subcentric localities (or linking localities of any kind). However, we think that the typical arguments used in the classification of finite simple groups can be implemented in the locality setting if one works simultaneously with a subcentric locality and a regular locality over the same fusion system $\F$. For E-balance this is demonstrated below. The approach will be taken further in a forthcoming paper with Valentina Grazian. 

\smallskip

If $(\L,\Delta,S)$ is a subcentric locality over a fusion system $\F$, there is always a natural way to find a regular locality $\L_\delta=\L|_{\delta(\F)}$ over $\F$ inside of $(\L,\Delta,S)$ (cf. Section~\ref{SS:Restrictions}). Moreover, by \cite[Theorem~A2]{ChermakII} or by \cite[Theorem~C(a)]{Henke:2020}, there is a unique partial normal subgroup $E(\L)$ of $\L$ with
\[E(\L)\cap\L_\delta=E(\L_\delta).\]
Again, we call $E(\L)$ the \emph{layer} of $\L$. With the definition of the layer of a subcentric locality in place, it is possible to formulate and prove the following theorem. 

\begin{Th}[E-balance, Chermak]\label{T:EbalanceMain}
Let $(\L,\Delta,S)$ be a subcentric locality over $\F$ and let $X\leq S$ be fully $\F$-normalized. Then 
\[E(\bN_\L(X))\subseteq E(\L).\] 
\end{Th}

Given the technical subtleties explained above, the reader might wonder why we propose to work with localities rather than with fusion systems to generalize arguments from the proof of the classification of finite simple groups. One reason is that the arguments in the fusion system setting often use the existence of models for constrained normalizers of $p$-subgroups. Since all these models can be seen in the subcentric locality, many arguments should become more transparent in the locality setting. The second and more serious reason is that there are difficulties encountered in Aschbacher's program because, in the fusion system setting, centralizers of certain components of involution centralizers are not defined. Therefore, it seems impossible to prove a fusion system version of Aschbacher's component theorem \cite{Aschbacher1975} in full generality (cf. \cite{AschbacherFSCT}). Perhaps the existence of appropriate centralizers can be shown inside of a subcentric locality. At least, as a starting point, centralizers in $S$ of local components are defined.

\section{Overview and comparison to Chermak's approach}\label{SS:ChermaksApproach}

The principal line of thought in the present paper is the same as in Chermak's preprint \cite{ChermakIII}. Namely, we first  define and study $\N^\perp$, and building on that we introduce $F^*(\L)$; then we use these notions to define regular localities and prove the results stated above. Compared to Chermak's paper, we study $\N^\perp$ however in much more detail and generality. This includes introducing the concepts defined in Definition~\ref{D:mainCommute}. We prove moreover many results about commuting partial normal subgroups and $\N^\perp$  for arbitrary localities (cf. Theorem~\ref{T:mainNperp}(a)-(d) and Section~\ref{SS:CommutingPartialNormal}). In contrast, Chermak introduces $\N^\perp$ only for linking localities. While Chermak proves in this special case that $\N^\perp$ commutes with $\N$ and some additional properties hold (cf. \cite[Theorem~5.7(c)]{ChermakIII}), we are able to give some stronger information in Theorem~\ref{T:mainNperp}(a),(b),(d). 

\smallskip

In Corollary~\ref{C:NperpCapS} we show that, in the case of linking localities, our definition of $\N^\perp$ is equivalent to Chermak's definition. The equality that we thereby show is actually important for many of the arguments used subsequently. The surrounding results in Chapter~\ref{S:Nperp} are very much inspired by Chermak's approach. However, in many cases we are able to give more general formulations of lemmas and theorems that Chermak proves only under special assumptions on the object sets of linking localities. In this context, both in Chermak's work and in ours some technical results are needed that make it possible to decompose elements of a locality $(\L,\Delta,S)$ as products of elements which lie in normalizers of certain elements of $\Delta$. While Chermak \cite[Theorem~3.12]{ChermakIII} shows a somewhat specialized decomposition result, we build on Alperin's Fusion Theorem for localities as it was proved by Molinier \cite{Molinier:2016}. In the form of Lemma~\ref{L:PartialNormalAlperin}, this enables us to decompose elements of a fixed partial normal subgroup in a certain way. Thereby we are naturally led to the concept of $\N$-radical elements of $\Delta$, which we study in Chapter~\ref{S:NRadical}. When we apply the results on $\N$-radical subgroups in later chapters, in many cases it is (at least a priori) necessary to make some assumption on the object set of a linking locality $(\L,\Delta,S)$. This motivates us to define $\N$-replete localities in Definition~\ref{D:NReplete} as linking localities with object sets that are in a certain sense large enough. We prove some sufficient conditions for a linking locality to be $\N$-replete (cf. Lemma~\ref{L:NRadKRadIntersect} and Corollary~\ref{C:GetNreplete}). In particular, it turns out that every subcentric locality $(\L,\Delta,S)$ is $\N$-replete for every $\N\unlhd\L$. This enables us to prove many results about arbitrary linking localities by reducing to the $\N$-replete case. The principal strategy to make reductions to linking localities with particularly nice object sets goes back to Chermak. It should be mentioned that there are also some cases (for example in the proof of Lemma~\ref{L:NNperpCommute}) where we are able to prove surprisingly strong results about arbitrary linking localities by direct arguments.

\smallskip

To point out a further technical detail, we work in this text with the concept of an (internal) central products of partial groups, which is not mentioned in Chermak's paper. We introduce such central products in Chapter~\ref{S:CentralProduct}. If $\M$ and $\N$ are partial normal subgroups of a partial group $\L$, then the following implications hold by Lemma~\ref{L:CentralProductsFactorsCentralize} and Lemma~\ref{L:CentralizerHelp}:
\begin{eqnarray*}
\mbox{$\M\N$ is a central product of $\M$ and $\N$}\Longrightarrow \M\subseteq C_\L(\N)
\Longrightarrow  \mbox{$\M$ commutes with $\N$}
\end{eqnarray*}
If $(\L,\Delta,S)$ is a \emph{regular locality}, then Lemma~\ref{L:PerpendicularPartialNormalCentralProduct} shows that the above implications are actually equivalences. 

\smallskip

We use the concept of central products in particular to formulate our results about components (cf. Theorem~\ref{T:mainComponents}(a),(c),(d) and Theorem~\ref{T:FstarCentralProduct}(a)-(d)). Parts (a)-(c) of Theorem~\ref{T:mainComponents} have some resemblance to \cite[Theorem~8.5]{ChermakIII}. However, since Chermak does not use the notion of a central product, our results are slightly stronger. Admittedly, we pay the price of having to introduce the technical language of central products. The reader who has no interest in this subtle point could skip Chapter~\ref{S:CentralProduct} and replace the statement that a given collection of partial subgroups $\N_1,\dots,\N_k$ forms a central product by the statement that $\N_i$ commutes with $\N_j$ for all $1\leq i,j\leq k$ with $i\neq j$. 

\smallskip

The E-balance theorem goes also back to Chermak \cite[Theorem~9.9]{ChermakIII} who states it however in a slightly weaker form. Our proof mainly follows his, even though we formulate and prove the necessary preliminary results differently (cf. Lemma~\ref{L:PartialNormalinbNLX} and Lemma~\ref{L:ProdCompinbNLD}). As Chermak communicated to us privately, he observed in the meantime independently that his proof generalizes to show the statement in Theorem~\ref{T:EbalanceMain}.

\section{Background references} We seek to keep the background references we rely on to a minimum. For fusion systems we build mainly on Sections~1-7 of \cite[Part~I]{Aschbacher/Kessar/Oliver:2011}. The proof of \cite[Theorem~C]{Henke:2020}, which we cite below, uses moreover the construction of factor systems as introduced in \cite[Section~II.5]{Aschbacher/Kessar/Oliver:2011}.

\smallskip

For the theory of localities we rely on Chermak's fundamental preprint  \cite{Chermak:2015} (with earlier versions of many results published in \cite{Chermak:2013}). For many properties of linking localities we build on \cite{Henke:2015}. In particular, the existence and uniqueness of centric linking systems \cite{Chermak:2013,Oliver:2013,Glauberman/Lynd} is used through \cite[Theorem~A]{Henke:2015} to define the set $\delta(\F)$ and to show the existence and uniqueness of regular localities (cf. Remark~\ref{R:ExistenceUniquenessCLS}). Our proofs use moreover Theorem~C, Theorem~5.1 and some elementary lemmas from \cite{Henke:2020}, as well as the main results from \cite{Henke:2015a} and \cite{Molinier:2016}. 

\smallskip

It should be mentioned that \cite[Theorem~A]{Henke:2015} is essentially the same as \cite[Theorem~A1]{ChermakII} and that the first part of \cite[Theorem~C]{Henke:2020} revisits \cite[Theorem~A2]{ChermakII}. We have chosen not to rely on any results from Chermak's preprints \cite{ChermakII,ChermakIII} for our proofs. Therefore, in Chapter~\ref{S:Residues}, we revisit definitions and results from the final section of \cite{ChermakII}.

\subsection*{Acknowledgement} The author thanks Andrew Chermak for many interesting discussions. This paper never would have been written without his ideas and deep insights. As explained before in more detail, many results proved in the present paper are similar or identical to the results stated in \cite{ChermakIII}.

\numberwithin{section}{chapter}

\chapter{Some group-theoretic lemmas}\label{S:GroupTheory}

Recall that a finite group $G$ is of characteristic $p$ if $C_G(O_p(G))\leq O_p(G)$. We will frequently use the following properties.

\begin{lemma}\label{L:MSCharp}
If $G$ is a group of characteristic $p$, then the following hold:
\begin{itemize}
 \item [(a)] Every normal subgroup of $G$ is of characteristic $p$.
 \item [(b)] For every $p$-subgroup $P$ of $G$, the normalizer $N_G(P)$ is of characteristic $p$. 
\end{itemize}
\end{lemma}

\begin{proof}
See \cite[Lemma~1.2(a),(c)]{MS:2012b}.
\end{proof}

\begin{lemma}\label{L:CGUleqU}
Let $G$ be a finite group of characteristic $p$ and let $U$ be a normal $p$-subgroup of $G$ with $C_{O_p(G)}(U)\leq U$. Then $C_G(U)\leq U$. 
\end{lemma}

\begin{proof}
As $[O_p(G),C_G(U)]\leq C_{O_p(G)}(U)\leq U$, we have 
\[[O_p(G),O^p(C_G(U))]=[O_p(G),O^p(C_G(U)),O^p(C_G(U))]=1.\]
Since $C_G(O_p(G))\leq O_p(G)$ does not contain any $p^\prime$-elements, it follows that $O^p(C_G(U))=1$. This means that $C_G(U)$ is a normal $p$-subgroup of $G$, i.e. $C_G(U)\leq C_{O_p(G)}(U)\leq U$. 
\end{proof}

\begin{lemma}\label{L:AutGroup}
Let $G$ be a finite group with a normal $p$-subgroup $U$. Let $\alpha\in\Aut(G)$ be a $p^\prime$-automorphism with $[G,\alpha]\leq U$ and $[U,\alpha]=1$. Then $\alpha=\id_G$.
\end{lemma}

\begin{proof}
Take the semidirect product $H:=\<\alpha\>G$ and consider $G$ as a subgroup of $H$. As $[G,\alpha]\leq U\unlhd G$, the group $G$ acts on $R:=\<\alpha\>U$. Since $[U,\alpha]=1$, we have $R=\<\alpha\>\times U$ and $\<\alpha\>=O_{p^\prime}(R)$ is normalized by $G$. Hence, $[G,\alpha]\leq U\cap\<\alpha\>=1$ and so $\alpha=\id_G$. 
\end{proof}

If $H$ is a group, then we denote by $H^\prime$ the commutator group $[H,H]$ of $H$.

\begin{lemma}\label{L:CGNCGQ}
Let $G$ be a finite group and $N\unlhd G$. Suppose $U$ is a normal $p$-subgroup of $G$ such that $U\leq N$ and $C_N(U)\leq U$. Then \[O^p(C_G(N))=O^p(C_G(U))\mbox{ and }C_G(U)^\prime\leq C_G(N).\]  
\end{lemma}

\begin{proof}
As $U\leq N$, we have $C_G(N)\leq C_G(U)$ and thus $O^p(C_G(N))\leq O^p(C_G(U))$. So it remains to show that $O^p(C_G(U))\leq C_G(N)$ and $C_G(U)'=[C_G(U),C_G(U)]\leq C_G(N)$. Our assumptions yield 
\[[N,C_G(U)]\leq C_N(U) \leq U.\]
Now Lemma~\ref{L:AutGroup} applied with $N$ in place of $G$ gives that every $p^\prime$-element of $C_G(U)$ acts trivially on $N$ by conjugation. So $O^p(C_G(U))\leq C_G(N)$. Moreover, $[N,C_G(U),C_G(U)]=1=[C_G(U),N,C_G(U)]$. So the Three-Subgroup-Lemma implies $[C_G(U),C_G(U),N]=1$. 
\end{proof}

\begin{lemma}\label{L:CharpCGNCGQ}
Let $G$ be a group of characteristic $p$, $N\unlhd G$ and $U:=O_p(N)$. Then 
\[O^p(C_G(N))=O^p(C_G(U))\mbox{ and }C_G(U)^\prime\leq C_G(N).\] 
\end{lemma}

\begin{proof}
Notice that $U\unlhd G$ as $U$ is characteristic in $N$. Since $G$ is of characteristic $p$, Lemma~\ref{L:MSCharp}(a) gives that $N$ is of characteristic $p$, which means $C_N(U)\leq U$. Hence the assertion follows from Lemma~\ref{L:CGNCGQ}.
\end{proof}

\begin{lemma}\label{L:GetintoOp}
Let $G$ be a finite group and $P$ a $p$-subgroup of $G$. Then the following properties hold:
\begin{itemize}
 \item [(a)] If $P$ is subnormal in $G$, then $P\leq O_p(G)$.
 \item [(b)]  If $P$ is normalized by $O^p(G)$, then $P\subn G$ and $P\leq O_p(G)$.
 \item [(c)] We have $P\leq O_p(G)$ if $[P,O^p(G)]$ is a $p$-group.
\end{itemize}
\end{lemma}

\begin{proof}
Part (a) follows by induction on the length of a subnormal series of $P$ in $G$.

\smallskip

Assume now that $P$ is normalized by $O^p(G)$. Then $P\unlhd O^p(G)P$. As $G/O^p(G)$ is a $p$-group, $O^p(G)P/O^p(G)$ is subnormal in $G/O^p(G)$ and thus $O^p(G)P\subn G$. Hence, $P\subn G$ and (b) follows from (a). 

\smallskip

Let now $P$ be an arbitrary $p$-subgroup of $G$ and assume that $Q:=[P,O^p(G)]$ is a $p$-group. By \cite[1.5.5]{KS}, $Q$ is normalized by $O^p(G)$ and so (b) gives $Q\leq O_p(G)$. This implies that $O^p(G)$ normalizes $O_p(G)P$. Using (b) again we get $P\leq O_p(G)P\leq O_p(G)$.  
\end{proof}

\chapter{Some lemmas on fusion systems}

In this chapter we will show some elementary lemmas on fusion systems which are needed in the proofs of our main theorems. We will use the results on fusion systems described in Sections~1-7 of \cite[Part~I]{Aschbacher/Kessar/Oliver:2011}. We will also adapt the notation and terminology from there, but with the following two caveats: Firstly, we will write homomorphisms on the right hand side of the argument and adapt the notation accordingly (similarly as in Part~II of \cite{Aschbacher/Kessar/Oliver:2011}). Secondly, following Chermak~\cite{ChermakII}, we define radical and centric radical subgroups differently, namely as in  Definition~\ref{D:Fcr} below.

\smallskip

\textbf{Throughout this chapter let $\F$ be a fusion system over a finite $p$-group $S$.}

\section{Centric radical subgroups} 

\begin{definition}[{cf. \cite[Definition~1.8]{ChermakII}}]\label{D:Fcr}~
\begin{itemize}
\item A subgroup $P\leq S$ of a fusion system $\F$ over $S$ is called \emph{$\F$-radical} if there exists a fully $\F$-normalized $\F$-conjugate $Q$ of $P$ such that $O_p(N_\F(Q))=Q$. 
\item We write $\F^{cr}$ for the set of all subgroups of $S$ which are $\F$-centric in the usual sense (cf. \cite[Definition~I.3.1]{Aschbacher/Kessar/Oliver:2011}) and $\F$-radical in the sense defined above.
\end{itemize}
\end{definition}

\begin{lemma}\label{L:FcrAKOinFcr}
Let $P\in\F^c$. If $O_p(\Aut_\F(P))=\Inn(P)$, then $P\in\F^{cr}$. If $\F$ is saturated, then the converse holds.
\end{lemma}

\begin{proof}
Fix $P\in\F^c$. Notice that we have $O_p(\Aut_\F(P))=\Inn(P)$ if and only if the corresponding property holds for an $\F$-conjugate of $P$. 

\smallskip

Assume first that $O_p(\Aut_\F(P))=\Inn(P)$ and pick $Q\in P^\F$ such that $Q$ is fully normalized. Setting $Q^*:=O_p(N_\F(Q))$, we have then $\Aut_{Q^*}(Q)\leq O_p(\Aut_\F(Q))=\Inn(Q)$. As $C_S(Q)\leq Q$, this implies $Q^*=Q$. Thus, $P\in\F^{cr}$.

\smallskip

Suppose now that $\F$ is saturated and $P\in\F^{cr}$. Let $R$ be a fully $\F$-normalized $\F$-conjugate of $P$ with $O_p(N_\F(R))=R$. Note that $N_\F(R)$ is saturated. So as $R\in\F^c$, the subsystem $N_\F(R)$ is constrained. Thus, by \cite[Theorem~III.5.10(a)]{Aschbacher/Kessar/Oliver:2011}, we may choose a model $G$ for $N_\F(R)$ with  $O_p(G)=R=O_p(N_\F(R))$. Then  
\[\Aut_\F(R)\cong G/C_G(R)\cong G/Z(R).\]
Hence, $O_p(\Aut_\F(R))\cong O_p(G/Z(R))=R/Z(R)\cong\Inn(R)$, which implies $O_p(\Aut_\F(R))= \Inn(R)$ and then $O_p(\Aut_\F(P))=\Inn(P)$.
\end{proof}

If $\F$ is saturated, then Lemma~\ref{L:FcrAKOinFcr} says that the set $\F^{cr}$ coincides with the set $\F^{cr}$ as it is defined in \cite[Definition~I.3.1]{Aschbacher/Kessar/Oliver:2011}). We will use this fact from now on without further reference.

\begin{lemma}\label{L:OpFinFRadical}
Let $P\leq S$ such that $P$ is $\F$-radical. Then $O_p(\F)\leq P$.
\end{lemma}

\begin{proof}
 Set $R:=O_p(\F)$ and fix $Q\in P^\F$ such that $O_p(N_\F(Q))=Q$. Then $N_R(Q)\leq O_p(N_\F(Q))=Q$ and thus $N_{RQ}(Q)=N_R(Q)Q=Q$. As $QR$ is a $p$-group, this implies $Q=QR\geq R$. Using that $R$ is normal in $\F$, one sees now that $R\leq P$. 
\end{proof}

\section{The set $\F_T^c$} \label{SS:FTc}

The following (non-standard) notation together with Lemma~\ref{L:FTcFRc} below will be very convenient to use in Chapters~\ref{S:NRadical} and \ref{S:NReplete}.

\begin{notation}\label{N:FRc}
Given a subgroup $R\leq S$, we write $\F^c_R$ for the set of all subgroups $U\leq R$ such that $C_R(U^*)\leq U^*$ for every $\F$-conjugate $U^*$ of $U$ with $U^*\leq R$.  
\end{notation}

If $R$ is strongly closed, then every $\F$-conjugate of a subgroup of $R$ lies in $R$. So $\F_R^c$ is in this case the set of all $U\leq R$ such that $C_R(U^*)\leq U^*$ for all $U^*\in U^\F$.

\begin{lemma}\label{L:FTcFRc}
Let $\F$ be saturated and suppose $T$ is a subgroup of $S$ which is strongly closed in $\F$. For parts (b),(c),(d) let $R\leq C_S(T)$ be strongly closed in $N_\F(T)$.
\begin{itemize}
 \item [(a)] If $U\leq T$ such that there exists a fully $\F$-normalized $\F$-conjugate $U^*$ of $U$ with $C_T(U^*)\leq U^*$, then $U\in\F_T^c$.
 \item [(b)] Let $V\leq R$ and suppose $V^*$ is a fully $\F$-normalized $\F$-conjugate of $V$. Then $V^*\leq R$. If $C_R(V^*)\leq V^*$, then $V\in \F_R^c$.
\item [(c)] We have $N_\F(T)_R^c=\F_R^c$.
\item [(d)] Let $T\leq P\in\F_{TR}^c$. Then $P\cap R\in\F_R^c$.
\end{itemize}
\end{lemma}

\begin{proof}
Notice that (a) follows from (b) applied with $(1,T)$ in place of $(T,R)$. Let now $R$, $V$ and $V^*$ be as in (b). By \cite[Lemma~2.6(c)]{Aschbacher/Kessar/Oliver:2011}, there exists $\alpha\in\Hom_\F(N_S(V),S)$ with $V\alpha=V^*$. As $V\leq R\leq C_S(T)$, we have $T\leq C_S(V)\leq N_S(V)$. Since $T$ is strongly closed, it follows that $\alpha$ is a morphism in $N_\F(T)$. As $R$ is strongly closed in $N_\F(T)$, we have thus $V^*=V\alpha\leq R$. Furthermore, if $C_R(V^*)\leq V^*$, then $C_R(V)\alpha\leq C_R(V^*)=Z(V^*)$ and thus $C_R(V)=Z(V)$. The latter property remains true if we replace $V$ by any $\F$-conjugate $\hat{V}$ of $V$ with $\hat{V}\leq R$. Hence (b) holds.

\smallskip

Assume now that $V\in N_\F(T)_R^c$. Since have seen that $V^*$ is conjugate to $V$ under $N_\F(T)$, we have then $C_R(V^*)\leq V^*$. Hence $V\in \F_R^c$ by (b). This proves $N_\F(T)_R^c\subseteq\F_R^c$. Clearly $\F_R^c\subseteq N_\F(T)_R^c$, so (c) holds. 

\smallskip

Let now $P$ be as in (d) and set $Q:=P\cap R$. By (c) it is sufficient to show that $Q\in N_\F(T)_R^c$. Note that $P=TQ$ by a Dedekind argument. So if $\phi\in\Hom_{N_\F(T)}(Q,R)$, then $\phi$ extends to $\hat{\phi}\in\Hom_\F(P,S)$ and we have $P\hat{\phi}=T(Q\phi)$. As $R\leq C_S(T)$ and $P\in\F_{TR}^c$, it follows $C_R(Q\phi)\leq C_{TR}(P\hat{\phi})\leq P\hat{\phi}$ and hence  $C_R(Q\phi)\leq R\cap P\hat{\phi}=(R\cap T)(Q\phi)$. Observe that $R\cap T\leq R\cap Z(TR)\leq R\cap C_{TR}(P)\leq R\cap P=Q$. Since $R$ and $T$ are both strongly closed in $N_\F(T)$, it follows $R\cap T=(R\cap T)\hat{\phi}\leq Q\phi$ and $C_R(Q\phi)\leq Q\phi$. This shows $Q\in N_\F(T)_R^c$ as required.
\end{proof}

\section{Properties of the focal and the hyperfocal subgroup}

We will use the definition of the focal subgroup also for fusion systems which are not necessarily saturated. Thus, if $\F$ is any fusion system over a finite $p$-group $S$, then we set
\[\foc(\F):=\<x^{-1}(x\alpha)\colon x\in P\leq S\mbox{ and }\alpha\in\Hom_\F(P,S)\>.\]

\begin{lemma}\label{L:focFinTStrCl}
Let $\foc(\F)\leq T\leq S$. Then $T$ is strongly closed in $\F$.
\end{lemma}

\begin{proof}
If $x\in T$ and $\alpha\in\Hom_\F(\<x\>,S)$, then by definition of the focal subgroup, we have $y:=x^{-1}(x\alpha)\in\foc(\F)\leq T$ and thus $x\alpha=xy\in T$. 
\end{proof}

\begin{lemma}\label{L:TFq}
Suppose $\F$ is saturated and $T\leq S$ is weakly closed in $\F$. Assume that $\hyp(C_\F(T))\leq T$. Then $T\in\F^q$.
\end{lemma}

\begin{proof}
Since $\F$ is saturated and $T\unlhd S$, the centralizer $C_\F(T)$ is saturated. Moreover, since $T^\F=\{T\}$, we only need to show that $C_\F(T)$ is the fusion system of a $p$-group. Pick $P\leq C_\F(T)^c$. By Alperin's fusion theorem (cf. \cite[Theorem~I.3.5]{Aschbacher/Kessar/Oliver:2011}), it is sufficient to show that $\Aut_{C_\F(T)}(P)$ is a $p$-group. If $\phi\in\Aut_{C_\F(T)}(P)$ is a $p^\prime$-element, then $[P,\phi]\leq \hyp(C_\F(T))\leq T$. As $\phi$ is a morphism in $C_\F(T)$, it follows $[P,\phi]=[P,\phi,\phi]=1$. Hence $\Aut_{C_\F(T)}(P)$ is a $p$-group as required.
\end{proof}

\section{Morphisms of fusion systems}\label{SS:FusionMorphisms}

Throughout this section let $\F$ and $\tF$ be fusion systems over $S$ and $\tS$ respectively. 

\begin{definition}
A group homomorphism $\alpha\colon S\longrightarrow \tS$ is said to \textit{induce a morphism} from $\F$ to $\tF$ if, for each $\phi\in\Hom_\F(P,Q)$, there exists $\psi\in\Hom_{\tF}(P\alpha,Q\alpha)$ such that $(\alpha|_P)\psi=\phi(\alpha|_Q)$. 
\end{definition}

If $\alpha$ induces a morphism from $\F$ to $\tF$, then for any $\phi\in\Hom_\F(P,Q)$, the map $\psi\in\Hom_{\tF}(P\alpha,Q\alpha)$ as in the above definition is uniquely determined. So if $\alpha$ induces a morphism from $\F$ to $\tF$, then $\alpha$ induces a map 
\[\alpha_{P,Q}\colon\Hom_\F(P,Q)\longrightarrow \Hom_{\tF}(P\alpha,Q\alpha).\] 
Together with the map $P\mapsto P\alpha$ from the set of objects of $\F$ to the set of objects of $\tF$ this gives a functor from $\F$ to $\tF$. Moreover, $\alpha$ together with the maps $\alpha_{P,Q}$ ($P,Q\leq S$) is a morphism of fusion systems in the sense of \cite[Definition~II.2.2]{Aschbacher/Kessar/Oliver:2011}.

\begin{definition}
Suppose $\alpha\colon S\longrightarrow \tS$ induces a morphism from $\F$ to $\tF$. Then $\alpha$ is said to \emph{induce an epimorphism} from $\F$ to $\tF$ if $\alpha$ is surjective as a map $S\longrightarrow \tS$ and, for all $P,Q\leq S$ with $\ker(\alpha)\leq P\cap Q$, the map $\alpha_{P,Q}$ is surjective. If $\alpha$ is in addition injective, then we say that $\alpha$ \emph{induces an isomorphism} from $\F$ to $\tF$. Write $\Aut(\F)$ for the set of automorphisms of $S$ which induce an automorphism of $\F$, i.e. an isomorphism from $\F$ to $\F$. 
\end{definition}

If $\alpha\colon S\longrightarrow \tS$ is an isomorphism of groups, then $\alpha$ induces an isomorphism from $\F$ to $\tF$ if and only if, for all $P,Q\leq S$ and every group homomorphism $\phi\colon P\longrightarrow Q$, $\phi\in\Hom_\F(P,Q)$ if and only if $\alpha^{-1}\phi\alpha\in\Hom_{\tF}(P\alpha,Q\alpha)$.

\section{Products of fusion systems which centralize each other} 

In this section we study certain products of fusion systems. Our results are tailored for use in Theorem~\ref{T:RegularN1timesN2}, which is one of the major steps towards the proof of Theorem~\ref{T:mainRegularPartialNormal}. We use the following notation.

\begin{notation} 
Suppose $\F_1$ and $\F_2$ are fusion systems over subgroups $S_1$ and $S_2$ of $S$ respectively such that $[S_1,S_2]=1$ and $S_1\cap S_2\leq Z(\F_i)$ for $i=1,2$. Then we use the following notation.
\begin{itemize}
\item Given $P_i,Q_i\leq S_i$ and $\phi_i\in\Hom_{\F_i}(P_i,Q_i)$ for $i=1,2$, we write $\phi_1*\phi_2$ for the map $P_1P_2\rightarrow Q_1Q_2$ sending $x_1x_2$ to $(x_1\phi_1)(x_2\phi_2)$ for all $x_1\in P_1$ and $x_2\in P_2$.
\item We write $\F_1*\F_2$ for the fusion system over $S_1S_2$ which is generated by the maps $\phi_1*\phi_2$ with $P_i\leq S_i$ and $\phi_i\in\Hom_{\F_i}(P_i,S_i)$ for $i=1,2$.
\item If $A\leq S_1S_2$ we set
\[A_i:=\{x_i\in S_i\colon \exists x_{3-i}\in S_{3-i}\mbox{ such that }x_1x_2\in A\}.\]
\end{itemize}   
\end{notation}

Under the hypothesis above, it can be shown that the fusion system $\F_1*\F_2$ is isomorphic to a certain quotient of the direct product $\F_1\times\F_2$ and thus isomorphic to a central product of $\F_1$ and $\F_2$ in the sense of Aschbacher \cite[Definition~2.8]{Aschbacher:2011}. In a special case this is shown in \cite[Proposition~3.3]{Henke:2018} and the arguments can be generalized. We will use this connection however only in the proof of Lemma~\ref{L:F1starF2Main}(g) below. 

\smallskip  

In the next lemma we use the following convention: If $\phi\colon A\rightarrow B$ and $\psi\colon B^*\rightarrow C$ are maps, then $\phi\psi$ denotes the map from $\phi^{-1}(B\cap B^*)$ to $C$ given by $x\mapsto (x\phi)\psi$. Moreover, if $\phi\colon A\rightarrow B$ is injective, then $\phi^{-1}$ denotes the inverse of the bijection $A\rightarrow A\phi$ induced by $\phi$. 

\begin{lemma}\label{L:F1starF2Help}
For $i=1,2$ let $\F_i$ be a fusion system over $S_i\leq S$ such that $S_1\cap S_2\leq Z(\F_i)$ and $[S_1,S_2]=1$. Let $P_i\leq S_i$ and $\phi_i\in\Hom_{\F_i}(P_i,S_i)$. Then the following hold:
\begin{itemize}
\item [(a)] The map $\phi_1*\phi_2$ is well-defined and an injective group homomorphism. In particular, $\F_1*\F_2$ is well-defined. 
\item [(b)] Setting $\hat{P}_i=P_i(S_1\cap S_2)$, the morphism $\phi_i$ extends to a morphism $\hat{\phi}_i\in\Hom_{\F_i}(\hat{P}_i,S_i)$ for $i=1,2$, and $\phi_1*\phi_2$ extends to $\hat{\phi}_1*\hat{\phi}_2$.
\item [(c)] We have $(\phi_1*\phi_2)^{-1}=\phi_1^{-1}*\phi_2^{-1}$. 
\item [(d)] Suppose $S_1\cap S_2\leq P_i$ for $i=1,2$. If $x_i\in S_i$ for $i=1,2$ with $x_1x_2\in P_1P_2$, then $x_i\in P_i$. In particular, if $A\leq P_1P_2$, then $A_i\leq P_i$ for each $i=1,2$.  
\item [(e)] Let $S_1\cap S_2\leq Q_i\leq S_i$, $R_i:=\phi_i^{-1}(Q_i)$ and $\psi_i\in\Hom_{\F_i}(Q_i,S_i)$ for each $i=1,2$. Set $\phi=\phi_1*\phi_2$ and $\psi=\psi_1*\psi_2$. Then $\phi_i\psi_i$ is a map $R_i\rightarrow S_i$. Moreover, $\phi^{-1}(Q_1Q_2)=R_1R_2$ and $\phi\psi=(\phi_1\psi_1)*(\phi_2\psi_2)$. If in addition $S_1\cap S_2\leq P_i$ for $i=1,2$, then $S_1\cap S_2\leq R_i$ for $i=1,2$.
\item [(f)] If $A\leq S_1S_2$ and $\phi\in\Hom_{\F_1*\F_2}(A,S_1S_2)$, then there exist $\phi_i\in\Hom_{\F_i}(A_i,S_i)$ for $i=1,2$ such that $\phi=(\phi_1*\phi_2)|_A$.
\end{itemize}
\end{lemma}

\begin{proof}
For $i=1,2$ let $x_i,y_i\in P_i$. Then we have the following equivalence:
\begin{eqnarray*}
 x_1x_2=y_1y_2&\Longleftrightarrow & y_1^{-1}x_1=y_2x_2^{-1}\\
&\Longleftrightarrow & (y_1^{-1}x_1)\phi_1=(y_2x_2^{-1})\phi_2\mbox{ (since $S_1\cap S_2\leq Z(\F_i)$ for $i=1,2$)}\\
&\Longleftrightarrow & (y_1\phi_1)^{-1}(x_1\phi_1)=(y_2\phi_2)(x_2\phi_2)^{-1}\\
&\Longleftrightarrow & (x_1\phi_1)(x_2\phi_2)=(y_1\phi_1)(y_2\phi_2).
\end{eqnarray*}
This shows that $\phi_1*\phi_2$ is well-defined and injective. Using $[S_1,S_2]=1$ and the fact that $\phi_1$ and $\phi_2$ are group homomorphisms, it is easy to see that $\phi_1*\phi_2$ is a group homomorphism. So (a) holds. 

\smallskip

The existence of $\hat{\phi}_i$ in (b) follows from $S_1\cap S_2\leq Z(\F_i)$ for $i=1,2$. By definition, $\phi_1*\phi_2$ and $\hat{\phi}_1*\hat{\phi}_2$ agree on $P_1P_2$, so (b) holds. Part (c) is immediate.

\smallskip

For the proof of (d) assume $S_1\cap S_2\leq P_i$ and $x_i\in S_i$ for $i=1,2$. Suppose furthermore $x_1x_2\in P_1P_2$. If $x_1x_2=y_1y_2$ where $y_i\in P_i$ for $i=1,2$, then $y_1^{-1}x_1=y_2x_2^{-1}\in S_1\cap S_2\leq P_i$ and thus $x_i\in P_i$ for $i=1,2$. This shows (d). 

\smallskip

Let now $Q_i$, $R_i$, $\psi_i$, $\phi$ and $\psi$ be as in (e). Let $x_1\in P_1$ and $x_2\in P_2$. Then $x_1x_2\in\phi^{-1}(Q_1Q_2)$ if and only if $(x_1\phi_1)(x_2\phi_2)=(x_1x_2)\phi\in Q_1Q_2$. Since $S_1\cap S_2\leq Q_i$ for each $i=1,2$, this is by (d) the case if and only if $x_i\phi_i\in Q_i$ and thus $x_i\in R_i$ for each $i=1,2$.  Hence, $\phi^{-1}(Q_1Q_1)=R_1R_2$ and it is easy to see that $\phi\psi=(\phi_1\psi_1)*(\phi_2\psi_2)$. If $S_1\cap S_2\leq P_i$ for each $i=1,2$, then $S_1\cap S_2\leq Z(\F_i)$ yields $S_1\cap S_2\leq R_i$ for each $i=1,2$. So (e) holds.

\smallskip

Let now $A\leq S_1S_2$ and $\phi\in\Hom_{\F_1*\F_2}(A,S_1S_2)$. It follows from parts (b),(c),(e) and from the definition of $\F_1*\F_2$ that there exist $P_i\leq S_i$ with $S_1\cap S_2\leq P_i$ and $\psi_i\in \Hom_{\F_i}(P_i,S_i)$ such that $A\leq P_1P_2$ and $\phi=(\psi_1*\psi_2)|_A$. By (d), we have $A_i\leq P_i$ for $i=1,2$. So $\phi_i:=\psi_i|_{A_i}\in\Hom_{\F_i}(A_i,S_i)$ is well-defined. Clearly $\psi=(\phi_1*\phi_2)|_A$.  
\end{proof}

\begin{definition}\label{D:SubsystemCentralize}
Suppose that for each $i=1,2$, we are given a subsystem $\F_i$ of $\F$ over $S_i\leq S$. We say that $\F_1$ and $\F_2$ centralize each other (in $\F$) if $\F_i\subseteq C_\F(S_{3-i})$ and $S_1\cap S_2\leq Z(\F_i)$ for $i=1,2$.
\end{definition}

\begin{lemma}\label{L:F1starF2inF}
Let $\F_i$ be a subsystem of $\F$ for $i=1,2$ such that $\F_1$ and $\F_2$ centralize each other in $\F$. Then $\F_1*\F_2$ is well-defined and a subsystem of $\F$. 
\end{lemma}

\begin{proof}
Let $\F_i$ be a subsystem over $S_i\leq S$ for $i=1,2$. Then $S_1\cap S_2\leq Z(\F_i)$ for $i=1,2$, so the subsystem $\F_1*\F_2$ is by Lemma~\ref{L:F1starF2Help}(a) well-defined. Let now $P_i\leq S_i$ and $\phi_i\in\Hom_{\F_i}(P_i,S_i)$ for $i=1,2$. As $\F_i\subseteq C_\F(S_{3-i})$, the morphism $\phi_i$ extends to a morphism $\hat{\phi}_i\in\Hom_\F(P_iS_{3-i},S_1S_2)$ which restricts to the identity on $S_{3-i}$. Hence, $\phi_1*\phi_2=(\hat{\phi}_1|_{P_1P_2})(\hat{\phi}_2|_{(P_1\phi_1)P_2})$ is a morphism in $\F$. This proves $\F_1*\F_2\subseteq\F$.  
\end{proof}

The connection between $\F_i^{cr}$ and $\F^{cr}$ given in part (f) of the following lemma will be particularly important in the proof of Theorem~\ref{T:RegularN1timesN2}. Because of our non-standard definition of $\F^{cr}$, we are not able to cite results from the literature about central products of fusion systems. It is important that we show the property (f) for fusion systems that are not known to be saturated, because this is needed to prove saturation.

\begin{lemma}\label{L:F1starF2Main}
For $i=1,2$ let $\F_i$ be a fusion system over $S_i$ such that $S_1\cap S_2\leq Z(\F_i)$ and $[S_1,S_2]=1$. Assume $\F=\F_1*\F_2$. Then the following hold:
\begin{itemize}
 \item [(a)] Given $Q\leq S_1S_2$, we have $N_{\F}(Q)\subseteq N_{\F_1}(Q_1)*N_{\F_2}(Q_2)$.
 \item [(b)] If $S_1\cap S_2\leq Q_i\leq S_i$ for $i=1,2$, then $N_\F(Q_1Q_2)=N_{\F_1}(Q_1)*N_{\F_2}(Q_2)$.
 \item [(c)] We have $O_p(\F)=O_p(\F_1)O_p(\F_2)$.
 \item [(d)] Suppose for $i=1,2$ we are given $Q_i\leq S_i$ such that $S_1\cap S_2\leq Q_i$ and $Q_i$ is fully $\F_i$-normalized. Then $Q_1Q_2$ is fully $\F$-normalized.
 \item [(e)] Given $P_i\in\F_i^c$ for $i=1,2$, we have $P_1P_2\in\F^c$.
 \item [(f)] If $P_i\in\F_i^{cr}$ for $i=1,2$, we have $P_1P_2\in\F^{cr}$.
 \item [(g)] Suppose $\F_i$ is saturated and $P_i\leq S_i$ for each $i=1,2$. Then $P_1P_2\in\F^s$ if and only if $P_i\in\F_i^s$ for each $i=1,2$. 
\end{itemize}
\end{lemma}

\begin{proof}
Let $Q\leq S_1S_2$ and $x_i\in Q_i$ for $i=1,2$ such that $x_1x_2\in Q$. If $s_i\in S_i$ for $i=1,2$ such that $s_1s_2\in N_S(Q)$, then $(x_1^{s_1})(x_2^{s_2})=(x_1x_2)^{s_1s_2}\in Q^{s_1s_2}=Q$ and so $x_i^{s_i}\in Q_i$ for $i=1,2$. This proves $Q_i^{s_i}\leq Q_i$ and thus $s_i\in N_{S_i}(Q_i)$ for $i=1,2$. Hence, $N_S(Q)\leq N_{S_1}(Q_1)N_{S_2}(Q_2)$. Let now $A\leq N_S(Q)$ and $\phi\in\Hom_{N_\F(Q)}(A,N_S(Q))$. Without loss of generality choose $A$ and $\phi$ such that $Q\leq A$ and thus $Q\phi=Q$. By Lemma~\ref{L:F1starF2Help}(f), there exist $\phi_i\in\Hom_{\F_i}(A_i,S_i)$ such that $\phi=(\phi_1*\phi_2)|_A$. Notice that $Q_i\leq A_i$. If $x_1$ and $x_2$ are as above, then $(x_1\phi_1)(x_2\phi_2)=(x_1x_2)\phi\in Q\phi=Q$ and hence $x_i\phi_i\in Q_i$ for $i=1,2$. This proves $Q_i\phi_i=Q_i$ and $\phi_i\in\Hom_{N_{\F_i}(Q_i)}(A_i,S_i)$. Hence $\phi=(\phi_1*\phi_2)|_A$ is a morphism in $N_{\F_1}(Q_1)*N_{\F_2}(Q_2)$ and (a) holds.
 
\smallskip

For the proof of (b) let now $Q_i\leq S_i$ be arbitrary with $S_1\cap S_2\leq Q_i$ for $i=1,2$. It follows from Lemma~\ref{L:F1starF2Help}(d) that $(Q_1Q_2)_i=Q_i$ for $i=1,2$. Hence, setting $Q:=Q_1Q_2$, it follows from part (a) that $N_\F(Q)\subseteq N_{\F_1}(Q_1)*N_{\F_1}(Q_2)$. Observe that $\F_1$ and $\F_2$ centralize each other in $\F$. This implies that $N_{\F_1}(Q_1)$ and $N_{\F_2}(Q_2)$ are contained in $N_\F(Q)$ and centralize each other in $N_\F(Q)$. So Lemma~\ref{L:F1starF2inF} implies now that (b) holds. 

\smallskip

Notice that $S_1\cap S_2\leq Z(\F_i)\leq O_p(\F_i)$ for $i=1,2$. So part (b) yields in particular that $\F=\F_1*\F_2=N_{\F_1}(O_p(\F_1))*N_{\F_1}(O_p(\F_2))=N_\F(O_p(\F_1)O_p(\F_2))$. Thus, $O_p(\F_1)O_p(\F_2)\leq R:=O_p(\F_1*\F_2)$ and $O_p(\F_i)\leq R_i$. Similarly $\F=N_\F(R)\subseteq N_{\F_1}(R_1)*N_{\F_2}(R_2)$ by (a). So if $i\in\{1,2\}$ is fixed, $A\leq S_i$ and $\phi\in\Hom_{\F_i}(A,S_i)$, then by Lemma~\ref{L:F1starF2Help}(f) applied with $N_{\F_1}(R_1)$ and $N_{\F_2}(R_2)$ in place of $\F_1$ and $\F_2$, we have $\phi=(\phi_1*\phi_2)|_A$ where $\phi_j\in \Hom_{N_{\F_j}(R_j)}(A_j,S_j)$ for $j=1,2$. As $A\leq S_i$, it follows that $A\leq A_i$ and $\phi=\phi_i|_A$ is a morphism in $N_{\F_i}(R_i)$. As $\phi$ was an arbitrary morphism in $\F_i$, this shows that  $R_i\unlhd \F_i$ and thus $R\leq R_1R_2\leq O_p(\F_1)O_p(\F_2)$. This proves (c).   

\smallskip

For the proof of (d) let now $S_1\cap S_2\leq Q_i\leq S_i$ such that $Q_i$ is fully $\F_i$-normalized for $i=1,2$. It follows from Lemma~\ref{L:F1starF2Help}(f) that an $\F$-conjugate of $Q_1Q_2$ is of the form $P_1P_2$ with $P_i\in Q_i^{\F_i}$ for $i=1,2$. Fix such $P_i$. As $S_1\cap S_2\leq Z(\F_i)$, we have also $S_1\cap S_2\leq P_i$ for $i=1,2$. By (b), we have now $N_S(Q_1Q_2)=N_{S_1}(Q_1)N_{S_2}(Q_2)$ and $N_S(P_1P_2)=N_{S_1}(P_1)N_{S_2}(P_2)$. Observe furthermore that $N_{S_1}(Q_1)\cap N_{S_2}(Q_2)=S_1\cap S_2$ and $N_{S_1}(P_1)\cap N_{S_2}(P_2)=S_1\cap S_2$. As $|N_{S_i}(Q_i)|\geq |N_{S_i}(P_i)|$ for $i=1,2$, it follows that $|N_S(Q_1Q_2)|\geq |N_S(P_1P_2)|$. This proves (d).

\smallskip

Fix now $P_i\in\F_i^c$ for $i=1,2$. Notice that $S_1\cap S_2\leq Z(S_i)\leq P_i$ for $i=1,2$. In particular, by Lemma~\ref{L:F1starF2Help}(d), we have $(P_1P_2)_i=P_i$ for $i=1,2$. It follows from Lemma~\ref{L:F1starF2Help}(f) that every $\F$-conjugate of $P_1P_2$ is of the form $Q_1Q_2$ with $Q_i\in P_i^{\F_i}$. Fixing such $Q_i$, we have $C_{S_i}(Q_i)=Z(Q_i)$ and, in particular, $S_1\cap S_2\leq Z(S_i)\leq Z(Q_i)$ for $i=1,2$. Hence,
\begin{eqnarray*}
C_S(Q_1Q_2)&=& C_S(Q_1)\cap C_S(Q_2)\\
&=&Z(Q_1)S_2\cap S_1Z(Q_2)\\
&=&Z(Q_1)(S_2\cap S_1Z(Q_2))\\
&=&Z(Q_1)(S_2\cap S_1)Z(Q_2)\leq Q_1Q_2.
\end{eqnarray*}
Thus $P_1P_2\in \F^c$. This proves (e). 

\smallskip

For the proof of (f) suppose now that $P_i\in\F_i^{cr}$. By (e), it is sufficient to show that $P_1P_2\in \F^r$. By Definition of $\F_i^r$, for each $i=1,2$ there exists $Q_i\in P_i^{\F_i}$ such $Q_i$ is fully $\F_i$-normalized and $O_p(N_{\F_i}(Q_i))=Q_i$. As $S_1\cap S_2\leq Z(\F_i)\leq Q_i$, part (b) yields that $N_\F(Q_1Q_2)=N_{\F_1}(Q_1)*N_{\F_2}(Q_2)$. Moreover, we get $S_1\cap S_2\leq Z(N_{\F_i}(Q_i))$. So by part (c), we have $O_p(N_\F(Q_1Q_2))=O_p(N_{\F_1}(Q_1))O_p(N_{\F_2}(Q_2))=Q_1Q_2$. Observe that $Q_1Q_2$ is $\F$-conjugate to $P_1P_2$ and fully $\F$-normalized by (d). Hence, $P_1P_2\in\F^r$ and (f) holds. 

\smallskip

For the proof of (g) suppose now that $\F_i$ is saturated and $P_i\leq S_i$ for each $i=1,2$. By \cite[Proposition~3.3]{Henke:2018}, $\F$ is a central product of $\F_1$ and $\F_2$, i.e. the map
\[\alpha\colon S_1\times S_2\rightarrow S,\;(s_1,s_2)\mapsto s_1s_2\]
is a group homomorphism which induces an epimorphism from $\F_1\times\F_2$ to $\F$ and has kernel $Z:=\{(x,x^{-1})\colon x\in S_1\cap S_2\}\leq Z(\F_1\times \F_2)$. By \cite[Lemma~2.7(g)]{Henke:2016}, we have $P_i\in\F_i^s$ for $i=1,2$ if and only if $P_1\times P_2\in (\F_1\times \F_2)^s$. Moreover, by \cite[Lemma~2.15]{Henke:2016}, this is the case if and only if $P_1P_2=(P_1\times P_2)\alpha\in\F^s$. This proves (g). 
\end{proof}

\chapter{Partial groups and localities}\label{S:PartialGroupsLocalities}

The reader is referred to \cite{Chermak:2015} for a detailed introduction to partial groups and localities. However, we will summarize the most important definitions and results here including some further background material from \cite{Henke:2015,Henke:2020}. We also take the opportunity to prove a couple of elementary lemmas which are not available in the literature.

\section{Partial groups}

Given a set $\L$, write $\W(\L)$ for the set of words in $\L$. The elements of $\L$ will be identified with the words of length one, and $\emptyset$ denotes the empty word. The concatenation of two words $u,v\in\W(\L)$ will be denoted $u\circ v$ and the concatenation of more than two words will be written similarly.

\begin{definition}[{\cite[Definition~1.1]{Chermak:2015}}]\label{partial}
Suppose $\L$ is a non-empty set and $\D\subseteq\W(\L)$. Let $\Pi \colon  \D \longrightarrow \L$ be a map, and let $(-)^{-1} \colon \L \longrightarrow \L$ be an involutory bijection, which we extend to a map 
\[(-)^{-1} \colon \W(\L) \longrightarrow \W(\L), w = (g_1, \dots, g_k) \mapsto w^{-1} = (g_k^{-1}, \dots, g_1^{-1}).\]
Then $\L$ is called a \emph{partial group} with product $\Pi$ and inversion $(-)^{-1}$ if the following hold for all words $u,v,w\in\W(\L)$:
\begin{itemize}
\item  $\L \subseteq \D$ and
\[  u \circ v \in \D \Longrightarrow u,v \in \D.\]
(So in particular, $\emptyset\in\D$.)
\item $\Pi$ restricts to the identity map on $\L$.
\item $u \circ v \circ w \in \D \Longrightarrow u \circ (\Pi(v)) \circ w \in \D$, and $\Pi(u \circ v \circ w) = \Pi(u \circ (\Pi(v)) \circ w)$.
\item $w \in  \D \Longrightarrow  w^{-1} \circ w\in \D$ and $\Pi(w^{-1} \circ w) = \One$ where $\One:=\Pi(\emptyset)$.
\end{itemize}
\end{definition}

\textbf{For the remainder of this chapter let $\L$ always be a partial group with product $\Pi\colon\D\rightarrow\L$. As above set $\One:=\Pi(\emptyset)$.}

\smallskip

If $w=(f_1,\dots,f_n)\in\D$, then we write sometimes $f_1f_2\cdots f_n$ for $\Pi(f_1,\dots,f_n)$. In particular, if $(x,y)\in\D$, then $xy$ stands for $\Pi(x,y)$. Notice that the first and the third of the axioms listed above imply the following lemma. 

\begin{lemma}[{\cite[Lemma~1.4(d)]{Chermak:2015}}]\label{L:Chermak14d}
Let $u,v\in\W(\L)$ with $u\circ v\in\D$. Then $u^{-1}\circ u\circ v\in\D$ and $u\circ v\circ v^{-1}\in\D$. Moreover, $\Pi(u^{-1}\circ u\circ v)=\Pi(v)$ and $\Pi(u\circ v\circ v^{-1})=\Pi(u)$.
\end{lemma}

We will also need the following lemma.

\begin{lemma}\label{L:AddOnes}
\begin{itemize}
\item [(a)] Let $u,v\in\W(\L)$ with $u\circ v\in\D$ and $w\in\W(\{\One\})$. Then $u\circ w\circ v\in \D$ and $\Pi(u\circ w\circ v)=\Pi(u\circ v)$.
\item [(b)] If $w,w'\in\W(\{\One\})$ and $u\in\D$, then $w\circ u\circ w'\in\D$ with $\Pi(w\circ u\circ w')=\Pi(u)$.
\end{itemize}
\end{lemma}

\begin{proof}
Observe that (b) follows from (a) applied twice, so it is sufficient to prove (a). Let $u,v,w$ be as in (a). As $u\circ v=u\circ\emptyset\circ v\in\D$, the axioms of a partial group imply $u\circ (\One)\circ v=u\circ(\Pi(\emptyset))\circ v\in\D$ and $\Pi(u\circ (\One)\circ v)=\Pi(u\circ \emptyset\circ v)=\Pi(u\circ v)$. Therefore, it follows by induction on $|w|$ that $u\circ w\circ v\in\D$ and $\Pi(u\circ w\circ v)=\Pi(u\circ v)$. By the axioms of a partial group, we have $\L\subseteq\D$ and $\Pi|_\L=\id_\L$. Hence the assertion follows. 
\end{proof}

\begin{definition}
 \begin{itemize}
 \item Given $f\in\L$, write $\D(f):=\{x\in\L\colon (f^{-1},x,f)\in\D\}$ for the set of all $x$ such that the conjugate $x^f:=\Pi(f^{-1},x,f)$ is defined. 
 \item Write $c_f$ for the conjugation map $c_f\colon \D(f)\rightarrow \L,x\mapsto x^f$.
 \item For $f\in\L$ and  $\X\subseteq\D(f)$, set $\X^f:=\{x^f\colon x\in\X\}$.
\item Set 
\[N_\L(\X):=\{f\in\L\colon \X\subseteq\D(f)\mbox{ and }\X^f=\X\}\]
and
\[C_\L(\X):=\{f\in\L\colon x\in\D(f)\mbox{ and }x^f=x\mbox{ for all }x\in\X\}.\]
\item For $\Y\subseteq\L$ define moreover $N_\Y(\X)=\Y\cap N_\L(\X)$ and $C_\Y(\X)=\Y\cap C_\L(\X)$. 
\item Call $Z(\L)=C_\L(\L)$ the \emph{center} of $\L$. 
\end{itemize}
\end{definition}

It is proved in \cite[Lemma~1.6(c)]{Chermak:2015} that the conjugation map $c_f\colon \D(f)\rightarrow \D(f^{-1})$ is a bijection with inverse map $c_{f^{-1}}$. The following lemma shows that, for all $\X,\Y\subseteq\L$, we have $\X\subseteq C_\L(\Y)$ if and only if $\Y\subseteq C_\L(\X)$. 

\begin{lemma}\label{L:CentralizerHelp}
Let $f,g\in\L$. Then 
\[g\in C_\L(f)\Longleftrightarrow f\in C_\L(g).\]
If so, then $(f,g),(g,f)\in\D$ and $fg=gf$.
\end{lemma}

\begin{proof}
Notice that the situation is symmetric in $f$ and $g$, so it is sufficient to prove one direction. Suppose $g\in C_\L(f)$, i.e. $f\in\D(g)$ and $f^g=f$. Then $(g^{-1},f,g)\in\D$, so by Lemma~\ref{L:Chermak14d}, we have $v=(f^{-1},g,g^{-1},f,g)\in\D$ and $u=(g,g^{-1},f,g)\in\D$. Using the axioms of a partial group and \cite[Lemma~2.2(a)]{Henke:2020}, we can conclude that $(f^{-1},g,f)=(f^{-1},g,f^g)\in\D$ and 
\[g^f=\Pi(f^{-1},g,f^g)=\Pi(v)=\Pi(f^{-1},\One,f,g)=\Pi(f^{-1},f,g)=\Pi(\One,g)=g.\]
Hence, $g\in C_\L(f)$. Observe that $(g^{-1},f,g)\in\D$ implies $(f,g)\in\D$ and similarly, $(f^{-1},g,f)\in\D$ implies $(g,f)\in\D$. Moreover, $fg=\Pi(u)=\Pi(g,f^g)=gf$.
\end{proof}

\begin{definition}\label{D:PartialSubgroup}
\begin{itemize}
 \item A subset $\H\subseteq\L$ is called a \emph{partial subgroup} of $\L$ if $h^{-1}\in\H$ for all $h\in\H$, and moreover $\Pi(w)\in\H$ for all $w\in\D\cap\W(\H)$. (Notice that $\H$ is then itself a partial group with product $\Pi|_{\W(\H)\cap\D}$.)
 \item A partial subgroup $\H$ of $\L$ is a called a \emph{subgroup} of $\L$ if $\W(\H)\subseteq\D(\L)$. (Observe that a subgroup of $\L$ forms actually a group.)
 \item By a \emph{$p$-subgroup} of $\L$ we mean a subgroup of $\L$ which is a finite $p$-group. 
 \item If $\N$ is a partial subgroup of $\L$, then $\N$ is called a \emph{partial normal subgroup} if $n^f\in\N$ for all $f\in\L$ and all $n\in\N\cap\D(f)$. Write $\N\unlhd\L$ to indicate that $\N$ is a partial normal subgroup of $\L$.
 \item A partial subgroup $\H$ of $\L$ is \emph{subnormal} in $\L$ if there exists a series $\H=\H_0\unlhd\H_1\unlhd\cdots\unlhd\H_k=\L$ with $k\geq 0$. We write $\H\subn\L$ to indicate that $\H$ is subnormal in $\L$. 
\end{itemize}
\end{definition}

\begin{lemma}\label{L:SubnormalinSubgroup}
Let $\H$ be a partial subgroup of $\L$.
\begin{itemize}
\item [(a)] If $\K\subn\L$ with $\K\subseteq\H$, then $\K\subn\H$.
\item [(b)] If $\K\subn\H$ and $\H\subn\L$, then $\K\subn\L$.
\end{itemize}
\end{lemma}

\begin{proof}
If $\K=\K_0\unlhd\K_1\unlhd\cdots\unlhd\K_n=\L$ is a subnormal series of $\K$ in $\L$, then \[\K=\K_0\cap\H\unlhd\K_1\cap\H\unlhd\cdots\unlhd\K_n\cap\H=\H\]
is a subnormal series of $\K$ in $\H$. Hence (a) holds. 

\smallskip

If $\K=\K_0\unlhd\K_1\unlhd\cdots \unlhd\K_m=\H$ is a subnormal series of $\K$ in $\H$ and $\H=\H_0\unlhd\H_1\unlhd\cdots\unlhd\H_n=\L$ is a subnormal series of $\H$ in $\L$, then
\[\K=\K_0\unlhd\K_1\unlhd\cdots \unlhd\K_m=\H=\H_0\unlhd\H_1\unlhd\cdots\unlhd\H_n=\L\]
is a subnormal series of $\K$ in $\L$.
\end{proof}

In various contexts we will consider products of subsets of partial groups as defined next.

\begin{definition}
For $\X,\Y\subseteq\L$ set
\[\X\Y:=\{\Pi(x,y)\colon x\in\X,\;y\in\Y,\;(x,y)\in\D\}.\]
More generally, for subsets $\X_1,\dots,\X_n$ of $\L$ set
\[\X_1\X_2\cdots \X_n:=\{\Pi(x_1,\dots,x_n)\colon (x_1,\dots,x_n)\in\D,\;x_i\in\X_i\mbox{ for }i=1,\dots,n\}.\]
\end{definition}

\section{Localities} 

\begin{definition}\label{locality}
Let $\L$ be a finite partial group, let $S$ be a $p$-subgroup of $\L$ and let $\Delta$ be a non-empty set of subgroups of $S$. We say that $(\L, \Delta, S)$ is a \emph{locality} if the following hold:
\begin{itemize}
\item $S$ is maximal with respect to inclusion among the $p$-subgroups of $\L$.
\item $\D$ is the set of words $w=(g_1, \dots, g_k) \in \W(\L)$ for which there exist $P_0, \dots ,P_k \in \Delta$ with 
\begin{itemize}
\item[(*)] $P_{i-1} \subseteq \D(g_i)$ and $P_{i-1}^{g_i} = P_i$ for all $1 \leq  i \leq k$.
\end{itemize} 
\item $\Delta$ is overgroup-closed in $S$ and closed under taking $\L$-conjugates in $S$; the latter condition means $P^g\in \Delta$ for every $P \in \Delta$ and $g\in \L$ with $P \subseteq \D(g)$.
\end{itemize}
If $\L$ is a partial group and $\Delta$ is a set of subgroups of $\L$, then we write $\D_\Delta$ for the set of $w=(g_1,\dots,g_k)\in\W(\L)$ such that (*) holds for some $P_0,\dots,P_k\in\Delta$. Moreover, if $w\in\W(\L)$ and (*) holds for some $P_0,\dots,P_k\in\Delta$, then we say  that $w\in\D_\Delta$ via $P_0,P_1,\dots,P_k$, or just that $w\in\D_\Delta$ via $P_0$.
\end{definition}

The definition of a locality above is a reformulation of the one given by Chermak \cite[Definition 2.8]{Chermak:2015}. The two definitions are equivalent by \cite[Remark~5.2]{Henke:2015}. 

\smallskip

\textbf{For the remainder of this section let $(\L,\Delta,S)$ be a locality.}

\smallskip

Notice that $N_\L(P)$ is a subgroup of $\L$ for all $P\in\Delta$. Let $(\L,\Delta,S)$ be a locality. Given $w=(f_1,\dots,f_n)\in\W(\L)$, write $S_w$ for the set of all $s\in S$ such that there exist elements $s=s_0,\dots,s_n\in S$ with $s_{i-1}^{f_i}=s_i$. In particular,
\[S_f:=\{s\in S\colon s\in\D(f),\;s^f\in S\}\mbox{ for all }f\in\L.\]
For every $f\in\L$ and $w\in\W(\L)$ the following properties hold by \cite[Proposition~2.6(a),(b), Corollary~2.7]{Chermak:2015}. We will use them from now on without further reference.
\begin{itemize}
 \item $S_w$ is a subgroup of $S$. Moreover, $S_w\in\Delta$ if and only if $w\in\D$.
 \item $S_f\in\Delta$. 
 \item $c_f\colon S_f\rightarrow S_{f^{-1}}=S_f^f$ is an isomorphism of groups with inverse map $c_{f^{-1}}$.
\end{itemize}

Because of these properties, the conjugation maps $c_f\colon S_f\rightarrow S$ with $f\in\L$ generate a fusion system over $S$, which is denoted by $\F_S(\L)$. More generally, for every partial subgroup $\H$ of $\L$, $\F_{S\cap\H}(\H)$ denotes the fusion system over $S\cap\H$ which is generated by all the maps of the form $c_h\colon \H\cap S_h\rightarrow S\cap\H$ with $h\in\H$.

\begin{lemma}\label{L:ConjugateNormalizer}
The following hold: 
\begin{itemize}
\item [(a)] Let $P\in\Delta$ and $f\in\L$ with $P\leq S_f$. Then $P^f\in\Delta$, $N_\L(P)\subseteq\D(f)$ and 
\[c_f\colon N_\L(P)\rightarrow N_\L(P^f),x\mapsto x^f\]
is an isomorphism of groups. Moreover, $N_\N(P)^f=N_\N(P^f)$ for all $\N\unlhd\L$.
\item [(b)] Let $w=(f_1,\dots,f_n)\in\D$ via $P_0,P_1,\dots,P_n\in\Delta$. Then 
\[c_{f_1}\circ c_{f_1}\circ \cdots \circ c_{f_n}=c_{\Pi(w)}\]
as a map from $N_\L(P_0)$ to $N_\L(P_n)$. 
\end{itemize} 
\end{lemma}

\begin{proof}
Part (b) holds by \cite[Lemma~2.3(c)]{Chermak:2015}. Moreover, using the definition of a locality given above, it is clear that $P^f\in\Delta$. Hence, by \cite[Lemma~2.3(b)]{Chermak:2015}, $N_\L(P)\subseteq \D(f)$ and 
$c_f\colon N_\L(P)\rightarrow N_\L(P^f)$
is an isomorphism of groups. Let now $\N$ be a partial normal subgroup of $\L$. Then $N_\N(P)^f\subseteq N_\L(P^f)\cap \N=N_\N(P^f)$. The properties stated above give $P^f\leq S_f^f=S_{f^{-1}}$ and $(P^f)^{f^{-1}}=P$, so a symmetric argument gives $N_\N(P^f)^{f^{-1}}\subseteq N_\N(P)$. By \cite[Proposition~1.6(c)]{Chermak:2015}, $c_f\colon \D(f)\rightarrow \D(f^{-1})$ is bijective with inverse map $c_{f^{-1}}$. Thus we can conjugate on both sides by $f$ to obtain $N_\N(P^f)\subseteq N_\N(P)^f$. Hence, $N_\N(P^f)=N_\N(P)^f$ as required.
\end{proof}

If $w=(f_1,\dots,f_n)\in\D$, then Lemma~\ref{L:ConjugateNormalizer}(b) applied with $P_0:=S_w$ and $P_i:=P_{i-1}^{f_i}$ gives in particular that $c_{f_1}|_{P_0}\circ c_{f_2}|_{P_1}\circ\cdots \circ c_{f_n}|_{P_{n-1}}=c_{\Pi(w)}|_{P_0}$ and $P_0=S_w\leq S_{\Pi(w)}$. We will use Lemma~\ref{L:ConjugateNormalizer} frequently in this form, most of the time without reference. 

\begin{lemma}\label{L:LocalityFusionSystem}
Let $\F=\F_S(\L)$ and $P\in\Delta$.
\begin{itemize}
\item [(a)] For every $\phi\in\Hom_\F(P,S)$, there exists $f\in\L$ with $P\leq S_f$ and $\phi=c_f|_P$. 
\item [(b)] If $\H$ is a partial subgroup of $\L$ with $P\leq S\cap\H$ and if $\phi\colon P\rightarrow S$ is a morphism in $\F_{S\cap \H}(\H)$, then $\phi=c_f|_P$ for some $f\in\H$ with $P\leq S_f$.
\end{itemize}
\end{lemma}

\begin{proof}
Notice that (b) implies (a). So assume that $\H$ and $\phi$ are as in (b). Then by definition of $\F_S(\H)$, the map $\phi$ can be written as a composition $\phi=(c_{f_1}|_{P_0})\circ\cdots \circ (c_{f_n}|_{P_{n-1}})$ where $P_0=P\in\Delta$, $f_i\in\H$ and $P_i=P_{i-1}^{f_i}\in\Delta$ for $i=1,\dots,n$. Hence, $w:=(f_1,\dots,f_n)\in\D$ via $P_0,\dots,P_n\in\Delta$. So setting $f:=\Pi(w)$, Lemma~\ref{L:ConjugateNormalizer} gives $P\leq S_w\leq S_f$ and $\phi=c_f$. As $w\in\W(\H)$, we have $f=\Pi(w)\in\H$.
\end{proof}

We repeat the following definition from \cite[Lemma~2.14]{Chermak:2015}.

\begin{definition}
If $(\L,\Delta,S)$ is a locality, then set
\[O_p(\L):=\bigcap\{S_w\colon w\in\W(\L)\}\] 
\end{definition}

\begin{lemma}\label{L:OpL}
If $(\L,\Delta,S)$ is a locality, then $O_p(\L)$ is the unique largest $p$-subgroup of $\L$, which is a partial normal subgroup of $\L$. Moreover, a subgroup $P\leq S$ is a partial normal subgroup of $\L$ if and only if $N_\L(P)=\L$.
\end{lemma}

\begin{proof}
By \cite[Lemma~2.14]{Chermak:2015}, $O_p(\L)$ is the unique largest subgroup of $S$ which is a partial normal subgroup of $\L$, and by \cite[Proposition~2.11(c)]{Chermak:2015}, every $p$-subgroup of $\L$, which is also a partial normal subgroup of $\L$, is contained in $S$. Hence, $O_p(\L)$ is the unique largest $p$-subgroup of $\L$, which is a partial normal subgroup of $\L$.

\smallskip 

Let $P\leq S$. If $\L=N_\L(P)$, then clearly $P\unlhd \L$. On the other hand, if $P\unlhd\L$ and $f\in\L$, then the property above gives $P\leq O_p(\L)\leq S_f$. Hence, $P^f$ is defined and then (because of $P\unlhd\L$) equal to $P$. So $f\in N_\L(P)$. This shows $\L=N_\L(P)$ if $P\unlhd\L$.
\end{proof}

The following lemma will be used in several places.

\begin{lemma}[{\cite[Lemma~2.8]{Henke:2020}}]\label{L:NLSbiset}
If $r\in N_\L(S)$ and $f\in\L$, then $(r,f)$, $(f,r)$ and $(r^{-1},f,r)$ are words in $\D$. Moreover,
\[S_{(f,r)}=S_{fr}=S_f,\;S_{(r,f)}=S_{rf}=S_f^{r^{-1}}\mbox{ and }S_{f^r}=S_f^r.\]
\end{lemma}

Notice that the lemma above implies $R\N=\{\Pi(r,n)\colon r\in R,\;n\in\N\}$ and $\N R=\{\Pi(n,r)\colon n\in\N,\;r\in R\}$ for $\N\subseteq\L$ and $R\subseteq N_\L(S)$. 

\begin{lemma}\label{L:ProductPartialSubgroup}
If $R\leq N_\L(S)$ and $\N\unlhd\L$, then $\N R=R\N$ is a partial subgroup of $\L$.
\end{lemma}

\begin{proof}
It follows from \cite[Lemma~3.4]{Chermak:2015} that $R\N=\N R$ and that $R\N$ is closed under taking products. If $r\in R$ and $n\in\N$, then by \cite[Lemma~1.4(f)]{Chermak:2015}, $(rn)^{-1}=n^{-1}r^{-1}\in \N R=R\N$. This shows that $R\N$ is a partial subgroup.
\end{proof}

\begin{lemma}\label{L:CentralizerPartialNormal}
For every $R\leq S$, the centralizer $C_\L(R)$ and the normalizer $N_\L(R)$ are partial subgroups of $\L$. Moreover, $C_\L(R)$ is a partial normal subgroup of $N_\L(R)$.
\end{lemma}

\begin{proof}
By \cite[Lemma~5.4]{Henke:2015}, $C_\L(R)$ and $N_\L(R)$ are partial subgroups of $\L$. Let $x\in C_\L(R)$ and $f\in N_\L(R)$ such that $x\in\D(f)$. Then $(f^{-1},x,f)\in\D$ via some objects $P_0,P_1,P_2,P_3$. Replacing $P_i$ by $P_iR$, we may assume $R\leq P_i$ for $i=0,1,2,3$. Recall that $c_{f^{-1}}=(c_f)^{-1}$. Using Lemma~\ref{L:ConjugateNormalizer}(a), we conclude 
\[c_{x^f}|_R=(c_{f^{-1}})|_R\circ (c_x)|_R\circ (c_f)|_R=(c_{f^{-1}})|_R\circ \id_R\circ (c_f)|_R=(c_f|_R)^{-1}\circ (c_f)|_R=\id_R\]
and so $x^f\in C_\L(R)$. 
\end{proof}

\begin{lemma}\label{L:HZHelp}
Let $\H$ be a partial subgroup of $\L$, and let $Z\leq Z(\L)$ with $\L=\H Z$. Then $\H\unlhd \L$.
\end{lemma}

\begin{proof}
Let $n\in\H$ and $f\in\L$ with $(f^{-1},n,f)\in\D$. Then $f=hz$ for some $h\in\H$ and $z\in Z$ with $(h,z)\in\D$. Notice that $Z(\L)\leq N_\L(S)$ and thus $S_f=S_{(h,z)}$ by Lemma~\ref{L:NLSbiset}. Hence, $u:=(z^{-1},h^{-1},n,h,z)\in\D$ via $S_{(f^{-1},n,f)}$ and $n^f=\Pi(u)=(n^h)^z$. As $z\in Z\leq Z(\L)$ and $n,h\in\H$, it follows $n^f=n^h\in\H$. This proves $\H\unlhd\L$.
\end{proof}

We will use the following general theorem.

\begin{theorem}\label{T:ProductsPartialNormal}
Let $(\L,\Delta,S)$ be a locality and $\M,\N\unlhd\L$. Then 
\[\M\N\unlhd\L,\;\M\N=\N\M\mbox{ and }(\M\N)\cap S=(\M\cap S)(\N\cap S).\]
Moreover, for every $f\in\M\N$, there exist $m\in\M$ and $n\in\N$ such that $(m,n)\in\D$, $m\in\M$, $n\in\N$, $f=mn$ and $S_f=S_{(m,n)}$. 
\end{theorem}

\begin{proof}
This is \cite[Theorem~1]{Henke:2015a}.
\end{proof}

We remark that the theorem above was proved by Chermak \cite[Theorem~5.1]{Chermak:2015} in the special case that $\M\subseteq N_\L(\N\cap S)$ and $\N\subseteq N_\L(\M\cap S)$. This is used in \cite{Henke:2015a} to prove Theorem~\ref{T:ProductsPartialNormal} in the general case. 

\smallskip

When working with localities it will be convenient to have the following definition for fusion systems in place. 

\begin{definition}\label{D:Fclosed}
Let $\F$ be a fusion system over a finite $p$-group $S$ and let $\Gamma$ be a set of subgroups of $S$. Then $\Gamma$ is called \emph{$\F$-closed} if $\Gamma$ is overgroup-closed in $S$ and every $\F$-conjugate of an element of $\Gamma$ is an element of $\Gamma$.
\end{definition}

Notice that $\Delta$ is $\F_S(\L)$-closed.

\section{Homomorphisms of partial groups.} \label{SS:LocalityHomomorphism} 
In this section, we will introduce natural notions of homomorphisms, projections, isomorphisms and automorphisms of partial groups and of localities.  Recall that $\L$ is a partial group with product $\Pi\colon \D\rightarrow\L$.

\smallskip

\textbf{Throughout this section let $\tL$ be a partial group with product $\tPi\colon\tD\longrightarrow\tL$.}

\begin{notation}\label{N:PartialHomWordMap}
If $\alpha\colon\L\longrightarrow\tL,f\mapsto f\alpha$ is a map, then $\alpha^*$ denotes the induced map on words 
\[\W(\L) \longrightarrow \W(\tL),\quad w=(f_1,\dots,f_n)\mapsto w\alpha^*=(f_1\alpha,\dots,f_n\alpha).\]
Set $\D\alpha^*:=\{w\alpha^*\colon w\in\D\}$. 
\end{notation}

\begin{definition}\label{D:PartialHom}
A map $\alpha\colon\L\longrightarrow\tL$ is called a \emph{homomorphism of partial groups} if 
\begin{enumerate}
\item $\D\alpha^* \subseteq \tD$; and 
\item $\Pi(w)\alpha = \tPi(w\alpha^*)$ for every $w \in \D$.
\end{enumerate}
If moreover $\D\alpha^* = \tD$, then we say that $\alpha$ is a \emph{projection} of partial groups. If $\alpha$ is injective and $\D\alpha^*=\tD$, then $\alpha$ is called an \emph{isomorphism}. Write $\Iso(\L,\tL)$ for the set of isomorphisms from $\L$ to $\tL$. The isomorphisms from $\L$ to itself are called \emph{automorphisms} and the set of these automorphisms is denoted by $\Aut(\L)$. 
\end{definition}

Notice that every projection $\L\longrightarrow\tL$ is surjective, as $\tD$ contains all the words of length one. In particular, isomorphisms are bijective. We now turn attention to maps between localities.

\begin{definition}\label{D:LocalityHomomorphism}
Let $(\L, \Delta, S)$ and $(\tL, \tDelta,\tS)$ be localities and let $\alpha \colon \L \longrightarrow \tL$ be a projection of partial groups. 
\begin{itemize}
\item For any set $\Gamma$ of subgroups of $\L$, set 
\[\Gamma\alpha:=\{P\alpha\colon P\in\Gamma\}.\]
\item $\alpha$ is called a \emph{projection of localities} from $(\L,\Delta,S)$ to $(\tL,\tDelta,\tS)$ if $\Delta\alpha= \tDelta$. 
\item An injective projection of localities is a called an \emph{isomorphism} of localities. Write 
\[\Iso((\L,\Delta,S),(\tL,\tDelta,\tS))\]
for the set of isomorphisms from $(\L,\Delta,S)$ to $(\tL,\tDelta,\tS)$. 
\item Write $\Aut(\L,\Delta,S)$ for the set of \emph{automorphisms} of $(\L,\Delta,S)$, i.e. for the isomorphisms from $(\L,\Delta,S)$ to itself.
\item Write $\Iso((\L,S),(\tL,\tS))$ for the set of all $\alpha\in\Iso(\L,\tL)$ with $S\alpha=\tS$. Set $\Aut(\L,S):=\Iso((\L,S),(\L,S))$. 
\end{itemize}
\end{definition}

If $\alpha$ is a projection of localities from $(\L,\Delta,S)$ to $(\tL,\tDelta,\tS)$, then $\alpha$ maps $S$ to $\tS$, as $S$ and $\tS$ are the unique maximal elements of $\Delta$ and $\tDelta$ respectively. Hence, 
\[\Iso((\L,\Delta,S),(\tL,\tDelta,\tS))\subseteq\Iso((\L,S),(\tL,S))\mbox{ and }\Aut(\L,\Delta,S)\leq \Aut(\L,S).\]

\smallskip

Let now $(\L,\Delta,S)$ be a locality and $\N\unlhd\L$.  Chermak \cite[Section~3]{Chermak:2015} considers the (right) cosets $\N f:=\{nf\colon n\in\N\}$ and shows in \cite[Proposition~3.14(d)]{Chermak:2015} that the cosets which are maximal with respect to inclusion form a partition of $\L$. Write $\L/\N$ for the set of these maximal cosets. Then $\L/\N$ can be given the structure of a partial group such that the map $\alpha\colon\L\rightarrow \L/\N$ sending every element $g\in\L$ to the unique maximal coset containing $g$ is a projection of partial groups (cf. \cite[Lemma~3.16]{Chermak:2015}). Therefore, by \cite[Theorem~4.3]{Chermak:2015}, $\alpha$ induces a locality structure on $\L/\N$ such that $\alpha$ becomes a projection of localities. The map $\alpha$ is called the \emph{natural projection} $\L\rightarrow\L/\N$. The partial normal subgroup $\N$ is the kernel of $\alpha$, i.e. the set of all elements of $\L$ which $\alpha$ maps to one.

\section{Restrictions} \label{SS:Restrictions}
Let $(\L^+,\Delta^+,S)$ be a locality with partial product $\Pi^+\colon\D^+\longrightarrow\L^+$. Moreover, let $\Delta\subseteq \Delta^+$ be an $\F_S(\L^+)$-closed collection of subgroups of $S$ (cf. Definition~\ref{D:Fclosed}). Suppose $\Delta$ is non-empty. Then there is a canonical way to construct a locality $(\L,\Delta,S)$ with $\L\subseteq\L^+$. Namely, set 
\[\L^+|_{\Delta}:=\{f\in\L^+\colon S_f\in\Delta\}\]
and write $\D$ for the set of words $w=(f_1,\dots,f_n)\in\W(\L^+)$ such that $w\in\D^+$ via some elements $P_0,P_1,\dots,P_n\in\Delta$ (cf. Definition~\ref{locality}). Observe that $\D\subseteq\W(\L^+|_\Delta)$ and that $\Pi^+(w)\in\L|_{\Delta}$ for all $w\in\D$ by Lemma~\ref{L:ConjugateNormalizer}(b). We call $\L:=\L^+|_{\Delta}$ together with  $\Pi^+|_{\D}\colon\D\longrightarrow \L$ and the restriction of the inversion map on $\L^+$ to $\L$ the \emph{restriction} of $\L^+$ to $\Delta$. It is shown in \cite[Lemma~2.21(a)]{Chermak:2015} and \cite[Lemma~2.23(a),(c)]{Henke:2020} that the restriction of $\L^+$ to $\Delta$ a partial group and that $(\L,\Delta,S)$ is a locality.

\begin{lemma}\label{L:RestrictionIntersectL}
Let $(\L,\Delta,S)$ and $(\L^+,\Delta^+,S)$ be localities with $\Delta\subseteq\Delta^+$ and $\L=\L^+|_\Delta$. Let $\K^+\subseteq\L^+$ and $R\subseteq S$. Then $(R\K^+)\cap \L=R(\K^+\cap\L)$.
\end{lemma}

\begin{proof}
 Clearly $R(\K^+\cap\L)\subseteq R\K^+\cap \L$ (where the product $R(\K^+\cap\L)$ is formed in $\L$ and the product $R\K^+$ is formed in $\L^+$). Let now $r\in R$ and $k\in\K^+$ such that $rk$ (meaning the product of $r$ and $k$ in $\L^+$) is an element of $\L$. As $\L=\L^+|_\Delta$, using Lemma~\ref{L:NLSbiset}, we get $S_{(r,k)}=S_{rk}\in\Delta$ and $S_k=S_{(r,k)}^r=S_{rk}^r\in\Delta$. Hence, $k\in\L$ and the product $rk$ is defined in $\L$. Thus, $rk\in R(\K^+\cap\L)$. 
\end{proof}

\section{Linking localities and saturation}

\begin{definition}
\begin{itemize}
\item A locality $(\L,\Delta,S)$ is of \emph{objective characteristic $p$} if $N_\L(P)$ is of characteristic $p$ for every $P\in\Delta$. 
\item A \emph{linking locality} is a locality $(\L,\Delta,S)$ of objective characteristic $p$ such that $\F_S(\L)^{cr}\subseteq \Delta$. 
\end{itemize}
\end{definition}

As already mentioned in the introduction, Chermak \cite{ChermakII,ChermakIII} uses the term proper locality instead of linking locality.

\begin{remark}\label{R:ExistenceUniquenessCLS}
It is shown in \cite[Theorem~A]{Henke:2015} that, given a saturated fusion system $\F$ over $S$ and an $\F$-closed collection $\Delta$ of subgroups of $S$ with $\F^{cr}\subseteq\Delta$, there is a linking locality $(\L,\Delta,S)$ over $\F$ with object set $\Delta$, which is unique up to an isomorphism that restricts to the identity on $S$. (A priori this uses a different definition of $\F^{cr}$, but according to Lemma~\ref{L:FcrAKOinFcr} the two notions coincide for saturated fusion systems.) The proof uses the existence and uniqueness of centric linking systems in the form Chermak \cite{Chermak:2013} proved it first, namely as a statement about localities. Chermak's proof relies on the classification of finite simple groups. The dependence on the classification can be avoided if one cites Oliver's version of the proof \cite{Oliver:2013} together with a result of Glauberman--Lynd \cite{Glauberman/Lynd} to conclude that a centric linking system over $\F$ exists and is unique; the arguments from \cite[Appendix~A]{Chermak:2015} (in the form of  \cite[Theorem~2.11]{Glauberman/Lynd:2020}) imply then that there exists a unique linking locality $(\L,\F^c,S)$ over $\F$. This leads also to a classification-free proof of \cite[Theorem~A]{Henke:2015}, which we partly use below. Thus, the results in this paper do not depend on the classification of finite simple groups.
\end{remark}

Our non-standard definition of $\F^{cr}$ makes it possible to prove the following theorem, which was communicated to us by Chermak.

\begin{theorem}\label{T:Saturation}
If $(\L,\Delta,S)$ is a locality with $\F_S(\L)^{cr}\subseteq\Delta$, then $\F_S(\L)$ is saturated. In particular, $\F_S(\L)$ is saturated if $(\L,\Delta,S)$ is a linking locality. 
\end{theorem}

\begin{proof}
Set $\F:=\F_S(\L)$. By \cite[Lemma~2.10]{Chermak:2015}, every element of $\Delta$ is $\F$-conjugate to an element $P\in\Delta$ such that $N_S(P)\in\Syl_p(N_\L(P))$. Hence, by Lemma~\ref{L:SaturationHelp} below, $\F$ is $\Delta$-saturated in the sense of \cite[Definition~I.3.9]{Aschbacher/Kessar/Oliver:2011}.

\smallskip

If $Q\leq S$ is $\F$-centric with $Q\not\in\Delta$, then the assumption that $\F^{cr}\subseteq\Delta$ implies that $Q$ is not $\F$-radical. Thus, if we fix a fully normalized subgroup $P\in Q^\F$, then we have  $P<P^*=O_p(N_\F(P))$. Observe that $\Out_{P^*}(P)\leq \Out_S(P)\cap O_p(\Out_\F(P))$. As $Q$ is $\F$-centric, we have $C_S(P)\leq P$ and so $\Out_{P^*}(P)\neq 1$. Thus the assertion follows from \cite[Theorem~2.2]{BCGLO1} (which is stated as Theorem~I.3.10 in \cite{Aschbacher/Kessar/Oliver:2011}).
\end{proof}

\begin{lemma}\label{L:SaturationHelp}
Let $(\L,\Delta,S)$ be a locality and $P\in\Delta$ such that $N_S(P)\in\Syl_p(N_\L(P))$. Then $P$ is fully automized and receptive in $\F_S(\L)$.
\end{lemma}

\begin{proof}
Set $\F=\F_S(\L)$. Then by Lemma~\ref{L:LocalityFusionSystem}(a), we have $\Aut_\F(P)\cong N_\L(P)/C_\L(P)$ and thus $P$ is fully automized. Let now $Q\in P^\F$, $\phi\in\Hom_\F(Q,P)$ and $R:=N_\phi$. Again by Lemma~\ref{L:LocalityFusionSystem}(a), there exists $f\in\L$ such that $Q\leq S_f$ and $\phi=c_f|_Q$. Moreover, by Lemma~\ref{L:ConjugateNormalizer}, the conjugation map $c_f\colon N_\L(Q)\rightarrow N_\L(P)$ is well-defined and an isomorphism of groups. By definition of $N_\phi$, we have $\phi^{-1}\Aut_R(Q)\phi\leq \Aut_S(P)$. Thus, $R^f\leq N_S(P)C_\L(P)$. Notice that $N_S(P)$ is a Sylow $p$-subgroup of $N_S(P)C_\L(P)\leq N_\L(P)$. Hence, by Sylow's Theorem there exists $c\in C_\L(P)$ such that $R^{fc}=(R^f)^c\leq N_S(P)$. Thus, $\phi=c_f|_Q=c_{fc}|_Q$ extends to $c_{fc}|_R\in\Hom_\F(R,N_S(P))$. This proves that $P$ is receptive. 
\end{proof}

From now on, the fact shown in Theorem~\ref{T:Saturation} that the fusion system $\F_S(\L)$ of a linking locality $(\L,\Delta,S)$ is saturated will be used  without further reference. 

\section{Two lemmas on partial normal subgroups}

The following two lemmas will be used in the proof of Lemma~\ref{L:RegularConjAutomorphisms}, but are also useful in other contexts.

\begin{lemma}\label{L:FTNWeaklyNormalNormal}
Let $(\L,\Delta,S)$ be a locality over $\F$ and $\N\unlhd\L$. Set $T:=\N\cap S$ and $\E:=\F_{S\cap\N}(\N)$. Suppose every element of $\Aut_\F(T)$ induces an automorphism of $\E$. Then $\E$ is $\F$-invariant. If $\E$ is saturated and $(\L,\Delta,S)$ is of objective characteristic $p$, then $\E$ is normal in $\F$.  
\end{lemma}

\begin{proof}
By \cite[Lemma~3.1(a)]{Chermak:2015}, $T$ is strongly closed in $\F$. To prove that $\E$ is $\F$-invariant, it remains to show the Frattini condition as stated in \cite[Definition~I.6.1]{Aschbacher/Kessar/Oliver:2011}. For that let $P\leq T$ and $\phi\in\Hom_\F(P,T)$. As $\F=\F_S(\L)$, we can factorize $\phi$ as $\phi=(c_{f_1}|_{P_0})(c_{f_2}|_{P_1})\cdots (c_{f_k}|_{P_{k-1}})$ for subgroups $P=P_0,\dots,P_k\leq T$ and $f_i\in\L$ with $P_{i-1}\leq S_{f_i}$ and $P_{i-1}^{f_i}=P_i$ for $i=1,\dots,k$. By the Frattini Lemma and the Splitting Lemma for localities \cite[Corollary~3.11, Lemma~3.12]{Chermak:2015}, for each $i=1,2,\dots,k$, there exist $n_i\in\N$ and $g_i\in N_\L(T)$ such that $f_i=n_ig_i$ and $S_{f_i}=S_{(n_i,g_i)}$. Hence, $c_{f_i}|_{P_{i-1}}=(c_{n_i}|_{P_{i-1}})(c_{g_i}|_T)$ is the composition of the morphism $\psi_i:=c_{n_i}|_{P_{i-1}}\in\Hom_\E(P_{i-1},T)$ and the automorphism $\alpha_i=c_{g_i}|_T\in\Aut_\F(T)$. So 
\[\phi=\psi_1\alpha_1\psi_2\alpha_2\cdots \psi_k\alpha_k.\]
By assumption, $\alpha_1,\dots,\alpha_{k-1}$ induce automorphisms of $\E$. Hence, in the decomposition of $\phi$ above we can move $\alpha_1,\dots,\alpha_{k-1}$ to the end to obtain $\phi=\psi\alpha$ where $\psi\in\Hom_\E(P,T)$ is the composition of conjugates of $\psi_1,\dots,\psi_k$ and $\alpha=\alpha_1\alpha_2\dots\alpha_k\in\Aut_\F(T)$. Thus the Frattini condition holds and $\E$ is $\F$-invariant.

\smallskip 

Suppose now $\E$ is saturated and $(\L,\Delta,S)$ is of objective characteristic $p$. To show that $\E\unlhd\F$, it remains to prove the extension condition as stated in \cite[Definition~I.6.1]{Aschbacher/Kessar/Oliver:2011}. For the proof let $\alpha=c_n|_T$ with $n\in N_\N(T)$. By \cite[Lemma~3.5]{Chermak:2015}, we have $n\in N_\N(T)\subseteq N_\N(TC_S(T))$. So $\hat{\alpha}:=c_n|_{TC_S(T)}$ extends $\alpha$. For $s\in TC_S(T)$ we see using \cite[Lemma~2.9]{Chermak:2015} that $(s^{-1},n^{-1},s,n)\in\D$ and $[s,\hat{\alpha}]=s^{-1}(s\hat{\alpha})=s^{-1}s^n=(n^{-1})^sn\in\N\cap S=T$. By definition of $\E$, the extension condition follows.
\end{proof}

\begin{lemma}\label{L:NLTAutL}
Let $(\L,\Delta,S)$ be a locality and $T\leq S$ strongly $\F_S(\L)$-closed such that
\begin{equation}\label{E:ObjectCondition}
P\cap T\in\Delta\mbox{ for every }P\in\Delta\mbox{ with }O_p(\L)\leq P.
\end{equation}
Then the following hold for every $f\in N_\L(T)$:
\begin{itemize}
 \item[(a)] If $w\in\W(\L)$, then $w\in\D$ if and only if $S_w\cap T\in\Delta$.
 \item[(b)] $\L=\D(f)$ and $c_f\in\Aut(\L,\Delta,S)$.
 \item[(c)] Let $\N\unlhd\L$ with $\N\cap S=T$. Then setting $\Gamma:=\{P\in\Delta\colon P\leq T\}$, the triple $(\N,\Gamma,T)$ is a locality and $c_f|_\N\in\Aut(\N,\Gamma,T)$. Moreover $(\N,\Gamma,T)$ is of objective characteristic $p$ if $(\L,\Delta,S)$ is of objective characteristic $p$.
 \item[(d)] $c_f|_S\in\Aut(\F)$ and, if $\N$ is as in (c), then $c_f|_T\in\Aut(\F_T(\N))$.
\end{itemize}
\end{lemma}

\begin{proof}
Let $w\in\W(\L)$. Notice that $O_p(\L)\leq S_w$ and, as $(\L,\Delta,S)$ is a locality, $w\in\D$ if and only $S_w\in\Delta$. Hence (a) follows from \eqref{E:ObjectCondition}. 

\smallskip

Assume now $w=(x_1,\dots,x_k)\in\D$. Then by (a) and as $T$ is strongly closed, we have $w\in\D$ via some $P_0,P_1,\dots,P_k\in\Gamma$. This implies that    \[u:=(f^{-1},x_1,f,f^{-1},x_2,f,\dots,f^{-1},x_k,f)\in\D\] via $P_0^f,P_0,P_1,P_1^f,P_1,\dots,P_{k-1}^f,P_{k-1},P_k,P_k^f$. In particular, applying this argument to words of length one, we see that $\L\subseteq\D(f)$. Using Notation~\ref{N:PartialHomWordMap}, it follows moreover from the partial group axioms that $wc_f^*=(x_1^f,\dots,x_k^f)\in\D$ and 
\[\Pi(wc_f^*)=\Pi(x_1^f,\dots,x_k^f)=\Pi(u)=\Pi((f^{-1})\circ w\circ (f))=\Pi(w)c_f;\] 
more precisely, we use here Lemma~\ref{L:AddOnes} to conclude that $\Pi(u)=\Pi(f^{-1},x_1,\One,x_2,\dots,x_{k-1},\One,x_k,f)=\Pi((f^{-1})\circ w\circ (f))$. This shows that $c_f\colon\L\rightarrow\L$ is a homomorphism of partial groups. As $c_f$ is bijective with inverse map $c_{f^{-1}}$, and since $c_{f^{-1}}$ is by a similar argument a homomorphism of partial groups, \cite[Lemma~2.17]{Henke:2020} gives that $c_f$ is an automorphism of $\L$. As $\Delta$ is closed under $\F$-conjugacy, this implies (b).

\smallskip

Let now $\N$ be as in (c). By \cite[Lemma~3.1(c)]{Chermak:2015}, $T$ is a maximal $p$-subgroup of $\N$. Since $\Delta$ is closed under taking $\L$-conjugates and overgroups in $S$, the set $\Gamma$ is closed under taking $\N$-conjugates and overgroups in $T$. Moreover, it follows from (a) that $w\in\W(\N)$ is an element of $\D$ if and only if $w\in\D_\Gamma$. This shows that $(\N,\Gamma,T)$ is a locality. For every $P\in\Gamma\subseteq\Delta$, it follows moreover from Lemma~\ref{L:MSCharp}(a) that $N_\N(P)\unlhd N_\L(P)$ is of characteristic $p$ if $(\L,\Delta,S)$ is of objective characteristic $p$. 

\smallskip

As $\N$ is a partial normal subgroup of $\L$, the map $c_f$ restricts to an automorphism of $\N$. As $T$ is strongly closed and $\Delta$ is closed under $\F_S(\L)$-conjugacy, it follows that $c_f\in\Aut(\N,\Gamma,T)$. This proves (c). Part (d) follows from (b),(c) and \cite[Lemma~2.21(b)]{Henke:2020}.
\end{proof}

\section{Varying the object sets of linking localities}\label{SS:VaryObjects}

If $(\L^+,\Delta^+,S)$ is a linking locality over a fusion system $\F$ and $\Delta$ is an $\F$-closed collection of subgroups of $S$ with $\F^{cr}\subseteq\Delta\subseteq\Delta^+$, then writing $\L=\L^+|_\Delta$ for the restriction of $\L^+$ to $\Delta$, it follows from Alperin's fusion theorem that $(\L,\Delta,S)$ is a locality over $\F$. Moreover, it is easy to see that  $N_\L(P)=N_{\L^+}(P)$ for all $P\in\Delta$ and that $(\L,\Delta,S)$ is thus a linking locality over $\F$. Writing $\fN(\H)$ for the set of partial normal subgroups of a partial group $\H$, notice that we have a map
\[\Phi_{\L^+,\L}\colon \fN(\L^+)\rightarrow\fN(\L),\N^+\mapsto \N^+\cap\L.\]
The next theorem will allow us to move between linking localities with different object sets. As in \cite[Definition~1]{Henke:2015} we define a subgroup $P\leq S$ to be \emph{subcentric} if $O_p(N_\F(Q))\in\F^c$ for some (and thus for every) fully $\F$-normalized $\F$-conjugate $Q$ of $\F$. Equivalently, by \cite[Lemma~3.1]{Henke:2015}, $P$ is subcentric if and only if, for some (and thus for every) fully $\F$-normalized $\F$-conjugate $Q$ of $\F$, the normalizer $N_\F(Q)$ is constrained. The set of subcentric subgroups will be denoted by $\F^s$.

\begin{definition}
We call a linking locality $(\L,\Delta,S)$ over $\F$ a \emph{subcentric locality} if $\Delta=\F^s$.
\end{definition}

\begin{theorem}\label{T:VaryObjects}
Let $(\L,\Delta,S)$ be a linking locality over $\F$ and let $\Delta^+$ be an $\F$-closed collection of subgroups of $S$ such that $\Delta\subseteq \Delta^+\subseteq \F^s$.
\begin{itemize}
\item [(a)] There exists a linking locality $(\L^+,\Delta^+,S)$ over $\F$ such that $\L=\L^+|_\Delta$. Moreover, $(\L^+,\Delta^+,S)$ is unique up to an isomorphism which restricts to the identity on $\L$; that is, if $(\wL^+,\Delta^+,S)$ is another linking locality over $\F$ with $\wL^+|_{\Delta}=\L$, then there exists an isomorphism of partial groups $\beta\colon\L^+\rightarrow\wL^+$ which is the identity on $\L$.
\item [(b)] If $(\L^+,\Delta^+,S)$ is a linking locality over $\F$ with $\L^+|_\Delta=\L$, then the map $\Phi_{\L^+,\L}$ is an inclusion-preserving bijection such that $\Phi_{\L^+,\L}^{-1}$ is also inclusion-preserving. If $\N^+\unlhd\L^+$ and $\N:=\L\cap\N^+\unlhd\L$ such that $\F_{S\cap\N}(\N)$ is $\F$-invariant, then $\F_{S\cap\N^+}(\N^+)=\F_{S\cap\N}(\N)$. 
\item [(c)] The set $\F^s$ of subcentric subgroups of $S$ is $\F$-closed and contains $\Delta$. In particular, there exists a subcentric locality $(\L^s,\F^s,S)$ over $\F$ with $\L^s|_\Delta=\L$, and such $\L^s$ is unique up to an isomorphism which restricts to the identity on $\L$. 
\end{itemize}
\end{theorem}

\begin{proof}
Part (a) follows from \cite[Theorem~7.2(a),(b)]{Henke:2015}, and part (c) follows then from Proposition~3.3 and Lemma~6.1 in \cite{Henke:2015}. Part (b) is \cite[Theorem~C(a),(b)]{Henke:2020}. 
\end{proof}

We point out that, alternatively, part (a) and most statements in part (b) of Theorem~\ref{T:VaryObjects}  follow from Chermak's Theorems~A1 and A2 in \cite{ChermakII}. The relevant statements in part (b) were indeed first proved by Chermak.

\begin{notation}\label{N:VaryObjects}
Whenever $(\L^+,\Delta^+,S)$ and $(\L,\Delta,S)$ are linking localities over the same fusion system $\F$ with $\Delta\subseteq\Delta^+$ and $\L=\L^+|_\Delta$, and whenever $\N$ is a partial normal subgroup of $\L$, we write $\N^+$ for $\Phi_{\L^+,\L}^{-1}(\N)$ (which is well-defined by Theorem~\ref{T:VaryObjects}(b)). So $\N^+$ denotes the unique partial normal subgroup of $\L^+$ with $\N^+\cap\L=\N$.

\smallskip

Similarly, if $(\L^s,\F^s,S)$ is a subcentric locality over $\F$ and $\L=\L^s|_\Delta$, then for every $\N\unlhd \L$, we write $\N^s$ for the unique partial normal subgroup of $\L^s$ with $\N^s\cap\L=\N$. 
\end{notation}

\begin{lemma}\label{L:TakePlusInS}
Let $(\L^+,\Delta^+,S)$ and $(\L,\Delta,S)$ be linking localities over the same fusion system $\F$ with $\Delta\subseteq\Delta^+$ and $\L=\L^+|_\Delta$. For every $R\leq S$, we have $R\unlhd \L$ if and only if $R\unlhd\L^+$. In particular, if $R\unlhd\L$, then $R^+=R$. 
\end{lemma}

\begin{proof}
Let $R\leq S$ be arbitrary. By Lemma~\ref{L:OpL}, $R\unlhd \L$ if and only if $\L=N_\L(R)$. Hence, using \cite[Proposition~5]{Henke:2015} it follows that
\[R\unlhd\L\Longleftrightarrow \L=N_\L(R)\Longleftrightarrow R\unlhd\F.\]
We have the same equivalence with $\L^+$ in place of $\L$. So $R\unlhd \L$ if and only if $R\unlhd\F$, and this is the case if and only if $R\unlhd\L^+$. As $R\cap\L=R$, it follows $R^+=R$ if $R\unlhd\L$.   
\end{proof}

\begin{lemma}\label{L:NNsTinN} 
Suppose $(\L,\Delta,S)$ and $(\L^s,\F^s,S)$ are linking localities over the same fusion system $\F$ such that $\L=\L^s|_\Delta$. Let $\N\unlhd\L$, set $T:=S\cap\N$ and adapt Notation~\ref{N:VaryObjects}. Then $Q:=O_p(N_{\L^s}(TC_S(T)))\in\F^{cr}$ and $N_{\N^s}(T)=N_\N(T)=N_\N(TC_S(T))=N_\N(Q)$. 
\end{lemma}

\begin{proof}
As $TC_S(T)\in\F^c\subseteq\F^s$, the normalizer $N_{\L^s}(TC_S(T))$ is a group. By definition of $Q$, we have $TC_S(T)\leq Q$ and $N_{\L^s}(TC_S(T))\leq  N_{\L^s}(Q)$. Since $T$ is strongly closed, it follows $N_{\L^s}(Q)=N_{\L^s}(TC_S(T))$ and thus $Q=O_p(N_{\L^s}(Q))$. Hence, \cite[Lemma~6.2]{Henke:2015} applied with $(\L^s,\F^s,S)$ in place of $(\L,\Delta,S)$ gives $Q\in\F^{cr}\subseteq\Delta$. In particular, $N_{\L^s}(TC_S(T))=N_{\L^s}(Q)\subseteq\L$. Using \cite[Lemma~3.5]{Chermak:2015}, it follows $N_{\N^s}(T)\subseteq N_{\L^s}(TC_S(T))\subseteq\L$. Thus, as $\N=\N^s\cap\L$ and $T$ is strongly closed, we have $N_{\N^s}(T)=N_\N(T)=N_\N(TC_S(T))=N_\N(Q)$.
\end{proof}

\section{Strongly closed subgroups and partial normal subgroups}

The following lemma is partly inspired by \cite[Corollary~1.5(a),(b)]{ChermakII}.

\begin{lemma}\label{L:NLTCLT}
Let $(\L,\Delta,S)$ be a locality over a fusion system $\F$ such that $\F^{cr}\subseteq\Delta$. Suppose $T\leq S$ is strongly closed in $\F$. Then the following hold.
\begin{itemize}
\item [(a)] $(N_\L(T),\Delta,S)$ is a locality over $N_\F(T)$ and $N_\F(T)^{cr}\subseteq\F^{cr}\subseteq \Delta$. 
\item [(b)] If $(\L,\Delta,S)$ is a linking locality, then $(N_\L(T),\Delta,S)$ is a linking locality over $N_\F(T)$. 
\item [(c)] We have $C_\L(T)\unlhd N_\L(T)$. Moreover, if $\F^c\subseteq \Delta$, then $\F_{C_S(T)}(C_\L(T))=C_\F(T)$.
\end{itemize}
\end{lemma}

\begin{proof}
We use throughout that $\F=\F_S(\L)$ is saturated by Theorem~\ref{T:Saturation}. 

\smallskip

As $T\unlhd S$, it follows from \cite[Lemma~2.13]{Chermak:2015} that $(N_\L(T),\Delta,S)$ is a locality. Clearly $\F_S(N_\L(T))\subseteq N_\F(T)$. Observe that $T\unlhd N_\F(T)$ is contained in every element of $N_\F(T)^{cr}$ by Lemma~\ref{L:OpFinFRadical}. Moreover, as $T$ is strongly closed, whenever $T\leq P\leq S$, we have $N_\F(P)=N_{N_\F(T)}(P)$ and $P^\F=P^{N_\F(T)}$. Hence,
\[N_\F(T)^{cr}=\{P\in\F^{cr}\colon T\leq S\}\subseteq \F^{cr}\subseteq\Delta.\]
Thus, to prove (a), it is sufficient to show that $N_\F(T)\subseteq \F_S(N_\L(T))$. Notice that $N_\F(T)$ is saturated by \cite[Theorem~I.5.5]{Aschbacher/Kessar/Oliver:2011} as $T\unlhd S$ and $\F$ is saturated. Fix $P\in N_\F(T)^{cr}$. By Alperin's Fusion Theorem \cite[Theorem~I.3.5]{Aschbacher/Kessar/Oliver:2011}, we only need to show that $\Aut_\F(P)=\Aut_{N_\F(T)}(P)\subseteq \F_S(N_\L(T))$. Notice that $N_\L(P)\subseteq N_\L(T)$ as $T\leq P$ and $T$ is strongly closed in $\F$. Hence, using Lemma~\ref{L:LocalityFusionSystem}(a) one sees that $\Aut_\F(P)=\{c_f|_P\colon f\in N_\L(P)\}\subseteq \F_S(N_\L(T))$. This proves (a).

\smallskip 

For the proof of (b) suppose now that $(\L,\Delta,S)$ is a linking locality. If $P\in\Delta$, then $TP\in\Delta$ and thus $N_\L(TP)$ is a group of characteristic $p$. Using that $T$ is strongly closed, one sees that $N_{N_\L(T)}(P)=N_{N_\L(TP)}(P)$. Hence, $N_{N_\L(T)}(P)$ is of characteristic $p$ by Lemma~\ref{L:MSCharp}(b). It follows now from (a) that $(N_\L(T),\Delta,S)$ is a linking locality over $N_\F(T)$, i.e. (b) holds.

\smallskip

It is a special case of Lemma~\ref{L:CentralizerPartialNormal} that $C_\L(T)\unlhd N_\L(T)$. Assume now that $\F^c\subseteq\Delta$ and set $\L_0:=\L|_{\F^c}$. By \cite[Lemma~9.12]{Henke:2015} applied with $(\L_0,T,\{\id\})$ in place of $(\L,Q,K)$, we have $C_\F(T)=\F_{C_S(T)}(C_{\L_0}(T))\subseteq \F_{C_S(T)}(C_\L(T))$. Clearly $\F_{C_S(T)}(C_\L(T))\subseteq C_\F(T)$, so (c) holds.
\end{proof}

If $\N\unlhd \L$, then $T=\N\cap S$ is strongly closed by \cite[Lemma~3.1(a)]{Chermak:2015}. We will apply the lemma above in this situation. The following result of Chermak gives then actually some more information. Since it is crucial in our later proofs, we restate it here.

\begin{prop}\label{P:ChermakEndPartI}
Let $(\L,\Delta,S)$ be a locality of objective characteristic $p$, let $\N\unlhd\L$ and set $T:=\N\cap S$. Suppose $\K$ is a partial normal subgroup of $N_\L(T)$ such that $\K\subseteq C_\L(T)$. Then $\<\K,\N\>\unlhd\L$ and $S\cap \<\K,\N\>=(\K\cap S)T$. In particular, $(\K\cap S)T$ is strongly closed in $\F_S(\L)$. 
\end{prop}

\begin{proof}
The first part is a restatement of \cite[Proposition~5.5]{Chermak:2015}. The last sentence follows then from the fact that the intersection of a partial normal subgroup of $\L$ with $S$ is strongly closed by \cite[Lemma~3.1(a)]{Chermak:2015}.
\end{proof}

\begin{corollary}\label{C:ChermakEndPartI}
Let $(\L,\Delta,S)$ be a locality of objective characteristic $p$, let $\N\unlhd\L$ and set $T:=\N\cap S$. Then $\<\N,C_\L(T)\>\unlhd\L$ and $\<\N,C_\L(T)\>\cap S=TC_S(T)$. In particular, $TC_S(T)$ is strongly closed in $\F_S(\L)$.
\end{corollary}

\begin{proof}
This is an immediate consequence of Lemma~\ref{L:NLTCLT}(c) and Proposition~\ref{P:ChermakEndPartI}.
\end{proof}

\section{Im-partial subgroups}\label{SS:Impartial}

Following Chermak we use the following definition.

\begin{definition}
 Let $\L$ and $\L_0$ be partial groups with products $\Pi\colon \D\rightarrow\L$ and $\Pi_0\colon \D_0\rightarrow \L_0$ respectively. Then we say that $\L_0$ (or more precisely $(\L_0,\D_0,\Pi_0)$) is an \emph{im-partial} subgroup of $\L$ if $\L_0\subseteq\L$, $\D_0\subseteq\D$ and $\Pi_0=\Pi|_{\D_0}$. 
\end{definition}

If $\L_0$ is an im-partial subgroup of $\L$ then the inclusion map $\L_0\hookrightarrow \L$ is a homomorphism of partial groups. The other way around, any image of a homomorphism of partial groups can be regarded as an im-partial subgroup (which is however not a partial subgroup in general). This justifies the name. 

\smallskip

Im-partial subgroups of localities occur very natural. For example, if $(\L,\Delta,S)$ is a locality over $\F$ and $\Gamma\subseteq \Delta$ is $\F$-closed, then the restriction $\L|_\Gamma$ is an im-partial subgroup of $\L$. In \cite[Subsection~9.2]{Henke:2015} restrictions were considered in a more general context. For the convenience of the reader we summarize this construction here. For that let $(\L,\Delta,S)$ be a locality, $\H$ a partial subgroup of $\L$ and $\Gamma$ a set of subgroups of $T:=S\cap\H$ such that $\Gamma$ is closed under passing to  $\H$-conjugates and overgroups in $T$. Suppose furthermore that there exists some subgroup $X\leq S$ such that the following properties hold:
\begin{itemize}
 \item [(Q1)] For all $P\in\Gamma$, we have $\<P,X\>\in\Delta$.
 \item [(Q2)] For all $P_1,P_2\in\Gamma$, we have $N_\H(P_1,P_2)\subseteq N_\L(\<P_1,X\>,\<P_2,X\>)$.  
\end{itemize}
Set
\[\H|_\Gamma:=\{f\in\H\colon S_f\cap T\in\Gamma\}.\]
Let $\D_0$ be the set of words $(f_1,\dots,f_n)$ in $\H$ such that there exist $P_0,P_1,\dots,P_n\in\Gamma$ with $P_{i-1}^{f_i}=P_i$ for all $i=1,\dots,n$. We allow here the case $n=0$ so that the empty word is an element of $\D_0$. By \cite[Lemma~9.6]{Henke:2015}, the set $\H|_\Gamma$ together with the restriction of the inversion on $\L$ to $\H|_\Gamma$ and with the product $\Pi_0:=\Pi|_{\D_0}\colon \D_0\rightarrow \H|_\Gamma$ forms a partial group, which is an im-partial subgroup of $\L$. Whenever we write $\H|_\Gamma$ in the following text, we mean implicitly that we regard $\H|_\Gamma$ as a partial group (and thus an im-partial subgroup of $\L$) in this way. We will now consider the case which is relevant in the present paper.

\bigskip

\textbf{For the remainder of this section let $(\L,\Delta,S)$ be a subcentric locality over a fusion system $\F$.}

\bigskip

If $X\leq S$ is fully $\F$-normalized and $\H=N_\L(X)$, then (Q2) holds for any set $\Gamma$ of subgroups of $T:=N_S(X)$. Moreover, if $(\L,\Delta,S)$ is a subcentric locality over $\F$, then (Q1) holds with $\Gamma=N_\F(X)^s$ as is implied by the following lemma.

\begin{lemma}\label{L:NFXs}
Let $\F$ be a saturated fusion systems and $X\leq S$ be fully $\F$-normalized. Then $N_\F(X)^s=\{P\leq N_S(X)\colon PX\in\F^s\}$. 
\end{lemma}

\begin{proof}
 By \cite[Lemma~3.14]{Henke:2015} we have $N_\F(X)^s\subseteq\{P\leq N_S(X)\colon PX\in\F^s\}$. Moreover, combining Lemma~3.4 and Lemma~3.16 in the same paper one gets the converse inclusion.
\end{proof}

Let now $(\L,\Delta,S)$ be a subcentric locality over $\F$. For every $X\leq S$ such that $X$ is fully $\F$-normalized, the set $N_\F(X)^s$ is $N_\F(X)$-closed. Moreover, by \cite[Lemma~9.12]{Henke:2015} we have $N_\F(X)=\F_{N_S(X)}(N_\L(X))$. So $N_\L(X)|_{N_\F(X)^s}$ forms a partial group as described above, and this partial group is an im-partial subgroup of $\L$. We denote this im-partial subgroup of $\L$ by $\bN_\L(X)$, i.e.
\[\bN_\L(X):=N_\L(X)|_{N_\F(X)^s}\mbox{ for all }X\in\F^f.\]
By \cite[Lemma~9.13]{Henke:2015}, the triple $(\bN_\L(X),N_\F(X)^s,N_S(X))$ is a subcentric locality over $N_\F(X)$. In the proof of the E-balance Theorem (Theorem~\ref{T:Ebalance}) at the very end of this paper, we will iterate the process we just described. In this context we will need the following lemma.

\begin{lemma}\label{L:NinNLXofY}
Let $\F$ be a saturated fusion system over $X$, let $X,Y$ be fully $\F$-normalized subgroups of $S$ which normalize each other. Suppose furthermore that $Y$ is fully $N_\F(X)$-normalized. Then $X$ is fully $N_\F(Y)$-normalized and $N_{N_\F(X)}(Y)=N_{N_\F(Y)}(X)$. Moreover,  
\[\Gamma:=N_{N_\F(X)}(Y)^s=\{P\leq N_S(X)\cap N_S(Y)\colon PXY\in \F^s\}\]
and, for every subcentric locality $(\L,\Delta,S)$ over $\F$, we have 
\[\bN_{\bN_\L(X)}(Y)=(N_\L(X)\cap N_\L(Y))|_{\Gamma}=\bN_{\bN_\L(Y)}(X).\]
\end{lemma}

\begin{proof}
It is well-known and easy to see that $N_{N_\F(X)}(Y)=N_{N_\F(Y)}(X)$. Applying Lemma~\ref{L:NFXs} twice we get moreover
\[N_{N_\F(X)}(Y)^s=\{P\leq N_S(X)\cap N_S(Y)\colon PY\in N_\F(X)^s\}=\{P\leq N_S(X)\cap N_S(Y)\colon (PY)X\in \F^s\}\]
and so 
\begin{equation}\label{E:NNFXYs}
 \Gamma:=N_{N_\F(X)}(Y)^s=\{P\leq N_S(X)\cap N_S(Y)\colon PXY\in \F^s\}.
\end{equation}
We show now that $X$ is fully $N_\F(Y)$ normalized. For the proof let 
\[\phi\in\Hom_{N_\F(Y)}(N_S(X)\cap N_S(Y),N_S(Y))\] such that $X\phi$ is  fully $N_\F(Y)$-normalized. As $X$ is fully $\F$-normalized, there exists then $\alpha\in\Hom_\F(N_S(X\phi),S)$ such that $X\phi\alpha=X$. Observe that $\phi\alpha\in\Hom_{N_\F(X)}(N_S(X)\cap N_S(Y),S)$ and $Y\phi\alpha=Y\alpha$. So the assumption that $Y$ is fully $N_\F(X)$-normalized yields
\[|N_S(X)\cap N_S(Y\alpha)|\leq |N_S(X)\cap N_S(Y)|.\]
At the same time, as $X\phi$ is fully $N_\F(Y)$-normalized, we have 
\[|N_S(X)\cap N_S(Y)|\leq |N_S(X\phi)\cap N_S(Y)|.\]
Moreover, $(N_S(X\phi)\cap N_S(Y))\alpha\leq N_S(X\phi\alpha)\cap N_S(Y\alpha)=N_S(X)\cap N_S(Y\alpha)$. Hence,
\[|N_S(X)\cap N_S(Y)|\leq |N_S(X\phi)\cap N_S(Y)|\leq |N_S(X)\cap N_S(Y\alpha)|\leq |N_S(X)\cap N_S(Y)|\]
and thus equality holds everywhere above. In particular, $|N_S(X\phi)\cap N_S(Y)|=|N_S(X)\cap N_S(Y)|$. As $X\phi$ is fully $N_\F(Y)$-normalized, it follows therefore that $X$ is fully $N_\F(Y)$-normalized. 

\smallskip

Let $(\L,\Delta,S)$ be a subcentric locality over $\F$. The situation is now symmetric in $X$ and $Y$. Thus, it remains only to prove
\begin{equation}\label{E:bNbN}
 \bN_{\bN_\L(X)}(Y)=(N_\L(X)\cap N_\L(Y))|_{\Gamma}.
\end{equation}
Observe first that (Q1) and (Q2) hold with $\H=N_\L(X)\cap N_\L(Y)$ and with $XY$ in place of $X$; to see that (Q1) holds in this setting one needs to apply \eqref{E:NNFXYs}. So $(N_\L(X)\cap N_\L(Y))|_{\Gamma}$ is well-defined. Write $\Pi_0\colon \D_0\rightarrow \bN_\L(X)$ for the partial product on $\bN_\L(X)$, $\Pi_1\colon \D_1\rightarrow \bN_{\bN_\L(X)}(Y)$ for the partial product on $\bN_{\bN_\L(X)}(Y)$ and $\Pi_2\colon\D_2\rightarrow (N_\L(X)\cap N_\L(Y))|_{\Gamma}$ for the partial product on $(N_\L(X)\cap N_\L(Y))|_{\Gamma}$. Notice that the set $\bN_{\bN_\L(X)}(Y)$ corresponds to the words of length one in $\D_1$ and similarly $(N_\L(X)\cap N_\L(Y))|_{\Gamma}$ corresponds to the words of length one in $\D_2$. Hence, as $\Pi_1=\Pi|_{\D_1}$ and $\Pi_2=\Pi|_{\D_2}$, it is sufficient to prove $\D_1=\D_2$. Set $T:=N_S(X)$. If $w=(f_1,\dots,f_n)\in\W(\bN_\L(X))$, then it follows from \cite[Lemma~9.7(b)]{Henke:2015} that 
\[T_w:=\{x\in T\colon \exists x=x_0,x_1,\dots,x_n\in T\mbox{ such that }x_{i-1}\in\D_0(f_i)\mbox{ and }x_{i-1}^{f_i}=x_i\mbox{ for }i=1,2,\dots,n\}\]
equals $S_w\cap T$. Using this for words of length one, it follows in particular that 
\begin{equation}\label{E:NbNLXY}
N_{\bN_\L(X)}(Y)=N_\L(Y)\cap\bN_\L(X) 
\end{equation}
We can conclude now that 
\begin{eqnarray*}
 \D_1&=&\{w\in\W(N_{\bN_\L(X)}(Y))\colon T_w \cap N_T(Y)\in \Gamma\}\;\;\mbox{ (by Definition of $\Gamma$ and of $\bN_{\bN_\L(X)}(Y)$)}\\
&=& \{w\in\W(N_\L(Y)\cap \bN_\L(X))\colon S_w\cap T\cap N_S(Y)\in \Gamma\}\;\;\mbox{ (as $T_w=S_w\cap T$ and by \eqref{E:NbNLXY})}\\
&\subseteq & \{w\in\W(N_\L(X)\cap N_\L(Y))\colon S_w\cap T\cap N_S(Y)\in \Gamma\}\\
&=& \D_2.
\end{eqnarray*}
Observe now that $w\in\W(\bN_\L(X))$ for every $w\in \W(N_\L(X)\cap N_\L(Y))$ with $S_w\cap T\cap N_S(Y)\in \Gamma=N_{N_\F(X)}(Y)^s$. This is because for such $w$ we have $Y\leq S_w\cap T\cap N_S(Y)$ and so $S_w\cap T\cap N_S(Y)\in N_\F(X)^s$ by Lemma~\ref{L:NFXs}. This shows that the inclusion above is an equality as well. Hence $\D_1=\D_2$ and \eqref{E:bNbN} holds.
\end{proof}

\chapter{Internal central products of partial groups}\label{S:CentralProduct}

Internal central products of partial groups with two factors were defined and studied in \cite{Henke:2016}. We introduce here internal central products of partial groups with possibly more than two factors. This will be important in Chapters~\ref{S:Regular} and \ref{S:Components}. We warn the reader that the notion of an internal central product of two partial subgroups introduced here differs slightly from the definition given in \cite{Henke:2016}; see Remark~\ref{R:CentralProductDifference} below.

\smallskip

\textbf{Throughout this chapter let $\L$ be a partial group with product $\Pi\colon\D\rightarrow\L$.}

\begin{definition}\label{D:CentralProductPartialGroup}
Let $\L_1,\L_2,\dots,\L_k$ be partial subgroups of $\L$ where $k\geq 1$. We say that $\L$ is an \textit{(internal) central product} of $\L_1,\L_2,\cdots,\L_k$ and write $\L=\L_1*\L_2*\cdots *\L_k$ if the following conditions hold:
\begin{itemize}
 \item [($\P$)] We have $\L=\L_1\L_2\cdots \L_k$.
 \item [($\C 1$)] We have $\L_1\times\L_2\times\cdots\times\L_k\subseteq\D$.
 \item [($\C 2$)] Whenever we are given words $w_j=(f_{1j},f_{2j},\dots,f_{kj})\in \L_1\times\L_2\times\cdots\times\L_k$ for $j=1,\dots,n$, then setting $w=(\Pi(w_1),\Pi(w_2),\dots,\Pi(w_n))$ and $u_i:=(f_{i1},f_{i2},\dots,f_{in})\in\W(\L_i)$ for $i=1,\dots,k$, the following equivalence holds: 
\begin{eqnarray*}
w\in\D\Longleftrightarrow u_i\in\D \mbox{ for all }i=1,\dots,k.
\end{eqnarray*} 
Moreover, if $w\in\D$, then
\begin{eqnarray*}
\Pi(w)=\Pi(\Pi(u_1),\dots,\Pi(u_k)).
\end{eqnarray*}
\end{itemize}
We say that $\L_1,\dots,\L_k$ \emph{form an internal central product in $\L$} if ($\C 1$) and ($\C 2$) hold. 
\end{definition}

\begin{remark}\label{R:CentralProductMatrix}
The reader might want to think of the property ($\C 2$) as follows: Suppose we are given a $k\times n$ matrix
\begin{eqnarray*}
A= \begin{pmatrix}
  f_{11} & f_{12} & \hdots & f_{1n}\\
 f_{21} & f_{22} & \hdots & f_{2n}\\
\vdots & \vdots &     & \vdots \\
f_{k1} & f_{k2} & \hdots & f_{kn}
 \end{pmatrix}
\end{eqnarray*}
such that the $i$th row $u_i=(f_{i1},f_{i2},\dots,f_{in})$ is an element of $\W(\L_i)$ for all $i=1,\dots,k$. Write $w$ for the word of length $n$ whose $j$th entry equals the the product of the word written in the $j$th column of $A$ for all $j=1,\dots,n$.  Property ($\C 2$) says that, whenever we are in this situation, we have $w\in\D$ if and only if $u_i\in\D$ for all $i=1,\dots,k$. Moreover, if so, then
\[\Pi(w)=\Pi(\Pi(u_1),\dots,\Pi(u_k)).\]\qed
\end{remark}

\begin{remark}\label{R:CentralProductDifference}
If $\L$ is a central product of $\L_1$ and $\L_2$ in the sense of Definition~\ref{D:CentralProductPartialGroup}, then one easily observes that $\L$ is also a central product of $\L_1$ and $\L_2$ in the sense defined in \cite[Definition~6.1]{Henke:2016} (so in particular all the results proved in that paper apply). However, the converse might not hold in general. It turns out that the two notions coincide if there is a locality structure on each of the partial subgroups $\L_i$. We plan to provide the details of this a forthcoming paper with Valentina Grazian. The main reason why we choose to give a different definition here is that it allows us to prove Lemma~\ref{L:CentralProductAssociative} below. 
\end{remark}

We have the following lemma. By $\Sigma_k$ we denote the symmetric group on $k$ letters.

\begin{lemma}\label{L:CentralProductPartialSubgroup}
Let $\L_1,\dots,\L_k$ be partial subgroups of $\L$ which form a central product in $\L$. Then the following hold:
\begin{itemize}
 \item [(a)] Given $g_i\in\L_i$ for all $i=1,\dots,k$ and $\sigma\in\Sigma_k$, we have $(g_{1\sigma},g_{2\sigma},\dots,g_{k\sigma})\in\D$ and $\Pi(g_1,g_2,\dots,g_k)=\Pi(g_{1\sigma},g_{2\sigma},\dots,g_{k\sigma})$.
 \item [(b)] The product $\L_1\L_2\cdots\L_k$ is a partial subgroup and $\L_1\L_2\cdots\L_k=\L_1*\L_2 *\cdots *\L_k$.
 \item [(c)] If $g_i\in\L_i$ for $i=1,\dots,k$, then $\Pi(g_1,\dots,g_k)^{-1}=\Pi(g_1^{-1},\dots,g_k^{-1})$. 
\end{itemize}
\end{lemma}

\begin{proof}
Let $g_i\in\L_i$ for $i=1,\dots,k$. As every element of $\Sigma_k$ is a product of transpositions of the form $(l,l+1)$ with $1\leq l<k$, for the proof of (a), we may assume without loss of generality that $\sigma=(l,l+1)$ for some $1\leq l<k$. So we need to show that 
\[w:=(g_1,\dots,g_{l-1},g_{l+1},g_l,g_{l+2},\dots,g_k)\] 
is an element of $\D$ and that $\Pi(w)=\Pi(g_1,\dots,g_k)$. Consider the $k\times k$ matrix
\begin{eqnarray*}
A= \begin{pmatrix}
  g_1 & \hdots & \One      &  \One & \One       & \One &\hdots & \One \\
  \vdots & \ddots & \vdots &  \vdots & \vdots   &  \vdots & & \vdots \\
 \One & \hdots & g_{l-1}   &  \One & \One       & \One &\hdots & \One \\
 \One & \hdots & \One      & \One & g_l         & \One & \hdots & \One\\
 \One & \hdots & \One      & g_{l+1} & \One     & \One & \hdots & \One \\
 \One &\hdots & \One       & \One & \One        & g_{l+2} & \hdots & \One\\
 \vdots & & \vdots         & \One & \One        & \vdots & \ddots & \vdots\\
 \One &\hdots & \One       & \vdots & \vdots    & \One & \hdots & g_k \\
 \end{pmatrix}
\end{eqnarray*}
Observe that the $i$th row $u_i$ of $A$ is an element of $\W(\L_i)$ for $i=1,\dots,k$. Moreover, by Lemma~\ref{L:AddOnes}(b), we have $u_i\in\D$ and $\Pi(u_i)=g_i$ for $i=1,\dots,k$. The same lemma yields also that $w$ is the word of length $k$ whose $j$th entry is the product of the word written into the $j$th column of $A$. Hence, property ($\C 2$) applied in the form of Remark~\ref{R:CentralProductMatrix} yields that $w\in\D$ and 
\[\Pi(w)=\Pi(\Pi(u_1),\cdots,\Pi(u_k))=\Pi(g_1,\dots,g_k).\]
This proves (a).

\smallskip

Set now $u=(g_1,\dots,g_k)$. By \cite[Lemma~1.4(f)]{Chermak:2015}, we have $u^{-1}=(g_k^{-1},\dots,g_1^{-1})\in\D$ and $f^{-1}=\Pi(u^{-1})$. As $\L_1,\dots,\L_k$ are partial subgroups of $\L$, we have $g_i^{-1}\in\L_i$ for $i=1,\dots,k$. Hence, using part (a), we can conclude that $f^{-1}=\Pi(g_1^{-1},\dots,g_k^{-1})$. This proves (c).

\smallskip

As any element $f\in \L_1\L_2\cdots\L_k$ is of the form $f=\Pi(g_1,\dots,g_k)$ where $(g_1,\dots,g_k)\in\L_1\times\L_2\times\cdots\times \L_k$, it follows from (c) that $\L_1\L_2\cdots\L_k$ is closed under inversion. Moreover, property ($\C 2$) yields that $\Pi(w)\in\L_1\L_2\cdots\L_k$ for all $w\in \W(\L_1\L_2\cdots\L_k)\cap\D$. So $\L_1\L_2\cdots\L_k$ is a partial subgroup of $\L$ and forms thus a partial group as well, whose product is the restriction of $\Pi$ to $\D\cap \W(\L_1\L_2\cdots\L_k)$. It is now easy to observe that $\L_1\L_2\cdots\L_k=\L_1*\L_2*\cdots *\L_k$. 
\end{proof}

If $\L_1,\dots,\L_k$ are partial subgroups of $\L$ and we write a term of the form $\L_1*\L_2*\cdots *\L_k$, then this indicates implicitly that $\L_1,\dots,\L_k$ form a central product in $\L$. In particular, writing $\L=(\L_1*\L_2*\cdots *\L_l)*(\L_{l+1}*\L_{l+2}*\cdots *\L_k)$ means that $\L_1,\dots,\L_l$ form a central product, that $\L_{l+1},\dots,\L_k$ form a central product, and that $\L$ is a central product of $\L_1\L_2\cdots \L_l$ and $\L_{l+1}\L_{l+2}\cdots \L_k$. We have the following lemma.

\begin{lemma}\label{L:CentralProductAssociative}
Let $\L_1,\dots,\L_k$ be partial subgroups of $\L$. The following two conditions are equivalent:
\begin{itemize}
 \item [(i)] $\L=\L_1*\L_2*\cdots *\L_k$;
 \item [(ii)] $\L=(\L_1*\L_2*\cdots *\L_l)*(\L_{l+1}*\L_{l+2}*\cdots *\L_k)$.
\end{itemize}
\end{lemma}

\begin{proof}
Assume first that (i) holds. Using the axioms of a partial group, one sees that property ($\P$) implies $\L=(\L_1\L_2\cdots\L_l)(\L_{l+1}\L_{l+2}\cdots\L_k)$. Moreover, property ($\C 1$) implies that $\L_1\times \L_2\times \cdots \times \L_l$, $\L_{l+1}\times\L_{l+2}\times\cdots\times\L_k$ and $(\L_1\L_2\cdots\L_l)\times (\L_{l+1}\L_{l+2}\cdots \L_k)$ are subsets of $\D$. 

\smallskip

To show that $\L_1,\dots,\L_l$ form a central product, it remains only to show that ($\C 2$) holds with $l$ in place of $k$. Let $A$ be an $l\times n$-matrix whose $i$th row $u_i$ is a word in $\L_i$ for $i=1,\dots,l$. Write $w$ for the word of length $n$ whose $j$th entry equals the product of the word written into the $j$th row of $A$. Let $E$ be the $(k-l)\times n$ matrix all of whose entries are equal to $\One$, and form the $k\times n$ block matrix 
\begin{eqnarray*}
 A'=\begin{pmatrix}
     A\\
\hline
E
    \end{pmatrix}.
\end{eqnarray*}
Consistent with our previous notation, we write $u_i$ for the $i$th row of $A'$ for $i=1,\dots,k$. Then $u_i\in\W(\L_i)$ for $i=1,\dots,k$. Lemma~\ref{L:AddOnes} yields that $u_i\in\D$, $\Pi(u_i)=\One$ for $i=l+1,\dots,k$ and that $w$ is also the word of length $n$ whose $j$th entry equals the product of the word written into the $j$th column of $A'$ for $j=1,\dots,n$. Hence, by ($\C 2$) we have
\[w\in\D\Longleftrightarrow u_i\in\D\mbox{ for all }i=1,\dots,k\Longleftrightarrow u_i\in\D\mbox{ for all }i=1,\dots,l.\] 
Moreover, if $w\in\D$, then 
\[\Pi(w)=\Pi(\Pi(u_1),\dots,\Pi(u_k))=\Pi(\Pi(u_1),\dots,\Pi(u_l),\One,\dots,\One)=\Pi(\Pi(u_1),\dots,\Pi(u_l)),\]
where the latter equality uses again Lemma~\ref{L:AddOnes}. This shows that $\L_1,\dots,\L_l$ form a central product. A similar argument shows that $\L_{l+1},\dots,\L_k$ form a central product; given a $(k-l)\times n$ matrix $B$ whose $i$the row is in $\W(\L_{l+i})$, take $E'$ to be the $l\times n$ matrix all of whose entries are equal to $\One$ and form the block matrix
\begin{eqnarray*}
 B'=\begin{pmatrix}
  E'\\
\hline
B
 \end{pmatrix}
\end{eqnarray*}
to prove ($\C 2$) with $\L_{l+1},\dots,\L_k$ in place of $\L_1,\dots,\L_k$. 

\smallskip

It remains to show that $\L_1\L_2\cdots \L_l$ and $\L_{l+1}\L_{l+2}\cdots\L_k$ form a central product or more precisely, that ($\C 2$) holds with $k=2$ and $(\L_1\L_2\cdots\L_l,\L_{l+1}\L_{l+2}\cdots\L_k)$ in place of $(\L_1,\L_2)$. Let
\begin{eqnarray*}
 A=\begin{pmatrix}
    g_1 & g_2 & \cdots & g_n\\
    h_1 & h_2 & \cdots & h_n
   \end{pmatrix}
\end{eqnarray*}
be a $2\times n$ matrix whose first row $w_1:=(g_1,\dots,g_n)$ is a word in $\L_1\L_2\cdots\L_l$ and whose second row $w_2:=(h_1,h_2,\dots,h_n)$ is a word in $\L_{l+1}\L_{l+2}\cdots\L_k$. By definition of $\L_1\L_2\cdots\L_l$, we may pick then an $l\times n$-matrix $B$ whose $i$th row $u_i$ is in $W(\L_i)$ for $i=1,\dots,l$ and such that $g_j$ is the product of the word written into the $j$th column of $B$. Similarly, by definition of $\L_{l+1}\L_{l+2}\cdots\L_k$, we can find a $(k-l)\times n$-matrix $C$ whose $i$th colum $u_{l+i}$ is in $\W(\L_{l+i})$ for $i=1,\dots,k-l$ and such that $h_j$ is the product of the word written into the $j$th column of $C$. If we form the block matrix 
\begin{eqnarray*}
A'=\begin{pmatrix}
 B\\
\hline
 C
\end{pmatrix},
\end{eqnarray*}
then $A'$ is the $k\times n$ matrix such that $u_i\in\W(\L_i)$ is the $i$th row of $A'$. Set $w:=(g_1h_1,\dots,g_nh_n)$. Note that the $j$th entry of $w$ equals both the product of the word written into the $j$th column of $A$ and the product written into the $j$th column of $A'$ for $j=1,\dots,n$. As $\L_1,\dots,\L_k$ form a central product, we have $w\in\D$ if and only if $u_i\in\D$ for all $i=1,\dots,n$. Moreover, if so then 
\[\Pi(w)=\Pi(\Pi(u_1),\dots,\Pi(u_k)).\]
As we have seen above, $\L_1,\dots,\L_l$ form a central product. Thus, we have $w_1\in\D$ if and only if $u_i\in\D$ for $i=1,\dots,l$; if so then $\Pi(w_1)=\Pi(\Pi(u_1),\dots,\Pi(u_l))$. Similarly, as $\L_{l+1},\dots,\L_k$ form a central product, we have $w_2\in\D$ if and only if $u_i\in\D$ for $i=l+1,\dots,k$; if so, then $\Pi(w_2)=\Pi(\Pi(u_{l+1}),\dots,\Pi(u_k))$. Putting everything together, we get
\[w\in\D\Longleftrightarrow u_i\in\D\mbox{ for all }i=1,\dots,k\Longleftrightarrow w_1,w_2\in\D.\]
Moreover, if $w\in\D$, then 
\[\Pi(w)=\Pi(\Pi(u_1),\dots,\Pi(u_k))=\Pi(\Pi(u_1),\dots,\Pi(u_l))\Pi(\Pi(u_{l+1}),\dots,\Pi(u_k))=\Pi(w_1)\Pi(w_2).\]
This proves that $\L_1\cdots\L_l$ and $\L_{l+1}\cdots\L_k$ form a central product. So we have shown that (i) implies (ii). 

\smallskip

Assume now that (ii) holds. We show first that ($\C 1$) holds. Let $g_i\in\L_i$ for $i=1,\dots,k$. As $\L_1,\dots,\L_l$ form a central product and $\L_{l+1},\dots,\L_k$ form a central product, the words $(g_1,\dots,g_l)$ and $(g_{l+1},\dots,g_k)$ are in $\D$. Hence, by Lemma~\ref{L:AddOnes}, the words $u_1:=(g_1,\dots,g_l,\One,\dots,\One)$ and $u_2:=(\One,\dots,\One,g_{l+1},\dots,g_k)$ of length $k$ are in $\D$. We consider now the $2\times k$-matrix $X$ whose $i$th row is $u_i$ for $i=1,2$. Observe that $u_1\in\W(\L_1\cdots\L_l)\cap\D$, $u_2\in\W(\L_{l+1}\cdots\L_k)\cap\D$ and $w=(g_1,\dots,g_k)$ is the word of length $k$ whose $j$th entry is the product of the word written into the $j$th row of $X$. So by Remark~\ref{R:CentralProductMatrix}, we have $w\in\D$. This proves property ($\C 1$). As $\L=(\L_1\cdots\L_l)(\L_{l+1}\cdots\L_k)$, it follows from ($\C 1$) and the axioms of a partial group that ($\P$) holds as well. Thus, it remains to show ($\C 2$). Consider a $k\times n$-matrix $A$ whose $i$th row $u_i$ is an element of $\W(\L_i)$ for $i=1,\dots,k$. Write $A$ as a block matrix
\begin{eqnarray*}
A=\begin{pmatrix}
 B\\
\hline
 C
\end{pmatrix}
\end{eqnarray*}
 where $B$ is an $l\times n$ matrix and $C$ is a $(k-l)\times n$-matrix. For each $Y\in\{A,B,C\}$ write $w_Y$ for the word of length $n$ whose $j$th entry equals the product of the word written in the $j$th column of $Y$ for all $j=1,\dots,n$. As $\L_1,\dots,\L_l$ form a central product, we have $w_B\in\D$ if and only if $u_i\in\D$ for $i=1,\dots,l$. Moreover, if so, then $\Pi(w_B)=\Pi(\Pi(u_1),\dots,\Pi(u_l))$. Similarly, as $\L_{l+1},\dots,\L_k$ form a central product, we have $w_C\in\D$ if and only if $u_i\in\D$ for $i=l+1,\dots,k$. Furthermore, if so, then $\Pi(w_C)=\Pi(\Pi(u_{l+1}),\dots,\Pi(u_k))$. Writing $A'$ for the $2\times n$-matrix whose first row is $w_B$ and whose second row is $w_C$, observe furthermore that the $j$th entry of $w_A$ equals the product of the word written into the $j$th column of $A'$. As $w_B\in \W(\L_1\cdots\L_l)$, $w_C\in\W(\L_{l+1}\cdots\L_k)$ and $\L=(\L_1\cdots\L_l)*(\L_{l+1}\cdots\L_k)$, we can conclude that the following equivalences hold:
\begin{eqnarray*}
 w_A\in\D\Longleftrightarrow w_B,w_C\in\D\Longleftrightarrow u_i\in\D\mbox{ for }i=1,\dots,k.
\end{eqnarray*}
Moreover, if $w_A\in\D$, then $\Pi(w_A)=\Pi(w_B)\Pi(w_C)=\Pi(\Pi(u_1),\dots,\Pi(u_l))\Pi(\Pi(u_{l+1}),\dots,\Pi(u_k))$. By ($\C 1$) we have $(\Pi(u_1),\dots,\Pi(u_k))\in\D$. Thus, if $w_A\in\D$, using the axioms of a partial group, we can conclude that $\Pi(w_A)=\Pi(\Pi(u_1),\dots,\Pi(u_k))$. This proves that ($\C 2$) holds. So we have shown that (ii) implies (i).
\end{proof}

\begin{lemma}\label{L:PermuteCentralProduct}
Let $\L_1,\dots,\L_k$ be partial subgroups of $\L$ and $\sigma\in\Sigma_k$. If $\L_1,\dots,\L_k$ form a central product in $\L$, then $\L_{1\sigma},\dots,\L_{k\sigma}$ form a central product in $\L$ and $\L_1*\L_2*\cdots *\L_k=\L_{1\sigma}*\L_{2\sigma}*\cdots *\L_{k\sigma}$. In particular, if $\L=\L_1*\L_2*\cdots *\L_k$, then we have $\L=\L_{1\sigma}*\L_{2\sigma}*\cdots *\L_{k\sigma}$.
\end{lemma}

\begin{proof}
As every element of $\Sigma_k$ is a product of transpositions of the form $(l,l+1)$ with $1\leq l<k$, we may assume without loss of generality that $\sigma=(l,l+1)$ for some $1\leq l<k$. Using Lemma~\ref{L:CentralProductAssociative} repeatedly, we may then reduce to the case that $k=2$. Suppose $\L_1$ and $\L_2$ form a central product. It follows from Lemma~\ref{L:CentralProductPartialSubgroup}(a) that ($\C 1$) holds with $k=2$ and $(\L_2,\L_1)$ in place of $(\L_1,\L_2)$ and that $\L=\L_1\L_2$ implies $\L=\L_2\L_1$. It remains thus to prove that ($\C 2$) holds for $(\L_2,\L_1)$. Let 
\begin{eqnarray*}
 A=
\begin{pmatrix}
 f_1 & f_2 & \hdots & f_n\\
 g_1 & g_2 & \hdots & g_n
\end{pmatrix}
\end{eqnarray*}
be a matrix such that the first row $u_1$ is in $\W(\L_2)$ and the second row $u_2$ is an element of $\W(\L_1)$. Set $w:=(f_1g_1,f_2g_2,\cdots,f_n g_n)$. By Lemma~\ref{L:CentralProductPartialSubgroup}(a), we have $w=(g_1f_1,g_2f_2,\dots,g_nf_n)$. Using that $\L=\L_1*\L_2$ and  applying Remark~\ref{R:CentralProductMatrix} to the matrix whose first row is $u_2$ and whose second row is $u_1$, we see that 
$w\in\D$ if and only $u_i\in\D$ for $i=1,2$. Moreover, if so then
\[\Pi(w)=\Pi(\Pi(u_2),\Pi(u_1)).\]
As $\L_1$ and $\L_2$ are partial subgroups, we have $\Pi(u_1)\in\L_2$ and $\Pi(u_2)\in\L_1$. Hence, Lemma~\ref{L:CentralProductPartialSubgroup}(a) allows us to conclude that $\Pi(w)=\Pi(\Pi(u_1),\Pi(u_2))$. This proves $\L=\L_2*\L_1$ as required.
\end{proof}

Lemma~\ref{L:PermuteCentralProduct} allows us to introduce the following definition.

\begin{definition}
Let $\mathfrak{N}$ be a set of partial subgroups of $\L$. Write $\L_1,\dots,\L_k$ for the pairwise distinct elements of $\mathfrak{N}$.
\begin{itemize}
\item We say that the elements of $\mathfrak{N}$ form a central product if $\L_1,\dots,\L_k$ form a central product. 
\item If $\L=\L_1*\L_2*\cdots *\L_k$, then $\L$ is called a central product of the elements of $\mathfrak{N}$.
\item If the elements of $\mathfrak{N}$ form a central product, then we set $\prod_{\N\in\mathfrak{N}}\N:=\L_1\L_2\cdots\L_k$. 
\end{itemize}
\end{definition}

The above definition will be important in Chapter~\ref{S:Components}. We also use it to formulate the next lemma.

\begin{lemma}\label{L:CentralProductsFactorsCentralize}
 Let $\mathfrak{N}_1$ and $\mathfrak{N}_2$ be disjoint sets of subgroups of $\L$. If the elements of $\mathfrak{N}_1\cup\mathfrak{N}_2$ form a central product, then the elements of $\mathfrak{N}_i$ form a central product for $i=1,2$. Moreover $\prod_{\N\in\mathfrak{N}_1}\N\subseteq C_\L(\prod_{\N\in\mathfrak{N}_2}\N)$. 
\end{lemma}

\begin{proof}
By Lemma~\ref{L:CentralProductPartialSubgroup}(b), we may reduce to the case that $\L$ is a central product of the elements in $\mathfrak{N}_1\cup\mathfrak{N}_2$. It follows then from Lemma~\ref{L:CentralProductAssociative} that the elements of $\mathfrak{N}_i$ form a central product for $i=1,2$ and that $\L$ is a central product of $\prod_{\N\in\mathfrak{N}_1}\N$ and $\prod_{\N\in\mathfrak{N}_2}\N$. Set $\L_i:=\prod_{\N\in\mathfrak{N}_i}\N$ for $i=1,2$ so that $\L=\L_1*\L_2$. 

\smallskip

Let $f\in\L_1$ and $g\in\L_2$. We need to show that $w:=(f^{-1},g,f)\in\D$ and $g^f=\Pi(w)=g$. By the axioms of a partial group we have $(f^{-1},f)\in\D$ and $\Pi(f^{-1},f)=\One$. Hence, by Lemma~\ref{L:AddOnes}(a), we have $(f^{-1},\One,f)\in\D$ and $\Pi(f^{-1},\One,f)=\One$. Applying Remark~\ref{R:CentralProductMatrix} to the matrix
\begin{eqnarray*}
 A=\begin{pmatrix}
 f^{-1} & \One & f\\
\One & g & \One
 \end{pmatrix}
\end{eqnarray*}
and using Lemma~\ref{L:AddOnes}(b), one sees now that indeed $w\in\D$ and $\Pi(w)=\Pi(\Pi(f^{-1},\One,f),\Pi(\One,g,\One))=\Pi(\One,g)=g$.  
\end{proof}

\begin{lemma}
Let $\L_1,\dots,\L_k$ be partial subgroups of $\L$ such that $\L=\L_1*\L_2*\cdots *\L_k$. Then $\L_i\unlhd\L$ for all $i\in\{1,\dots,k\}$.
\end{lemma}

\begin{proof}
By Lemma~\ref{L:PermuteCentralProduct}, it is sufficient to prove $\L_1\unlhd\L$. Moreover, by Lemma~\ref{L:CentralProductAssociative}, we may reduce to the case $k=2$. Let now $g\in\L_1$ and $f\in \L$ such that $g\in\D(f)$. Write $f=\Pi(f_1,f_2)$ where $f_i\in\L_i$ for $i=1,2$. By Lemma~\ref{L:CentralProductPartialSubgroup}(c), $f^{-1}=\Pi(f_1^{-1},f_2^{-1})$. Notice $w:=(f^{-1},g,f)\in\D$ and $g=\Pi(g,\One)$ by Lemma~\ref{L:AddOnes}. So ($\C 2$) yields $u_1:=(f_1^{-1},g,f_1)\in\D$ and $u_2:=(f_2^{-1},\One,f_2)\in\D$ with $g^f=\Pi(w)=\Pi(\Pi(u_1),\Pi(u_2))$. Combining the axioms of a partial group with Lemma~\ref{L:AddOnes}, we obtain $\Pi(u_2)=\Pi(f_2^{-1},f_2)=\One$ and $g^f=\Pi(\Pi(u_1),\One)=\Pi(u_1)\in\L_1$. This proves $\L_1\unlhd\L$ as required.
\end{proof}

\chapter{Commuting subsets and partial normal subgroups}

\textbf{Throughout this chapter, let $(\L,\Delta,S)$ be a locality.}

\smallskip

The results we prove in this chapter will play a role in the proof of Theorem~\ref{T:mainNperp}, which we give later on in Section~\ref{SS:proofmainNperp}. After proving some general lemmas concerning commuting subsets of localities, we develop a theory of commuting partial normal subgroups of $\L$ as far as this seems possible without the additional assumption that $(\L,\Delta,S)$ is a linking locality. At the end, in Section~\ref{SS:ProductsPartialNormal}, we show two general lemmas about products of partial normal subgroups of $\L$. The second one of these lemmas will be applied to commuting partial normal subgroups of linking localities in the proof of Theorem~\ref{T:mainNperp}.

\section{Commuting subsets}

In this section, we will study the following concepts, which were partly introduced already in the introduction.

\begin{definition}\label{D:Commute}
Let $\X,\Y\subseteq \L$. Then we say that $\X$ \emph{commutes with} $\Y$ (in $\L$) if for all $x\in\X$ and $y\in\Y$ the following implication holds:
\[(x,y)\in\D\Longrightarrow (y,x)\in\D\mbox{ and }xy=yx.\]
Moreover, we say that $\X$ \emph{commutes strongly with} $\Y$ (in $\L$) if for all $x\in\X$ and $y\in\Y$ the following implication holds:
\[(x,y)\in\D\Longrightarrow (y,x)\in\D,\;xy=yx\mbox{ and }S_{(x,y)}=S_{(y,x)}.\]
We say that $\X$ \emph{fixes $\Y$ under conjugation} if for all $x\in\X$ and $y\in\Y$ the following implication holds:
\[y\in\D(x)\Longrightarrow y^x=y.\]
\end{definition}

\begin{lemma}\label{L:PerpendicularPartialNormalPrepare}
Let $\X,\Y\subseteq\L$ such that $\X$ is closed under inversion. If $\X$ commutes with $\Y$, then $\X$ fixes $\Y$ under conjugation.  
\end{lemma}

\begin{proof}
Let $x\in\X$ and $y\in\Y$ with $y\in\D(x)$. The latter condition means $(x^{-1},y,x)\in\D$ and thus $(x^{-1},y)\in\D$. As $\X$ commutes with $\Y$ and $x^{-1}\in\X$ by assumption, we have $(y,x^{-1})\in\D$ and $x^{-1}y=yx^{-1}$. Using Lemma~\ref{L:Chermak14d} we can thus conclude that $(y,x^{-1},x)\in\D$. Hence, $y^x=\Pi(x^{-1},y,x)=\Pi(x^{-1}y,x)=\Pi(yx^{-1},x)=\Pi(y,x^{-1},x)=y$.   
\end{proof}

\begin{lemma}\label{L:XfixesRXcentralizesR}
Let $\X\subseteq\L$ and $R\leq S$. Assume that $R$ fixes $\X$ under conjugation or suppose that $R\unlhd S$ and $\X$ fixes $R$ under conjugation. Then $\X\subseteq C_\L(R)$ and $R\subseteq C_\L(\X)$. 
\end{lemma}

\begin{proof}
By Lemma~\ref{L:CentralizerHelp}, it is sufficient to show that $\X\subseteq C_\L(R)$ or $R\subseteq C_\L(\X)$. Moreover, for every  $x\in\X$ and $r\in R$ we have $(r^{-1},x,r)\in\D$ via $S_x^r$ and thus $x\in\D(r)$. Hence, if $R$ fixes $\X$ under conjugation, then $R\subseteq C_\L(\X)$. 

\smallskip 

So assume now that $R\unlhd S$ and $\X$ fixes $R$ under conjugation. Let $x\in\X$. As $\X$ fixes $R$ under conjugation, it is  enough to show that $R\subseteq \D(x)$. For $r\in N_R(S_x)$, we have $(x^{-1},r,x)\in\D$ via $S_x^x$. Hence, $r\in\D(x)$ and, since $\X$ fixes $R$ under conjugation, we have $r^x=r\in S$. So $r\in S_x$ proving $N_R(S_x)\leq S_x$ and thus $N_{R S_x}(S_x)=N_R(S_x)S_x=S_x$. Notice that $RS_x$ is a $p$-group, as $R\unlhd S$. Hence, it follows that $R S_x=S_x$, which implies $R\leq S_x\subseteq\D(x)$. This completes the proof. 
\end{proof}

\begin{lemma}\label{L:XCSXcommutesStrongly}
 Let $\X\subseteq\L$. Then $\X$ commutes strongly with $C_S(\X)$ and vice versa.
\end{lemma}

\begin{proof}
Let $x\in\X$ and $s\in C_S(\X)$. Then $(x,s)\in\D$ and $(s,s^{-1},x,s)\in\D$ via $S_x$. Using the axioms of a partial group, it follows $(s,x)=(s,x^s)\in\D$ and $xs=\Pi(s,s^{-1},x,s)=\Pi(s,x^s)=sx$, where the last equality uses $s\in C_S(\X)$. Using Lemma~\ref{L:NLSbiset}, it follows that $S_{(x,s)}=S_{xs}=S_{sx}=S_{(s,x)}$. 
\end{proof}

\begin{lemma}\label{L:CommuteStronglyProducts}
 Suppose $\X$ commutes strongly with $\Y$. Let $u,v\in\W(\L)$, $a\in\W(\X)$ and $b\in\W(\Y)$ with $w:=u\circ a\circ b\circ v\in\D$. Then $w':=u\circ b\circ a\circ v\in\D$, $S_w=S_{w'}$ and $\Pi(w)=\Pi(w')$. 
\end{lemma}

\begin{proof}
By the axioms of partial groups, we have $a\circ b\in \D$. If $S_{a\circ b}=S_{b\circ a}$, $b\circ a\in\D$ and $\Pi(a\circ b)=\Pi(b\circ a)$, then $S_w=S_{u\circ a\circ b\circ v}=S_{u\circ b\circ a\circ v}=S_{w'}$, in particular $S_{w'}=S_w\in\Delta$ and $w'\in\D$. Moreover, $\Pi(w)=\Pi(u\circ \Pi(a\circ b)\circ v)=\Pi(u\circ \Pi(b\circ a)\circ v)=\Pi(w')$. Hence, we may assume without loss of generality that $u=v=\emptyset$, $w=a\circ b$ and $w'=b\circ a$.

\smallskip

We prove now the assertion by induction on the length $|a\circ b|$ of $a\circ b$. The claim is clear if $a=\emptyset$ or $b=\emptyset$. So we may assume that both $a$ and $b$ have length at least one. Write $a=a'\circ (x)$ and $b=(y)\circ b'$ with $a'\in\W(\X)$, $x\in\X$, $b'\in\W(\Y)$, $y\in\Y$. Notice that $a'\circ (x,y)\circ b'=a\circ b=w\in\D$ and thus $(x,y)\in\D$. As $\X$ commutes strongly with $\Y$, it follows that $(y,x)\in\D$, $S_{(x,y)}=S_{(y,x)}$ and $xy=yx$. Therefore,
\[\Delta\ni S_w=S_{a\circ b}=S_{a'\circ (x,y)\circ b'}=S_{a'\circ (y,x)\circ b'}\]
and so $a'\circ (y,x)\circ b'\in\D$. Moreover,
\[\Pi(w)=\Pi(a'\circ (x,y)\circ b')=\Pi(a'\circ (\Pi(x,y))\circ b')=\Pi(a'\circ (\Pi(y,x))\circ b')=\Pi(a'\circ (y,x)\circ b').\]
Notice that $a'\circ (y)\in\D$ as $a'\circ (y,x)\circ b'\in\D$. Since $|a'\circ (y)|\leq |a'\circ b|<|a\circ b|$, we may assume by induction that $S_{a'\circ (y)}=S_{(y)\circ a'}$, $(y)\circ a'\in\D$ and $\Pi(a'\circ (y))=\Pi((y)\circ a')$. Hence,
\[S_w=S_{a'\circ (y,x)\circ b'}=S_{(y)\circ a'\circ (x)\circ b'}=S_{(y)\circ a\circ b'}\]
and thus $(y)\circ a'\circ (x)\circ b'=(y)\circ a\circ b'\in\D$. In particular, $a\circ b'\in\D$ and so, again by induction, we may assume that $S_{a\circ b'}=S_{b'\circ a}$, $b'\circ a\in\D$ and $\Pi(a\circ b')=\Pi(b'\circ a)$. Hence,
\[S_w=S_{(y)\circ a\circ b'}=S_{(y)\circ b'\circ a}=S_{b\circ a}=S_{w'}\]
and so $w'=b\circ a=(y)\circ b'\circ a\in\D$. Moreover,
\begin{eqnarray*}
\Pi(w) &=& \Pi(a'\circ (y,x)\circ b')
= \Pi((\Pi(a'\circ (y)),x)\circ b')
= \Pi((\Pi((y)\circ a'),x)\circ b')\\
&=& \Pi((y)\circ a'\circ (x)\circ b')=\Pi((y)\circ a\circ b')
= \Pi(y,\Pi(a\circ b'))\\
&=& \Pi(y,\Pi(b'\circ a))
= \Pi((y)\circ b'\circ a)
= \Pi(b\circ a)=\Pi(w').
\end{eqnarray*}
\end{proof}

\section{Commuting partial normal subgroups}\label{SS:CommutingPartialNormal}

In this section we develop a theory of commuting partial normal subgroups of localities, basically proving the parts of Theorem~\ref{T:mainNperp}, which are stated to be true without the assumption that $(\L,\Delta,S)$ is a linking locality.

\smallskip

\textbf{Throughout this section let $\N\unlhd\L$ and $T:=\N\cap S$.}

\begin{lemma}\label{L:PerpendicularNormLTNormL}
Let $\Y\subseteq\L$. Suppose $\N$ commutes with $\Y$ or assume more generally that $\N$ fixes $\Y$ under conjugation. Then $\Y\subseteq C_\L(T)$. Moreover, $\Y\unlhd\L$ if and only if $\Y\unlhd N_\L(T)$. 
\end{lemma}

\begin{proof}
By Lemma~\ref{L:PerpendicularPartialNormalPrepare}, we may just assume that $\N$ fixes $\Y$ under conjugation. Then $T\subseteq\N$ fixes $\Y$ under conjugation. Hence, by Lemma~\ref{L:XfixesRXcentralizesR}, we have $\Y\subseteq C_\L(T)\subseteq N_\L(T)$. In particular, if $\Y\unlhd\L$, then $\Y\unlhd N_\L(T)$. 

\smallskip

Suppose now that $\Y\unlhd N_\L(T)$. Let $g\in\L$ and $y\in\Y$ such that $(g^{-1},y,g)\in\D$. By the Frattini Lemma and the Splitting Lemma \cite[Corollary~3.11, Lemma~3.12]{Chermak:2015}, there exists $n\in\N$ and $f\in N_\L(T)$ such that $g=nf$ and $S_g=S_{(n,f)}$. Then $g^{-1}=f^{-1}n^{-1}$ and $S_{g^{-1}}=S_{(f^{-1},n^{-1})}$. So 
$S_{(f^{-1},n^{-1},y,n,f)}=S_{(g^{-1},y,g)}\in\Delta$ and thus $(f^{-1},n^{-1},y,n,f)\in\D$. Hence, $y\in\D(n)$ and $y^g=(y^n)^f$. As $\N$ fixes $\Y$ under conjugation, we have $y^n=y$. Since we assume that $\Y\unlhd N_\L(T)$, it follows $y^g=y^f\in \Y$, which proves $\Y\unlhd\L$. 
\end{proof}

We caution the reader that there are cases where a subset $\X$ commutes with a subset $\Y$ in $\L$, but $\Y$ does not commute with $\X$. Similarly, it can happen that $\X$ fixes $\Y$ under conjugation, but $\Y$ does not fix $\X$ under conjugation. However, we have the following lemma.

\begin{lemma}\label{L:XfixNNcommuteswithX}
Let $\X\subseteq\L$ such that $\X$ fixes $\N$ under conjugation. Then $\N$ commutes strongly with $\X$ and thus fixes $\X$ under conjugation. Moreover, $S_{(n,f)}=S_{(f,n)}=S_n\cap S_f$ for all $n\in\N$, $f\in\X$ with $(n,f)\in\D$. 
\end{lemma}

\begin{proof}
Suppose $\X$ fixes $\N$ under conjugation. Then in particular, $\X$ fixes $T=\N\cap S\unlhd S$ under conjugation. Hence, by Lemma~\ref{L:XfixesRXcentralizesR}, we have $\X\subseteq C_\L(T)=\C_T\subseteq\L_T$. Let now $n\in\N$ and $f\in\X$ with $(n,f)\in\D$. Then by \cite[Lemma~3.2(a)]{Chermak:2015}, we have $(f,f^{-1},n,f)\in\D$, $nf=fn^f$ and $S_{(n,f)}=S_{(f,n^f)}=S_f\cap S_n$. In particular, $n\in\D(f)$ and $(f,n^f)\in\D$. As $\X$ fixes $\N$ under conjugation, it follows that $n^f=n$, so $(f,n)\in\D$, $nf=fn$ and $S_{(n,f)}=S_{(f,n)}=S_f\cap S_n$. In particular, $\N$ commutes strongly with $\X$ and thus, by Lemma~\ref{L:PerpendicularPartialNormalPrepare}, $\N$ fixes $\X$ under conjugation.
\end{proof}

\begin{corollary}\label{C:MperpNNperpM}
Let $\M,\N\unlhd\L$. Then the following conditions are equivalent:
\begin{itemize}
 \item[(i)] $\M$ commutes with $\N$;
 \item[(ii)] $\N$ commutes with $\M$;
 \item[(iii)] $\M$ commutes strongly with $\N$;
 \item[(iv)] $\N$ commutes strongly with $\M$;
 \item[(v)] $\M$ fixes $\N$ under conjugation;
 \item[(vi)] $\N$ fixes $\M$ under conjugation.
\end{itemize}
Moreover, if one of these conditions holds, then $S_{(m,n)}=S_{(n,m)}=S_m\cap S_n$ for all $m\in\M$, $n\in\N$ with $(m,n)\in\D$.
\end{corollary}

\begin{proof}
By Lemma~\ref{L:PerpendicularPartialNormalPrepare}, (i) implies (v) and (ii) implies (vi). Moreover, by Lemma~\ref{L:XfixNNcommuteswithX}, (v) implies (iv) and (vi) implies (iii). Trivially, (iii) implies (i) and (iv) implies (ii). Hence, all six conditions are equivalent. Lemma~\ref{L:XfixNNcommuteswithX} gives also that (vi) implies $S_{(m,n)}=S_{(n,m)}=S_m\cap S_n$ for all $m\in\M$ and $n\in\N$ with $(m,n)\in\D$.
\end{proof}

\begin{remark}\label{R:Implications}
For two partial normal subgroups $\M$ and $\N$ of $\L$ the following holds by Lemma~\ref{L:CentralizerHelp} and Lemma~\ref{L:CentralProductsFactorsCentralize}.
\begin{eqnarray*}
\mbox{$\M$ and $\N$ form a central product}&\Longrightarrow &\M\subseteq C_\L(\N)\\
&\Longleftrightarrow & \N\subseteq C_\L(\M)\\
&\Longrightarrow & \mbox{conditions (i)-(vi) of Corollary~\ref{C:MperpNNperpM} hold.}
\end{eqnarray*}
We will introduce regular localities in Chapter~\ref{S:Regular} and show in Lemma~\ref{L:PerpendicularPartialNormalCentralProduct} that all the implications above are equivalences if $(\L,\Delta,S)$ is regular.  
\end{remark}

\begin{lemma}\label{L:ProductPartialNormalCommutes}
Let $\M,\N\unlhd\L$ and $\Y\subseteq\L$ such that $\M$ and $\N$ both commute strongly with $\Y$. Then $\M\N$ commutes strongly with $\Y$. In particular, there is a (with respect to inclusion) largest partial normal subgroup of $\L$ that commutes strongly with $\Y$.
\end{lemma}

\begin{proof}
Let $x\in\M\N$ and $y\in\Y$ with $(x,y)\in\D$. By Theorem~\ref{T:ProductsPartialNormal}, $\M\N\unlhd\L$ and there exist $m\in\M$ and $n\in\N$ such that $(m,n)\in\D$, $x=mn$ and $S_x=S_{(m,n)}$. As $S_{(m,n,y)}=S_{(x,y)}\in\Delta$, it follows $(m,n,y)\in\D$. As $\M$ and $\N$ commute strongly with $\Y$, the subset $\{m,n\}$ commutes strongly with $\Y$. Hence, by Lemma~\ref{L:CommuteStronglyProducts}, we have $S_{(m,n,y)}=S_{(y,m,n)}$, $(y,m,n)\in\D$ and $\Pi(m,n,y)=\Pi(y,m,n)$. Since $S_x=S_{(m,n)}$, this implies  $S_{(x,y)}=S_{(m,n,y)}=S_{(y,m,n)}=S_{(y,x)}$ and $(y,x)\in\D$. By the axioms of a partial group, $xy=\Pi(m,n,y)=\Pi(y,m,n)=yx$. As $x$ and $y$ were arbitrary, this shows that $\M\N$ commutes strongly with $\Y$ as required.   
\end{proof}

\begin{corollary}\label{C:NperpExistence}
With respect to inclusion there is a largest partial normal subgroup of $\L$ which commutes with $\N$. 
\end{corollary}

\begin{proof}
By Corollary~\ref{C:MperpNNperpM}, a partial normal subgroup commutes with $\N$ if and only if it commutes strongly with $\N$. Hence, the assertion follows from Lemma~\ref{L:ProductPartialNormalCommutes} applied with $\N$ in the role of $\Y$. 
\end{proof}

\begin{notation}\label{N:NperpL}
By $\N^\perp$ we denote the largest partial normal subgroup of $\L$ which commutes with $\N$. Notice that $\N^\perp$ exists by Corollary~\ref{C:NperpExistence}. If we want stress the dependence of $\N^\perp$ on $\L$, we write $\N^\perp_\L$ for $\N^\perp$. 
\end{notation}

\begin{cor}\label{C:NinMperp}
Let $\M$ and $\N$ be partial normal subgroups of $\L$. Then $\M\subseteq\N^\perp$ if and only if $\N\subseteq\M^\perp$.
\end{cor}

\begin{proof}
This follows from Corollary~\ref{C:MperpNNperpM} and the definitions of $\M^\perp$ and $\N^\perp$.
\end{proof}

\begin{cor}\label{C:NperpCapSGeneral}
We have $\N^\perp\cap S\leq C_S(\N)$, $\N\subseteq C_\L(\N^\perp\cap S)$ and $\N^\perp\subseteq C_\L(T)$.
\end{cor}

\begin{proof}
By Corollary~\ref{C:MperpNNperpM}, $\N^\perp\cap S$ fixes $\N$ under conjugation, and $T=\N\cap S$ fixes $\N^\perp$ under conjugation. Hence the assertion follows from Lemma~\ref{L:XfixesRXcentralizesR}.
\end{proof}

\begin{lemma}\label{L:Opperp}
If $R\leq S$ such that $R\unlhd\L$, then $R^\perp=C_\L(R)$. In particular, $O_p(\L)^\perp=C_\L(O_p(\L))$.
\end{lemma}

\begin{proof}
It is a special case of Corollary~\ref{C:NperpCapSGeneral} that $R^\perp\subseteq C_\L(R)$. By Lemma~\ref{L:CentralizerPartialNormal}, we have $C_\L(R)\unlhd N_\L(R)=\L$ and, by Lemma~\ref{L:CentralizerHelp}, $C_\L(R)$ commutes with $R$. Hence, $C_\L(R)\subseteq R^\perp$. This implies the assertion.
\end{proof}

\begin{lemma}\label{L:MinNimpliesNperpinMperp}
If $\M,\N\unlhd\L$ with $\M\subseteq\N$, then $\N^\perp\subseteq\M^\perp$. 
\end{lemma}

\begin{proof}
 As $\N^\perp$ commutes with $\N$, it commutes in particular with $\M$. Since $\N^\perp\unlhd\L$, it follows thus from the definition of $\M^\perp$ that $\N^\perp\subseteq \M^\perp$.
\end{proof}

\begin{lemma}\label{L:NinNperperp}
We have $\N\subseteq (\N^\perp)^\perp$. 
\end{lemma}

\begin{proof}
By Corollary~\ref{C:MperpNNperpM}, $\N$ commutes with $\N^\perp$ and is thus by definition of $(\N^\perp)^\perp$ contained in $(\N^\perp)^\perp$.
\end{proof}

Notice that the statement of the following lemma makes sense by Theorem~\ref{T:ProductsPartialNormal}.

\begin{lemma}\label{L:PerpProduct}
Let $\M$ and $\N$ be partial normal subgroups of $\L$. Then $(\M\N)^\perp=\M^\perp\cap\N^\perp$.
\end{lemma}

\begin{proof}
Since $(\M\N)^\perp$ commutes with $\M\N$ by definition, it commutes also with the subsets $\M$ and $\N$ of $\M\N$. Hence, as $(\M\N)^\perp\unlhd\L$, by definition of $\M^\perp$ and $\N^\perp$, we have $(\M\N)^\perp\subseteq\M^\perp\cap\N^\perp$. 

\smallskip

It remains to show that $\M^\perp\cap\N^\perp\subseteq (\M\N)^\perp$. As $\M^\perp$ and $\N^\perp$ are partial normal subgroups of $\L$, the intersection $\M^\perp\cap\N^\perp$ is also a partial normal subgroup. Therefore, we only need to see that $\M^\perp\cap\N^\perp$ commutes with $\M\N$ or equivalently, by Corollary~\ref{C:MperpNNperpM}, that $\M\N$ commutes with $\M^\perp\cap\N^\perp$. Corollary~\ref{C:MperpNNperpM} gives also that $\M$ and $\N$ commute strongly with $\M^\perp\cap\N^\perp$. Hence the assertion follows from Lemma~\ref{L:ProductPartialNormalCommutes}.
\end{proof}

\begin{lemma}\label{L:AutNperp}
For every locality $(\tL,\tDelta,\tS)$ and every $\alpha\in\Iso(\L,\tL)$, we have $\N^\perp\alpha=(\N\alpha)^\perp$. In particular, if $\alpha\in\Aut(\L)$ with $\N\alpha=\N$, then $\N^\perp\alpha=\N^\perp$. 
\end{lemma}

\begin{proof}
Notice that $\alpha$ induces an inclusion preserving one-to-one correspondence between the partial normal subgroups of $\L$ and the ones of $\tL$.  Moreover $\X\subseteq\N$ commutes with $\N$ if and only if $\X\alpha$ commutes with $\N\alpha$. This implies the assertion.
\end{proof}

\section{Products of partial normal subgroups}\label{SS:ProductsPartialNormal}

If $\M$ and $\N$ are partial normal subgroups of $\L$, then by Theorem~\ref{T:ProductsPartialNormal}, for every element $f\in\M\N$, there exist $m\in\M$ and $n\in\N$ such that $f=mn$, $(m,n)\in\D$ and $S_f=S_{mn}=S_{(m,n)}$. The goal of this section is to prove in Lemma~\ref{L:DecomposeEltofMN} below that $S_{mn}=S_{(m,n)}$ \emph{for all} $m\in\M$ and $n\in\N$ with $(m,n)\in\D$, provided we have the additional assumption that $\M\cap\N\leq N_\L(S)$. If $(\L,\Delta,S)$ is a linking locality and $\M$ commutes with $\N$, then $\M\cap\N\leq S\leq N_\L(S)$. Lemma~\ref{L:DecomposeEltofMN} will be applied later on in this context.

\smallskip

In our analysis of the situation we will use the relation $\uparrow$ introduced in \cite[Definition~3.6]{Chermak:2015}. Given $\N\unlhd\L$, this is a relation on the set $\L\circ \Delta$ of pairs $(f,P)$ where $f\in\L$ and $P\leq S_f$. We write $\uparrow_\N$ rather than $\uparrow$ to indicate the dependence of this relation on $\N$. So by definition, for all pairs $(f,P)$ and $(g,Q)$ in $\L\circ\Delta$, 
\[(f,P)\uparrow_\N (g,Q)\Longleftrightarrow \mbox{ there exist } x\in N_\N(P,Q)\mbox{ and }y\in N_\N(P^f,Q^g)\mbox{ with }xg=fy.\]
This defines a reflexive and transitive relation on $\L\circ\Delta$. An element $(f,P)\in\L\circ\Delta$ is called \emph{maximal} in $\L\circ\Delta$ if $(f,P)\uparrow_\N (g,Q)$ implies $|P|=|Q|$ for all $(g,Q)\in\L\circ\Delta$. One observes easily that $(f,P)\uparrow_\N (f,S_f)$ for every $(f,P)\in\L\circ\Delta$. Thus, the maximal elements are of the form $(f,S_f)$. An element $f\in\L$ is called $\uparrow_\N$-maximal, if $(f,S_f)$ is $\uparrow_\N$-maximal (cf. \cite[Definition~3.6]{Chermak:2015}).

\begin{lemma}\label{L:FinduparrowMaxEltinN}
Let $\M,\N\unlhd\L$. Then for every $n\in \N$, there exists $g\in\N$ such that $g$ is $\uparrow_\M$-maximal and $(n,S_n)\uparrow_\M (g,S_g)$. 
\end{lemma}

\begin{proof}
As $\L$ is finite and $\uparrow_\M$ is transitive with $(x,Q)\uparrow (x,S_x)$ for all $x\in \L$ and $Q\leq S_x$, we can find $a\in\L$ such that $(n,S_n)\uparrow_\M(a,S_a)$ and $(a,S_a)$ is $\uparrow_\L$-maximal. The latter condition means that $a$ is $\uparrow_\M$-maximal. 

\smallskip

By \cite[Proposition~3.14(b)]{Chermak:2015}, $n\in \M a$, i.e. there exists $m\in\M$ with $(m,a)\in\D$ and $n=ma$. It follows then from the axioms of a partial group that $(m^{-1},m,a)\in\D$ and $a=\Pi(m^{-1},m,a)=m^{-1}n\in\M\N$. Hence, by Theorem~\ref{T:ProductsPartialNormal}, there exist $f\in\M$ and $g\in\N$ such that $(f,g)\in\D$, $a=fg$ and $S_a=S_{(f,g)}$. Then
\[S_a^f=S_{(f,g)}^f\leq S_g\mbox{ and }S_a^a=S_{(f,g)}^{fg}\leq S_g^g.\]
Hence, $f\in N_\M(S_a,S_g)$, $\One\in N_\M(S_a^a,S_g^g)$ and $fg=a=a\One$ by Lemma~\ref{L:AddOnes}. By definition of $\uparrow_\M$, this means that $(a,S_a)\uparrow_\M (g,S_g)$. As $a$ is $\uparrow_\M$-maximal, \cite[Lemma~3.7]{Chermak:2015} yields now that $g$ is $\uparrow_\M$-maximal. Since $\uparrow_\M$ is transitive, we have moreover $(n,S_n)\uparrow_\M (g,S_g)$. As $g$ was chosen to be an element of $\N$ this yields the assertion.
\end{proof}

\begin{lemma}\label{L:DecomposeEltofMN}
Let $\M,\N\unlhd\L$ such that $\M\cap\N\leq N_\L(S)$. Then every element of $\N$ is $\uparrow_\M$-maximal. In particular, if $m\in\M$ and $n\in\N$ such that $(m,n)\in\D$, then 
\[S_{mn}=S_{(m,n)}.\]
\end{lemma}

\begin{proof}
By the Splitting Lemma \cite[Lemma~3.12]{Chermak:2015}, it is enough to show that every element of $\N$ is $\uparrow_\M$-maximal. Let $n\in\N$. By Lemma~\ref{L:FinduparrowMaxEltinN}, there exists $g\in\N$ such that $g$ is $\uparrow_\M$-maximal and $(n,S_n)\uparrow_\M (g,S_g)$. Then by \cite[Proposition~3.14(b)]{Chermak:2015}, $n\in \M g$. Hence, there exists $m\in\M$ with $(m,g)\in\D$ and $n=mg$. Using Lemma~\ref{L:Chermak14d} and the axioms of a partial group, one observes now that $(m,g,g^{-1})\in\D$ and
\[m=\Pi(m,g,g^{-1})=\Pi(n,g^{-1})\in\N\cap \M\]
since $m\in\M$ and $n,g\in\N$. Thus, we have $m\in\M\cap\N\leq N_\L(S)$, so Lemma~\ref{L:NLSbiset} gives $S_n=S_{mg}=S_g^{m^{-1}}$. As $(m^{-1},m,g)\in\D$ by Lemma~\ref{L:Chermak14d}, it follows 
\[g=\Pi(m^{-1},m,g)=m^{-1}n\]
and hence $S_g^g=S_g^{m^{-1}n}=(S_g^{m^{-1}})^n=S_n^n$. Therefore 
\[m^{-1}\in N_\M(S_g,S_n),\;\One\in N_\M(S_g^g,S_n^n)\mbox{ and }m^{-1}n=g=g\One,\]
where we use Lemma~\ref{L:AddOnes} at the end. This proves $(g,S_g)\uparrow_\M (n,S_n)$. As $g$ is $\uparrow_\M$-maximal, it follows therefore from \cite[Lemma~3.7]{Chermak:2015} that $n$ is $\uparrow_\M$-maximal.    
\end{proof}

\chapter{$\N$-radical subgroups}\label{S:NRadical}

Throughout let $(\L,\Delta,S)$ be a locality, $\N\unlhd\L$ and $T:=\N\cap S$. 

\smallskip

The reader should observe that $\L$ itself is a partial normal subgroup of $\L$, so everything we define for $\N$ below is in particular defined for $\L$. Recall that a subgroup $H$ of a finite group $G$ is called strongly $p$-embedded if $H\neq G$, $p$ divides $|H|$, but $p$ does not divide $|H\cap H^g|$ for all $g\in G\backslash H$.

\begin{definition}\label{D:Radical}~
\begin{itemize}
\item An object $P\in\Delta$ is called \emph{$\N$-radical} if $O_p(N_\N(P))\leq P$. 
\item We write $R_\Delta(\L)$ for the set of objects $P\in\Delta$ such that either $P=S$, or $N_\L(P)/P$ has a strongly $p$-embedded subgroup and $N_S(P)\in\Syl_p(N_\L(P))$. 
\end{itemize}
\end{definition}

As $S$ is a maximal $p$-subgroup of $\L$, it is a Sylow $p$-subgroup of $N_\L(S)$. Thus, $O_p(N_\N(S))\leq S$, i.e. $S$ is always $\N$-radical (and in particular $\L$-radical). If a group $G$ has a strongly $p$-embedded subgroup, then $O_p(G)=1$. Hence, the elements of $R_\Delta(\L)$ are all $\L$-radical. This is indeed the main property of the elements of $R_\Delta(\L)$ that we will use.

\smallskip

For every  $P\in\Delta$, we have $P\unlhd N_\L(P)$ and $P\cap T=P\cap\N\unlhd N_\N(P)$. Observe also that $O_p(N_\N(P))\leq P\cap\N=P\cap T$ if $P\in\Delta$ is $\N$-radical. Hence, an object $P\in\Delta$ is $\N$-radical if and only if $O_p(N_\N(P))=P\cap T$. In particular, $P$ is $\L$-radical if and only if $O_p(N_\L(P))=P$.

\begin{lemma}\label{L:NradMrad}
 Let $\M\unlhd\L$ with $\M\subseteq\N$. If $P\in\Delta$ is $\N$-radical, then $P$ is $\M$-radical.
\end{lemma}

\begin{proof}
If $P$ is $\N$-radical then, since $N_\M(P)\unlhd N_\N(P)$, we have $O_p(N_\M(P))\leq O_p(N_\N(P))\leq P$ and thus $P$ is $\M$-radical.
\end{proof}

By \cite[Lemma~2.9, Lemma~4.1]{Chermak:2015}, $\N S=S\N$ is a partial subgroup of $\L$ and $(S\N,\Delta,S)$ is a locality. So in particular, an object $P\in\Delta$ is called \emph{$S\N$-radical} if $O_p(N_{S\N}(P))=P$. Moreover, by  $R_\Delta(S\N)$ we denote the set of all objects $P\in\Delta$ such that $P=S$, or $N_{S\N}(P)/P$ has a strongly $p$-embedded subgroup and $N_S(P)\in\Syl_p(N_{S\N}(P))$. Note that every element of $R_\Delta(S\N)$ is $S\N$-radical.

\begin{lemma}\label{L:SNradNrad}
If $P\in\Delta$ is $\L$-radical, then $P$ is $\N$-radical and $O_p(\L)\leq P$. In particular, if $P\in\Delta$ is $S\N$-radical or if $P\in R_\Delta(S\N)$, then $P$ is $\N$-radical and $O_p(\L)\leq O_p(S\N)\leq P$.
\end{lemma}

\begin{proof}
As $(S\N,\Delta,S)$ is a locality with $\N\unlhd S\N$ and $O_p(\L)\leq O_p(S\N)$, it is sufficient to prove the first sentence in the assertion. So suppose that $P\in\Delta$ is $\L$-radical. By Lemma~\ref{L:NradMrad}, $P$ is $\N$-radical. Moreover, $N_{O_p(\L)}(P)\unlhd N_\L(P)$ and thus $N_{O_p(\L)}(P)\leq O_p(N_\L(P))=P$. Hence, $N_{O_p(\L)P}(P)=N_{O_p(\L)}(P)P=P$. As $O_p(\L)P$ is a $p$-group, this implies $P=O_p(\L)P\geq O_p(\L)$. 
\end{proof}

We note next that the property of being $\N$-radical is preserved under conjugation by appropriate elements of $\L$.

\begin{lemma}\label{L:NradicalConjugate}
If $P\in\Delta$ is $\N$-radical and $f\in\L$ with $P\leq S_f$, then $P^f$ is $\N$-radical. If $P\in R_\Delta(S\N)$ and $f\in S$, then $P^f\in R_\Delta(S\N)$.
\end{lemma}

\begin{proof}
This follows basically since the conjugation map $c_f\colon N_\L(P)\rightarrow N_\L(P^f)$ is by Lemma~\ref{L:ConjugateNormalizer} a well-defined isomorphism of groups which maps $N_\N(P)$ onto $N_\N(P^f)$. As $(S\N,\Delta,S)$ is also a locality, if $f\in S$, then similarly $c_f\colon N_{\N S}(P)\rightarrow N_{\N S}(P^f)$ is a group isomorphism with $N_S(P)^f\leq N_S(P^f)$. This implies the assertion.
\end{proof}

The elements of $R_\Delta(S\N)$ play an important role as the following lemma shows. It can be regarded as a version of Alperin's Fusion Theorem for partial normal subgroups. 

\begin{lemma}\label{L:PartialNormalAlperin}
Let $(\L,\Delta,S)$ be a locality. If $\N$ is a partial normal subgroup of $\L$ and $n\in\N$, then there exist $k\in\mathbb{N}$, $R_1,R_2,\dots,R_k\in R_\Delta(S\N)$ and $w=(t,n_1,n_2,\dots,n_k)\in\D$ such that the following hold:
\begin{itemize}
\item[(i)] $S_n=S_w$ and $n=\Pi(w)$;
\item[(ii)] $n_i\in O^p(N_\N(R_i))$ and $S_{n_i}=R_i$ for all $i=1,\dots,k$; and
\item[(iii)] $t\in T$.
\end{itemize}
\end{lemma}

\begin{proof}
This is \cite[Lemma~6.2]{Henke:2020}.
\end{proof}

As a consequence of the above lemma, we will show now that $\foc(\F_T(\N))$ and $C_S(\N)$ have also nice descriptions in terms of the elements of $R_\Delta(S\N)$. 

\begin{lemma}\label{L:DescribeFocE}
Set $\E:=\F_T(\N)$. Then 
\[\foc(\E)=\<[P\cap T,N_\N(P)]\colon P\in R_\Delta(S\N)\>.\] 
\end{lemma}

\begin{proof}
Set $R:=\<[P\cap T,N_\N(P)]\colon P\in R_\Delta(S\N)\>$. Observe that, for any $P\in R_\Delta(S\N)$, the elements of $N_\N(P)$ induce $\E$-automorphisms of $P\cap T$ by conjugation. Hence, $R\leq \foc(\E)$. To show the converse inclusion let $x\in A\leq T$ and $\alpha\in\Hom_\E(A,T)$. We need to show that $x^{-1}(x\alpha)\in R$.

\smallskip

By definition of $\E$, the morphism $\alpha$ is the composition of restrictions of maps of the form $c_n\colon S_n\cap T\rightarrow T$ with $n\in\N$. For any $n\in\N$, we can in turn find a decomposition of $n$ as in Lemma~\ref{L:PartialNormalAlperin}. Note that $S\in R_\Delta(S\N)$ and $T\leq N_\N(S)$. Hence, we can conclude that $\alpha$ can be written as a product $\alpha=(c_{n_1}|_{A_0,A_1})(c_{n_2}|_{A_1,A_2})\cdots (c_{n_k}|_{A_{k-1},A_k})$ where $A=A_0,A_1,\dots,A_k$ are subgroups of $T$, $\<A_{i-1},A_i\>\leq P_i:=S_{n_i}\in R_\Delta(S\N)$, $n_i\in N_\N(P_i)$ and $A_{i-1}^{n_i}=A_i$. Then $x_0:=x\in A=A_0$ and $x_i:=x_{i-1}^{n_i}\in A_i$ for $i=1,\dots,k$. In particular, \[x_{i-1}^{-1}x_i=x_{i-1}^{-1}x_{i-1}^{n_i}=[x_{i-1},n_i]\in [P_i\cap T,N_\N(P_i)]\]
                                                                                                                                                                                                                                                                                                                                                                                                                                                                                                                                                                                                                                                                                                                                                                              for all $i=1,\dots,k$. Note also that $x\alpha=x_k$. Hence, it follows 
\[
x^{-1}(x\alpha)=x_0^{-1}x_k=(x_0^{-1}x_1)(x_1^{-1}x_2)\cdots (x_{k-1}^{-1}x_k)\in \prod_{i=1}^k[P_i\cap T,N_\N(P_i)]\subseteq R.
\]
This shows the assertion.
\end{proof}

\begin{lemma}\label{L:CSNAlperin}
\[C_S(\N)=\bigcap_{P\in R_\Delta(S\N)}C_S(N_\N(P))=C_S(T)\cap \left(\bigcap_{P\in R_\Delta(S\N)}C_S(O^p(N_\N(P)))\right).\]
\end{lemma}

\begin{proof}
Note that $S\in R_\Delta(S\N)$ and $T\subseteq N_\N(S)$. Hence, we have 
\[C_S(\N)\subseteq\bigcap_{P\in R_\Delta(S\N)}C_S(N_\N(P))\subseteq R:=C_S(T)\cap \left(\bigcap_{P\in R_\Delta(S\N)}C_S(O^p(N_\N(P)))\right).\]
As $C_\L(R)$ is by Lemma~\ref{L:CentralizerPartialNormal} a partial subgroup of $\L$, it follows from Lemma~\ref{L:CentralizerHelp} and Lemma~\ref{L:PartialNormalAlperin} that $\N\subseteq C_\L(R)$ and $R\subseteq C_S(\N)$. So equality holds everywhere above as required. 
\end{proof}

Since the elements of $R_\Delta(S\N)$ and thus the $S\N$-radical subgroups play an important role, we now want to study properties of $S\N$-radical subgroups in more detail. The following lemma is particularly useful for linking localities. 

\begin{lemma}\label{L:CSTinSNradical}
Let $P\in\Delta$ such that $N_\L(P)$ is of characteristic $p$ and $P$ is $S\N$-radical. Then $C_S(T)\leq P$.
\end{lemma}

\begin{proof}
Set $X:=N_{C_S(T)}(P)$. By Lemma~\ref{L:SNradNrad}, $P$ is $\N$-radical, i.e. $Q:=O_p(N_\N(P))=P\cap T$ and thus $X\leq C_{N_\L(P)}(Q)$. Since $N_\L(P)$ is of  characteristic $p$, $N_\N(P)\unlhd N_\L(P)$ is by Lemma~\ref{L:MSCharp}(a) of characteristic $p$, i.e. $C_{N_\N(P)}(Q)\leq Q$. Hence,
\[ [X,N_\N(P)]\leq [C_{N_\L(P)}(Q),N_\N(P)]\leq C_{N_\N(P)}(Q)\leq Q.\]
As $O^p(N_{S\N}(P))=O^p(N_\N(P))$ by \cite[Lemma~6.1(b)]{Henke:2020}, it follows that $[X,O^p(N_{S\N}(P))]$ is a $p$-group. Hence, by  Lemma~\ref{L:GetintoOp}(c), we have $X\leq O_p(N_{S\N}(P))=P$. Thus, $N_{PC_S(T)}(P)=PX=P$ and, as $PC_S(T)$ is a $p$-group, it follows $C_S(T)\leq PC_S(T)=P$.  
\end{proof}

The following lemma will be helpful when moving between localities with different object sets.

\begin{lemma}\label{L:TakePlusRadical}
 Let $(\L^+,\Delta^+,S)$ be a locality with $\Delta\subseteq\Delta^+$ and $\L=\L^+|_\Delta$. Fix $P\in\Delta$ and $\N^+\unlhd\L^+$ such that $\N^+\cap\L=\N$. Then 
\[N_{\L^+}(P)=N_\L(P),\;N_{\N^+}(P)=N_\N(P)\mbox{ and }N_{S\N^+}(P)=N_{S\N}(P).\]
In particular, $P$ is $\N$-radical if and only if $P$ is $\N^+$-radical, and $P$ is $S\N$-radical if and only if $P$ is $S\N^+$-radical.
\end{lemma}

\begin{proof}
As $\L=\L^+|_\Delta$, we have $N_{\L^+}(P)=N_\L(P)$ and thus $N_{\N^+}(P)=N_\N(P)$. Moreover, Lemma~\ref{L:RestrictionIntersectL} gives  $S\N^+\cap\L=S(\N^+\cap\L)=S\N$ and thus $N_{S\N^+}(P)=N_{S\N}(P)$. This implies the assertion.
\end{proof}

Recall that every $S\N$-radical subgroup is $\N$-radical by Lemma~\ref{L:SNradNrad}. In many contexts, it seems indeed more convenient to work with $\N$-radical objects rather than with $S\N$-radical objects. In Lemma~\ref{L:NradMrad} and Lemma~\ref{L:NradicalConjugate} we saw already that this concept has nice inheritance properties. More such properties will be shown in Lemmas~\ref{L:NradNOpLrad} and \ref{L:PradQrad} below.

\begin{lemma}\label{L:PRintersectT}
Suppose $P\in\Delta$ is $\N$-radical.  If $R\leq S$ such that $N_\N(P)P\subseteq N_\L(R)$, then $(PR)\cap T=P\cap T$.  
\end{lemma}

\begin{proof}
As $P\subseteq N_\L(R)$, the product $PR$ is a $p$-group. Thus, $X:=(PR)\cap T$ is a $p$-group normalized by $P$, which implies that $XP$ is a $p$-group as well. Since $N_\N(P)\subseteq N_\L(R)$, we have $N_\N(P)\subseteq N_\N(RP)\subseteq N_\N(X)$ and thus $N_X(P)\leq O_p(N_\N(P))=P$. Hence,  $N_{XP}(P)=N_X(P)P\leq P$. We can conclude that $XP=P$ and thus $X=P\cap T$ as required.
\end{proof}

\begin{lemma}\label{L:NradNOpLrad}
 Let $P\in\Delta$ be $\N$-radical with $P\leq T$. Then $PO_p(\L)$ is $\N O_p(\L)$-radical.
\end{lemma}

\begin{proof}
 Set $Q:=PO_p(\L)$. By Lemma~\ref{L:PRintersectT} applied with $O_p(\L)$ in place of $R$, we have $Q\cap T=P\cap T=P$. It follows now from the definition of $O_p(\L)$ and the fact that $T$ is strongly closed that
\[N:=N_\N(P)=N_\N(Q)\unlhd G:=N_{\N O_p(\L)}(Q).\]
As $P$ is $\N$-radical, we conclude that $O_p(G)\cap N=O_p(N)=P$. Observe that $O_p(\L)\leq N_\L(Q)$ and so the Dedekind argument for partial groups \cite[Lemma~1.10]{Chermak:2015} gives $G=N_\L(Q)\cap \N O_p(\L)=N_\N(Q)O_p(\L)=NO_p(\L)$. Since $O_p(\L)\leq O_p(G)$, using the Dedekind argument for groups, it follows that $O_p(G)=(O_p(G)\cap N)O_p(\L)=PO_p(\L)=Q$. This means that $Q$ is $\N O_p(\L)$-radical.
\end{proof}

\begin{lemma}\label{L:Lemma0}
 Let $P\in\Delta$ and $Q\leq P$ such that $P\cap T\leq Q$. Then $N_\N(P)\leq N_\N(Q)$.
\end{lemma}

\begin{proof}
Computing in the group $N_\L(P)$, we have 
\[[Q,N_\N(P)]\leq [P,N_\N(P)] \leq P\cap N_\N(P)=P\cap T\leq Q.\] 
\end{proof}

\begin{lemma}\label{L:PradQrad}
 Let $P,Q\in\Delta$ with $P\cap T\leq Q\leq P$. If $P$ is $\N$-radical, then $Q$ is $\N$-radical.
\end{lemma}

\begin{proof}
Let $(P,Q)$ be a counterexample such that $|P:Q|$ is minimal. Clearly, $Q\neq P$, so $Q<P_0:=N_P(Q)$ and $|P:P_0|<|P:Q|$. Hence, $(P,P_0)$ is not a counterexample and thus $P_0$ is $\N$-radical. If $P_0$ were properly contained in $P$, then $|P_0:Q|<|P:Q|$, so $(P_0,Q)$ were not a counterexample and $Q$ were $\N$-radical. Hence $P_0=P$, i.e. $Q\unlhd P$. In particular, $P$ normalizes
\[R:=O_p(N_\N(Q))\]
and so $RP$ is a $p$-group. By Lemma~\ref{L:Lemma0}, we have $N_\N(P)\leq N_\N(Q)$. Hence, $N_R(P)=R\cap N_\N(P)\leq O_p(N_\N(P))\leq P$ and so  $N_{RP}(P)=N_R(P)P=P$. As $RP$ is a $p$-group, it follows that $RP=P$. This implies $R\leq P\cap\N=P\cap T\leq Q$, i.e. $Q$ is $\N$-radical. 
\end{proof}

Set now 
\[\F:=\F_S(\L)\mbox{ and }\E:=\F_T(\N).\]
Our next aim is to investigate the properties of $\N$-radical subgroups in terms of the fusion systems $\F$ and $\E$. 

\begin{lemma}\label{L:NradEcr}
Let $P\in\Delta$ such that $P$ is $\N$-radical. Assume that $(\L,\Delta,S)$ is a linking locality or assume more generally that $\F$ is saturated and $N_\L(P)$ is of characteristic $p$. Then \[P\cap T\in\F_T^c\subseteq\E^c.\] 
\end{lemma}

\begin{proof}
Set $Q:=P\cap T$. By Theorem~\ref{T:Saturation}, we may just assume that $\F$ is saturated and $N_\L(P)$ is of characteristic $p$. We show first that 
\begin{equation}\label{E:CTQinQ}
C_T(Q)\leq Q.
\end{equation}
As $N_\L(P)$ has characteristic $p$, the normal subgroup $N_\N(P)$ of $N_\L(P)$ has characteristic $p$ by Lemma~\ref{L:MSCharp}(a). Moreover, since $P$ is $\N$-radical, we have $Q=O_p(N_\N(P))$. Hence, $N_{C_T(Q)}(P)\leq C_{N_\N(P)}(Q)\leq Q\leq P$ and thus $N_{C_T(Q)P}(P)=N_{C_T(Q)}(P)P=P$. As $C_T(Q)P$ is a $p$-group, this means $C_T(Q)P=P$ and thus $C_T(Q)\leq P\cap T=Q$. So \eqref{E:CTQinQ} holds. 

\smallskip

By \cite[Lemma~2.6(c)]{Aschbacher/Kessar/Oliver:2011}, there exists a morphism in $\Hom_\F(N_S(Q),S)$ such that the image of $Q$ under this morphism is fully $\F$-normalized. Since $P\leq N_S(Q)$, we have $N_S(Q)\in\Delta$ and thus such a morphism is realized as a conjugation homomorphism by an element of $\L$ (cf. Lemma~\ref{L:LocalityFusionSystem}(a)). So there exists $f\in\L$ such that $P\leq N_S(Q)\leq S_f$ and $Q^f$ is fully $\F$-normalized. By Lemma~\ref{L:NradicalConjugate}, $P^f\in\Delta$ is $\N$-radical. As $Q^f=P^f\cap T$, property \eqref{E:CTQinQ} applied with $Q^f$ in place of $Q$ gives $C_T(Q^f)\leq Q^f$. By Lemma~\ref{L:FTcFRc}(a), this implies $Q\in\F_T^c\subseteq\E^c$.  
\end{proof}

We need a preliminary lemma to prove our next result about $\N$-radical subgroups.

\begin{lemma}\label{L:fullyEnorm}
Let $P\in\Delta$ such that $P\leq T$ and $P$ is fully $\E$-normalized. Then $N_T(P)\in\Syl_p(N_\N(P))$.
\end{lemma}

\begin{proof}
By \cite[Lemma~2.10]{Chermak:2015}, there exists $g\in\L$ such that $N_S(P)\leq S_g$ and $N_S(P^g)\in\Syl_p(N_\L(P^g))$. By the Frattini Lemma and the Splitting Lemma (cf. Corollary~3.11 and Lemma~3.12 in \cite{Chermak:2015}), there exist $n\in\N$ and $f\in N_\L(T)$ such that $(n,f)\in\D$, $g=nf$ and $S_g=S_{(n,f)}$. As $N_\N(P^g)$ is a normal subgroup of $N_\L(P^g)$, we have $N_T(P^g)=N_S(P^g)\cap N_\N(P^g)\in\Syl_p(N_\N(P^g))$. By Lemma~\ref{L:ConjugateNormalizer}, $c_f\colon N_\L(P^n)\rightarrow N_\L(P^g)$ is an isomorphism of groups which takes $N_\N(P^n)$ onto $N_\N(P^g)$. It follows that $N_T(P^g)^{f^{-1}}$ is a Sylow $p$-subgroup of $N_\N(P^n)$. Since $f\in N_\L(T)$, we have $N_T(P^g)^{f^{-1}}\leq N_T(P^n)$. As $N_T(P^n)$ is a $p$-subgroup of $N_\N(P^n)$, we conclude that $N_T(P^n)\in\Syl_p(N_\N(P^n))$. Note that $N_T(P)\leq N_S(P)\leq S_g=S_{(n,f)}\leq S_n$ and $T$ is strongly closed, so  $N_T(P)^n\leq N_T(P^n)$. At the same time, the assumption that $P$ is fully $\E$-normalized yields that $|N_T(P)^n|=|N_T(P)|\geq |N_T(P^n)|$. Hence, we must have $N_T(P)^n=N_T(P^n)$. Again by Lemma~\ref{L:ConjugateNormalizer}, $c_n\colon N_\L(P)\rightarrow N_\L(P^n)$ is a group isomorphism which induces an isomorphism from $N_\N(P)$ to $N_\N(P^n)$. So $N_T(P)$ is a Sylow $p$-subgroup of $N_\N(P)$.  
\end{proof}

\begin{lemma}\label{L:EcrNradical}
Let $P\in\Delta$ such that $P\leq T$ and $P$ is fully $\E$-normalized, $O_p(N_\E(P))=P$ and $N_\L(P)$ is of characteristic $p$. Then $P$ is $\N$-radical.
\end{lemma}

\begin{proof}
By Lemma~\ref{L:fullyEnorm}, we have $N_T(P)\in\Syl_p(N_\N(P))$. Since $P\in\Delta$, it follows from $\E=\F_T(\N)$ and Lemma~\ref{L:LocalityFusionSystem}(b) that $\F_{N_T(P)}(N_\N(P))=N_\E(P)$. Hence $O_p(N_\N(P))\leq O_p(N_\E(P))=P$ and so $P$ is $\N$-radical. 
\end{proof}

\chapter{$p$-Residuals and $p^\prime$-residuals of partial normal subgroups}\label{S:Residues}

Throughout let $(\L,\Delta,S)$ be a locality and let $\N$ be a partial normal subgroup of $\L$. Set $T:=S\cap\N$.

\section{Definition and basic properties}

Following Chermak \cite[Section~7]{ChermakII} we use the following definition.

\begin{definition}
Set
\[\mathbb{K}_\N:=\{\K\unlhd\L\colon T\K=\N\}\mbox{ and }\mathbb{K}_\N^\prime:=\{\K\unlhd\L\colon T\subseteq\K\subseteq\N\}.\]
Moreover set 
\[O_\L^p(\N):=\bigcap \mathbb{K}_\N\mbox{ and }O_\L^{p^\prime}(\N)=\bigcap\mathbb{K}_\N^\prime.\]
Write $O^p(\L)$ for $O^p_\L(\L)$ and $O^{p^\prime}(\L)$ for $O^{p^\prime}_\L(\L)$. A partial normal subgroup $\K\unlhd\L$ is said to have \emph{$p$-power index in $\N$} if $\K\in\mathbb{K}_\N$ and \emph{index prime to $p$ in $\N$} if $\K\in\mathbb{K}_\N^\prime$.
\end{definition}

Observe that $O^{p^\prime}_\L(\N)\in\mathbb{K}_\N^\prime$. It was shown by Chermak \cite[Proposition~7.2]{ChermakII} that $O^p_\L(\N)\in\mathbb{K}_\N$. In the following lemma we give an alternative proof of this property which fits better into our framework.

\begin{lemma}\label{L:ResidueEquivalences}
Let $\K\unlhd\L$ with $\K\subseteq\N$. Then the following conditions are equivalent:
\begin{itemize}
 \item [(i)] $\K\in\mathbb{K}_\N$.
 \item [(ii)] The image of $\N$ under the natural projection $\L\rightarrow\L/\K$ is a $p$-group.
 \item [(iii)] $O^p(N_\N(P))\subseteq \K$ for all $P\in\Delta$.
 \item [(iv)] $O^p(N_\N(P))\subseteq\K$ for all $P\in R_\Delta(S\N)$.
\end{itemize}
In particular, $O^p_\L(\N)\in\mathbb{K}_\N$ and each of the conditions (ii)-(iv) holds with $O^p_\L(\N)$ in place of $\K$. 
\end{lemma}

\begin{proof}
Write $\alpha\colon\L\rightarrow\L/\K$ for the natural projection from $\L$ to $\L/\K$. Then $\alpha$ is a projection of partial groups with kernel $\K$. Hence, if (i) holds, then $\N=T\K$ and $\N\alpha=T\alpha\cong T/(T\cap \K)$ is a $p$-group. So (i) implies (ii).

\smallskip

Assume now that (ii) holds, i.e.  $\N\alpha$ is a $p$-group. Then for every $P\in\Delta$, the restriction of $\alpha$ to a map $\alpha_P\colon N_\N(P)\rightarrow \N\alpha$ is a group homomorphism whose image is a $p$-group. Hence, $O^p(N_\N(P))\leq \ker(\alpha_P)\subseteq\ker(\alpha)=\K$ for every $P\in\Delta$. So (ii) implies (iii). 

\smallskip

Clearly (iii) implies (iv). It follows moreover from Lemma~\ref{L:PartialNormalAlperin} that (iv) implies (i). So the conditions (i)-(iv) are equivalent. Clearly, the intersection of a set of partial normal subgroups of $\L$ is a partial normal subgroup of $\L$. So $O^p_\L(\N)\unlhd\L$. Since (i) implies (iii), we have furthermore
\[O^p(N_\N(P))\subseteq \bigcap \mathbb{K}_\N=O^p_\L(\N)\mbox{ for all }P\in\Delta.\]
So (iii) and thus each of the conditions (i)-(iv) holds with $O^p_\L(\N)$ in place of $\K$ and the proof is complete.                                                                                                                                                                                                                                                                                                                                                                                                                                                                                                                                                              
\end{proof}

%We obtain a similar characterization of partial normal subgroups of $\N$ of index prime to $p$.

%\begin{lemma}\label{L:ResiduepprimeEquivalences}
% Let $\K\unlhd\L$ with $\K\subseteq\N$. Write $\alpha\colon \L\rightarrow\L/\K$ for the natural projection. Then the following conditions are equivalent:
%\begin{itemize}
% \item [(i)] $\K\in\mathbb{K}_\N^\prime$.
% \item [(ii)] $\N\alpha\cap S\alpha=\{\One\}$.
% \item [(iii)] There is no non-trivial $p$-subgroup of $\N\alpha$.
% \item [(iv)] $O^{p^\prime}(N_\N(P))\subseteq \K$ for all $P\in\Delta$.
%\end{itemize}
%\end{lemma}

%\begin{proof}
%As $\K\subseteq\N$ it follows from \cite[Lemma~4.9]{Chermak:2015} that $\N\alpha\cap S\alpha=(\N\cap S)\alpha$. Therefore (i) and (ii) are equivalent. 

%\smallskip

%By \cite[Theorem~4.3]{Chermak:2015}, $(\L/\K,\Delta\alpha,S\alpha)$ is a locality and by the partial subgroup correspondence \cite[Proposition~4.7]{Chermak:2015}, $\N\alpha\unlhd \L/\K$. Hence, by \cite[Proposition~2.11(c)]{Chermak:2015}, every $p$-subgroup of $\N\alpha$ is conjugate in $\L/\K$ to a subgroup of $S\alpha$ and thus of $S\alpha\cap\N\alpha$. Hence, (ii) implies (iii). 

%\smallskip

%For every $P\in\Delta$, $\alpha$ induces a group homomorphism from $N_\L(P)$ to $N_\L(P)\alpha$. Hence, if (iii) holds, then $N_\N(P)$ gets mapped to a $p^\prime$-group, which yields $O^{p^\prime}(N_\N(P))\subseteq \ker(\alpha)=\K$. So (iii) implies (iv). Clearly (iv) implies $\N\cap S\subseteq O^{p^\prime}(N_\N(S))\subseteq\K$ and thus (i). 
%\end{proof}

The following lemma was first proved by Chermak \cite[Lemma~7.4]{ChermakII}. Again we give an alternative proof for the first part of the assertion.

\begin{lemma}\label{L:ResidueNinM}
 Let $\N\unlhd \L$ with $\M\subseteq\N$. Then $O^p_\L(\M)\subseteq O^p_\L(\N)$ and $O^{p^\prime}_\L(\M)\subseteq O^{p^\prime}_\L(\N)$. 
\end{lemma}

\begin{proof}
Set $\K:=O^p_\L(\N)$. Then by Lemma~\ref{L:ResidueEquivalences} applied with $\N$ in place of $\M$, we have $O^p(N_\M(P))\subseteq O^p(N_\N(P))\subseteq \K$ for all $P\in\Delta$. Hence, $O^p(N_\M(P))\subseteq \K\cap\M$ for all $P\in\Delta$. Observe that $\K\cap\M$ is a partial normal subgroup of $\L$ contained in $\M$. So again by Lemma~\ref{L:ResidueEquivalences}, we have $\K\cap\M\in\mathbb{K}_\M$. Thus, $O^p_\L(\M)\subseteq\K\cap\M\subseteq\K=O^p_\L(\N)$. 

\smallskip

Notice that $\M\cap S\subseteq \N\cap S\subseteq O^{p^\prime}_\L(\N)$ and hence $\M\cap S\subseteq O^{p^\prime}_\L(\N)\cap \M$. As $O^{p^\prime}_\L(\N)\cap\M\unlhd\L$, this implies $O^{p^\prime}_\L(\N)\cap\M\in\mathbb{K}_\M^\prime$ and thus $O^{p^\prime}_\L(\M)\subseteq O^{p^\prime}_\L(\N)\cap\M\subseteq O^{p^\prime}_\L(\N)$.
\end{proof}

The following more specialized lemma will be needed later on.

\begin{lemma}\label{L:OupperpGenerate}
Suppose we are given $R\leq S$ with $R\unlhd\L$ such that
\[Q\cap (TR)\in\Delta\mbox{ for all }Q\in\Delta.\]
Then $\Gamma:=\{(P\cap T)^hR\;\colon P\in R_\Delta(S\N),\;h\in N_\L(T)\}\subseteq\Delta$ and
\[O^p_\L(\N)=\<O^p(N_\N(X))\colon X\in\Gamma\>.\] 
\end{lemma}

\begin{proof}
By Lemma~\ref{L:SNradNrad}, for all $P\in R_\Delta(S\N)$, we have $R\leq O_p(\L)\leq P$ and thus $(P\cap T)R=P\cap (TR)\in\Delta$. Moreover, $TR=S\cap TR\in\Delta$ and thus $N_\L(T)=N_\L(TR)$ is group, which acts on $TR$ and $R$ via conjugation. As $\Delta$ is closed under taking $\L$-conjugates in $S$, it follows that $\Gamma\subseteq\Delta$ and $N_\L(T)$ acts on $\Gamma$. Set
\[\K:= \<O^p(N_\N(X))\colon X\in\Gamma\>.\]
By Lemma~\ref{L:ResidueEquivalences}, we have $O^p(N_\N(X))\subseteq O^p_\L(\N)$ for all $X\in\Delta$ and thus $\K\subseteq O^p_\L(\N)$. Notice moreover that $N_\N(P)\subseteq N_\N((P\cap T)R)$ and thus   
\begin{equation}\label{E:OpNNpinK}
O^p(N_\N(P))\leq O^p(N_\N((P\cap T)R))\subseteq \K \mbox{ for every }P\in R_\Delta(S\N).
\end{equation}
Hence, again by Lemma~\ref{L:ResidueEquivalences}, we have $\K\supseteq O^p_\L(\N)$ if we can prove that $\K\unlhd \L$. We show this in two steps.

\smallskip

\textit{Step~1:} We show that $x\in\D(h)$ and $x^h\in \K$ for all $x\in\K$ and $h\in N_\L(T)$. For the proof fix $h\in N_\L(T)$, set 
\[\Y_0:=\bigcup_{X\in\Gamma}O^p(N_\N(X))\]
and $\Y_{i+1}:=\{\Pi(w)\colon w\in\D\cap\W(\Y_i)\}$. By \cite[Lemma~1.9]{Chermak:2015}, we have $\K=\bigcup_{i\geq 0}\Y_i$. Hence, it is sufficient to show that, for all $i\geq 0$, we have $x\in\D(h)$ and $x^h\in \Y_i$ whenever $x\in \Y_i$. We show this by induction on $i$. As $N_\L(T)$ acts on $\Gamma$, it follows from Lemma~\ref{L:ConjugateNormalizer} that the claim is true for $i=0$. Let now $i\geq 0$ be arbitrary such that $x\in\D(h)$ and $x^h\in\Y_i$ for all $x\in\Y_i$. Let $y\in\Y_{i+1}$. Then there exists $w=(x_1,\dots,x_n)\in\D\cap\W(\Y_i)$ with $y=\Pi(w)$. As $R\unlhd\L$, we have $R\leq S_w$. Our assumption yields that $Q:=(S_w\cap T)R=S_w\cap (TR)\in\Delta$. Hence, $u:=(h^{-1},x_1,h,h^{-1},x_2,h,\dots,h^{-1},x_n,h)\in\D$ via $Q^h$. It follows now from the axioms of a partial group that $y\in\D(h)$ and 
\[y^h=\Pi(w)^h=\Pi(u)=\Pi(x_1^h,x_2^h,\dots,x_n^h).\]
By induction hypothesis, we have $x_j^h\in\Y_i$ for all $j=1,\dots,n$ and thus $y^h\in\Y_{i+1}$ as required. This completes Step~1.

\smallskip

\textit{Step~2:} We prove now that $\K\unlhd\L$ and thus the assertion holds as we have argued above. Let $k\in\K$ and $f\in\L$ with $(f^{-1},k,f)\in\D$. By the Frattini Lemma and the Splitting Lemma \cite[Corollary~3.11, Lemma~3.12]{Chermak:2015}, there exists $n\in\N$ and $h\in N_\L(T)$ with $f=hn$ and $S_f=S_{(h,n)}$. Moreover, by Lemma~\ref{L:PartialNormalAlperin}, there exists $w=(t,n_1,\dots,n_l)\in\D$ with $n=\Pi(w)$, $S_n=S_w$, $t\in T$ and $n_i\in O^p(N_\N(R_i))$ for some $R_i\in R_\Delta(S\N)$ ($i=1,\dots,l$). It follows from \eqref{E:OpNNpinK} that $m:=\Pi(n_1,\dots,n_l)\in\K$. Notice that $(h,t,m)\in\D$ via $S_f=S_{(h,n)}$, $htm=f$ and $(m^{-1},(ht)^{-1},k,ht,m)\in\D$ via $S_{(f^{-1},k,f)}$. So $k^f=(k^{ht})^m$. As $ht\in N_\L(T)$, it follows from Step~1 that $k^{ht}\in\K$. Thus, as $\K$ is a partial subgroup and $m\in\K$, we have $k^f\in\K$. This shows $\K\unlhd\L$ as required.
\end{proof}

\section{$p$-Residuals and $p^\prime$-residuals of expansions}

The following lemma was first stated by Chermak \cite[Lemma~7.3]{ChermakII}. There is a small gap in Chermak's proof that $O^p_\L(\N)^+=O^p_{\L^+}(\N^+)$, which we fix by citing \cite[Theorem~C(c)]{Henke:2020}. For the convenience of the reader we repeat the complete argument.

\begin{lemma}\label{L:OupperpVaryObjects}
Let $(\L^+,\Delta^+,S)$ and $(\L,\Delta,S)$ be linking localities over the same fusion system $\F$ with $\Delta\subseteq\Delta^+$ and $\L=\L^+|_\Delta$. Adopt Notation~\ref{N:VaryObjects}. Then 
\[O^p_\L(\N)^+=O^p_{\L^+}(\N^+)\mbox{ and }O^{p^\prime}_\L(\N)^+=O^{p^\prime}_{\L^+}(\N^+).\]
\end{lemma}

\begin{proof}
Let $*$ be one of the symbols $p$ or $p^\prime$. Set $\mathbb{K}_\N^*=\mathbb{K}_\N$ and $\mathbb{K}^*_{\N^+}:=\{\K^+\unlhd\L^+\colon T\K^+=\N^+\}$ if $*$ stands for $p$, and set $\mathbb{K}_\N^*=\mathbb{K}_\N^\prime$ and $\mathbb{K}^*_{\N^+}:=\{\K^+\unlhd\L^+\colon T\subseteq\K^+\subseteq\N^+\}$ if $*$ stands for the symbol $p^\prime$. Note that in either case
\[O^*(\L)=\bigcap\mathbb{K}_\N^*\mbox{ and }O^*_{\L^+}(\N^+)=\bigcap\mathbb{K}^*_{\N^+}.\]
We will now consider the map $\Phi_{\L^+,\L}$ introduced at the beginning of  Section~\ref{SS:VaryObjects}, which is by Theorem~\ref{T:VaryObjects}(b) an inclusion-preserving bijection whose inverse map is also inclusion-preserving. It turns out now that $\Phi_{\L^+,\L}$ induces a bijection
\[\mathbb{K}^+_{\N^+}\rightarrow \mathbb{K}_\N,\K^+\mapsto \K^+\cap \L.\]
If $*$ stands for $p$, then this is true by \cite[Theorem~C(c)]{Henke:2020}, and if $*$ stands for $p^\prime$, then this follows directly from the properties of $\Phi_{\L^+,\L}$. Observe furthermore that $\Phi_{\L^+,\L}$ commutes with taking intersections. Hence, we have
\[\Phi_{\L^+,\L}(O^*_{\L^+}(\N^+))=\Phi_{\L^+,\L}\left(\bigcap\mathbb{K}^*_{\N^+}\right)=\bigcap \mathbb{K}^*_\N=O^*_\L(\N).\]
This proves the assertion.
\end{proof}

\begin{remark}\label{R:TakePlus}
Let $(\L^+,\Delta^+,S)$ be a linking locality over $\F$ with $\Delta\subseteq\Delta^+$ and $\L^+|_\Delta=\L$. Recall from Lemma~\ref{L:NLTCLT}(b),(c) that $(N_{\L^+}(T),\Delta^+,S)$ and $(N_\L(T),\Delta,S)$ are linking localities over $N_\F(T)$, $C_{\L^+}(T)\unlhd N_{\L^+}(T)$ and $C_\L(T)\unlhd N_\L(T)$. Using \cite[Lemma~2.23(b)]{Henke:2020}, one observes moreover that $N_{\L^+}(T)|_\Delta=N_{\L^+}(T)\cap\L=N_\L(T)$ and $C_{\L^+}(T)\cap N_\L(T)=C_\L(T)$. Hence, by Lemma~\ref{L:OupperpVaryObjects}, we have 
\[O^p_{N_{\L^+}(T)}(C_{\L^+}(T))\cap \L=O^p_{N_{\L^+}(T)}(C_{\L^+}(T))\cap N_\L(T)=O^p_{N_\L(T)}(C_\L(T)).\]
In particular,
\[O^p_{N_{\L^+}(T)}(C_{\L^+}(T))\cap S=O^p_{N_\L(T)}(C_\L(T))\cap S.\]
\end{remark}

\section{Images of $p$-residuals and $p^\prime$-residuals under isomorphisms} Recall the definitions from Section~\ref{SS:LocalityHomomorphism}.

\begin{lemma}\label{L:AutResidue}
Let $(\tL,\tDelta,\tS)$ be a locality and $\alpha\in\Iso(\L,\tL)$. Then $\N\alpha\unlhd\tL$, $O^p_\L(\N)\alpha=O^p_{\tL}(\N\alpha)$ and $O^{p^\prime}_\L(\N)\alpha=O^{p^\prime}_{\tL}(\N\alpha)$.
\end{lemma}

\begin{proof}
By \cite[Proposition~2.17]{Chermak:2015}, $\alpha$ can be factored as the composition of an element of $\Iso((\L,S),(\tL,\tS))$ and an inner automorphism of $\L$. Thus, as $O^p_\L(\N)$ and $O^{p^\prime}_\L(\N)$ are partial normal subgroups of $\L$, we may assume $S\alpha=\tS$. Then $T\alpha=\N\alpha\cap \tS$. It is now easy to check that $\alpha$ induces a bijection $\K\mapsto \K\alpha$ from $\mathbb{K}_\N$ to $\mathbb{K}_{\N\alpha}$ and similarly from $\mathbb{K}_\N^\prime$ to $\mathbb{K}_{\N\alpha}^\prime$. Thus, $\alpha$ maps $O^p_\L(\N)=\bigcap\mathbb{K}_\N$ onto $O^p_{\tL}(\N\alpha)=\bigcap\mathbb{K}_{\N\alpha}$ and $O^{p^\prime}_\L(\N)=\bigcap\mathbb{K}_\N^\prime$ onto $O^{p^\prime}_{\tL}(\N\alpha)=\bigcap\mathbb{K}_{\N\alpha}^\prime$.
\end{proof}

\begin{lemma}\label{L:AutResidueCT}
Let $(\tL,\tDelta,\tS)$ be a locality and $\alpha\in\Iso((\L,S),(\tL,\tS))$. Suppose $T\leq S$ is strongly closed in $\F_S(\L)$. Setting 
\[\L_T:=N_\L(T),\;\C_T:=C_\L(T),\;\L_{T\alpha}:=N_{\tL}(T\alpha)\mbox{ and }\C_{T\alpha}:=C_{\tL}(T\alpha),\] we have then $\alpha|_{\L_T}\in\Iso(\L_T,\L_{T\alpha})$ and $O^p_{\L_T}(\C_T)\alpha=O^p_{\L_{T\alpha}}(\C_{T\alpha})$. 
\end{lemma}

\begin{proof}
Observe that $T\alpha$ is strongly closed in $\F$. By Lemma~\ref{L:NLTCLT}, $(\L_T,\Delta,S)$ and $(\L_{T\alpha},\tDelta,\tS)$ are localities with $\C_T\unlhd\L_T$ and $\C_{T\alpha}\unlhd\L_{T\alpha}$. So $O^p_{\L_T}(\C_T)$ and $O^p_{\L_{T\alpha}}(\C_{T\alpha})$ are well-defined. One checks easily that $\L_T\alpha=\L_{T\alpha}$ and $\C_T\alpha=\C_{T\alpha}$. In particular $\alpha|_{\L_T}\in\Iso(\L_T,\L_{T\alpha})$. Now  Lemma~\ref{L:AutResidue}  yields the assertion. 
\end{proof}

\section{Quasisimple linking localities}

\begin{definition}\label{D:SimpleQuasisimple}~
\begin{itemize}
\item A partial group $\L$ is called \emph{simple} if there are exactly two partial normal subgroups of $\L$ (namely $\{\One\}$ and $\L$).  
\item A linking locality $(\L,\Delta,S)$ is called \emph{quasisimple} if $\L=O^p(\L)$ and $\L/Z(\L)$ is simple. We say then also that $\L$ is quasisimple. 
\end{itemize}
\end{definition}

Notice that $\L\neq\{\One\}$ for every simple partial group $\L$. The following lemma is essentially the same as \cite[Lemma~8.2]{ChermakIII}.

\begin{lemma}\label{L:Quasisimple}
Let $(\L,\Delta,S)$ be a quasisimple linking locality and $\N\subn\L$. Then either $\N\leq Z(\L)$ or $\N=\L$. In particular, $O_p(\L)=Z(\L)$, $O^{p^\prime}(\L)=\L$ and $S$ is not abelian.
\end{lemma}

\begin{proof}
Set $Z:=Z(\L)$. If $\N\neq \L$, then $\N$ is contained in a partial normal subgroup of $\L$ which is properly contained in $\L$. Hence, we may assume without generality that $\N\unlhd\L$. Consider the canonical projection $\alpha\colon\L\rightarrow\L/Z$ which exists by 
\cite[Corollary~4.5]{Chermak:2015}. By Theorem~\ref{T:ProductsPartialNormal}, $Z\N$ is a partial normal subgroup of $\L$. Thus, by the Partial Subgroup Correspondence \cite[Proposition~4.7]{Chermak:2015}, $\N\alpha=(Z\N)\alpha$ is a partial normal subgroup of $\L/Z$. As $\L/Z$ is simple, it follows $\N\alpha=1$ or $\N\alpha=\L$. In the first case, $\N\leq \ker(\alpha)=Z$. In the second case, again by the Partial Subgroup Correspondence, we have $\L=Z\N$ and so in particular, $\L=S\N$. Therefore, it follows in this case from the definition of $O^p(\L)$ that $\L=O^p(\L)\subseteq\N$ and thus $\N=\L$. This proves the first part of the assertion. 

\smallskip

As $\{\One\}\neq \L=O^p(\L)$, we have $\L\neq O_p(\L)$. Hence, $O_p(\L)\leq Z$. On the other hand, as $(\L,\Delta,S)$ is a linking locality, $Z\leq C_\L(S)\leq S$ and thus $Z\leq O_p(\L)$. This shows $O_p(\L)=Z(\L)$. 

\smallskip

If $S\leq Z$, then $\L=C_\L(S)\leq S$, a contradiction to $O^p(\L)=\L\neq\{\One\}$. In particular, it follows from $S\subseteq O^{p^\prime}(\L)\unlhd\L$ that  $\L=O^{p^\prime}(\L)$. If $S$ is abelian, then by Alperin's fusion theorem for localities \cite{Molinier:2016}, we have $S\unlhd\L$ and thus $S=O_p(\L)=Z(\L)$, a contradiction as before. 
\end{proof}

\chapter{$\N$-Replete localities}\label{S:NReplete}

\textbf{Throughout this chapter let $(\L,\Delta,S)$ be a linking locality over a fusion system $\F$, let $\N\unlhd\L$ and $T:=S\cap\N$. Set $\L_T:=N_\L(T)$ and $\C_T:=C_\L(T)$.}

\smallskip

Roughly speaking, we will say that $(\L,\Delta,S)$ is $\N$-replete (or weakly $\N$-replete under weaker assumptions) if the object set $\Delta$ is large enough for $(\L,\Delta,S)$ to have nice properties with respect to $\N$.
After introducing the precise definitions, we will prove in Section~\ref{SS:ProveNReplete} some criterions for $(\L,\Delta,S)$ to be $\N$-replete. The results in Section~\ref{SS:ApplyNReplete} will then give a first indication that localities which are $\N$-replete or weakly $\N$-replete are convenient to work with, and that general properties of linking localities can be proved by reducing to the $\N$-replete case. The concept of $\N$-repleteness will also be crucial in later chapters.

\smallskip

Throughout this chapter we will use without further reference that, by Lemma~\ref{L:NLTCLT}(b),(c), the triple $(\L_T,\Delta,S)$ is a linking locality over $N_\F(T)$ and $\C_T\unlhd \L_T$. In particular, it makes sense to say that an object $Q\in\Delta$ is $S\C_T$-radical or $\C_T$-radical (in $\L_T$), and the set $R_\Delta(S\C_T)$ is defined.

\section{Main definition} 

The results in this chapter will be centered around the following definition. 

\begin{definition}\label{D:NReplete}~
\begin{itemize}
\item For every $P,Q\in\Delta$ such that $P$ is $S\N$-radical in $\L$ and $Q$ is $S\C_T$-radical in $\L_T$, we set
\[X_{P,Q}:=(P\cap T)(Q\cap C_S(T))O_p(\L).\]
\item The linking locality $(\L,\Delta,S)$ is called \emph{$\N$-replete} if $X_{P,Q}\in\Delta$ for all $P\in R_\Delta(S\N)$ and all $Q\in R_\Delta(S\C_T)$.  
\item We say that $(\L,\Delta,S)$ is \emph{weakly $\N$-replete} if 
\[X_{P,S}:=(P\cap T)C_S(T)O_p(\L)\in \Delta\mbox{ for all }P\in R_\Delta(S\N).\]  
\end{itemize}
\end{definition}

\section{Proving $\N$-repleteness}\label{SS:ProveNReplete}

In this section, we will give some sufficient conditions for $(\L,\Delta,S)$ to be $\N$-replete. In particular, it will turn out that $(\L,\Delta,S)$ is $\N$-replete if $(\L,\Delta,S)$ is a linking locality with $\F^q\subseteq\Delta$. The following lemma will be useful. The reader might want to recall Notation~\ref{N:FRc}.

\begin{lemma}\label{L:GetFTcIntoFq}
If $\hyp(C_\F(T))\leq T$ or $T\in\F^q$, then $\F_T^c\subseteq\F^q$.
\end{lemma}

\begin{proof}
By Theorem~\ref{T:VaryObjects}, there is a subcentric locality $(\L^s,\F^s,S)$ over $\F$ and  a partial normal subgroup $\N^s\unlhd\L^s$ with $\L=\L^s|_\Delta$ and $\N^s\cap\L=\N$. We have then $\N^s\cap S=T$. So replacing $(\L,\Delta,S)$ and $\N$ by $(\L^s,\F^s,S)$ and $\N^s$, we may assume without loss of generality that  
\[\Delta=\F^s.\]
Assume now that $\hyp(C_\F(T))\leq T$ or $T\in\F^q$. By Lemma~\ref{L:TFq}, we have in either case $T\in\F^q$. Suppose there exists $U\in\F_T^c$ such that $U\not\in\F^q$ and choose such $U$ of maximal order. As $T$ is strongly closed, the set $\F_T^c$ is closed under $\F$-conjugacy. Thus, we may choose $U$ such that $U$ is fully $\F$-normalized. Then $C_\F(U)$ is saturated. Moreover, as $U$ is not quasicentric, $C_\F(U)$ is not the fusion system of a $p$-group. Thus, by Alperin's fusion theorem, there exists $R\in C_\F(U)^c$ such that $\Aut_{C_\F(U)}(R)$ is not a $p$-group. By \cite[Lemma~3.14]{Henke:2015}, we have $Q:=UR\in\F^c\subseteq\Delta$. In particular, every $\F$-automorphism of $Q$ is realized by an element in $N_\L(Q)$. A $p^\prime$-automorphism in $\Aut_{C_\F(U)}(R)$ extends to an automorphism of $Q=UR$ which acts trivially on $U$ and is then a $p^\prime$-automorphism of $Q$. Hence, $C_{N_\L(Q)}(U)$ is not a $p$-group. Moreover, as $R\leq C_S(U)$, we have $R\cap T\leq C_T(U)\leq U$ and $Q\cap T=U(R\cap T)=U$. 

\smallskip

By \cite[Lemma~2.10]{Chermak:2015}, we may pick $f\in\L$ such that $Q\leq S_f$ and $N_S(Q^f)\in\Syl_p(N_\L(Q^f))$. Set $Q^*:=Q^f$ and $U^*:=U^f$. As $Q\cap T=U$, we have $Q^*\cap T=U^*$. By Lemma~\ref{L:ConjugateNormalizer}, $c_f\colon N_\L(Q)\rightarrow N_\L(Q^*)$ is an isomorphism of groups which takes $U$ to $U^*$. 
In particular, $C_{N_\L(Q^*)}(U^*)$ is not a $p$-group. 

\smallskip

Notice that $N_\N(Q^*)$ is normal in $N_\L(Q^*)$. So $N_T(Q^*)=N_\N(Q^*)\cap N_S(Q^*) \in\Syl_p(N_\N(Q^*))$ and, by Lemma~\ref{L:MSCharp}(a), $N_\N(Q^*)$ is of characteristic $p$. As $U\in\F_T^c$, the former fact implies $C_{O_p(N_\N(Q^*))}(U^*)\leq C_{N_T(Q^*)}(U^*)\leq U^*$. Notice that $U^*=Q^*\cap T$ is normal in $N_\L(Q^*)$ and in $N_\N(Q^*)$. Hence, it follows from Lemma~\ref{L:CGUleqU} applied with $(N_\N(Q^*),U^*)$ in place of $(G,U)$ that  $C_{N_\N(Q^*)}(U^*)\leq U^*$. Thus, we get from Lemma~\ref{L:CGNCGQ} applied with $(N_\L(Q^*),N_\N(Q^*),U^*)$ in place of $(G,N,U)$ that  $O^p(C_{N_\L(Q^*)}(U^*))=O^p(C_{N_\L(Q^*)}(N_\N(Q^*)))\leq C_{N_\L(Q^*)}(N_T(Q^*))$. Set
\[V:=N_T(Q^*).\]
As $C_{N_\L(Q^*)}(U^*)$ is not a $p$-group, we can conclude from the above that $X:=C_{N_\L(Q^*)}(V)$ is not a $p$-group. Since $V\unlhd N_S(Q^*)\in\Syl_p(N_\L(Q^*))$, we have $N_S(Q^*)\cap C_S(V)\in\Syl_p(X)$. In particular, $P:=O_p(X)\leq C_S(V)$ and $\Aut_X(P)\leq \Aut_{C_\F(V)}(P)$.  By Lemma~\ref{L:MSCharp}(b), $X$ is of characteristic $p$. So $\Aut_X(P)\leq \Aut_{C_\F(V)}(P)$ is not a $p$-group and thus $V\not\in\F^q$. Since $T\in\F^q$, we have $Q^*\cap T=U^*<T$ and thus $Q^*<TQ^*$. This implies $Q^*<N_{TQ^*}(Q^*)=N_T(Q^*)Q^*$ and therefore $V=N_T(Q^*)\not\leq Q^*$. So $U^*=Q^*\cap T$ is properly contained in $V$. As $U^*\leq V\leq T$ with $U^*\in\F_T^c$, we have $V\in\F_T^c$. Hence, the maximality of $|U|=|U^*|$ yields $V\in\F^q$. Because of this contradiction we have proved the assertion.   
\end{proof}

The proof of the next lemma uses a very nice trick, which is taken from Chermak's proof of \cite[Lemma~4.1]{ChermakIII}. The Lemma itself is however new. Again we use Notation~\ref{N:FRc}.

\begin{lemma}\label{L:NiceTrick}
Let $\K$ be a partial normal subgroup of $\L_T$ such that $\K\subseteq \C_T$. Set $R:=\K\cap S$ and assume that $\F_{TR}^c\subseteq \Delta$. Then $UV\in \F_{TR}^c$ for every $U\in\F_T^c$ and every $V\in\F_R^c$. 
\end{lemma}

\begin{proof}
Suppose the assertion is false. Among all  $(U,V)\in\F_T^c\times \F_R^c$ with $UV\not\in\F_{TR}^c$ pick $(U,V)$ such that $|U||V|$ is as large as possible. Set $X:=UV$ and fix $X^*\in X^\F$ fully $\F$-normalized.  

\smallskip

\emph{Step~1:} We argue that $N_{TR}(X)\in\F_{TR}^c\subseteq\Delta$. Since $TR\in\F_{TR}^c$, we have $U<T$ or $V<R$. So $U<N_T(U)$ or $V<N_R(V)$, which implies $|N_T(U)| |N_R(V)|>|U||V|$. Moreover, as $[R,T]=1$, we have $N_T(U)N_R(V)\leq N_{TR}(X)$. Observe that, for every subgroup $Y\leq S$, the set $\F_Y^c$ is overgroup closed in $Y$. So $N_T(U)\in\F_T^c$ and $N_R(V)\in\F_R^c$. Moreover, because of the maximality of $|U||V|$, it follows that $N_{TR}(X)\geq N_T(U)N_R(V)$ is an element of $\F_{TR}^c$ and thus of $\Delta$. 

\smallskip

\emph{Step~2:} We argue that we can choose $(U,V)$ such that $X^*=X^n$ for some $n\in\N$ with $X\leq S_n$. To see this pick $\alpha\in\Hom_\F(N_S(X),S)$ such that $X^*=X\alpha$; such $\alpha$ exists by \cite[Lemma~2.6(c)]{Aschbacher/Kessar/Oliver:2011}. By Step~1, we have $N_S(X)\in\Delta$ and thus, by Lemma~\ref{L:LocalityFusionSystem}(a), $\alpha=c_g$ for some $g\in\L$ with $N_S(X)\leq S_g$. So $X^*=X^g=U^gV^g$. By the Frattini Lemma and the Splitting Lemma \cite[Corollary~3.11, Lemma~3.12]{Chermak:2015}, we can write $g=fn$, where $f\in \L_T$, $n\in \N$ and $S_g=S_{(f,n)}$. Similarly, applying the Frattini Lemma and the Splitting Lemma to $\K\unlhd \L_T$, we can write $f=hk$ where $h\in N_{\L_T}(R)=N_\L(TR)$, $k\in \K\subseteq \C_T$ and $S_f=S_{(h,k)}$. Then $X\leq S_g=S_{(h,k,n)}$, $(h,k,n)\in\D$ and $g=hkn$.  Notice that $U^{hk}\leq T$ and $V^{hk}\leq R$, so $U^{hk}\in\F_T^c$ and $V^{hk}\in\F_R^c$. Since $X^*=X^g=((U^{hk})(V^{hk}))^n$, we may replace $(U,V)$ by $(U^{hk},V^{hk})$ to get $X^*=X^n$. This completes Step~2.

\smallskip

\emph{Step~3:} We will reach now the final contradiction. According to Step~2, we may assume that $X^*=X^n=U^nV^n$ for some $n\in\N$ with $X\leq S_n$. By Proposition~\ref{P:ChermakEndPartI}, $TR$ is strongly closed in $\F$. Hence, by Lemma~\ref{L:FTcFRc}(a), we reach a contradiction if we can show that $C_{TR}(X^*)\leq X^*$. Set $U^*=U^n$ and $V^*=V^n$ so that $X^*=U^*V^*$. Observe that $U^*\leq T\leq C_S(R)$. As $U\in\F_T^c$, we have $C_T(U^*)\leq U^*$. Moreover, $[U,V]=[T,R]=1$ implies $[U^*,V^*]=1$. By \cite[Lemma~3.1(b)]{Chermak:2015}, we have also
\[VT=V^*T.\]
As $R\leq C_S(T)$, it follows $C_R(V)=C_R(VT)=C_R(V^*T)=C_R(V^*)$. Hence, we are able to compute now
\begin{eqnarray*}
 C_{TR}(X^*)&=& C_{TR}(U^*V^*)\\
&=& C_{TR}(U^*)\cap C_{TR}(V^*)\\
&=& Z(U^*)R \cap C_{TR}(V^*)\mbox{ \;(as $C_T(U^*)=Z(U^*)$)}\\
&=& Z(U^*)C_R(V^*)\mbox{ \;(as $Z(U^*)\leq U^*\leq C_{TR}(V^*)$)}\\
&=& Z(U^*)C_R(V)\mbox{ \;(as $C_R(V)=C_R(V^*)$)}\\
&=& Z(U^*)Z(V)\mbox{ \;(as $V\in\F_R^c$)}.
\end{eqnarray*}
Since $Z(X^*)\leq C_{TR}(X^*)$, in order to show that $C_{TR}(X^*)\leq X^*$ we only need to show that $|Z(U^*)Z(V)|$ equals $|Z(X^*)|=|Z(X)|=|Z(U)Z(V)|$. Since $C_T(U)=Z(U)$ and $V\leq C_S(T)\leq C_S(U)$, we have $Z(V)\cap Z(U)\leq Z(V)\cap U\leq Z(V)\cap T\leq Z(V)\cap C_T(U)\leq Z(V)\cap Z(U)$. So equality holds and $Z(V)\cap Z(U)=Z(V)\cap T$. Similarly, as $C_T(U^*)=Z(U^*)$ and $V\leq C_S(T)\leq C_S(U^*)$, we have $Z(V)\cap Z(U^*)=Z(V)\cap T$. Thus $Z(V)\cap Z(U)=Z(V)\cap T=Z(V)\cap Z(U^*)$ and $|Z(U)Z(V)|=|Z(U^*)Z(V)|$. As argued above, this yields $C_{TR}(X^*)\leq X^*$ and thus $X\in\F_{TR}^c$, contradicting the choice of $(U,V)$.  
\end{proof}

\begin{lemma}\label{L:NRadKRadIntersect}
Let $\K$ be a partial normal subgroup of $\L_T$ such that $R:=\K\cap S\leq C_S(T)$. Let $P,Q\in\Delta$ such that $P$ is $\N$-radical in $\L$ and $Q$ is $\K$-radical in $\L_T$. Then $P\cap T\in\F_T^c$ and $Q\cap R\in\F_R^c$. In particular, if $\F_{TR}^c\subseteq \Delta$, then
\[(P\cap T)(Q\cap R)\in\F_{TR}^c\subseteq\Delta\]
and $(\L,\Delta,S)$ is $\N$-replete.
\end{lemma}

\begin{proof}
It follows from Lemma~\ref{L:NradEcr} that $P\cap T\in\F_T^c$ and $Q\cap R\in N_\F(T)_R^c$. By Lemma~\ref{L:FTcFRc}(c) we have $N_\F(T)_R^c=\F_R^c$, so indeed $Q\cap R\in \F_R^c$.

\smallskip

If we suppose now that $\F_{TR}^c\subseteq\Delta$, then it follows from Lemma~\ref{L:NiceTrick} that $(P\cap T)(Q\cap R)\in\F_{TR}^c\subseteq\Delta$. By Lemma~\ref{L:SNradNrad}, every $S\N$-radical subgroup is $\N$-radical and every $S\C_T$-radical subgroup is $\C_T$-radical. As $\Delta$ is overgroup-closed and $R\leq C_S(T)$, it follows thus from Lemma~\ref{L:SNradNrad} that $(\L,\Delta,S)$ is then $\N$-replete. 
\end{proof}

\begin{lemma}\label{L:PutTogetherK}
Assume $\F^q\subseteq \Delta$. Let $\K\unlhd \L_T$ such that $\K\subseteq \C_T$ and $\hyp(C_\F(T))\leq R:=\K\cap S$. Then
\[\F_{TR}^c\subseteq\F^q\subseteq\Delta\]
and
\[(P\cap T)(Q\cap R)\in\F_{TR}^c\subseteq\Delta\] 
whenever $P\in\Delta$ is $\N$-radical and $Q\in\Delta$ is $\K$-radical in $\L_T$. In particular $(\L,\Delta,S)$ is $\N$-replete.
\end{lemma}

\begin{proof}
By Proposition~\ref{P:ChermakEndPartI}, $\M:=\<\N,\K\>$ is a partial normal subgroup of $\L$ with $\M\cap S=TR$. Note that $C_\F(TR)\subseteq C_\F(T)$ and thus $\hyp(C_\F(TR))\leq \hyp(C_\F(T))\leq R\leq TR$. It follows thus from Lemma~\ref{L:GetFTcIntoFq} that $\F_{TR}^c\subseteq \F^q\subseteq\Delta$. Now Lemma~\ref{L:NRadKRadIntersect} implies that $(P\cap T)(Q\cap R)\in\F_{TR}^c\subseteq\Delta$. So $(\L,\Delta,S)$ is $\N$-replete by Lemma~\ref{L:SNradNrad} as $\Delta$ is overgroup-closed in $S$.
\end{proof}

\begin{cor}\label{C:GetNreplete}
If $\F^q\subseteq\Delta$, then $\F_{TC_S(T)}^c\subseteq\F^q\subseteq\Delta$ and
\[(P\cap T)(Q\cap C_S(T))\in\F_{TC_S(T)}^c\subseteq\Delta\] 
whenever $P\in\Delta$ is $\N$-radical and $Q\in\Delta$ is $\C_T$-radical in $\L_T$. In particular, $(\L,\Delta,S)$ is $\N$-replete. 
\end{cor}

\begin{proof}
Recall that $\K:=\C_T$ is a partial normal subgroup of $\L_T$ and note that $\hyp(C_\F(T))\leq C_S(T)=\K\cap S$. The assertion follows now from Lemma~\ref{L:PutTogetherK}.
\end{proof}

\section{Applications and results}\label{SS:ApplyNReplete}

In this section, we prove some results about weakly $\N$-replete localities. Moreover, in Lemma~\ref{L:focCFT} and Lemma~\ref{L:OpCTcapSSpecial}, we give applications of the results from the previous sections. 

\smallskip

We start with the following general lemma. If $(\L,\Delta,S)$ is $\N$-replete and $\K$ is as in the lemma, then the hypothesis is fulfilled for every $P\in R_\Delta(S\N)$, every $Q\in R_\Delta(S\C_T)$, and for $X:=X_{P,Q}$.

\begin{lemma}\label{L:CommuteInNLX}
Let $\K\unlhd\L_T$ such that $\K\subseteq\C_T$ and set $R:=\K\cap S$. Suppose we are given $P,Q,X\in\Delta$ such that $P$ is $\N$-radical, 
\[P\cap T\leq X\leq P\mbox{ and }Q\cap R\leq X\leq Q.\]
Then the following hold:
\begin{itemize}
 \item [(a)] $X$ is $\N$-radical in $\L$.
 \item [(b)] $N_\N(P)$ and $N_{\K}(Q)$ are contained in $N_\L(X)$. Moreover, computing in the group $N_\L(X)$, we have $[N_\N(P),O^p(N_{\K}(Q))]=1$ and $[N_\N(P),N_\K(Q)']=1$.  
\end{itemize}
\end{lemma}

\begin{proof}
Lemma~\ref{L:PradQrad} implies that (a) holds and  Lemma~\ref{L:Lemma0} (applied once with $P,X,\N,\L$ and once with $Q,X,\K,\L_T$ in the roles of $P,Q,\N,\L$ of that lemma) gives that $N_\N(P)$ and $N_\K(Q)$ are contained in $G:=N_\L(X)$. In particular
\[N_\N(P)\leq N_\N(X)\mbox{ and }N_{\K}(Q)\leq G\cap \C_T\leq C_G(U).\] 
As $X$ is $\N$-radical by (a), we have $U:=O_p(N_\N(X))=X\cap T$. By Lemma~\ref{L:CharpCGNCGQ}, $[N_\N(X),O^p(C_G(U))]=1=[N_\N(X),C_G(U)']$. This implies (b).
\end{proof}

\begin{lemma}\label{L:focCFT}
We have $\foc(C_\F(T))\leq C_S(\N)$. 
\end{lemma}

\begin{proof}
By Theorem~\ref{T:VaryObjects}(b),(c), there exist a subcentric locality $(\L^s,\F^s,S)$ over $\F$ and $\N^s\unlhd\L^s$ such that $\L=\L^s|_\Delta$ and $\N^s\cap\L=\N$. Then $\N^s\cap S=\N\cap S=T$ and $C_S(\N^s)\leq C_S(\N)$. Hence, replacing $(\L,\Delta,S)$ and $\N$ by $(\L^s,\F^s,S)$ and $\N^s$, we may assume without loss of generality that
\[\Delta=\F^s.\]
By Lemma~\ref{L:NLTCLT}(c), we have then $\F_{C_S(T)}(\C_T)=C_\F(T)$. Hence, Lemma~\ref{L:DescribeFocE} gives 
\[\foc(C_\F(T))=\<[Q\cap C_S(T),N_{\C_T}(Q)]\colon Q\in R_\Delta(S\C_T)\>.\]
As $\Delta=\F^s$, we get from Corollary~\ref{C:GetNreplete} that $(\L,\Delta,S)$ is $\N$-replete. In particular, fixing $P\in R_\Delta(S\N)$ and all $Q\in R_\Delta(S\C_T)$, we have $X:=X_{P,Q}\in\Delta$. By Lemma~\ref{L:SNradNrad} and Lemma~\ref{L:CSTinSNradical}, $P$ is $\N$-radical, $Q$ is $\C_T$-radical, we have $C_S(T)O_p(\L)\leq P$ and $TO_p(\L)\leq O_p(\L_T)\leq Q$. In particular, $P\cap T\leq X\leq P$ and $Q\cap C_S(T)\leq X\leq Q$. Hence, by Lemma~\ref{L:CommuteInNLX}(b) applied with $\K=\C_T$, we have $[N_\N(P),N_{\C_T}(Q)^\prime]=1$ in $N_\L(X)$ and thus 
\[[Q\cap C_S(T),N_{\C_T}(Q)]\leq N_{\C_T}(Q)^\prime\cap S\leq C_S(N_\N(P)).\]
Using Lemma~\ref{L:CSNAlperin} we conclude that
\[\foc(C_\F(T))\subseteq \bigcap_{P\in R_\Delta(S\N)}C_S(N_\N(P))=C_S(\N).\]
\end{proof}

\begin{lemma}\label{L:CSNstronglyclosed}
If $(\L,\Delta,S)$ is weakly $\N$-replete, then $C_S(\N)$ is strongly closed in $\F$.
\end{lemma}

\begin{proof}
Set $H:=N_\L(TC_S(T))$ and notice that $H$ acts on $TC_S(T)O_p(\L)$. We will first show that $C_S(\N)$ is $H$-invariant. As $(\L,\Delta,S)$ is weakly $\N$-replete, we have $X_{P,S}\in\Delta$ for every $P\in R_\Delta(S\N)$. Since $\Delta$ is $\F$-closed, it follows 
\[\Gamma:=\{X_{P,S}^h\colon P\in R_\Delta(S\N),\;h\in H\}\subseteq\Delta.\]
By Lemma~\ref{L:CSTinSNradical}, for every $P\in R_\Delta(S\N)$, we have $C_S(T)\leq P$ and thus $X_{P,S}=(P\cap T)C_S(T)O_p(\L)=(P\cap TC_S(T))O_p(\L)$.  Since $TC_S(T)$ is by Corollary~\ref{C:ChermakEndPartI} strongly closed, it follows that $N_\N(P)\subseteq N_\N(X_{P,S})$ and thus $C_{C_S(T)}(N_\N(X_{P,S}))\leq C_{C_S(T)}(N_\N(P))$ for any such $P$. Hence, using Lemma~\ref{L:CSNAlperin}, we see that
\begin{eqnarray*}
C_S(\N)&\subseteq & \bigcap_{X\in\Gamma}C_{C_S(T)}(N_\N(X))\\
&\subseteq &\bigcap_{P\in R_\Delta(S\N)}C_{C_S(T)}(N_\N(X_{P,S}))\subseteq \bigcap_{P\in R_\Delta(S\N)}C_{C_S(T)}(N_\N(P))\subseteq C_S(\N) 
\end{eqnarray*}
and thus equality holds everywhere above. So $C_S(\N)=\bigcap_{X\in\Gamma}C_{C_S(T)}(N_\N(X))$. Notice that $H$ acts on $C_S(T)$ and on $\Gamma$. As $\Gamma\subseteq\Delta$, it follows thus from Lemma~\ref{L:ConjugateNormalizer} that $H$ acts also on the set $\{N_\N(X)\colon X\in\Gamma\}$ and thus on the set $\{C_S(N_\N(X))\colon X\in\Gamma\}$. Hence, $C_S(\N)$ is $H$-invariant.

\smallskip

We are now in a position to show that $C_S(\N)$ is strongly closed in $\F$. Let $Y\leq C_S(\N)$ and $g\in\L$ such that $Y\leq S_g$. Since $\F_S(\L)$ is generated by group homomorphisms between subgroups of $S$ which are induced by conjugation with an element of $\L$, it is sufficient to show that $Y^g\leq C_S(\N)$. By the Frattini Lemma and the Splitting Lemma \cite[Corollary~3.11, Lemma~3.12]{Chermak:2015}, there exist $n\in\N$ and $f\in \L_T$ such that $(n,f)\in\D$, $g=nf$ and $S_g=S_{(n,f)}$. Applying the Frattini Lemma and the Splitting Lemma now to $\C_T\unlhd \L_T$, we see that there exist also $c\in \C_T$ and $h\in N_{\L_T}(C_S(T))=H$ such that $(c,h)\in\D$, $f=ch$ and $S_f=S_{(c,h)}$. Then $S_{(n,c,h)}=S_{(n,f)}=S_g\in\Delta$ and $g=nch$. As $Y\leq C_S(\N)$, it follows $Y^g=(Y^c)^h$. By Lemma~\ref{L:focCFT} we have $\foc(C_\F(T))\leq C_S(\N)$. So by Lemma~\ref{L:focFinTStrCl}, $C_S(\N)$ is strongly closed in $C_\F(T)$, which implies that $Y^c\leq C_S(\N)$. As $C_S(\N)$ is $H$-invariant, it follows that $Y^g=(Y^c)^h\leq C_S(\N)$. So $C_S(\N)$ is strongly closed in $\F$. 
\end{proof}

\begin{lemma}\label{L:CSNCSNPlus}
Suppose $(\L,\Delta,S)$ is weakly $\N$-replete and $(\L^+,\Delta^+,S)$ is a linking locality over $\F$ with $\Delta\subseteq\Delta^+$ and $\L=\L^+|_\Delta$. Adapting Notation~\ref{N:VaryObjects}, we have then
\[C_S(\N)=C_S(\N^+).\]
\end{lemma}

\begin{proof}
 As $\N\subseteq\N^+$, we have $C_S(\N^+)\leq C_S(\N)$. Thus, it remains to show that $R:=C_S(\N)\leq C_S(\N^+)$. By Lemma~\ref{L:CSNstronglyclosed}, $R$ is strongly closed in $\F$. Notice also that $\N\subseteq C_\L(R)$ by Lemma~\ref{L:CentralizerHelp}. Using that $(\L,\Delta,S)$ is weakly $\N$-replete and $S\in R_\Delta(S\N)$, we see moreover that
\[X:=TC_S(T)O_p(\L)\in \Delta.\]
By \cite[Lemma~3.5]{Chermak:2015}, we have $N_{\N^+}(T)\subseteq N_{\L^+}(TC_S(T))$. By Lemma~\ref{L:TakePlusInS}, we have moreover $O_p(\L)=O_p(\L^+)$. Hence, as $\L=\L^+|_\Delta$, we can conclude that $N_{\N^+}(T)\subseteq N_{\L^+}(X)\subseteq\L$ and thus $N_{\N^+}(T)\subseteq \N^+\cap\L=\N\subseteq C_\L(R)\subseteq C_{\L^+}(R)$. Thus, by \cite[Proposition~8.2]{Henke:2015} applied with $(\L^+,\Delta^+,S)$ and $\N^+$ in place of $(\L,\Delta,S)$ and $\N$, we have $\N^+\subseteq C_{\L^+}(R)$. It follows now from Lemma~\ref{L:CentralizerHelp} that $R\leq C_S(\N^+)$ as required.
\end{proof}

In the following lemma we consider a very special case of $\N$-replete localities. Lemma~\ref{L:OpCTcapSSpecial}(d) will be particularly useful in combination with Lemma~\ref{L:CSNCSNPlus}.

\begin{lemma}\label{L:OpCTcapSSpecial}
 Let $\Delta$ be the set of overgroups in $S$ of the elements in $\F_{TC_S(T)}^c$. Then the following hold:
\begin{itemize}
 \item [(a)] $(\L,\Delta,S)$ is $\N$-replete.
 \item [(b)] Setting $\Gamma:=\{X_{S,Q}^h\colon Q\in R_\Delta(S\C_T),\;h\in N_\L(TC_S(T))\}$, we have 
\[O^p_{\L_T}(\C_T)=\<O^p(N_{\C_T}(R))\colon R\in\Gamma\>.\]
 \item [(c)] We have $N_\N(P)\subseteq C_\L(O^p_{\L_T}(\C_T))$ for all $P\in R_\Delta(S\N)$.
 \item [(d)] We have $\N\subseteq C_\L(O^p_{\L_T}(\C_T)\cap S)$ and $O^p_{\L_T}(\C_T)\cap S\leq C_S(\N)$. 
\end{itemize}
\end{lemma}

\begin{proof}
By Lemma~\ref{L:NRadKRadIntersect}, $(\L,\Delta,S)$ is $\N$-replete, i.e. (a) holds. Notice that $TO_p(\L)\unlhd\L_T$ and $N_{\L_T}(C_S(T))=N_\L(TC_S(T))$. Moreover, as $\Delta$ is the set of overgroups in $S$ of the elements of $\F_{TC_S(T)}^c$, we have $Q\cap TC_S(T)\in\Delta$. This implies  $Q\cap (C_S(T)TO_p(\L))\in\Delta$ for all $Q\in\Delta$. If $Q\in R_\Delta(S\C_T)$ and $h\in N_\L(TC_S(T))$, then
\[X_{S,Q}^h=(Q\cap C_S(T))^h(TO_p(\L)).\]
Hence, part (b) follows from Lemma~\ref{L:OupperpGenerate} with $(\L_T,\C_T,TO_p(\L))$ in place of $(\L,\N,R)$. 

\smallskip

For the proof of (c) fix $P\in R_\Delta(S\N)$, set $\Y_0:=\bigcup_{R\in \Gamma}O^p(N_{\C_T}(R))$ and $\Y_{i+1}:=\{\Pi(w)\colon w\in\D\cap\W(\Y_i)\}$ for all $i\geq 0$. It follows from part (b) and \cite[Lemma~1.9]{Chermak:2015} that 
\[O^p_{\L_T}(\C_T)=\bigcup_{i\geq 0}\Y_i\]
Hence, for (c) it is sufficient to show $N_\N(P)\subseteq C_\L(\Y_i)$ for all $i\geq 0$. We show this by induction on $i$. To verify the claim for $i=0$ let $Q\in R_\Delta(S\C_T)$, $h\in N_\L(TC_S(T))$, set 
\[R:=X_{S,Q}^h=T(Q\cap C_S(T))^hO_p(\L)\mbox{ and }X:=(P\cap T)(Q\cap C_S(T))^h.\]
By Lemma~\ref{L:NRadKRadIntersect}, we have $P\cap T\in \F_T^c$ and $Q\cap C_S(T)\in\F_{C_S(T)}^c$, hence also $(Q\cap C_S(T))^h\in \F_{C_S(T)}^c$. So by Lemma~\ref{L:NiceTrick} and assumption, we have $X\in\F_{TC_S(T)}^c\subseteq\Delta$. By Lemma~\ref{L:CSTinSNradical}, we have $C_S(T)\leq P$ and thus $P\cap T\leq X\leq P$. By Lemma~\ref{L:SNradNrad}, we have $TO_p(\L)\leq O_p(\L_T)\leq Q$ and thus $X_{S,Q}\leq Q$. So 
\[R\cap C_S(T)=R\cap C_S(T)^h=(X_{S,Q}\cap C_S(T))^h\leq (Q\cap C_S(T))^h\leq X\leq R.\]
Hence, it follows from Lemma~\ref{L:CommuteInNLX} applied with $(\C_T,R)$ in place of $(\K,Q)$ that $N_\N(P)$ and $O^p(N_{\C_T}(R))$ are contained in $N_\L(X)$ and $[N_\N(P),O^p(N_{\C_T}(R))]=1$. This shows $N_\N(P)\subseteq C_\L(\Y_0)$. 

\smallskip

Let now $i\geq 0$ such that $N_\N(P)\subseteq C_\L(\Y_i)$. Let $y\in\Y_{i+1}$. Then there exists 
\[w=(y_1,\dots,y_k)\in \W(\Y_i)\cap\D\mbox{ with }\Pi(w)=y.\]
Note that $\Y_i\subseteq\C_T$ and so $T\leq S_w$. As $S_w\in\Delta$ and $\Delta$ is by assumption the set of overgroups in $S$ of the elements in $\F_{TC_S(T)}^c$, it follows $S_w\cap TC_S(T)\in\F_{TC_S(T)}^c\subseteq\Delta$. By Lemma~\ref{L:FTcFRc}(d), we have thus $S_w\cap C_S(T)\in\F_{C_S(T)}^c$ and therefore 
\[A:=(P\cap T)(S_w\cap C_S(T))\in\F_{TC_S(T)}^c\subseteq\Delta\]
by Lemma~\ref{L:NiceTrick}. Let $f\in N_\N(P)$. As $C_S(T)\leq P\leq S_f$ by Lemma~\ref{L:CSTinSNradical} 
and since $P\cap T$ is centralized by $y_j\in\Y_i\subseteq\C_T$ for $j=1,\dots,k$, we have 
\[u:=(f^{-1},y_1,f,f^{-1},y_2,f,\dots,f^{-1},y_k,f)\in\D\]
via $A^f$. By the axioms of a partial group, this implies $y=\Pi(w)\in\D(f)$ and $y^f=\Pi(w)^f=\Pi(u)=\Pi(y_1^f,\dots,y_k^f)=\Pi(y_1,\dots,y_k)=y$, where the second last equality uses $f\in N_\N(P)\subseteq C_\L(\Y_i)$. This shows $N_\N(P)\subseteq C_\L(\Y_{i+1})$ and thus (c).

\smallskip

It follows from (c), Lemma~\ref{L:CentralizerHelp} and Lemma~\ref{L:CSNAlperin} that (d) holds.
\end{proof}

\chapter{Commuting partial normal subgroups of linking localities}\label{S:Nperp}

Given a partial normal subgroup $\N$ of a locality $(\L,\Delta,S)$, we introduced in Section~\ref{SS:CommutingPartialNormal} a partial normal subgroup $\N^\perp$, which morally plays the role of a ``centralizer'' of $\N$. More precisely, $\N^\perp=\N^\perp_\L$ was defined to be the largest partial normal subgroup that commutes with $\N$. In this chapter we give a concrete description of $\N^\perp$ in the case that $(\L,\Delta,S)$ is a linking locality. This shows in particular that our partial normal subgroup $\N^\perp$ coincides with the one introduced by Chermak \cite[Definition~5.5]{ChermakIII} in the case of linking localities. If $(\L,\Delta,S)$ is a linking locality, the concrete description of $\N^\perp$ will enable us to prove further properties of $\N^\perp$ and to study the generalized Fitting subgroup $F^*(\L)$. In our proofs we take over some ideas of Chermak (cf. Sections~5 and 6 in \cite{ChermakIII}), but we are able to prove results about $\N^\perp$ and $F^*(\L)$ which are stronger and more general.

\smallskip

\textbf{Throughout this chapter let $(\L,\Delta,S)$ be a linking locality over a fusion system $\F$ and let $\N\unlhd\L$. Set 
\[T:=S\cap\N,\;\L_T:=N_\L(T)\mbox{ and }\C_T:=C_\L(T).\]
Fix moreover a linking locality $(\L^s,\F^s,S)$ over $\F$ such that $\L^s|_\Delta=\L$. Let $\N^s\unlhd\L^s$ be the unique partial normal subgroup of $\L^s$ such that $\N=\N^s\cap \L$.} 

\smallskip

Notice that $\L^s$ and $\N^s$ exist according to Theorem~\ref{T:VaryObjects}, and that $\N^s$ is unique by that theorem if $\L^s$ is given. As before, we will use that $(\L_T,\Delta,S)$ is by Lemma~\ref{L:NLTCLT} a linking locality over $N_\F(T)$ with $\C_T\unlhd \L_T$. In particular, $O^p_{\L_T}(\C_T)$ is defined. 

\section{A concrete description of $\N^\perp$}\label{SS:Nperp} Recall from Corollary~\ref{C:NperpCapSGeneral} that $\N^\perp\subseteq\C_T$. To ease notation, we will now define a subset $\N^\dagger$ of $\C_T$ and show afterwards that $\N^\dagger$ equals $\N^\perp$. 

\begin{notation}\label{N:Nperp}
Set
\[C_S^\circ(\N):=C_S(\N^s)\]
and
\[\N^\dagger:=C_S^\circ(\N)O^p_{\L_T}(\C_T).\]
\end{notation}

The locality $\L^s$ is by Theorem~\ref{T:VaryObjects}(c) unique up to an isomorphism which restricts to the identity on $\L$. This can be used to show that $C_S^\circ(\N)$ and $\N^\dagger$ do not depend on the choice of $\L^s$. However, these properties will also follow along the way. As stated above, we will show that $\N^\dagger$ equals $\N^\perp$. One inclusion is shown in the next lemma. It is worth mentioning that the subset which we denote by $\N^\dagger$ is by definition the same as the subset $\N^\perp$ in the notation of Chermak \cite[Definition~5.5]{ChermakIII}, who defined $\N^\perp$ only for linking localities (which he calls proper localities). 

\begin{lemma}\label{L:PerpendicularNormalinNperp}
We have $\N^\perp\cap S\leq C_S^\circ(\N)$ and $\N^\perp\subseteq\N^\dagger$.
\end{lemma}

\begin{proof}
Recall that $\N^\perp$ is defined to be a partial normal subgroup of $\L$. By Corollary~\ref{C:NperpCapSGeneral}, we have $\N^\perp\subseteq \C_T$. Hence, for every $Q\in R_\Delta(S\N^\perp)$, we see that $O^p(N_{\N^\perp}(Q))\subseteq O^p(N_{\C_T}(Q))\subseteq O^p_{\L_T}(\C_T)\subseteq\N^\dagger$, where the second inclusion uses Lemma~\ref{L:ResidueEquivalences}. By Lemma~\ref{L:PartialNormalAlperin}, it is thus sufficient to show that $R:=\N^\perp\cap S\leq C_S^\circ(\N)$. Notice that $R$ is strongly closed in $\F=\F_S(\L)$ as $\N^\perp\unlhd\L$. Moreover, by Corollary~\ref{C:NperpCapSGeneral}, we have $\N\subseteq C_\L(R)$, so Lemma~\ref{L:NNsTinN} gives $N_{\N^s}(T)=N_\N(T)\subseteq\N\subseteq C_\L(R)\subseteq C_{\L^s}(R)$. Hence, \cite[Proposition~8.2]{Henke:2015} applied with $\L^s$ and $\N^s$ in place of $\L$ and $\N$ gives $\N^s\subseteq C_{\L^s}(R)$. Now by Lemma~\ref{L:CentralizerHelp}, we have $R\leq C_S(\N^s)=C_S^\circ(\N)$ as required. 
\end{proof}

\begin{lemma}\label{L:CScircN}
The subgroup $C_S^\circ(\N)$ is strongly closed in $\F$ and contained in $C_S(\N)$. If $(\L,\Delta,S)$ is weakly $\N$-replete, then $C_S^\circ(\N)=C_S(\N)$. 
\end{lemma}

\begin{proof}
As $\N\subseteq\N^s$, we have clearly $C_S^\circ(\N)=C_S(\N^s)\leq C_S(\N)$. Recall that $(\L^s,\F^s,S)$ is $\N$-replete by Corollary~\ref{C:GetNreplete}. Hence it follows from Lemma~\ref{L:CSNstronglyclosed} that $C_S^\circ(\N)$ is strongly closed. If $(\L,\Delta,S)$ is weakly $\N$-replete, then Lemma~\ref{L:CSNCSNPlus} gives $C_S^\circ(\N)=C_S(\N)$. 
\end{proof}

The next lemma is somewhat similar but not identical to \cite[Theorem~5.7(b)]{ChermakIII}. In particular the property that $\N^\dagger\cap S=C_S^\circ(\N)$ is not immediately clear from Chermak's results.

\begin{lemma}\label{L:NperpCapS0}
We have $\N^\dagger\cap S=C_S^\circ(\N)\leq C_S(\N)$. In particular, if $(\L,\Delta,S)$ is weakly $\N$-replete, we have $\N^\dagger\cap S=C_S(\N)$.  
\end{lemma}

\begin{proof}
By Lemma~\ref{L:CScircN}, it is sufficient to show $\N^\dagger\cap S=C_S^\circ(\N)$. As $\N^\dagger=C_S^\circ(\N)O^p_{\L_T}(\C_T)$, it follows from the Dedekind Lemma \cite[Lemma~1.10]{Chermak:2015} that 
\[\N^\dagger\cap S=C_S^\circ(\N)(O^p_{\L_T}(\C_T)\cap S).\]
Hence, we only need to show that $O^p_{\L_T}(\C_T)\cap S\leq C_S^\circ(\N)$. 

\smallskip

Let $\Delta_0$ be the set of overgroups of $\F_{TC_S(T)}^c$. Recall from Corollary~\ref{C:ChermakEndPartI} that $\M:=\<\N,\C_T\>\unlhd\L$ with $\M\cap S=TC_S(T)$. In particular $TC_S(T)$ is strongly closed in $\F$ and thus $\F_{TC_S(T)}^c$ is closed under $\F$-conjugacy. Hence, $\Delta_0$ is $\F$-closed. If $R\in\F^{cr}$, then by \cite[Lemma~6.2]{Henke:2015}, $R$ is $\L$-radical. Hence, $R$ is $\M$-radical by Lemma~\ref{L:SNradNrad}, so Lemma~\ref{L:NradEcr} implies $R\cap TC_S(T)\in\F_{TC_S(T)}^c$ and thus $R\in\Delta_0$. This shows $\F^{cr}\subseteq\Delta_0$. By Lemma~\ref{L:GetFTcIntoFq} applied with $\M$ in place of $\N$, we have $\F_{TC_S(T)}^c\subseteq\F^q$ and thus $\Delta_0\subseteq\F^q\subseteq\F^s$. Therefore, setting 
\[\L_0:=\L^s|_{\Delta_0}\mbox{ and }\N_0:=\N^s\cap \L_0,\]
the triple $(\L_0,\Delta_0,S)$ is a linking locality, $\N_0\unlhd\L_0$ and $\N_0\cap S=\N^s\cap S=T$. Now by Lemma~\ref{L:OpCTcapSSpecial}(a),(d), the locality $(\L_0,\Delta_0,S)$ is $\N_0$-replete and $O^p_{N_{\L_0}(T)}(C_{\L_0}(T))\cap S\leq C_S(\N_0)$. Using Remark~\ref{R:TakePlus} twice, one observes that 
\[O^p_{\L_T}(\C_T)\cap S=O^p_{N_{\L^s}(T)}(C_{\L^s}(T))\cap S=O^p_{N_{\L_0}(T)}(C_{\L_0}(T))\cap S.\]
Moreover, using Lemma~\ref{L:CScircN} with $(\L_0,\N_0)$ in place of $(\L,\N)$, we obtain $C_S(\N_0)=C_S^\circ(\N_0)=C_S(\N^s)=C_S^\circ(\N)$. This implies the assertion.
\end{proof}

\begin{lemma}\label{L:NperpNormLT}
We have $\N^\dagger\unlhd \L_T$, $\N^\dagger\subseteq\C_T$ and $S\N^\dagger=S\C_T$.
\end{lemma}

\begin{proof}
Notice that $C_S^\circ(\N)\subseteq C_S(\N)\leq C_S(T)$ and thus  $\N^\dagger\subseteq\C_T$. As $\K:=O^p_{\L_T}(\C_T)\unlhd \L_T$ it follows from Lemma~\ref{L:ProductPartialSubgroup} that $\N^\dagger=\K C_S^\circ(\N)$ is a partial subgroup of $\L_T$. Set $H:=N_\L(TC_S(T))=N_{\L_T}(C_S(T))$. As $\C_T\unlhd \L_T$, the Frattini Lemma \cite[Corollary~3.11]{Chermak:2015} gives $\L_T=\C_T H$. Moreover, by Lemma~\ref{L:ResidueEquivalences}, we have $\K C_S(T)=\C_T$. As $C_S(T)\leq H$, it follows
\[\L_T=\K H.\]
If $\alpha\colon \L_T\rightarrow \L_T/\K$ is the natural projection, we have thus $\L_T\alpha=H\alpha$. Since $C_S^\circ(\N)\leq TC_S(T)$ is by Lemma~\ref{L:CScircN} strongly $\F$-closed, $C_S^\circ(\N)\unlhd H$. Notice also that $\alpha$ restricts to a group homomorphism $\H\rightarrow \L_T\alpha=H\alpha$. Thus, $\N^\dagger\alpha=C_S^\circ(\N)\alpha\unlhd H\alpha=\L_T\alpha$. As $\N^\dagger$ is a partial subgroup of $\L_T$ containing $\K$, the partial subgroup correspondence \cite[Proposition~4.7]{Chermak:2015}  yields $\N^\dagger\unlhd \L_T$. Since $\K\subseteq\N^\dagger$ and $C_S(T)\K=\K C_S(T)=\C_T$, we have $C_S(T)\N^\dagger=\C_T$ and thus $S\N^\dagger=S\C_T$.
\end{proof}

\begin{lemma}\label{L:CommuteInNLXGeneral}
Let $P\in R_\Delta(S\N)$, $Q\in R_\Delta(S\C_T)$, $n\in N_\N(P)$ and $c\in O^p(N_{\C_T}(Q))$ such that $P=S_n$ and $Q=S_c$. Suppose $(n,c)\in\D$. Then
\[X:=S_{(n,c)}=P\cap Q=S_{(c,n)}\in\Delta,\]
$n,c\in N_\L(X)$ and $[n,c]=1$, where the commutator is computed in the group $N_\L(X)$. 
\end{lemma}

\begin{proof}
As $(n,c)\in\D$, we have $X:=S_{(n,c)}\in\Delta$. Moreover, as $c\in \C_T\subseteq N_\L(T)$, \cite[Lemma~3.2(a)]{Chermak:2015} gives $(c,c^{-1},n,c)\in\D$ and $X=S_{(n,c)}=S_n\cap S_c=S_{(c,n^c)}$, i.e.
\[X=S_{(n,c)}=P\cap Q=S_{(c,n^c)}\in\Delta.\]
Since $T\leq S_c=Q$, we have $P\cap T\leq P\cap Q=X$. Furthermore, $C_S^\circ(\N)\leq C_S(\N)\leq S_n=P$ yields $Q\cap C_S^\circ(\N)\leq P\cap Q=X$. So $X\in\Delta$ with $P\cap T\leq X\leq P$ and $Q\cap C_S^\circ(\N)\leq X\leq Q$. By Lemma~\ref{L:NperpCapS0} and Lemma~\ref{L:NperpNormLT}, we have $\N^\dagger\cap S=C_S^\circ(\N)$ and $\N^\dagger\unlhd\L_T$ with $\N^\dagger\subseteq\C_T$. Hence, Lemma~\ref{L:CommuteInNLX} applied with $\N^\dagger$ in place of $\K$ yields that $N_\N(P)$ and $O^p(N_{\N^\dagger}(Q))$ are contained in $N_\L(X)$ and (computing in $N_\L(X)$), we have $[N_\N(P),O^p(N_{\N^\dagger}(Q))]=1$. Since $O^p_{\L_T}(\C_T)\subseteq\N^\dagger\subseteq\C_T$, it follows from Lemma~\ref{L:ResidueEquivalences} that $O^p(N_{\N^\dagger}(Q))=O^p(N_{\C_T}(Q))$ and so $c\in O^p(N_{\N^\dagger}(Q))$. Thus, $n,c\in N_\L(X)$,  $[n,c]=1$ and $n^c=n$. This implies the assertion. 
\end{proof}

For the following lemma and subsequent results the reader might want to recall Definition~\ref{D:Commute}.

\begin{lemma}\label{L:NNperpCommute}
$\N$ commutes strongly with $\N^\dagger$.
\end{lemma}

\begin{proof}
Let $n\in \N$ and $c\in \N^\dagger$ such that $(n,c)\in\D$. We need to show that $S_{(n,c)}=S_{(c,n)}$, $(c,n)\in\D$ and $nc=cn$. 

\smallskip

If $t\in T$ notice that $t\in N_\N(S)$ and $S=S_t\in R_\Delta(S\N)$. So by Lemma~\ref{L:PartialNormalAlperin}, there exists $k\in\mathbb{N}$, $P_1,\dots,P_k\in R_\Delta(S\N)$ and $u=(n_1,\dots,n_k)\in\D$ such that $n=\Pi(u)$, $S_n=S_u$, $n_i\in N_\N(P_i)$ and $S_{n_i}=P_i$ for $i=1,\dots,k$. 

\smallskip

By Lemma~\ref{L:NperpCapS0} and Lemma~\ref{L:NperpNormLT}, $\N^\dagger$ is a partial normal subgroup of $\L_T$ with $\N^\dagger\subseteq\C_T$, $S\N^\dagger=S\C_T$ and $\N^\dagger\cap S=C_S^\circ(\N)\leq C_S(\N)$. As $O^p_{\L_T}(\C_T)\subseteq\N^\dagger\subseteq \C_T$, using Lemma~\ref{L:ResidueEquivalences} we see that $O^p(N_{\N^\dagger}(Q^*))=O^p(N_{\C_T}(Q^*))$ for all $Q^*\in\Delta$. Hence, it follows from Lemma~\ref{L:PartialNormalAlperin} applied with $(\N^\dagger,\L_T)$ in place of $(\N,\L)$ that there exist $l\in\mathbb{N}$, $Q_1,\dots,Q_l\in R_\Delta(S\N^\dagger)=R_\Delta(S\C_T)$ and $v=(s,c_1,\dots,c_l)\in\D$ such that $c=\Pi(v)$, $S_c=S_v$, $s\in C_S^\circ(\N)\leq C_S(\N)$, $c_j\in O^p(N_{\N^\dagger}(Q_j))=O^p(N_{\C_T}(Q_j))$ and $S_{c_j}=Q_j$ for $j=1,\dots,l$. 

\smallskip

As $(n,c)\in\D$, $S_n=S_u$ and $S_c=S_v$, we have $S_{u\circ v}=S_{(n,c)}\in\Delta$ and thus $u\circ v\in\D$. 
By Lemma~\ref{L:XCSXcommutesStrongly}, $\N$ commutes strongly with $C_S(\N)$, and by Lemma~\ref{L:CommuteInNLXGeneral}, $n_i$ commutes strongly with $c_j$ for all $i=1,\dots,k$ and $j=1,\dots,l$. Hence, $\{n_1,\dots,n_k\}$ commutes strongly with $\{s,c_1,\dots,c_l\}$. Therefore, by Lemma~\ref{L:CommuteStronglyProducts}, we have $S_{u\circ v}=S_{v\circ u}$, $v\circ u\in\D$ and $\Pi(u\circ v)=\Pi(v\circ u)$. This implies
\[\Delta\ni S_{(n,c)}=S_{u\circ v}=S_{v\circ u}=S_{(c,n)}\]
and, in particular, $(c,n)\in\D$. Moreover, $nc=\Pi(u\circ v)=\Pi(v\circ u)=cn$. Therefore the assertion holds. 
\end{proof}

The following result is (with different notation) stated as Theorem~5.7(a) in \cite{ChermakIII}.

\begin{corollary}\label{C:NperpNormal}
We have $\N^\dagger\unlhd \L$ and $O^p_{\L_T}(\C_T)\unlhd\L$.
\end{corollary}

\begin{proof}
We have shown in Lemma~\ref{L:NNperpCommute} that $\N$ commutes with $\N^\dagger$ and thus also with $O^p_{\L_T}(\C_T)\subseteq\N^\dagger$. Moreover, $O^p_{\L_T}(\C_T)\unlhd \L_T$ by definition and $\N^\dagger\unlhd\L_T$ by Lemma~\ref{L:NperpNormLT}. So the assertion follows from Lemma~\ref{L:PerpendicularNormLTNormL} applied twice, once with $\N^\dagger$ and once with $O^p_{\L_T}(\C_T)$ in place of $\Y$. 
\end{proof}

\begin{corollary}\label{C:NperpCapS}
We have $\N^\perp=\N^\dagger=C_S^\circ(\N)O^p_{\L_T}(\C_T)$. In particular, 
\[\N^\perp\cap S=C_S^\circ(\N)\leq C_S(\N),\]
$C_S^\circ(\N)$ does not depend on the choice of $\L^s$ and, if $(\L,\Delta,S)$ is weakly $\N$-replete, then $\N^\perp\cap S=C_S(\N)$.  
\end{corollary}

\begin{proof}
 By Corollary~\ref{C:NperpNormal}, $\N^\dagger$ is a normal subgroup of $\L$. Moreover, by Lemma~\ref{L:NNperpCommute}, $\N$ commutes with $\N^\dagger$ and thus, by Corollary~\ref{C:MperpNNperpM}, $\N^\dagger$ commutes with $\N$. So $\N^\dagger\subseteq\N^\perp$ by definition of $\N^\perp$. The converse inclusion is true by Lemma~\ref{L:PerpendicularNormalinNperp}.
\end{proof}

\section{The proof of Theorem~\ref{T:mainNperp}}\label{SS:proofmainNperp}

We will need the following elementary lemma in the proof of Theorem~\ref{T:mainNperp}. 

\begin{lemma}\label{L:CNTinZT}
We have $C_\N(T)\leq C_{\N^s}(T)\leq Z(T)$ and in particular $\N\cap\N^\perp\leq Z(T)$.
\end{lemma}

\begin{proof}
Clearly $C_\N(T)\subseteq C_{\N^s}(T)$. Moreover, by Corollary~\ref{C:NperpCapSGeneral}, we have $\N^\perp\subseteq C_\L(T)$. Hence, it remains to argue that $C_{\N^s}(T)\leq Z(T)$. It follows from \cite[Lemma~3.5]{Chermak:2015} and $T$ being strongly closed that $N_{\N^s}(T)=N_{\N^s}(TC_S(T))$. As $TC_S(T)\in\F^c\subseteq\F^s$, the normalizer $N_{\L^s}(TC_S(T))$ is a group of characteristic $p$. Hence, by Lemma~\ref{L:MSCharp}(b), $N_{\N^s}(T)=N_{\N^s}(TC_S(T))\unlhd N_{\L^s}(TC_S(T))$ is of characteristic $p$. Since $T$ is by \cite[Lemma~3.1(c)]{Chermak:2015} a maximal $p$-subgroup of $\N^s$, it follows $C_{\N^s}(T)\leq Z(T)$. 
\end{proof}

We are now in a position to prove most parts of Theorem~\ref{T:mainNperp}. Only for part (f) we will need to refer to Corollary~\ref{C:RegularNReplete} below. 

\begin{proof}[Proof of Theorem~\ref{T:mainNperp} (using Corollary~\ref{C:RegularNReplete} below)]
Part (a) follows from Corollaries~\ref{C:NperpExistence} and \ref{C:NperpCapSGeneral} as well as Lemma~\ref{L:NinNperperp}. Statement (b) and the first part of (d) follow from Corollary~\ref{C:MperpNNperpM}. Part (c) was stated and proved as Lemma~\ref{L:PerpProduct}. It remains thus to show the statements about linking localities. For the proof we keep the standing hypothesis that $(\L,\Delta,S)$ is a linking locality and $\N\unlhd\L$. Fix moreover  $\M\unlhd\L$. 

\smallskip

If $\M$ commutes with $\N$, then by Lemma~\ref{L:CNTinZT}, we have $\M\cap\N\subseteq\N^\perp\cap\N\leq Z(T)\leq S$. Hence, Lemma~\ref{L:DecomposeEltofMN} implies in this case that $S_{mn}=S_{(m,n)}$ for all $m\in\M$ and $n\in\N$ with $(m,n)\in\D$. This completes the proof of (d). 

\smallskip

We have seen already that (b) is true, i.e. conditions (i)-(iii) are equivalent. We show now that these conditions are also equivalent to the conditions (iv) and (v) stated in (e). By definition of $\M^\perp$ and $\N^\perp$, the conditions (i)-(iii)  are equivalent to $\N\subseteq\M^\perp$ and to $\M\subseteq\N^\perp$. Hence, conditions (i)-(iii) imply (v). By (a) we have $\M^\perp\cap S\leq C_S(\M)$. Hence, (v) implies (iv). Thus it suffices to show that (iv) implies $\M\subseteq\N^\perp$. So assume that (iv) holds, i.e. $\M\cap S\subseteq \N^\perp$ and $T=\N\cap S\subseteq C_S(\M)$. The latter condition implies by Lemma~\ref{L:CentralizerHelp} that $\M\subseteq\C_T\subseteq\L_T$. Note that $\M\unlhd\L_T$. So Lemma~\ref{L:ResidueNinM} applied with $\L_T,\M,\C_T$ in the roles of $\L,\N,\M$ implies that $O^p_{\L_T}(\M)\subseteq O^p_{\L_T}(\C_T)$. Hence, by Corollary~\ref{C:NperpCapS}, we have $O^p_{\L_T}(\M)\subseteq \N^\perp$. Using Lemma~\ref{L:ResidueEquivalences} we can conclude now that $\M=(\M\cap S)O^p_{\L_T}(\M)\subseteq \N^\perp$ as required. This proves (e).

\smallskip

If $(\L,\Delta,S)$ is $\N$-replete, then $(\L,\Delta,S)$ is weakly $\N$-replete and it follows from Corollary~\ref{C:NperpCapS} that $\N^\perp\cap S=C_S(\N)$. If $\F^q\subseteq\Delta$, then $(\L,\Delta,S)$ is $\N$-replete by Corollary~\ref{C:GetNreplete}. If $\delta(\F)\subseteq\Delta$, then $(\L,\Delta,S)$ is $\N$-replete by Corollary~\ref{C:RegularNReplete} below. This implies (b) and thus the proof of the theorem is complete once we have proved Corollary~\ref{C:RegularNReplete}.
\end{proof}

\section{Further properties of $\N^\perp$}

Recall that we have shown in Corollary~\ref{C:NperpCapS} that $\N^\perp=C_S^\circ(\N)O^p_{\L_T}(\C_T)$. We will now use this property to prove some more results about $\N^\perp$, which are in particular needed to show Theorem~\ref{T:GeneralizedFitting}. 

\smallskip

We start by showing that $\N^\perp$ behaves well with respect to extensions and restrictions. This was proved before by Chermak \cite[Theorem~5.7(d)]{ChermakIII}.

\begin{lemma}\label{L:NperpPlus}
Let $(\L^+,\Delta^+,S)$ be a linking locality over $\F$ with $\Delta\subseteq\Delta^+$ and $\L=\L^+|_\Delta$. Adapting Notation~\ref{N:VaryObjects}, we have then $C_S^\circ(\N^+)=C_S^\circ(\N)$ and $(\N^\perp)^+=(\N^+)^\perp$.
\end{lemma}

\begin{proof}
By Corollary~\ref{C:NperpCapS}, we have $\N^\perp=C_S^\circ(\N)O^p_{\L_T}(\C_T)$ and similarly, setting $\L_T^+:=N_{\L^+}(T)$ and $\C_T^+=C_{\L^+}(T)$, it follows that  
\[(\N^+)^\perp=C_S^\circ(\N^+)O^p_{\L_T^+}(\C_T^+).\]
Corollary~\ref{C:NperpCapS} gives also that $C_S^\circ(\N):=C_S(\N^s)$ does not depend on the choice of $\L^s$. So by Theorem~\ref{T:VaryObjects}, we may choose $\L^s$ such that $\L^s|_{\Delta^+}=\L^+$. As $\N^s\cap \L^+\unlhd\L^+$ with $(\N^s\cap\L^+)\cap\L=\N^s\cap\L=\N$, we have $\N^s\cap\L^+=\N^+$. Hence, 
\[C_S^\circ(\N^+)=C_S(\N^s)=C_S^\circ(\N).\]
By Remark~\ref{R:TakePlus}, we have $O^p_{\L_T^+}(\C_T^+)\cap \L=O^p_{\L_T}(\C_T)$. So Lemma~\ref{L:RestrictionIntersectL} yields 
\[(\N^+)^\perp\cap\L=C_S^\circ(\N)(O^p_{\L_T^+}(\C_T^+)\cap\L)=C_S^\circ(\N)O^p_{\L_T}(\C_T)=\N^\perp.\]
As $(\N^\perp)^+$ is characterized as the unique partial normal subgroup of $\L^+$ whose intersection with $\L$ equals $\N^\perp$, it follows that $(\N^+)^\perp=(\N^\perp)^+$. 
\end{proof}

For our next results it will be convenient to use the following notation.

\begin{notation}\label{N:Zcirc}
Set $Z^\circ(\N):=Z(\N^s)$. If we want to to stress the dependence of $Z^\circ(\N)$ on $\L$, we write $Z^\circ_\L(\N)$. 
\end{notation}

\begin{lemma}\label{L:ZcircBasic}
We have $Z^\circ(\N)=T\cap C_S^\circ(\N)\leq Z(\N)\leq Z(T)$. If $(\L,\Delta,S)$ is weakly $\N$-replete, then $Z^\circ(\N)=Z(\N)$.
\end{lemma}

\begin{proof}
By Lemma~\ref{L:CNTinZT}, we have $C_{\N^s}(T)\leq Z(T)$ and hence $Z^\circ(\N)=Z(\N^s)\leq Z(T)$. So $Z^\circ(\N)=C_T(\N^s)=T\cap C_S(\N^s)=T\cap C_S^\circ(\N)\leq C_T(\N)\leq Z(\N)\subseteq C_\N(T)\subseteq C_{\N^s}(T)\leq Z(T)$. If $(\L,\Delta,S)$ is weakly $\N$-replete, then Lemma~\ref{L:CScircN} gives $C_S^\circ(\N)=C_S(\N)$ and thus $Z^\circ(\N)=C_T(\N)$. Since $Z(\N)\leq C_{\N^s}(T)=Z(T)$, we have $Z(\N)=C_T(\N)$. Thus the assertion holds. 
\end{proof}

\begin{lemma}\label{L:NcapNperp}
 We have $\N\cap\N^\perp=Z^\circ(\N)\leq Z^\circ(\N^\perp)\leq S$. 
\end{lemma}

\begin{proof}
By Lemma~\ref{L:ZcircBasic} and Corollary~\ref{C:NperpCapS}, we have $Z^\circ(\N)\leq \N\cap C_S^\circ(\N)\subseteq \N\cap\N^\perp$ and $Z^\circ(\N^\perp)\leq S$. Hence, it remains only to show that $\N\cap\N^\perp\subseteq Z^\circ(\N)\cap Z^\circ(\N^\perp)$. By Lemma~\ref{L:NperpPlus}, we have $(\N^s)^\perp=(\N^\perp)^s$. Thus, if $\N^s\cap (\N^s)^\perp\subseteq Z(\N^s)\cap Z((\N^s)^\perp)$, then 
\[\N\cap\N^\perp\subseteq \N^s\cap (\N^\perp)^s\subseteq Z(\N^s)\cap Z((\N^\perp)^s)=Z^\circ(\N)\cap Z^\circ(\N^\perp).\] Hence, we only need to show $\N^s\cap (\N^s)^\perp\subseteq Z(\N^s)\cap Z((\N^s)^\perp)$. To ease notation, we may assume without loss of generality that $(\L,\Delta,S)=(\L^s,\F^s,S)$ and need to show that $R:=\N\cap\N^\perp\leq Z(\N)\cap Z(\N^\perp)$.

\smallskip

As $\N^\perp\subseteq \C_T=C_\L(T)$ by Corollary~\ref{C:NperpCapSGeneral}, we have $R=\N\cap\N^\perp\subseteq C_\N(T)\leq Z(T)\leq S$, where the second inclusion uses  Lemma~\ref{L:CNTinZT}. By Corollary~\ref{C:MperpNNperpM}, $R$ fixes $\N$ and $\N^\perp$ under conjugation. It follows now from Lemma~\ref{L:XfixesRXcentralizesR} applied twice (once with $\X=\N$ and once with $\X=\N^\perp$) that $R\subseteq C_\L(\N)$ and $R\subseteq C_\L(\N^\perp)$. This yields $R\leq Z(\N)\cap Z(\N^\perp)$.
\end{proof}

The following Lemma is a generalization of \cite[Lemma~6.5]{ChermakIII}.

\begin{lemma}\label{L:OpNNPinMperp}
If $\M$ is a partial normal subgroup of $\L$ with $\M\cap\N\leq Z^\circ(\M)$, then 
\begin{equation*}
O^p_\L(\M)\subseteq O^p_{\L_T}(\C_T)\subseteq\N^\perp.
\end{equation*} 
\end{lemma}

\begin{proof}
Adapt notation~\ref{N:VaryObjects} and note that $(\M\cap\N)^s=\M^s\cap\N^s$. The property $\M\cap\N\leq Z^\circ(\M)$ implies by Lemma~\ref{L:ZcircBasic} that $\M\cap\N\leq S$  and so, by Lemma~\ref{L:TakePlusInS}, $(\M\cap\N)^s=\M\cap\N$. Hence
\[\M^s\cap\N^s=\M\cap\N\leq Z^\circ(\M).\]
If $O^p_{\L^s}(\M^s)\subseteq O^p_{N_{\L^s}(T)}(C_{\L^s}(T))$, then it follows from Lemma~\ref{L:OupperpVaryObjects} and Remark~\ref{R:TakePlus} that 
\[O^p_\L(\M)=O^p_{\L^s}(\M^s)\cap\L\subseteq O^p_{N_{\L^s}(T)}(C_{\L^s}(T))\cap \L=O^p_{\L_T}(\C_T)\subseteq\N^\perp.\]
Hence, replacing $(\L,\Delta,S)$ by $(\L^s,\F^s,S)$ we may assume that
\[\Delta=\F^s\]
and we only need to show that $O^p_\L(\M)\subseteq \M^*:=O^p_{\L_T}(\C_T)$. Under this assumption
\[Z^\circ(\M)=Z(\M)\]
by Corollary~\ref{C:GetNreplete} and Lemma~\ref{L:ZcircBasic}. As seen in Corollary~\ref{C:NperpNormal}, we have $\M^*\unlhd\L$ and thus $\M^*\cap\M\unlhd\L$. Thus, it is sufficient to show that $\M^*\cap\M$ has $p$-power index in $\M$. Hence, by Lemma~\ref{L:PartialNormalAlperin}, it is sufficient to prove that $O^p(N_\M(P))\subseteq \M^*$ for all $P\in R_\Delta(\M S)$. 

\smallskip

We fix now $P\in R_\Delta(\M S)$ and set $R:=\M\cap S$. By Lemma~\ref{L:SNradNrad} and Lemma~\ref{L:CSTinSNradical}, $P$ is $\M$-radical and $O_p(\L)C_S(R)\leq P$. Set $U:=P\cap R$. By Lemma~\ref{L:NradEcr}, we have $C_R(U)\leq U$. Hence, using Lemma~\ref{L:ZcircBasic}, we can  conclude that $[U,T]\leq \M\cap \N\subseteq Z^\circ(\M)\leq Z(R)\leq C_R(U)\leq U$ and $T$ normalizes $U$. As $(\L,\Delta,S)$ is by Corollary~\ref{C:GetNreplete} $\M$-replete and thus weakly $\M$-replete, we have   
\[X:=UC_S(R)O_p(\L)\in\Delta.\]
As $T$ normalizes $U$ and $C_S(R)O_p(\L)\unlhd S$, we have $T\leq N_\L(X)$. Observe also that $P\cap R=U\leq X\leq P$, so Lemma~\ref{L:Lemma0} gives $N_\M(P)\leq N_\M(X)$. We see now that the commutator group $[T,N_\M(P)]$ is defined inside of $N_\L(X)$ and 
\[[T,N_\M(P)]\leq \N\cap\M\leq Z:=Z^\circ(\M)=Z(\M)\leq Z(R).\]
Hence $N_\M(P)$ acts on the $p$-group $TZ$ and $[TZ,O^p(N_\M(P))]=[TZ,O^p(N_\M(P)),O^p(N_\M(P))]\leq [Z,O^p(N_\M(P))]=1$. Hence, $O^p(N_\M(P))\subseteq \C_T$. Using Lemma~\ref{L:ResidueEquivalences}, it follows that $O^p(N_\M(P))\leq O^p(N_{\C_T}(P))\subseteq \M^*$ as required.
\end{proof}

\begin{cor}\label{C:OpNNPinMperp}
If there exists a partial normal subgroup $\M$ of $\L$ such that $\M\cap\N\subseteq Z^\circ(\N)$ and $\M^\perp\subseteq S$, then $\N\subseteq S$. 
\end{cor}

\begin{proof}
Applying Lemma~\ref{L:OpNNPinMperp} with the roles of $\M$ and $\N$ reversed, we get $O^p_\L(\N)\subseteq\M^\perp\subseteq S$. The assertion follows now from Lemma~\ref{L:ResidueEquivalences}. 
\end{proof}

\section{Centric radical partial normal subgroups}

\begin{definition}\label{D:GeneralizedFitting}~
\begin{itemize}
\item The partial normal subgroup $\N$ is called \emph{centric in $\L$} if $\N^\perp\subseteq\N$. 
\item We call $\N$ \emph{radical in $\L$} if $O_p(\L)\subseteq\N$.
\item $\N$ is said to be \emph{centric radical in $\L$} if $\N$ is both centric and radical in $\L$. 
\item The intersection of all centric radical partial normal subgroups of $\L$ is called the \emph{generalized Fitting subgroup} of $\L$ and denoted by $F^*(\L)$. 
\end{itemize}
\end{definition}

The reader should note that the definitions above fit well with the notions we introduced already. Namely, if $\N\leq S$ and $\N$ is centric in $\L$, then $\N^s=\N$ by Lemma~\ref{L:TakePlusInS} and hence, by Corollary~\ref{C:NperpCapS}, we have  $C_S(\N)=C_S(\N^s)=C_S^\circ(\N)\subseteq\N^\perp\subseteq \N$. So we obtain in this case that $\N$ is centric in $\F$. In the case that  $\N$ is an element of $\Delta$, we have $N_\L(\N)=\L$; so $\N$ is $\L$-radical in the sense of Definition~\ref{D:Radical} if and only if it is radical in $\L$ in the sense of Definition~\ref{D:GeneralizedFitting}.

\smallskip

It follows from Lemma~\ref{L:CSNinN} below that a partial normal subgroup $\N$ is centric as defined above if and only if it is large in the sense defined by Chermak \cite[Definition~6.1]{ChermakIII}. Our definition of $F^*(\L)$ is a priori slightly different from Chermak's definition of $F^*(\L)$ (cf. \cite[Definition~6.8]{ChermakIII}), but Lemma~\ref{L:TakePlusGeneralizedFitting}(d) below implies that the two definitions are equivalent. We will now see that the concepts introduced in Definition~\ref{D:GeneralizedFitting} are well-behaved with respect to expansion.

\begin{lemma}\label{L:TakePlusGeneralizedFitting}
Let $(\L^+,\Delta^+,S)$ be a linking locality with $\Delta\subseteq\Delta^+$ and $\L=\L^+|_\Delta$. Adapting Notation~\ref{N:VaryObjects} the following hold:
\begin{itemize}
 \item [(a)] $\N$ is centric in $\L$ if and only if $\N^+$ is centric in $\L^+$.
 \item [(b)] $\N$ is radical in $\L$ if and only if $\N^+$ is radical in $\L^+$.
 \item [(c)] $\N$ is centric radical in $\L$ if and only if $\N^+$ is centric radical in $\L^+$.
 \item [(d)] We have $F^*(\L^+)=F^*(\L)^+$.
\end{itemize}
\end{lemma}

\begin{proof}
 By Lemma~\ref{L:NperpPlus}, we have $(\N^\perp)^+=(\N^+)^\perp$. Using Theorem~\ref{T:VaryObjects}(b), it follows $\N^\perp\subseteq\N$ if and only if $(\N^+)^\perp\subseteq\N^+$, i.e. (a) holds. By Lemma~\ref{L:TakePlusInS}, we have $O_p(\L)=O_p(\L^+)$ and so (b) holds. Clearly (a) and (b) imply (c). By Theorem~\ref{T:VaryObjects}(b), the map $\Phi_{\L^+,\L}$ defined at the beginning of Section~\ref{SS:VaryObjects} is a bijection from the set of partial normal subgroups of $\L^+$ to the set of partial normal subgroups of $\L$. Hence, (c) says that $\Phi_{\L^+,\L}$ induces also a bijection between the set of centric radical partial normal subgroups of $\L^+$ and the set of centric radical partial normal subgroups of $\L$. Thus, as $\Phi_{\L^+,\L}$ commutes with taking the intersection of a finite number of partial normal subgroups, it follows that $\Phi_{\L^+,\L}$ maps $F^*(\L^+)$ to $F^*(\L)$. This proves (d). 
\end{proof}

The next result says basically that $F^*(\L)$ can be regarded as a ``characteristic'' partial normal subgroup of $\L$.

\begin{lemma}\label{L:AutGeneralizedFitting}
If $(\tL,\tDelta,\tS)$ is a linking locality and $\alpha\in\Iso(\L,\tL)$, then $F^*(\L)\alpha=F^*(\tL)$. In particular, for every $\alpha\in\Aut(\L)$, we have $F^*(\L)\alpha=F^*(\L)$. 
\end{lemma}

\begin{proof}
Let $(\tL,\tDelta,\tS)$ be a linking locality with $\alpha\in\Iso(\L,\tL)$. Then $\alpha$ induces a bijection from the set of partial normal subgroups of $\L$ to the set of partial normal subgroups of $\tL$. As $O_p(\L)$ is by Lemma~\ref{L:OpL} the largest $p$-subgroup of $\L$, which is a partial normal subgroup of $\L$ (and similarly for $O_p(\tL)$), the isomorphism $\alpha$ takes $O_p(\L)$ to $O_p(\tL)$. Moreover, it follows from Lemma~\ref{L:AutNperp} that $\N\unlhd\L$ is centric in $\L$ if and only if $\N\alpha$ is centric in $\tL$. Thus, $\alpha$ induces a bijection from the set of centric radical partial normal subgroups of $\L$ to the set of centric radical partial normal subgroups of $\tL$ and thus takes $F^*(\L)$ to $F^*(\tL)$. 
\end{proof}

\begin{lemma}\label{L:CentricInherit}
For every $\M\unlhd\L$ the following hold:
\begin{itemize}
 \item [(a)] If $\N\N^\perp\subseteq\M$, then $\M$ is centric in $\L$.
 \item [(b)] If $\N$ is centric in $\L$ and $\N\subseteq\M$, then $\M$ is centric in $\L$.
 \item [(c)] The product $\N\N^\perp$ is a centric partial normal subgroup of $\L$, and $\N\N^\perp O_p(\L)$ is a centric radical partial normal subgroup of $\L$. 
\end{itemize}
\end{lemma}

\begin{proof}
If $\N\N^\perp\subseteq\M$, then it follows from Lemma~\ref{L:MinNimpliesNperpinMperp} that $\M^\perp\subseteq \N^\perp\subseteq\M$. This shows (a). Part (b) is a special case of (a). By Theorem~\ref{T:ProductsPartialNormal}, the products  $\N\N^\perp$ and $\N\N^\perp O_p(\L)$ are partial normal subgroups of $\L$. So (c) follows from (a).
\end{proof}

The following lemma generalizes part of \cite[Lemma~6.4]{ChermakIII}.

\begin{lemma}\label{L:CSNinN}
The following conditions are equivalent:
\begin{itemize}
 \item [(i)] $\N$ is centric;
 \item [(ii)] $C_S^\circ(\N)\subseteq\N$;
 \item [(iii)] $\N^\perp=Z^\circ(\N)$.
\end{itemize}
In particular, if $\N$ is centric, then $\N^\perp\leq S$. 
\end{lemma}

\begin{proof}
Notice that (i) implies (ii) by Corollary~\ref{C:NperpCapS}. By Lemma~\ref{L:ZcircBasic}, we have $Z^\circ(\N)\subseteq Z(\N)\leq Z(T)\leq S$. In particular, (iii) implies (i). Hence, it is sufficient to show that (ii) implies (iii). 

\smallskip

Assume  (ii). Then $C_S^\circ(\N)=T\cap C_S^\circ(\N)=Z^\circ(\N)$ by Lemma~\ref{L:ZcircBasic}. By Corollary~\ref{C:NperpCapS}, we have thus $\N^\perp\cap S=C_S^\circ(\N)=Z^\circ(\N)\leq T$. Let $Q\in R_\Delta(S\N^\perp)$. Then $[Q,N_{\N^\perp}(Q)]\leq Q\cap \N^\perp=Q\cap (\N^\perp\cap S)\leq Q\cap T$. As $\N^\perp\subseteq\C_T=C_\L(T)$, we have also $[Q\cap T,N_{\N^\perp}(Q)]=1$. Hence, $[Q,O^p(N_{\N^\perp}(Q))]=1$. By Lemma~\ref{L:SNradNrad}, $Q$ is $\N^\perp$-radical and thus $O_p(N_{\N^\perp}(Q))\leq Q$. As $N_{\N^\perp}(Q)$ is by Lemma~\ref{L:MSCharp}(b) of characteristic $p$, it follows that $O^p(N_{\N^\perp}(Q))=1$. As $Q\in R_\Delta(S\N^\perp)$ was arbitrary, Lemma~\ref{L:PartialNormalAlperin} yields thus $\N^\perp=\N^\perp\cap S=Z^\circ(\N)$, i.e. (iii) holds. 
\end{proof}

\begin{lemma}\label{L:GetFTcIntoFqCentric}
If $\N$ is centric in $\L$, then $\F_T^c\subseteq \F^q$.
\end{lemma}

\begin{proof}
This follows from Lemma~\ref{L:focCFT} applied with $(\L^s,\N^s)$ in place of $(\L,\N)$ and from Corollary~\ref{C:NperpCapS} that 
\[\hyp(C_\F(T))\leq \foc(C_\F(T))\leq C_S(\N^s)=C_S^\circ(\N)\subseteq \N^\perp\subseteq\N.\]
Hence the assertion holds by Lemma~\ref{L:GetFTcIntoFq}. 
\end{proof}

Then next lemma is inspired by \cite[Lemma~6.6]{ChermakIII}, but works under a more general hypothesis.

\begin{lemma}\label{L:ChermakLemma66}
 Let $\M$ be a centric partial normal subgroup of $\L$ such that $\M\cap\N^\perp\subseteq Z^\circ(\N)$. Then $\N^\perp=C_S^\circ(\N)$.
\end{lemma}

\begin{proof}
Recall that $Z^\circ(\N)\leq Z^\circ(\N^\perp)$ by Lemma~\ref{L:NcapNperp} and $\M^\perp=Z^\circ(\M)\leq S$ by Lemma~\ref{L:CSNinN}. Hence, Corollary~\ref{C:OpNNPinMperp} applied with $\N^\perp$ in place of $\N$ yields $\N^\perp\subseteq S$. Now Corollary~\ref{C:NperpCapS} gives $\N^\perp=C_S^\circ(\N)$. 
\end{proof}

We are now able to show Theorem~\ref{T:GeneralizedFitting}. We take inspiration from the proof of \cite[Lemma~6.7, Corollary~6.9]{ChermakIII}, but by reducing to the case $\Delta=\F^s$ we are able to show a more general statement. 

\begin{proof}[Proof of Theorem~\ref{T:GeneralizedFitting}]
Theorem~\ref{T:VaryObjects} and Lemma~\ref{L:TakePlusGeneralizedFitting} allow us to assume throughout the proof without loss of generality that $\Delta=\F^s$. So by Corollary~\ref{C:GetNreplete} and Lemma~\ref{L:ZcircBasic}, we have in particular that $Z^\circ(\N)=Z(\N)$ for every partial normal subgroup of $\N$ of $\L$.

\smallskip

Fix now two centric radical partial normal subgroups $\M$ and $\N$ of $\L$ and set $\K:=\M\cap\N$. Clearly $\K$ is radical, so the difficulty is to show that $\K$ is centric. Note that, by Lemma~\ref{L:NcapNperp}, we have 
\[\M\cap (\N\cap\K^\perp)=\K\cap\K^\perp=Z(\K)\leq Z(\K^\perp).\]
Hence, since $Z(\K)\subseteq\K\subseteq\N$, it follows $\M\cap (\N\cap\K^\perp)\leq Z(\N\cap\K^\perp)$. As $\M$ is centric, Lemma~\ref{L:CSNinN} gives $\M^\perp\subseteq S$. Hence, an application of Corollary~\ref{C:OpNNPinMperp} with $\N\cap\K^\perp$ in place of $\N$ gives that $\N\cap\K^\perp\subseteq S$, i.e. $\N\cap\K^\perp$ is a normal $p$-subgroup of $\L$ and thus contained in $O_p(\L)$. As $O_p(\L)\subseteq\K$, we obtain
\[\N\cap\K^\perp\leq \K\cap\K^\perp=Z(\K).\]
Now Lemma~\ref{L:ChermakLemma66} applied with $(\N,\K)$ in place of $(\M,\N)$ gives $\K^\perp=C_S^\circ(\K)$. So $\K^\perp=C_S^\circ(\K)$ is a normal $p$-subgroup of $\L$ and thus contained in $O_p(\L)$. As $\K$ is radical in $\L$, it follows $\K^\perp\subseteq O_p(\L)\subseteq \K$. Hence $\K$ is centric in $\L$ and the assertion holds.
\end{proof}

\chapter{Regular localities}\label{S:Regular}

In this chapter we introduce regular localities and show that they have very nice properties. In particular, as already mentioned in the introduction, partial normal and partial subnormal subgroups of regular localities can be regarded as regular localities. We also show that every regular locality is $\N$-replete for every choice of a partial normal subgroup $\N$. 

\smallskip

The most important results of this chapter and the basic strategy of the proof are similar as in \cite[Section~7]{ChermakIII}, but there are also several differences. For example, since we work with $\N$-radical subgroups and the concept of  $\N$-repleteness, we argue in a slightly different framework. Moreover, in Theorem~\ref{T:RegularN1timesN2}(c) and Lemma~\ref{L:PerpendicularPartialNormalCentralProduct}, it is demonstrated how central products of partial groups (as introduced in Chapter~\ref{S:CentralProduct}) occur naturally in this context.

\smallskip

\textbf{Throughout let $\F$ be a saturated fusion system over $S$.}

\section{Definition and elementary results}

\begin{definition}\label{D:Regular}
Let $(\L^s,\F^s,S)$ be a subcentric locality attached to $\F$ and set 
\[\delta(\F):=\{P\leq S\colon P\cap F^*(\L^s)\in\F^s\}.\]
A linking locality $(\L,\Delta,S)$ over $\F$ is called \emph{regular} if $\Delta=\delta(\F)$. 
\end{definition}

By \cite[Theorem~A]{Henke:2015} there exists always a subcentric locality attached to $\F$ and such a subcentric locality is unique up to an isomorphism which restricts to the identity on $S$. We will see next that, as the notation suggests, the set $\delta(\F)$ does not depend on the choice of $\L^s$. Actually, the set $\delta(\F)$ could be characterized using any linking locality over $\F$. 

\begin{lemma}\label{L:deltaFDefinition}
The set $\delta(\F)$ depends only on $\F$ and not on the choice of $\L^s$. Moreover, if $(\L,\Delta,S)$ is a linking locality over $\F$, then
\[\delta(\F)=\{P\leq S\colon P\cap F^*(\L)\in\F^s\}.\]
\end{lemma}

\begin{proof}
If $(\L^s,\F^s,S)$ and $(\tL^s,\F^s,S)$ are subcentric localities over $\F$, then as mentioned above, there exists an isomorphism $\alpha\colon\L^s\rightarrow\tL^s$ with $\alpha|_S=\id_S$. By Lemma~\ref{L:AutGeneralizedFitting}, we have $F^*(\L^s)\alpha=F^*(\tL^s)$. In particular, $F^*(\L^s)\cap S=(F^*(\L^s)\cap S)\alpha=F^*(\L^s)\alpha\cap S\alpha=F^*(\tL^s)\cap S$. Therefore, $\delta(\F)$ does not depend on the choice of $\L^s$. 

\smallskip

If $(\L,\Delta,S)$ is any linking locality over $\F$, then by Theorem~\ref{T:VaryObjects}(c), we may choose $\L^s$ such that $\L^s|_\Delta=\L$. By Lemma~\ref{L:TakePlusGeneralizedFitting}(d), we have $F^*(\L^s)\cap\L=F^*(\L)$ and thus $F^*(\L^s)\cap S=F^*(\L)\cap S$. This implies the assertion.
\end{proof}

\begin{lemma}\label{L:GetVastRadintodeltaF}
Let $(\L,\Delta,S)$ be a linking locality with a centric radical partial normal subgroup $\N$ of $\L$. If $P\in\Delta$ is $\N$-radical, then $P\in\delta(\F)$. 
\end{lemma}

\begin{proof}
Suppose $P\in\Delta$ is $\N$-radical. We have $F^*(\L)\subseteq \N$ as $\N$ is centric radical. Hence, by Lemma~\ref{L:NradMrad}, $P$ is also $F^*(\L)$-radical. Setting $T^*:=S\cap F^*(\L)$, by Lemma~\ref{L:deltaFDefinition}, we have $P\in\delta(\F)$ if and only if $P\cap T^*=P\cap F^*(\L)\in\F^s$.  The latter condition is true since Lemma~\ref{L:NradEcr}, Theorem~\ref{T:GeneralizedFitting} and Lemma~\ref{L:GetFTcIntoFqCentric} imply  $P\cap T^*\in\F_{T^*}^c\subseteq\F^q\subseteq\F^s$.  
\end{proof}

\begin{lemma}\label{L:deltaFBasic}
The set $\delta(\F)$ is $\F$-closed and $\F^{cr}\subseteq\delta(\F)\subseteq \F^s$. In particular, there exists always a regular locality over $\F$ which is unique up to an isomorphism which restricts to the identity on $S$. 
\end{lemma}

\begin{proof} 
As before let $(\L^s,\F^s,S)$ be a linking locality over $\F$. As $F^*(\L^s)\cap S$ is strongly closed in $\F$ and $\F^s$ is $\F$-closed, it follows that $\delta(\F)$ is $\F$-closed and that $\delta(\F)\subseteq\F^s$. Let now $P\in\F^{cr}$. Then $P$ is an element of the object set $\F^s$ of the linking locality $(\L^s,\F^s,S)$ and so, by \cite[Lemma~6.2]{Henke:2015}, $P$ is $\L^s$-radical. As $\L^s$ is a centric radical partial normal subgroup of $\L^s$, Lemma~\ref{L:GetVastRadintodeltaF} implies now $P\in\delta(\F)$. So $\delta(\F)$ is $\F$-closed and $\F^{cr}\subseteq \delta(\F)\subseteq\F^s$. The unique existence of a regular locality follows now from \cite[Theorem~A]{Henke:2015} (see also Remark~\ref{R:ExistenceUniquenessCLS}). 
\end{proof}

The definition of $\delta(\F)$ given above is a priori slightly different from Chermak's definition \cite[p.36]{ChermakIII}, but we show in Remark~\ref{R:deltaF} below that the two definitions are equivalent. Apart from this difference, Lemmas~\ref{L:deltaFDefinition} and \ref{L:deltaFBasic} state the same as \cite[Lemma~6.10]{ChermakIII}.

\begin{corollary}\label{C:FstarcapSindeltaF}
Let $(\L,\Delta,S)$ be a linking locality over $\F$. Then every subgroup of $S$ containing $F^*(\L)\cap S$ is an element of $\delta(\F)$.
\end{corollary}

\begin{proof}
As a particular consequence of Lemma~\ref{L:deltaFBasic} we have $S\in\F^{cr}\subseteq\delta(\F)$. From the definition of $\delta(\F)$ it follows thus that $S\cap F^*(\L)\in\delta(\F)$ and that every overgroup of $S\cap F^*(\L)$ is in $\delta(\F)$.
\end{proof}

The following very important lemma states essentially the same as \cite[Lemma~7.3]{ChermakIII}.

\begin{lemma}\label{L:deltaFtimesNormal}
Let $(\L,\Delta,S)$ be a linking locality over $\F$ and $R\leq S$ with $R\unlhd\L$ or $R\unlhd \F$. For every $P\leq S$, we have then $P\in\delta(\F)$ if and only if $PR\in\delta(\F)$. 
\end{lemma}

\begin{proof}
By \cite{Henke:2015}, $R\unlhd\L$ if and only if $R\unlhd\F$. Thus, we may assume that $R\leq O_p(\L)$ and $R\unlhd\F$. Fix now $P\leq S$. As $\delta(\F)$ is by Lemma~\ref{L:deltaFBasic} overgroup-closed in $S$, we have $PR\in\delta(\F)$ if $P\in\delta(\F)$. Assume now $Q:=PR\in\delta(\F)$. Then $Q\cap F^*(\L)\in\F^s$. Since $R\leq O_p(\L)\subseteq F^*(\L)$, the Dedekind Lemma \cite[Lemma~1.10]{Chermak:2015} gives $Q\cap F^*(\L)=(P\cap F^*(\L))R$. Since $R\unlhd\F$, it follows now from \cite[Lemma~3.4]{Henke:2015} that $P\cap F^*(\L)\in\F^s$ and thus $P\in\delta(\F)$.  
\end{proof}

\begin{lemma}\label{L:GetCentricRadintodeltaF}
 Let $(\L,\Delta,S)$ be a linking locality with a centric partial normal subgroup $\N$. If $P\in\Delta$ is $\N$-radical and $P\leq S\cap\N$, then $P\in\delta(\F)$. 
\end{lemma}

\begin{proof}
By Theorem~\ref{T:ProductsPartialNormal}, $\N O_p(\L)$ is normal in $\L$. Clearly $\N O_p(\L)$ is radical in $\L$, and by  Lemma~\ref{L:CentricInherit}(b), $\N O_p(\L)$ is also centric. If $P\in\Delta$ is $\N$-radical with $P\leq S\cap\N$, then by Lemma~\ref{L:NradNOpLrad}, $PO_p(\L)$ is $\N O_p(\L)$-radical. Hence, by Lemma~\ref{L:GetVastRadintodeltaF}, we have $P O_p(\L)\in\delta(\F)$. So Lemma~\ref{L:deltaFtimesNormal} implies $P\in\delta(\F)$.
\end{proof}

\begin{lemma}\label{L:IsoSubcentricIsoRegular}
 Let $(\L,\Delta,S)$ and $(\tL,\tDelta,\tS)$ be subcentric localities over fusion systems $\F$ and $\tF$ respectively. Let $\alpha\colon\L\rightarrow \tL$ be an isomorphism from $(\L,\Delta,S)$ to $(\tL,\tDelta,\tS)$. Then $\delta(\F)\alpha=\delta(\tF)$ and $\alpha$ induces an isomorphism from $(\L|_{\delta(\F)},\delta(\F),S)$ to $(\tL|_{\delta(\tF)},\delta(\tF),\tS)$. 
\end{lemma}

\begin{proof}
By Lemma~\ref{L:AutGeneralizedFitting}, we have $F^*(\L)\alpha=F^*(\tL)$. It is a consequence of \cite[Lemma~2.21(b)]{Henke:2020} that $\alpha|_S$ induces an isomorphism from $\F$ to $\tF$. Hence, by \cite[Lemma~3.6]{Henke:2015}, we have $\F^s\alpha=\tF^s$. This yields $\delta(\F)\alpha=\delta(\tF)$. It follows then easily that $\alpha$ induces an isomorphism from $(\L|_{\delta(\F)},\delta(\F),S)$ to $(\tL|_{\delta(\tF)},\delta(\tF),\tS)$. 
\end{proof}

\section{$\N$-repleteness of regular localities}

\begin{lemma}\label{L:RegularNReplete}
 Let $(\L,\Delta,S)$ be subcentric locality over $\F$, $\N\unlhd\L$ and $T=S\cap\N$. Then the following hold:
\begin{itemize}
 \item [(a)] We have $X_{P,Q}\in\delta(\F)$ for all $P\in R_\Delta(S\N)$ and $Q\in R_\Delta(SC_\L(T))$.
 \item [(b)] Every $S\N$-radical subgroup $P\in\Delta$ is an element of $\delta(\F)$.
 \item [(c)] Setting $\L_\delta:=\L|_{\delta(\F)}$ and $\N_\delta:=\N\cap\L_\delta$, we have $\F_{S\cap \N}(\N)=\F_{S\cap\N_\delta}(\N_\delta)$.
\end{itemize}
\end{lemma}

\begin{proof}
If $P\in\Delta$ is $S\N$-radical, then by Lemma~\ref{L:SNradNrad} and Lemma~\ref{L:CSTinSNradical}, we have $C_S(T)O_p(\L)\leq P$ and thus $X_{P,S}=(P\cap T)C_S(T)O_p(\L)$ is contained in $P$. Since $\delta(\F)$ is overgroup-closed in $S$ by Lemma~\ref{L:deltaFBasic}, this shows that part (b) follows from part (a). Moreover, part (b) together with Lemma~\ref{L:PartialNormalAlperin} implies (c). Hence, it remains only to show (a).

\smallskip

Fix $P,Q\in\Delta$ such that $P$ is $S\N$-radical and $Q$ is $S C_\L(T)$-radical. As $\N^\perp\subseteq C_\L(T)$ by Corollary~\ref{C:NperpCapSGeneral}, it follows from Lemma~\ref{L:NradMrad} that $Q$ is also $\N^\perp$-radical. Set  
\[U:=P\cap T,\;V:=Q\cap\N^\perp,\; X:=UV.\]
As $X\leq X_{P,Q}=U(Q\cap C_S(T))O_p(\L)$ and $\delta(\F)$ is by Lemma~\ref{L:deltaFBasic} overgroup-closed in $S$, it is sufficient to prove $X\in \delta(\F)$. 

\smallskip

By Corollary~\ref{C:GetNreplete}, $(\L,\Delta,S)$ is $\N$-replete. Hence, by Corollary~\ref{C:NperpCapS}, we have $\N^\perp\cap S=C_S(\N)$. Now by Lemma~\ref{L:focCFT}, we have $\foc(C_\F(T))\leq C_S(\N)=\N^\perp\cap S$, i.e. the hypothesis of Lemma~\ref{L:PutTogetherK} is fulfilled with $\N^\perp$ in place of $\K$. Hence, as $\Delta=\F^s$, it follows from that lemma that $X=UV\in \Delta$. 

\smallskip

Set $\M:=\N\N^\perp$. By Lemma~\ref{L:CentricInherit}(c), $\M$ is a centric partial normal subgroup of $\L$. So by Lemma~\ref{L:GetCentricRadintodeltaF}, in order to show that $X\in\delta(\F)$, it is sufficient to show that $X$ is $\M$-radical. To see this we prove first 
\begin{equation}\label{E:NNNperpX}
N_{\M}(X)=N_\N(X)N_{\N^\perp}(X).
\end{equation}
Clearly, $N_\N(X)N_{\N^\perp}(X)\subseteq N_{\M}(X)$. Let $f\in N_{\M}(X)$. Then by \cite[Lemma~5.2]{Chermak:2015}, we can write $f$ as $f=nm$ with $n\in\N$, $m\in\N^\perp$ and $S_f=S_{(n,m)}$. Note that $n$ centralizes  $V\leq \N^\perp\cap S\leq C_S(\N)$ and $m\in \N^\perp\subseteq C_\L(T)$ centralizes $\<U^n,U\>\leq T$. Hence, $X=X^f=(U^nV)^m=U^nV^m$, which implies $X^n=U^nV\leq X$ and $X^m=UV^m\leq X$. This means that $n\in N_\N(X)$ and $m\in N_{\N^\perp}(X)$ proving \eqref{E:NNNperpX}. 

\smallskip

By Lemma~\ref{L:SNradNrad}, we have $TO_p(\L)\leq O_p(N_\L(T))\leq Q$. We have now $P\cap T=U\leq X\leq UC_S(T)O_p(\L)\leq P$ and $Q\cap \N^\perp=V\leq X\leq TVO_p(\L)\leq Q$. Hence, $X$ is $\N$-radical and $\N^\perp$-radical by Lemma~\ref{L:PradQrad}. By Lemma~\ref{L:NNperpCommute} and since $N_\L(X)$ is a group, we have $[N_\N(X),N_{\N^\perp}(X)]=1$. Moreover, $N_\N(X)\cap N_{\N^\perp}(X)\leq \N\cap\N^\perp=Z^\circ(\N)\leq S$ is a $p$-group by Lemma~\ref{L:NcapNperp}. Hence, by \eqref{E:NNNperpX}, we have $O_p(N_{\M}(X))=O_p(N_\N(X))O_p(N_{\N^\perp}(X))\leq X$, where the inclusion uses that $X$ is $\N$-radical and $\N^\perp$-radical. So $X$ is $\M$-radical and as argued above this yields the assertion. 
\end{proof}

\begin{cor}\label{C:RegularNReplete}
Let $(\L,\Delta,S)$ be a regular locality over $\F$ or assume more generally that $(\L,\Delta,S)$ is a linking locality over $\F$ with $\delta(\F)\subseteq\Delta$. Let $\N$ be a partial normal subgroup of $\L$. Then $(\L,\Delta,S)$ is $\N$-replete and, in particular, $\N^\perp\cap S=C_S^\circ(\N)=C_S(\N)$ and $Z^\circ(\N)=Z(\N)$.  
\end{cor}

\begin{proof}
By Theorem~\ref{T:VaryObjects}, there exists a subcentric locality $(\L^s,\F^s,S)$ over $\F$ such that $\L=\L^s|_\Delta$. Adapt Notation~\ref{N:VaryObjects}, let $P\in R_\Delta(S\N)$ and $Q\in R_\Delta(SC_\L(T))$. It follows from Lemma~\ref{L:TakePlusRadical} that $P\in R_{\F^s}(S\N^s)$. Similarly, Lemma~\ref{L:TakePlusRadical} combined with Remark~\ref{R:TakePlus} gives $Q\in R_{\F^s}(S C_{\L^+}(T))$. Hence, Lemma~\ref{L:RegularNReplete}(a) implies $X_{P,Q}\in\delta(\F)\subseteq\Delta$ showing that $(\L,\Delta,S)$ is $\N$-replete. Therefore, $(\L,\Delta,S)$ is  weakly $\N$-replete and the assertion follows from Corollary~\ref{C:NperpCapS} and Lemma~\ref{L:ZcircBasic}.
\end{proof}

\section{Centric partial normal subgroups}

The next goal is to show that partial normal subgroups of regular localities are regular localities (in a sense that will be made more precise below). As a first step, we study the special case of centric partial normal subgroups. A similar strategy was used by Chermak who studied first centric radical partial normal subgroups of regular localities. 
In the case that $\N$ is centric radical, Lemma~\ref{L:RegularCentricMain} below is the same as \cite[Lemma~7.4]{ChermakIII}; the generalization to the centric case is then basically given by \cite[Lemma~7.5]{ChermakIII}.

\begin{lemma}\label{L:RegularConjAutomorphisms}
Let $(\L,\Delta,S)$ be a regular locality and let $\N\unlhd\L$ be centric in $\L$. Set $T:=\N\cap S$ and $\E:=\F_T(\N)$. Then the following hold:
\begin{itemize}
\item [(a)] If $O_p(\L)\leq P\leq S$, then $P\in\Delta$ if and only if $P\cap T\in\Delta$. In particular, for every $w\in \W(\L)$, we have $w\in\D$ if and only if $S_w\cap T\in\Delta$. 
\item [(b)] Setting
\[\Gamma:=\{P\in\Delta\colon P\leq T\},\]
the triple $(\N,\Gamma,T)$ is a linking locality over $\E$. In particular, $\E$ is saturated.
\item [(c)] Every element $f\in N_\L(T)$ induces an automorphism of $\L$ and of $\N$ via conjugation, i.e. $\L=\D(f)$ and $c_f$ is an automorphism of $\L$, which restricts to an automorphism of $\N$. Moreover, \[c_f\in\Aut(\L,\Delta,S),\;c_f|_\N\in\Aut(\N,\Gamma,T),\;c_f|_S\in\Aut(\F)\mbox{ and }c_f|_T\in\Aut(\E).\]
\item [(d)] $\E^c\subseteq\F^q$ and $\E^{cr}\subseteq\Delta=\delta(\F)$.
\item [(e)] $\E$ is normal in $\F$.
\item [(f)] $\E^s=\{P\leq T\colon P\in\F^s\}$.
\end{itemize}
\end{lemma}

\begin{proof}
By Theorem~\ref{T:ProductsPartialNormal}, $\N O_p(\L)\unlhd\L$ and $(\N O_p(\L))\cap S=TO_p(\L)$. Clearly $\N O_p(\L)$ is radical in $\L$ and Lemma~\ref{L:CentricInherit}(b) gives that $\N O_p(\L)$ is centric in $\L$. Thus, we have 
\[F^*(\L)\subseteq \N O_p(\L).\]
Hence, it follows from Lemma~\ref{L:deltaFDefinition} that a subgroup $P\leq S$ is an element of $\delta(\F)=\Delta$ if and only if $P\cap (TO_p(\L))\in\Delta$. If $O_p(\L)\leq P$, then by the Dedekind Lemma and by Lemma~\ref{L:deltaFtimesNormal}, we have $P\cap (TO_p(\L))=(P\cap T)O_p(\L)\in\Delta$ if and only if $P\cap T\in\Delta$. This proves the first part of (a). Therefore, Lemma~\ref{L:NLTAutL} gives that (a) and (c) hold and that $(\N,\Gamma,T)$ is a locality over $\E$ of objective characteristic $p$ if $\Gamma$ is as in (b). Hence, (b) holds by Theorem~\ref{T:Saturation} provided we can show $\E^{cr}\subseteq \Gamma$.

\smallskip

By (c), the elements of $N_\L(T)$ induce automorphisms of $\E$. In particular, $N_\L(T)$ acts on $\E^c$. Hence, by the Frattini Lemma and the Splitting Lemma \cite[Corollary~3.11, Lemma~3.12]{Chermak:2015}, $\E^c$ is closed under $\F$-conjugacy. Thus, 
\[\E^c=\F_T^c\subseteq\F^q\] by Lemma~\ref{L:GetFTcIntoFqCentric}. Let $(\L^s,\F^s,S)$ be a subcentric locality with $\L^s|_\Delta=\L$, and let $\N^s\unlhd\L^s$ with $\N^s\cap\L=\N$; notice that $\L^s$ and $\N^s$ exist by Theorem~\ref{T:VaryObjects}. By Lemma~\ref{L:RegularNReplete}(c), we have $\F_T(\N^s)=\F_T(\N)=\E$. 

\smallskip

Fix now $P\in\E^{cr}$. Then there exists a fully $\E$-normalized $\E$-conjugate $Q$ of $P$ such that $O_p(N_\E(Q))=Q$. Note also that $Q\in\E^{cr}\subseteq\E^c\subseteq \F^q$ is an object in $(\L^s,\F^s,S)$. So it follows from Lemma~\ref{L:EcrNradical} applied with $(\L^s,\F^s,S)$ in place of $(\L,\Delta,S)$ that $Q$ is $\N^s$-radical. As $\N$ is centric in $\L$, Lemma~\ref{L:TakePlusGeneralizedFitting}(a) yields that $\N^s$ is centric in $\L^s$. Hence, $Q\in\Delta=\delta(\F)$ by Lemma~\ref{L:GetCentricRadintodeltaF}. As $\Delta$ is closed under $\F$-conjugacy, it follows $P\in\Delta$ showing $\E^{cr}\subseteq\Delta$. This proves (d) and, as argued above, also (b). 

\smallskip

We argue now that (e) holds. As $O_p(\L)\leq S\in\Delta$, it is a special case of (a) that $T=S\cap T\in\Delta$. In particular, every element of $\Aut_\F(T)$ is by Lemma~\ref{L:LocalityFusionSystem}(a) realized by an element of $N_\L(T)$ and induces thus by (c) an automorphism of $\E$. Since we have seen in (b) that $\E$ is saturated, it follows thus from Lemma~\ref{L:FTNWeaklyNormalNormal} that $\E$ is normal in $\F$. So (e) holds.

\smallskip

As $\E\unlhd\F$, \cite[Lemma~3.10]{Henke:2015} gives $P\in\E^s$ for every $P\in\F^s$ with $P\leq T$. Let now $P\in\E^s$. We need to show that $P\in\F^s$. By \cite[Lemma~3.7]{Henke:2015} (or by a direct argument similar to the one given above for $\E^c$), the set $\E^s$ is invariant under $\F$-conjugacy. Thus, we may assume $P\in\F^f$ and thus $P\in\E^f$ by \cite[Lemma~2.3]{Henke:2018}. Then by definition of $\E^s$ and by (d), we have $Q:=O_p(N_\E(P))\in \E^c\subseteq \F^q$. Notice that the elements of $N_S(P)$ induce automorphisms of $N_\E(P)$ and thus $N_S(P)\leq N_S(Q)$. By \cite[Lemma~2.6(c)]{Aschbacher/Kessar/Oliver:2011}, there exists $\alpha\in\Hom_\F(N_S(Q),S)$ such that $Q\alpha\in\F^f$. As $N_S(P)\alpha\leq N_S(P\alpha)$ and $P\in\F^f$, we have then $N_S(P)\alpha=N_S(P\alpha)$ and $P\alpha\in\F^f$. In particular, $N_T(P)\alpha=N_T(P\alpha)$. So $\alpha|_{N_T(P)}$ induces an isomorphism from $N_\E(P)$ to $N_\E(P\alpha)$ and maps thus $Q$ to $O_p(N_\E(P\alpha))$. Hence, replacing $(P,Q)$ by $(P\alpha,Q\alpha)$, we may assume $P,Q\in\F^f$. As $Q\in\F^q$, it follows from \cite[Lemma~3.1]{Henke:2015} that $Q\in\F^s$ and $N_\F(Q)$ is constrained. Using \cite[Proposition~I.6.4]{Aschbacher/Kessar/Oliver:2011}, it is easy to observe that $N_\E(P)$ is $N_\F(P)$-invariant. Thus, $N_\E(P)$ is weakly normal in $N_\F(P)$ as $P\in\E^f$ and therefore $N_\E(P)$ is saturated. Thus, \cite[Lemma~2.12(b)]{Henke:2015} gives $Q\unlhd N_\F(P)$ and so $N_\F(P)=N_{N_\F(Q)}(P)$. Notice that $P\in N_\F(Q)^f$ since $P\in\F^f$ and $N_S(P)\leq N_S(Q)$. Therefore, $N_\F(P)$ is constrained by \cite[Lemma~2.11]{Henke:2015}. By \cite[Lemma~3.1]{Henke:2015}, this means $P\in\F^s$ as required. This proves (f).
\end{proof}

\begin{remark}\label{R:deltaF}
If $(\L,\Delta,S)$ is a regular locality, it is a special case of Lemma~\ref{L:RegularConjAutomorphisms}(e),(f) that $\F_{S\cap F^*(\L)}(F^*(\L))\unlhd \F$ and 
\[\F_{S\cap F^*(\L)}(F^*(\L))^s=\{P\leq S\cap F^*(\L)\colon P\in\F^s\}.\]
By Lemma~\ref{L:RegularNReplete}(c), the same holds if $(\L,\Delta,S)$ is a subcentric locality. In either case,  Lemma~\ref{L:deltaFBasic} yields in particular that  
\[\delta(\F)=\{P\leq S\colon P\cap F^*(\L)\in\F_{S\cap F^*(\L)}(F^*(\L))^s\}.\]
So our definition of $\delta(\F)$ agrees with Chermak's definition in \cite[p.37]{ChermakIII}.   
\end{remark}

\begin{lemma}\label{L:RegularCentricMain}
Let $(\L,\Delta,S)$ be a regular locality, let $\N\unlhd\L$ be centric in $\L$, set $T:=\N\cap S$ and $\E:=\F_T(\N)$. Then $\E$ is saturated, $(\N,\delta(\E),T)$ is a regular locality over $\E$, and 
\[\delta(\E)=\{P\in\Delta\colon P\leq T\}.\]
Moreover, $O_p(\N)\unlhd\L$ and $F^*(\N)O_p(\L)=F^*(\L)$. 
\end{lemma}

\begin{proof}
Set $\Gamma:=\{P\in\Delta\colon P\leq T\}$. By Lemma~\ref{L:RegularConjAutomorphisms}(b), the fusion system $\E$ is saturated and $(\N,\Gamma,T)$ is a linking locality over $\E$. In particular, $\delta(\E)$, $O_p(\N)$ and $F^*(\N)$ are defined. It follows from Lemma~\ref{L:RegularConjAutomorphisms}(c) that $O_p(\N)^f=O_p(\N)$ for every $f\in N_\L(T)$. Hence, by  \cite[Corollary~3.13]{Chermak:2015}, we have $O_p(\N)\unlhd\L$. Thus, it remains to prove that $F^*(\N)O_p(\L)=F^*(\L)$ and $\delta(\E)=\Gamma$. We will show the former property in two steps and then use it in a third step to conclude $\delta(\E)=\Gamma$. 

\smallskip

\emph{Step~1:} We prove $F^*(\N)\subseteq F^*(\L)$. For the proof observe first that 
\[\K:=F^*(\L)\cap\N\unlhd\L.\]
In particular, $\K\unlhd\N$. So, using Notations~\ref{N:NperpL} and \ref{N:Zcirc}, $\K_\N^\perp\unlhd\N$ and $Z_\N^\circ(\K)$ are defined.  By Lemma~\ref{L:AutNperp} and Lemma~\ref{L:RegularConjAutomorphisms}(c), we have $\K_\N^\perp c_f=\K_\N^\perp$ for every $f\in N_\L(T)$. Hence, by \cite[Corollary~3.13]{Chermak:2015}, we have \[\K_\N^\perp\unlhd\L.\]
Now observe that
\[\K_\N^\perp\cap F^*(\L)=\K_\N^\perp\cap \N\cap F^*(\L)=\K_\N^\perp\cap \K\leq Z_\N^\circ(\K_\N^\perp),\]
where the inclusion at the end uses Lemma~\ref{L:NcapNperp}. As $(\L,\Delta,S)$ is a regular locality,  Corollary~\ref{C:RegularNReplete} gives $Z^\circ(\K_\N^\perp)=Z(\K_\N^\perp)$ (where $Z^\circ(\K_\N^\perp)$ means more precisely $Z_\L^\circ(\K_\N^\perp)$). So by Lemma~\ref{L:ZcircBasic}, we have
\[\K_\N^\perp\cap F^*(\L)\leq Z_\N^\circ(\K_\N^\perp)\leq Z(\K_\N^\perp)=Z^\circ(\K_\N^\perp).\]
As $F^*(\L)$ is by Theorem~\ref{T:GeneralizedFitting} centric in $\L$, we get from Lemma~\ref{L:CSNinN} that $F^*(\L)^\perp\leq S$. Now Corollary~\ref{C:OpNNPinMperp} yields $\K_\N^\perp\subseteq S$. As $\K_\N^\perp\unlhd\L$, it follows thus from Lemma~\ref{L:OpL} that $\K_\N^\perp\subseteq O_p(\L)\subseteq F^*(\L)$. So $\K_\N^\perp\subseteq F^*(\L)\cap \N=\K$ and $\K$ is centric in $\N$. As $O_p(\N)\unlhd\L$, we have $O_p(\N)\subseteq  O_p(\L)\cap\N\subseteq F^*(\L)\cap\N=\K$. So $\K$ is centric radical in $\N$, which implies $F^*(\N)\subseteq\K\subseteq F^*(\L)$. This completes the first step. 

\smallskip

\emph{Step~2:} We show that $F^*(\L)=F^*(\N)O_p(\L)$. For the proof set $\M:=F^*(\N)$. By Step~1 it is sufficient to prove that $F^*(\L)\subseteq \M O_p(\L)$. By Lemma~\ref{L:AutGeneralizedFitting} and Lemma~\ref{L:RegularConjAutomorphisms}(c), we have $\M c_f=\M$ for every $f\in N_\L(T)$. It follows now from  $\M\unlhd\N$ and \cite[Corollary~3.13]{Chermak:2015} that $\M\unlhd\L$. As $\M$ commutes with $\M^\perp=\M^\perp_\L$ in $\L$ by Lemma~\ref{L:NNperpCommute}, one can observe that $\M$ commutes with $\N\cap\M^\perp$ in $\N$. As $\M^\perp\unlhd\L$, we have moreover $\N\cap\M^\perp\unlhd\N$. Hence, by Lemma~\ref{L:PerpendicularNormalinNperp} applied with $(\N,\Gamma,T)$ in place of $(\L,\Delta,S)$, we have 
\[\N\cap\M^\perp\subseteq\M_\N^\perp.\]
According to Theorem~\ref{T:GeneralizedFitting}, $\M$ is centric in $\N$ and so Lemma~\ref{L:CSNinN} gives 
\[\M_\N^\perp=Z_\N^\circ(\M).\]
By Lemma~\ref{L:ZcircBasic} and Lemma~\ref{L:NcapNperp}, we have $Z_\N^\circ(\M)\leq Z(\M)=Z_\L^\circ(\M)\leq Z_\L^\circ(\M^\perp)$. So
\[\N\cap\M^\perp\subseteq\M_\N^\perp=Z_\N^\circ(\M)\leq Z_\L^\circ(\M^\perp).\]
As $\N$ is by assumption centric, it follows from Lemma~\ref{L:CSNinN} that $\N^\perp\subseteq S$. Hence, using Corollary~\ref{C:OpNNPinMperp}, we can conclude that $\M^\perp\subseteq S$. As $\M^\perp\unlhd\L$, Lemma~\ref{L:OpL} yields $\M^\perp\subseteq O_p(\L)$. By Theorem~\ref{T:ProductsPartialNormal}, $\M O_p(\L)$ is a partial normal subgroup of $\L$. Using $\M^\perp\subseteq O_p(\L)\subseteq\M O_p(\L)$, Lemma~\ref{L:CentricInherit}(a) allows us to conclude that $\M O_p(\L)$ is centric radical in $\L$. Hence, $F^*(\L)\subseteq \M O_p(\L)$ as required. This completes Step~2.

\smallskip

\emph{Step~3:} We show $\Gamma=\delta(\E)$. Let $P\leq T$. We have seen above that $O_p(\N)\unlhd\L$ and this implies $O_p(\N)=O_p(\L)\cap \N$. Using the Dedekind Lemma \cite[Lemma~1.10]{Chermak:2015}, observe now that 
\begin{eqnarray*}
PO_p(\L)\cap F^*(\N)&=&PO_p(\L)\cap \N\cap F^*(\N)\\
&=&P(O_p(\L)\cap\N)\cap F^*(\N)\\
&=&PO_p(\N)\cap F^*(\N)\\
&=&(P\cap F^*(\N))O_p(\N). 
\end{eqnarray*}
Using Step~2, one sees moreover \[PO_p(\L)\cap F^*(\L)=PO_p(\L)\cap F^*(\N)O_p(\L)=(PO_p(\L)\cap F^*(\N))O_p(\L).\]
Putting both properties together we get thus    
\begin{equation}\label{E:POpL}
PO_p(\L)\cap F^*(\L)=(P\cap F^*(\N))O_p(\N)O_p(\L)=(P\cap F^*(\N))O_p(\L).
\end{equation}
We have now the following equivalences:
\begin{eqnarray*}
 P\in\delta(\E)&\Longleftrightarrow & P\cap F^*(\N)\in\E^s\mbox{ (by Lemma~\ref{L:deltaFDefinition})}\\
&\Longleftrightarrow & P\cap F^*(\N)\in\F^s\mbox{ (by Lemma~\ref{L:RegularConjAutomorphisms}(f))}\\
&\Longleftrightarrow & (P\cap F^*(\N))O_p(\L)\in\F^s\mbox{ (by \cite[Lemma~3.4]{Henke:2015})}\\
&\Longleftrightarrow & PO_p(\L)\cap F^*(\L)\in\F^s\mbox{ (by \eqref{E:POpL})}\\
&\Longleftrightarrow & PO_p(\L)\in \delta(\F)\mbox{ (by Lemma~\ref{L:deltaFBasic})}\\
&\Longleftrightarrow & P\in\delta(\F)\mbox{ (by Lemma~\ref{L:deltaFtimesNormal})}\\
&\Longleftrightarrow & P\in\Gamma\mbox{ (as $P\leq T$ and $\Delta=\delta(\F)$).}
\end{eqnarray*}
This shows $\Gamma=\delta(\E)$ and the proof is complete.
\end{proof}

\section{Commuting partial normal subgroups whose product is centric}

\begin{lemma}\label{L:Step4Seperate}
Let $(\L,\Delta,S)$ be a regular locality and let $\N_1$ and $\N_2$ be partial normal subgroups of $\L$ such that $\N_1$ commutes with $\N_2$ in $\L$. If $\H\unlhd\N_i$ for some $i\in\{1,2\}$, then $\H\unlhd\N_1\N_2=\N_2\N_1\unlhd\L$. 
\end{lemma}

\begin{proof}
By Theorem~\ref{T:ProductsPartialNormal}, $\N_1\N_2=\N_2\N_1$ is a partial normal subgroup of $\L$. Moreover, Lemma~\ref{C:MperpNNperpM} implies that $\N_2$ commutes with $\N_1$. So the situation is symmetric in $\N_1$ and $\N_2$. Without loss of generality, assume $\H\unlhd\N_1$. Let $h\in\H$ and $f\in\N_1\N_2$ with $h\in\D(f)$. By \cite[Lemma~5.2]{Chermak:2015}, there exist $n_1\in\N_1$ and $n_2\in\N_2$ such that $f=n_1n_2$ and $S_f=S_{(n_1,n_2)}$. Since $(f^{-1},h,f)\in\D$, it follows that  $u:=(n_2^{-1},n_1^{-1},h,n_1,n_2)\in\D$ via $S_{(f^{-1},h,f)}$. So $h^f=\Pi(u)=(h^{n_1})^{n_2}$. As $\H\unlhd \N_1$, we have $h^{n_1}\in\H$. Hence, as $\N_2$ commutes with $\N_1\supseteq\H$, it follows from Lemma~\ref{L:PerpendicularPartialNormalPrepare} that $h^f=h^{n_1}\in\H$. This shows $\H\unlhd\N_1\N_2$ as required. 
\end{proof}

Parts (a) and (b) of the following lemma state essentially the same as \cite[Lemma~7.6]{ChermakIII}. For the proof of part (a), we use the properties of fusion systems stated in Lemma~\ref{L:F1starF2Main}. In contrast, Chermak works with localities throughout and refers in his proof to \cite[Lemma~2.10]{ChermakIII}.

\begin{lemma}\label{T:RegularN1timesN2}
Let $(\L,\Delta,S)$ be a regular locality and let $\N_1$ and $\N_2$ be partial normal subgroups of $\L$. Suppose that $\N_1$ commutes with  $\N_2$ in $\L$ and that $\N_1\N_2$ is a centric partial normal subgroup of $\L$. Set $T_i=\N_i\cap S$ and $\E_i:=\F_{T_i}(\N_i)$ for $i=1,2$. Then the following hold:
\begin{itemize}
 \item [(a)] For $i=1,2$, the subsystem $\E_i$ is saturated, $(\N_i,\delta(\E_i),T_i)$ is a regular locality over $\E_i$ and 
\[\delta(\E_i)=\{P\leq T_i\colon PT_{3-i}\in\Delta\}.\]
\item [(b)] We have $S\cap\N_1\N_2=T_1T_2$, $O_p(\N_1\N_2)=O_p(\N_1)O_p(\N_2)$ and $F^*(\N_1\N_2)=F^*(\N_1)F^*(\N_2)$. 
\item [(c)] The partial group $\N_1\N_2$ is an internal central product of $\N_1$ and $\N_2$ in the sense of Definition~\ref{D:CentralProductPartialGroup}. Moreover, $\F_{T_1T_2}(\N_1\N_2)=\E_1*\E_2$.
\item [(d)] Let $P_i\leq T_i$ for $i=1,2$. Then $P_1P_2\cap F^*(\N_1\N_2)=(P_1\cap F^*(\N_1))(P_2\cap F^*(\N_2))$. Moreover, 
$P_1P_2\in\Delta$ if and only if $P_i\in\delta(\E_i)$ for each $i=1,2$.
\end{itemize}
\end{lemma}

\begin{proof}
Setting $T:=(\N_1\N_2)\cap S$ and $\Gamma:=\{P\leq T\colon P\in\Delta\}$, Lemma~\ref{L:RegularCentricMain} gives that $(\N_1\N_2,\Gamma,T)$ is a regular locality over $\F_T(\N_1\N_2)$. As $\N_1$ commutes with $\N_2$ in $\L$, it follows moreover that $\N_1$ commutes also with $\N_2$ in $\N_1\N_2$. Hence, replacing $(\L,\Delta,S)$ by $(\N_1\N_2,\Gamma,T)$, we may assume without loss of generality that
\[\L=\N_1\N_2.\]
By Lemma~\ref{C:MperpNNperpM}, $\N_2$ also commutes with $\N_1$ and so the situation is symmetric in $\N_1$ and $\N_2$. By Lemma~\ref{L:PerpendicularNormLTNormL}, we have $\N_i\subseteq C_\L(T_{3-i})$ for $i=1,2$. In particular, $\E_1$ and $\E_2$ centralize each other in the sense of Definition~\ref{D:SubsystemCentralize}. So $\E_1*\E_2$ is by Lemma~\ref{L:F1starF2inF} well-defined and contained in $\F$. 
For $i=1,2$ set
\[\Gamma_i=\{P\leq T_i\colon PT_{3-i}\in\Delta\}.\]
We proceed now in several steps to show the assertion.

\smallskip

\emph{Step~1:} Let $s_i\in T_i$ and $n_i\in\N_i$ for $i=1,2$ such that $s_1s_2\in S_{(n_1,n_2)}$. We show that $s_i\in S_{n_i}$ for $i=1,2$ and $((s_1s_2)^{n_1})^{n_2}=s_1^{n_1}s_2^{n_2}$. In particular,  $S_{(n_1,n_2)}=(S_{n_1}\cap T_1)(S_{n_2}\cap T_2)$. 

\smallskip

As $\N_1\subseteq C_\L(T_2)$, we have $s_1s_2\in S_{(n_1,n_2)}\leq S_{n_1}\geq T_2\ni s_2$, which implies $s_1\in S_{n_1}$. Thus, using $\N_2\subseteq C_\L(T_1)$, we can conclude that $s_1^{n_1}s_2=(s_1s_2)^{n_1}\in S_{n_2}\geq T_1\ni s_1^{n_1}$. Therefore, we have $s_2\in S_{n_2}$ and $((s_1s_2)^{n_1})^{n_2}=s_1^{n_1}s_2^{n_2}$. This proves the first part of the assertion and $S_{(n_1,n_2)}\leq (S_{n_1}\cap T_1)(S_{n_2}\cap T_2)$. The converse inclusion follows from $\N_i\subseteq C_\L(T_{3-i})$ for $i=1,2$.

\smallskip

\emph{Step~2:} We argue that $S=T_1T_2$ and $\F=\E_1*\E_2$. Indeed, $S=T_1T_2$ holds by Theorem~\ref{T:ProductsPartialNormal}. Moreover, by \cite[Lemma~5.2]{Chermak:2015}, an element  $f\in\L$ can be written as product $f=n_1n_2$ with $n_i\in \N_i$ for $i=1,2$ and $S_f=S_{(n_1,n_2)}$. By Step~1 we have $S_f=(S_{n_1}\cap T_1)(S_{n_2}\cap T_2)$ and $c_f|_{S_{n_i}\cap T_i}=c_{n_i}|_{S_{n_i}\cap T_i}\in\E_i$. Hence, $\F\subseteq \E_1*\E_2$. We have already seen above that $\E_1*\E_2\subseteq\F$, so Step~2 is complete. 

\smallskip

\emph{Step~3:} We show that $(\N_i,\Gamma_i,T_i)$ is a linking locality over $\E_i$ and thus $\E_i$ is saturated by Theorem~\ref{T:Saturation} for $i=1,2$. For the proof fix $i\in\{1,2\}$. By \cite[Lemma~3.1(c)]{Chermak:2015}, $T_i$ is a maximal $p$-subgroup of $\N_i$. Notice that $\Gamma_i$ is overgroup-closed in $T_i$ as $\Delta$ is overgroup-closed in $S$. Moreover, since $\N_i\subseteq C_\L(T_{3-i})$ and $\Delta$ is closed under taking $\L$-conjugates in $S$, it follows that $\Gamma_i$ is closed under taking $\N_i$-conjugates in $T_i$. Fix $w\in\W(\N_i)$. If $w\in\D_{\Gamma_i}$, then $\N_i\subseteq C_\L(T_{3-i})$ implies also $w\in\D=\D_{\Delta}$. On the other hand, if $w\in\D$, then $w=(n_1,\dots,n_k)\in\D$ via some $P_0,P_1,\dots,P_k\in\Delta$, and because of $\N_i\subseteq C_\L(T_{3-i})$, we may choose the $P_j$ such that $T_{3-i}\leq P_j$ for $j=0,1,\dots,k$. As $S=T_iT_{3-i}$, we have $P_j=(P_j\cap T_i)T_{3-i}$ and thus $Q_j:=P_j\cap T_i\in\Gamma_i$ for $i=0,1,\dots,k$. Now $w\in\D_{\Gamma_i}$ via $Q_0,Q_1,\dots,Q_k$. This shows that $(\N_i,\Gamma_i,T_i)$ is a locality. Clearly it is a locality over $\E_i$ by definition of $\E_i$. Note that $T_{3-i}\in\E_{3-i}^{cr}$. Thus, for every $P\in \E_i^{cr}$, it follows from Lemma~\ref{L:F1starF2Main}(f) and Step~2 that $PT_{3-i}\in \F^{cr}$. Hence, by Lemma~\ref{L:deltaFBasic}, we have $PT_{3-i}\in\Delta$ and thus $P\in\Gamma_i$. This proves $\E_i^{cr}\subseteq\Gamma_i$. For every $Q\in\Gamma_i$, we have $R:=QT_{3-i}\in\Delta$. Hence,  Lemma~\ref{L:MSCharp} implies that $N_{\N_i}(R)\unlhd N_\L(R)$ is of characteristic $p$ and thus $N_{\N_i}(Q)=N_{N_{\N_i}(R)}(Q)$ is of characteristic $p$. Thus, $(\N_i,\Gamma_i,T_i)$ is a linking locality over $\E_i$. 

\smallskip 

\emph{Step~4:} We prove $O_p(\L)=O_p(\N_1)O_p(\N_2)$. By \cite[Proposition~5]{Henke:2015} and Lemma~\ref{L:OpL}, we have $O_p(\F)=O_p(\L)$ and $O_p(\N_i)=O_p(\E_i)$ for $i=1,2$; we use here that $(\N_i,\Gamma_i,T_i)$ is a linking locality over $\E_i$ by Step~3. Hence, by Step~2 and Lemma~\ref{L:F1starF2Main}(c), we have $O_p(\L)=O_p(\F)=O_p(\E_1)O_p(\E_2)=O_p(\N_1)O_p(\N_2)$. 

\smallskip

We will use Step~3 from now without further reference. In particular, we use that 
\[\H_i:=F^*(\N_i)\mbox{ for }i=1,2\]
is well-defined.

\smallskip

\emph{Step~5:} We show that $T_1\cap T_2\leq \N_1\cap \N_2\subseteq \H_1\cap \H_2\cap S$. Fix $i\in\{1,2\}$ and notice that $\N_{3-i}\subseteq \N_i^\perp$ as $\N_{3-i}$ commutes with $\N_i$. Hence, by Lemma~\ref{L:ZcircBasic} and Lemma~\ref{L:NcapNperp}, we have $\N_1\cap \N_2\subseteq\N_i\cap\N_i^\perp\subseteq Z(\N_i)\cap S$. In particular, $\N_1\cap\N_2\leq O_p(\N_i)\subseteq \H_i$. 
This completes Step~5. 

\smallskip

\emph{Step~6:} We prove that $F^*(\L)\subseteq\H_1\H_2$. Since $\H_1$ and $\H_2$ are partial normal subgroups of $\L$ by Lemma~\ref{L:Step4Seperate}, it follows from Theorem~\ref{T:ProductsPartialNormal} that $\M:=\H_1\H_2\unlhd\L$. Notice that $\M$ is radical by Step~4. We show now that $\M$ is centric. By Lemma~\ref{L:PerpProduct}, we have $\M^\perp=\H_1^\perp\cap\H_2^\perp$. Observe that, for $i=1,2$, $\N_{3-i}$ commutes with $\H_i\subseteq \N_i$ and thus $\N_{3-i}\subseteq \H_i^\perp=(\H_i)^\perp_\L$. Hence, applying the Dedekind Lemma \cite[Lemma~1.10]{Chermak:2015} twice, we obtain $\H_1^\perp=(\H_1^\perp\cap\N_1)\N_2$ and $\H_1^\perp\cap\H_2^\perp=(\H_1^\perp\cap \N_1)(\H_2^\perp\cap\N_2)$. Observe that $\H_i^\perp\cap\N_i\unlhd\N_i$ commutes with $\H_i$ and thus $\H_i^\perp\cap\N_i\subseteq (\H_i)^\perp_{\N_i}\subseteq\H_i$, as $\H_i$ is centric in $\N_i$ by Theorem~\ref{T:GeneralizedFitting}. Hence, $\M^\perp=\H_1^\perp\cap\H_2^\perp= (\H_1^\perp\cap \N_1)(\H_2^\perp\cap\N_2)\subseteq\H_1\H_2=\M$. This shows that $\M$ is centric in $\L$ and thus centric radical in $\L$. Hence, $F^*(\L)\subseteq\M$ and Step~6 is complete. 

\smallskip

\emph{Step~7:} We show that $F^*(\L)=\H_1\H_2$. By Step~6 it is sufficient to prove $\H_1\H_2\subseteq F^*(\L)$. Hence, fixing $i\in\{1,2\}$, we need to prove that $\H_i$ is contained in $\K:=F^*(\L)\cap \N_i$. Notice that $\K\unlhd\L$ and $\K^\perp_{\N_i}\unlhd\L$ by Lemma~\ref{L:Step4Seperate}. Using  Lemma~\ref{L:NcapNperp} we see moreover
\[\K^\perp_{\N_i}\cap F^*(\L)=\K^\perp_{\N_i}\cap \N_i\cap F^*(\L)=\K^\perp_{\N_i}\cap\K=Z_{\N_i}^\circ(\K)\leq Z^\circ_{\N_i}(\K_{\N_i}^\perp).\]
Hence, by Lemma~\ref{L:ZcircBasic} and Corollary~\ref{C:RegularNReplete}, we have 
\[\K^\perp_{\N_i}\cap F^*(\L)\leq Z(\K^\perp_{\N_i})=Z_\L^\circ(\K^\perp_{\N_i}).\]
As $F^*(\L)^\perp\subseteq S$ by Lemma~\ref{L:CSNinN} and Theorem~\ref{T:GeneralizedFitting}, it follows now from Corollary~\ref{C:OpNNPinMperp} that $\K^\perp_{\N_i}\subseteq S$. So Lemma~\ref{L:OpL} yields $\K^\perp_{\N_i}\subseteq O_p(\L)$. Hence, $\K^\perp_{\N_i}\subseteq O_p(\L)\cap\N_i\subseteq \K$ and $\K$ is centric in $\N_i$. By Step~4, $O_p(\N_i)\subseteq O_p(\L)\cap \N_i\subseteq\K$, so $\K$ is centric radical in $\N_i$ and $\H_i\subseteq\K$. This completes Step~7.

\smallskip

\emph{Step~8:} We prove (a) and (d). For the proof let $P_i\leq T_i$ and set $T_i^*:=S\cap \H_i=T_i\cap \H_i$ for $i=1,2$. It follows from Step~5 that $T_1\cap T_2\subseteq \H_1\cap \H_2$. So
\begin{equation}\label{E:StarStep8}
D:=T_1\cap T_2=T_1^*\cap T_2^*. 
\end{equation}
This allows us to conclude that, for $i=1,2$, we have  $P_{3-i}\cap T_i\leq T_i^*$ and hence
\begin{equation}\label{E:StarStarStep8}
P_1P_2\cap T_i^*=P_iP_{3-i}\cap T_i\cap T_i^*=P_i(P_{3-i}\cap T_i)\cap T_i^*=(P_i\cap T_i^*)(P_{3-i}\cap T_i).
\end{equation}
We compute now
\begin{eqnarray*}
 P_1P_2\cap T_1^*T_2^* &=& P_1P_2\cap P_1T_2\cap T_1^*T_2^*\\
&=& P_1P_2\cap (P_1T_2\cap T_1^*)T_2^*\\
&=& P_1P_2\cap (P_1\cap T_1^*)(T_2\cap T_1)T_2^*\mbox{ (by \eqref{E:StarStarStep8} applied with $T_2$ in place of $P_2$)}\\
&=& P_1P_2\cap (P_1\cap T_1^*)T_2^*\mbox{ (by \eqref{E:StarStep8})}\\
&=& (P_1\cap T_1^*)(P_1P_2\cap T_2^*)\\
&=& (P_1\cap T_1^*)(P_2\cap T_2^*)(P_1\cap T_2)\mbox{ (by \eqref{E:StarStarStep8})}\\
&=& (P_1\cap T_1^*)(P_2\cap T_2^*)\mbox{ (as $P_1\cap T_2\leq P_1\cap D\leq P_1\cap T_1^*$ by \eqref{E:StarStep8}).}
\end{eqnarray*}
By Step~7 and Theorem~\ref{T:ProductsPartialNormal}, we have $F^*(\L)\cap S=T_1^*T_2^*$. Hence, the equation above translates to
\begin{equation}\label{E:3StarStep8}
 P_1P_2\cap F^*(\L)=(P_1\cap F^*(\N_1))(P_2\cap F^*(\N_2)).
\end{equation}
Hence,
\begin{eqnarray*}
 P_1P_2\in\Delta&\Longleftrightarrow & P_1P_2\cap F^*(\L)\in\F^s\mbox{ (by Lemma~\ref{L:deltaFDefinition})}\\
&\Longleftrightarrow & (P_1\cap F^*(\N_1))(P_1\cap F^*(\N_2))\in\F^s\mbox{ (by \eqref{E:3StarStep8})}\\
&\Longleftrightarrow & P_i\cap F^*(\N_i)\in\E_i^s\mbox{ for each }i=1,2\mbox{ (by Step~2 and Lemma~\ref{L:F1starF2Main}(g))}\\
&\Longleftrightarrow & P_i\in\delta(\E_i)\mbox{ for each }i=1,2\mbox{ (by Step~3 and Lemma~\ref{L:deltaFDefinition}).}
\end{eqnarray*}
This completes the proof of (d). Fix now $i\in\{1,2\}$. Proving (a) amounts by Step~3 to showing that $\Gamma_i=\delta(\E_i)$. Notice that $T_{3-i}\in \delta(\E_{3-i})$. Hence it follows from (d) that we have the following equivalence:
\begin{equation*}
 P_i\in \Gamma_i\Longleftrightarrow P_iT_{3-i}\in\Delta\Longleftrightarrow P_i\in\delta(\E_i).
\end{equation*}
This shows $\Gamma_i=\delta(\E_i)$ and Step~8 is thus complete. 

\smallskip

\emph{Step~10:} We finish the proof. Parts (a) and (d) hold by Step~8 and part (b) holds by Step~2, Step~4 and Step~7. It  remains thus to prove (c). By Step~2, we only need to prove that $\L$ is an internal central product of $\N_1$ and $\N_2$. That is, we need to prove that the conditions ($\C 1$) and ($\C 2$) in Definition~\ref{D:CentralProductPartialGroup} hold for $k=2$ and $(\N_1,\N_2)$ in place of $(\L_1,\L_2)$. 

\smallskip

For the proof of ($\C 1$) let $n_i\in\N_i$ for $i=1,2$. As $\N_i\subseteq C_\L(T_{3-i})$ for $i=1,2$, we have $(S_{n_1}\cap T_1)(S_{n_2}\cap T_2)\leq S_{(n_1,n_2)}$. Part (a) implies $S_{n_i}\cap T_i\in\delta(\E_i)$ for each $i=1,2$. Thus, part (d) gives $(S_{n_1}\cap T_1)(S_{n_2}\cap T_2)\in\Delta$. As $\Delta$ is overgroup closed, it follows that $S_{(n_1,n_2)}\in\Delta$ and  $(n_1,n_2)\in\D$. So ($\C 1$) holds.

\smallskip

Let now
\begin{eqnarray*}
 A=\begin{pmatrix}
    f_1 & f_2 & \hdots & f_n\\
    g_1 & g_2 & \hdots & g_n
   \end{pmatrix}
\end{eqnarray*}
be a matrix such that $u:=(f_1,\dots,f_n)\in\W(\N_1)$ and $v=(g_1,\dots,g_n)\in\W(\N_2)$. Set $w:=(f_1g_1,\dots,f_ng_n)$. To prove ($\C 2$) we need to show that $w\in\D$ if and only if $u,v\in\D$. Moreover, if so, we need to show that $\Pi(w)=\Pi(\Pi(u),\Pi(v))$. 

\smallskip

By Step~5 we have $\N_1\cap\N_2\leq S$. Hence, Lemma~\ref{L:DecomposeEltofMN} gives $S_{fg}=S_{(f,g)}$ for all $f\in\N_1$ and $g\in\N_2$. Combining this property with Step~1, we obtain that, for all $j=1,\dots,n$, we have \[S_{f_jg_j}=S_{(f_j,g_j)}=(S_{f_j}\cap T_1)(S_{g_j}\cap T_2)\]
and, whenever $s_i\in S_i$ for $i=1,2$ with $s_1s_2\in S_{f_jg_j}$, we have $s_1\in S_{f_j}$, $s_2\in S_{g_j}$ and  $(s_1s_2)^{f_jg_j}=s_1^{f_j}s_2^{g_j}$. This implies that $S_w=S_{(f_1,g_1,f_2,g_2,\dots,f_n,g_n)}\leq (S_u\cap T_1)(S_v\cap T_2)$. The converse inclusion holds since $\N_i\subseteq C_\L(T_{3-i})$ for $i=1,2$. Hence, $S_w=S_{(f_1,g_1,f_2,g_2,\dots,f_n,g_n)}=(S_u\cap T_1)(S_v\cap T_2)$. Therefore, (d) implies that $S_w\in\Delta$ if and only if $S_u\cap T_1\in\delta(\E_1)$ and $S_v\cap T_2\in\delta(\E_2)$. Using  (a) and the fact that $(\L,\Delta,S)$ is a locality, one sees now that $w\in\D$ if and only if $u,v\in\D$. Moreover, this is the case if and only if $(f_1,g_1,f_2,g_2,\dots,f_n,g_n)\in\D$.  

\smallskip

Assume now that $w\in\D$. As $(f_1,g_1,\dots,f_n,g_n)\in\D$ we have
\[\Pi(w)=\Pi(f_1,g_1,\dots,f_n,g_n).\]
We claim now that $u\circ v\in\D$ via $S_w$ and $\Pi(w)=\Pi(u\circ v)$. This will in particular imply $\Pi(w)=\Pi(\Pi(u),\Pi(v))$ and thus complete the proof of ($\C 2$). By the axioms of a partial group, we have $(f_1g_1,\dots,f_{n-1}g_{n-1})\in\D$. So to prove our claim, by induction on $n$, we can assume that $(f_1,\dots,f_{n-1},g_1,\dots,g_{n-1})\in\D$ via $S_{(f_1g_1,\dots,f_{n-1}g_{n-1})}$ and $\Pi(f_1g_1,\dots,f_{n-1}g_{n-1})=\Pi(f_1,\dots,f_{n-1},g_1,\dots,g_{n-1})$. Then $(f_1,\dots,f_{n-1},g_1,\dots,g_{n-1},f_n,g_n)\in\D$ via $S_w$ and 
\begin{eqnarray*}
 \Pi(w)&=&\Pi(\Pi(f_1g_1,\dots,f_{n-1}g_{n-1}),f_n,g_n)\\
&=& \Pi(\Pi(f_1,\dots,f_{n-1},g_1,\dots,g_{n-1}),f_n,g_n)\\
&=& \Pi(f_1,\dots,f_{n-1},g_1,\dots,g_{n-1},f_n,g_n).
\end{eqnarray*}
By Corollary~\ref{C:MperpNNperpM}, $\N_2$ commutes strongly with $\N_1$. Hence, it follows from Lemma~\ref{L:CommuteStronglyProducts} applied with $\N_2,\N_1,(f_1,\dots,f_{n-1}),(g_1,\dots,g_{n-1}),(f_n),(g_n)$ in the roles of $\X,\Y,u,a,b,v$ that 
\[u\circ v=(f_1,\dots,f_{n-1},f_n,g_1,\dots,g_{n-1},g_n)\in\D\mbox{ via }S_w\leq S_{(f_1,\dots,f_{n-1},g_1,\dots,g_{n-1},f_n,g_n)}=S_{u\circ v}\]
and 
\[\Pi(w)=\Pi(f_1,\dots,f_{n-1},g_1,\dots,g_{n-1},f_n,g_n)=\Pi(u\circ v).\]
As argued above, the proof is now complete.
\end{proof}

\section{Partial normal and partial subnormal subgroups} We are now able to prove that partial normal subgroups of $\L$ form regular localities. Indeed, some more precise information is given in the following theorem. Most parts are also stated in  \cite[Theorem~7.7]{ChermakIII}.

\begin{theorem}\label{T:RegularPartialNormal}
Let $(\L,\Delta,S)$ be a regular locality, let $\N\unlhd\L$, $T:=\N\cap S$ and $\E:=\F_T(\N)$. Then the following hold: 
\begin{itemize}
\item [(a)] The subsystem $\E$ is saturated, $(\N,\delta(\E),T)$ is a regular locality over $\E$ and 
\[\delta(\E)=\{P\leq T\colon PC_S(\N)\in\Delta\}.\]
\item [(b)] $\N\N^\perp\unlhd\L$ is an internal central product of $\N$ and $\N^\perp$; 
\item [(c)] $O_p(\N\N^\perp)=O_p(\N)O_p(\N^\perp)$ and $O_p(\N)=\N\cap O_p(\L)\unlhd \L$;
\item [(d)] $F^*(\N\N^\perp)=F^*(\N)F^*(\N^\perp)$, $F^*(\L)=\F^*(\N)F^*(\N^\perp)O_p(\L)$ and $F^*(\N)=F^*(\L)\cap\N\unlhd\L$;
\item [(e)] $\N^\perp=C_\L(\N)\unlhd\L$. 
\item [(f)] For every $f\in N_\L(T)$ we have $\N\subseteq \D(f)$, $c_f|_\N\in\Aut(\N,\delta(\E),T)$ and $c_f|_T\in\Aut(\E)$. Moreover, the map $N_\L(T)\rightarrow \Aut(\N),f\mapsto c_f|_\N$ is a homomorphism of partial groups.
\end{itemize}
\end{theorem}

\begin{proof}
By definition, $\N^\perp$ commutes with $\N$ and thus, by Corollary~\ref{C:MperpNNperpM}, $\N$ commutes with $\N^\perp$. Moreover, by Lemma~\ref{L:CentricInherit}(c), the product  $\N\N^\perp$ is a centric partial normal subgroup of $\L$. Hence, it follows from Theorem~\ref{T:RegularN1timesN2}(a),(b),(c) that parts (a) and (b) hold, that $O_p(\N\N^\perp)=O_p(\N)O_p(\N^\perp)$ and that $F^*(\N\N^\perp)=F^*(\N)F^*(\N^\perp)$. Using Lemma~\ref{L:RegularCentricMain}, it follows $F^*(\L)=F^*(\N\N^\perp)O_p(\L)=F^*(\N)F^*(\N^\perp)O_p(\L)$.

\smallskip

To show $O_p(\N)=O_p(\L)\cap\N\unlhd\L$ we use Lemma~\ref{L:OpL} throughout. It follows from Lemma~\ref{L:NcapNperp} that $\N\cap\N^\perp\leq O_p(\N)$. By Lemma~\ref{L:RegularCentricMain}, we have $O_p(\N)O_p(\N^\perp)=O_p(\N\N^\perp)\unlhd\L$ and so $O_p(\N)=O_p(\N\N^\perp)\cap \N\unlhd\L$. In particular $O_p(\N)\leq O_p(\L)$. Clearly $O_p(\L)\cap\N\leq O_p(\N)$, so equality holds. This completes the proof of (c).

\smallskip

Since $F^*(\L)=F^*(\N\N^\perp)O_p(\L)$, we have $F^*(\L)\cap \N\N^\perp=F^*(\N\N^\perp)(\N\N^\perp\cap O_p(\L))=F^*(\N\N^\perp)O_p(\N\N^\perp)=F^*(\N\N^\perp)$. Hence, using $\N\cap\N^\perp\leq O_p(\N)$, we conclude  $F^*(\L)\cap\N=F^*(\N\N^\perp)\cap\N=F^*(\N)F^*(\N^\perp)\cap \N=F^*(\N)$. Thus (d) holds.

\smallskip

As $\N\N^\perp$ is an internal central product of $\N$ and $\N^\perp$, Lemma~\ref{L:CentralProductsFactorsCentralize} implies $\N^\perp\subseteq C_{\N\N^\perp}(\N)\subseteq C_\L(\N)$. By Lemma~\ref{L:ResidueEquivalences} and Corollary~\ref{C:NperpCapS}, we have moreover
\[C_\L(T)=C_S(T)O^p_{N_\L(T)}(C_\L(T))=C_S(T)\N^\perp.\]
So an element $c\in C_\L(\N)\subseteq C_\L(T)$ can be written as $c=sm$ where $s\in C_S(T)$ and $m\in \N^\perp\subseteq C_\L(\N)$. We have then $S_c=S_{(s,m)}$ by Lemma~\ref{L:NLSbiset}. For $n\in\N$, it follows that $u=(m^{-1},s^{-1},n,s,m)\in\D$ via $S_{(c^{-1},n,c)}$ and $n^s\in\N$. As $c,m\in C_\L(\N)$, we can conclude that
\[n=n^c=\Pi(u)=(n^s)^m=n^s\] 
and thus $s\in C_S(\N)$. By Corollary~\ref{C:RegularNReplete}, we have $C_S(\N)=C_S^\circ(\N)=\N^\perp\cap S$, so in particular  $s\in\N^\perp$. This proves $c=sm\in\N^\perp$ and thus $C_\L(\N)\subseteq\N^\perp$. Hence (e) holds.

\smallskip

For the proof of (f) let $f\in N_\L(T)$. As $\N^\perp\unlhd N_\L(T)$, it follows from the Frattini Lemma \cite[Corollary~3.11]{Chermak:2015} and (e) that there exists $c\in\N^\perp=C_\L(\N)$ and $h\in N_\L(TC_S(\N))=N_{N_\L(T)}(C_S(\N))$ with $(c,h)\in\D$ and $f=ch$. We show now first that
\begin{equation}\label{E:NDf}
 \N\subseteq\D(f)\mbox{ and }c_f|_\N=c_h|_\N\in\Aut(\N).
\end{equation}
For the proof of \eqref{E:NDf} let $n\in\N$. As $c\in C_\L(\N)$, we have $(c^{-1},n,c)\in\D$ and $n^c=n$. In particular, by Lemma~\ref{L:RegularConjAutomorphisms}(a), we have $P:=S_{(c^{-1},n,c)}\cap TC_S(\N)\in\Delta$ as $\N\N^\perp=\N C_\L(\N)$ is centric. So $u:=(h^{-1},c^{-1},n,c,h)\in\D$ via $P^h$. Hence, $(f^{-1},n,f)\in\D$ and $n^f=\Pi(u)=(n^c)^h=n^h$. This proves $\N\subseteq\D(f)$ and $c_f|_\N=c_h|_\N$. By Lemma~\ref{L:RegularConjAutomorphisms}(c) applied with $\N\N^\perp$ in place of $\N$, we have $c_h\in\Aut(\L)$ and thus $c_h|_\N\in\Aut(\N)$. Hence \eqref{E:NDf} holds. 

\smallskip

As $\E=\F_T(\N)$, it is now easy to check that $c_f|_T\in\Aut(\E)$. In particular, by \cite[Lemma~3.6]{Henke:2015}, $c_f|_T$ leaves $\E^s$ invariant. As $F^*(\N)c_f=F^*(\N)$, it follows thus from Lemma~\ref{L:deltaFDefinition} that $\delta(\E)c_f=\delta(\E)$. Hence, $c_f|_\N\in\Aut(\N,\delta(\E),T)$.

\smallskip

It remains to prove that the map $N_\L(T)\rightarrow \Aut(\N),f\mapsto c_f|_\N$ is a homomorphism of partial groups. This means that, fixing $(f_1,\dots,f_k)\in\W(N_\L(T))\cap\D$, we need to show that $c_{\Pi(f_1,\dots,f_k)}=c_{f_1}\circ c_{f_2}\circ\cdots\circ c_{f_k}$. With a similar argument as above, we can write $f_i=c_ih_i$ with $c_i\in C_\L(\N)$ and $h_i\in N_\L(TC_S(\N))$ for $i=1,\dots,k$. The Frattini Lemma together with the Splitting Lemma \cite[Corollary~3.11, Lemma~3.12]{Chermak:2015} even allows us to assume that $S_{f_i}=S_{(c_i,h_i)}$ so that $w:=(c_1,h_1,\dots,c_k,h_k)\in\D$ with $\Pi(w)=\Pi(f_1,\dots,f_k)$. By (b), $\N$ and $\N^\perp=C_\L(\N)$ form a central product. Hence, by the definition of a central product, we have $(c_1,\dots,c_n)\in\D$, $(h_1,\dots,h_k)\in\D$ and $\Pi(w)=ch$ where $c:=\Pi(c_1,\dots,c_k)$ and $h:=\Pi(h_1,\dots,h_k)$. By \eqref{E:NDf}, we have thus $c_{\Pi(f_1,\dots,f_n)}=c_{\Pi(w)}=c_h$ and $c_{f_i}=c_{h_i}$ for $i=1,\dots,k$. Hence, we only need to show that $c_h=c_{h_1}\circ\cdots\circ c_{h_k}$. Fix $n\in\N$. By Lemma~\ref{L:RegularConjAutomorphisms}(a), we have $Q:=S_n\cap TC_S(\N)\in\Delta$. Hence, $v:=(h_k^{-1},\dots,h_1^{-1},n,h_1,\dots,h_k)\in\D$ via $Q^h$ and, by the axioms of partial groups, $nc_h=n^h=\Pi(v)=n(c_{h_1}\circ c_{h_2}\circ\cdots\circ c_{h_k})$. This completes the proof.
\end{proof}

If $(\L,\Delta,S)$ is a regular locality and $\N\unlhd\L$, it is a particular consequence of Theorem~\ref{T:RegularPartialNormal} that the $p$-residual $O^p(\N)$ of $\N$ and the $p^\prime$-residual $O^{p^\prime}(\N)$ are defined. The following two lemmas are similiar to Lemma~7.10 and Corollary~7.11 in \cite{ChermakIII}.

\begin{lemma}\label{L:CharacteristicNormalIntersection}
Let $(\L,\Delta,S)$ be a regular locality, $\N\unlhd\L$, $T:=\N\cap S$ and $\E=\F_T(\N)$. Suppose $\mathbb{K}$ is a set of partial normal subgroups of $\N$ such that $\M:=\bigcap\mathbb{K}$ is invariant under $\Aut(\N,\delta(\E),T)$. Then $\M\unlhd\L$. If $\M\in\mathbb{K}$, then $\M=\bigcap\{\K\in\mathbb{K}\colon \K\unlhd\L\}$.
\end{lemma}

\begin{proof}
As the elements of $\mathbb{K}$ are partial normal subgroups of $\N$, we have $\M\unlhd\N$. Since $\M$ is invariant under $\Aut(\N,\delta(\E),T)$, it follows from Theorem~\ref{T:RegularPartialNormal}(f) that $\M$ is invariant under conjugation by elements of $N_\L(T)$. Hence, \cite[Corollary~3.13]{Chermak:2015} yields that $\M\unlhd\L$. In particular, if $\M\in\mathbb{K}$, then $\M\in\mathbb{K}_0:=\{\K\in\mathbb{K}\colon \K\unlhd\L\}$. Hence, it follows in this case that $\M=\bigcap\mathbb{K}\subseteq\bigcap\mathbb{K}_0\subseteq \M$ and thus $\M=\bigcap\mathbb{K}_0$.  
\end{proof}

\begin{lemma}\label{L:ResidueRegular}
Let $(\L,\Delta,S)$ be a regular locality and $\N\unlhd\L$. Then $O^p_\L(\N)=O^p(\N)$ and $O^{p^\prime}_\L(\N)=O^{p^\prime}(\N)$.  
\end{lemma}

\begin{proof}
Set $T:=\N\cap S$, 
\[\mathbb{K}:=\{\K\unlhd\N\colon T\K=\N\}\mbox{ and }\mathbb{K}':=\{\K\unlhd\N\colon T\subseteq\N\}.\]
Notice that 
\[O^p(\N)=\bigcap\mathbb{K}\mbox{ and }O^{p^\prime}(\N)=\bigcap\mathbb{K}^\prime\]
are invariant under $\Aut(\N)$ by Lemma~\ref{L:AutResidue}. Clearly $O^{p^\prime}(\N)\in\mathbb{K}^\prime$ and, by Lemma~\ref{L:ResidueEquivalences}, also $O^p(\N)\in\mathbb{K}$. As $\mathbb{K}_\N=\{\K\in\mathbb{K}\colon \K\unlhd\L\}$ and $\mathbb{K}_\N^\prime=\{\K\in\mathbb{K}'\colon \K\unlhd\L\}$, it follows from Lemma~\ref{L:CharacteristicNormalIntersection} that
\[O^p(\N)=\bigcap\mathbb{K}=\bigcap\mathbb{K}_\N=O^p_\L(\N)\mbox{ and }O^{p^\prime}(\N)=\bigcap\mathbb{K}^\prime=\bigcap\mathbb{K}_\N^\prime=O^{p^\prime}(\N).\]
\end{proof}

\begin{proof}[Proof of Theorem~\ref{T:mainRegularPartialNormal}]
Using Corollary~\ref{C:NperpCapS} and Lemma~\ref{L:ResidueRegular}, the assertion follows from  Theorem~\ref{T:RegularPartialNormal}(a),(e).
\end{proof}

We now turn attention to partial subnormal subgroups. Theorem~\ref{T:RegularPartialNormal} implies easily the following corollary which is partly the same as \cite[Corollary~7.9]{ChermakIII}.

\begin{corollary}\label{C:RegularSubnormal}
Let $(\L,\Delta,S)$ be a regular locality, $\H\subn\L$, $T:=\H\cap S$ and $\E:=\F_T(\H)$. Then $\E$ is saturated and  $(\H,\delta(\E),T)$ is a regular locality over $\E$. Moreover, the following hold:
\begin{itemize}
 \item [(a)] $PC_S(\H)\in\Delta$ for every $P\in\delta(\E)$.
 \item [(b)] $O_p(\H)=O_p(\L)\cap\H$ and $F^*(\H)=F^*(\L)\cap\H$.
 \item [(c)] For every $\M\unlhd\L$ with $\M\subseteq\H$, we have $\M^\perp_\H=\M^\perp\cap\H$.
\end{itemize}  
\end{corollary}

\begin{proof}
Let $\H=\H_0\unlhd\H_1\unlhd\cdots\unlhd\H_k=\L$ and assume without loss of generality that $k\geq 1$. Set $\N:=\H_{k-1}\unlhd\L$, $T_1:=\N\cap S$ and $\E_1:=\F_{T_1}(\N)$. We argue first that the claim is true with $\N$ in place of $\H$ and conclude then by induction on $k$ that it holds for $\H$.

\smallskip

By Theorem~\ref{T:RegularPartialNormal}(a),(c),(d), $(\N,\delta(\E_1),T_1)$ is a regular locality over $\E_1$ with $\delta(\E_1)=\{P\leq T_1\colon PC_S(\N)\in\Delta\}$; moreover $O_p(\N)=O_p(\L)\cap\N$ and $F^*(\N)=F^*(\L)\cap\N$. So except for the claim in (c), the assertion holds with $\N$ in place of $\H$.

\smallskip

Let now $\M\unlhd\L$ with $\M\subseteq\N$. Then $\M^\perp_\N\unlhd\N$ by definition of $\M^\perp_\N$. Set $T:=\N\cap S$. By Theorem~\ref{T:RegularPartialNormal}(f), the elements of $N_\L(T)$ induces automorphisms of $\N$ which then leave $\M$ invariant as $\M\unlhd\L$. Hence, by Lemma~\ref{L:AutNperp}, $\M^\perp_\N$ is invariant under conjugation by elements of $N_\L(T)$. Now \cite[Corollary~3.13]{Chermak:2015} yields $\M^\perp_\N \unlhd\L$. Thus $\M^\perp_\N$ is a partial normal subgroup of $\L$ which commutes with $\M$ and is thus contained in $\M^\perp=\M^\perp_\L$. On the other hand, $\M^\perp\cap\N$ is a partial normal subgroup of $\N$ which commutes with $\M$ and thus contained in $\M^\perp_\N$. Hence, $\M^\perp\cap \N=\M^\perp_\N$. This shows that the assertion holds with $\N$ in place of $\H$. 

\smallskip

So the assertion is true if $k=1$ and $\H=\N$. Suppose now $k>1$ and let $\M$ be as in (c). By induction on $k$, we may assume then that $\E$ is saturated, $(\H,\delta(\E),T)$ is a regular locality, $PC_{T_1}(\H)\in\delta(\E_1)$ for every $P\in\delta(\E)$, $O_p(\H)=O_p(\N)\cap\H$, $F^*(\H)=F^*(\N)\cap\H$ and $\M^\perp_\H=\M^\perp_\N\cap\H$. Then for every $P\in\delta(\E)$, we have $PC_{T_1}(\H)C_S(\N)\in\Delta$. Thus, as $\Delta$ is overgroup-closed and $C_{T_1}(\H)C_S(\N)\leq C_S(\H)$, it follows that $PC_S(\H)\in\Delta$ for all $P\in\delta(\E)$. Moreover, we have
\[O_p(\H)=O_p(\N)\cap \H=(O_p(\L)\cap \N)\cap\H=O_p(\L)\cap\H,\]
\[F^*(\H)=F^*(\N)\cap\H=(F^*(\L)\cap\N)\cap\H=F^*(\L)\cap \H\]
and
\[\M^\perp_\H=\M^\perp_\N\cap \H=(\M^\perp\cap\N)\cap\H=\M^\perp\cap\H.\]
This shows the assertion.
\end{proof}

\begin{lemma}\label{L:H1H2subN1N2}
Let $(\L,\Delta,S)$ be a regular locality with partial normal subgroups $\N_1$ and $\N_2$ of $\L$ such that $\N_1$ commutes with $\N_2$. Suppose for each $i=1,2$, we are given $\H_i\subn\N_i$. Then $\H_1\H_2\subn\N_1\N_2\unlhd\L$ and thus $\H_1\H_2\subn\L$.
\end{lemma}

\begin{proof}
Suppose $(\L,\N_1,\N_2,\H_1,\H_2)$ is a counterexample with $|\L|+|\N_1|+|\N_2|$ minimal. By Theorem~\ref{T:ProductsPartialNormal}, $\N_1\N_2\unlhd\L$. Hence, $(\N_1\N_2,\N_1,\N_2,\H_1,\H_2)$ must be a counterexample as well, so the minimality of $|\L|+|\N_1|+|\N_2|$ together with Theorem~\ref{T:RegularPartialNormal}(a) yields $\L=\N_1\N_2$. We must have $\H_1\neq \N_1$ or  $\H_2\neq \N_2$ as otherwise $(\L,\N_1,\N_2,\H_1,\H_2)$ would not be a counterexample. As $\N_1$ commutes with $\N_2$, we have $\N_1\N_2=\N_2\N_1$ and $\H_1\H_2=\H_2\H_1$. Moreover, $\N_2$ commutes with $\N_1$ by Corollary~\ref{C:MperpNNperpM}. So the situation is symmetric in $(\N_1,\H_1)$ and $(\N_2,\H_2)$. Thus, without loss of generality, we may assume $\H_1\neq \N_1$. As $\H_1\subn\N_1$, there exists then $\N_1'\unlhd\N_1$ with $\H_1\subn\N_1'$ and $\N_1'\neq\N_1$. By Lemma~\ref{L:Step4Seperate}, we have $\N_1'\unlhd\N_1\N_2=\L$. Clearly $\N_1'\subseteq\N_1$ commutes with $\N_2$. The minimality of $|\L|+|\N_1|+|\N_2|$ yields that $(\L,\N_1',\N_2,\H_1,\H_2)$ cannot be a counterexample. Thus, $\H_1\H_2\subn\N_1'\N_2\unlhd \N_1\N_2=\L$, which contradicts the assumption that $(\L,\N_1,\N_2,\H_1,\H_2)$ is a counterexample.  
\end{proof}

\begin{lemma}\label{L:ResidueSubnormalContainment}
 Let $(\L,\Delta,S)$ be a regular locality. Let $\K$ and $\H$ be subnormal in $\L$ with $\K\subseteq\H$. Then $O^p(\K)\subseteq O^p(\H)$ and $O^{p^\prime}(\K)\subseteq O^{p^\prime}(\H)$.  
\end{lemma}

\begin{proof}
We will use that, by Corollary~\ref{C:RegularSubnormal}, every partial subnormal subgroup of $\L$ is a regular locality. In particular, the statement of the lemma makes sense. Moreover, $\H$ is a regular locality and by Lemma~\ref{L:SubnormalinSubgroup}(a), we have $\K\subn\H$. Thus, we may assume without loss of generality that $\H=\L$. Thus, we only need to show that $O^p(\K)\subseteq O^p(\L)$ and $O^{p^\prime}(\K)\subseteq O^{p^\prime}(\L)$. Let $\K=\K_0\unlhd\K_1\unlhd\cdots\unlhd\K_n=\L$ be a subnormal series for $\K$ in $\L$. Then $\N:=\K_{n-1}$ is normal in $\L$. By induction on $n$, we may assume $O^p(\K)\subseteq O^p(\N)$ and $O^{p^\prime}(\K)\subseteq O^{p^\prime}(\N)$. So using  Lemma~\ref{L:ResidueNinM} and Lemma~\ref{L:ResidueRegular}, we can conclude now that $O^p(\K)\subseteq O^p(\N)=O^p_\L(\N)\subseteq O^p_\L(\L)=O^p(\L)$ and similiarly $O^{p^\prime}(\K)\subseteq O^{p^\prime}(\N)=O^{p^\prime}_\L(\N)\subseteq O^{p^\prime}_\L(\L)=O^{p^\prime}(\L)$. 
\end{proof}

The following lemma says that the implications in Remark~\ref{R:Implications} are equivalences for regular localities.

\begin{lemma}\label{L:PerpendicularPartialNormalCentralProduct}
Let $(\L,\Delta,S)$ be a regular locality and let $\N_1$ and $\N_2$ be partial normal subgroups such that $\N_1$ commutes with $\N_2$. Then $\N_1\N_2$ is an internal central product of $\N_1$ and $\N_2$. 
\end{lemma}

\begin{proof}
By Theorem~\ref{T:ProductsPartialNormal}, the product $\N:=\N_1\N_2$ is a partial normal subgroup of $\L$. Hence, by Theorem~\ref{T:RegularPartialNormal}, $\N$ can be regarded as a regular locality. Notice that $\N_1$ commutes also with $\N_2$ in $\N$. Moreover, clearly $\N$ is centric in $\N$. Hence, it follows from Theorem~\ref{T:RegularN1timesN2}(c) applied with $\N$ in the role of $\L$ that $\N_1\N_2$ is an internal central product of $\N_1$ and $\N_2$.
\end{proof}

\chapter{Components of regular localities}\label{S:Components}

\textbf{Throughout let $(\L,\Delta,S)$ be a regular locality.}

\smallskip

In this chapter we will define and study components of $\L$ in a similar fashion as for finite groups. Recall that we introduced quasisimple localities in Definition~\ref{D:SimpleQuasisimple}. We will use without further reference that, by Corollary~\ref{C:RegularSubnormal}, every partial subnormal subgroup $\K$ of $\L$ can be regarded as a regular locality (or more precisely that $\F_{S\cap\K}(\K)$ is saturated and $(\K,\delta(\F_{S\cap\K}(\K)),S\cap\K)$ is a regular locality). In particular, the next definition makes sense.

\begin{definition}
A \emph{component} of $\L$ is a partial subnormal subgroup $\K$ of $\L$ such that $\K$ is quasisimple. The set of components of $\L$ will be denoted by $\Comp(\L)$. 
\end{definition}

It follows from Lemma~\ref{L:SubnormalinSubgroup} that we have the following Remark.

\begin{remark}\label{R:SubnormalinSubgroupComponents}
Let $\H_1\subn \H_2\subn\L$. Then 
\begin{itemize}
 \item [(a)] For every $\K\in\Comp(\H_2)$ with $\K\subseteq\H_1$, we have $\K\in\Comp(\H_1)$.
 \item [(b)] $\Comp(\H_1)\subseteq\Comp(\H_2)$.
\end{itemize}
\end{remark}

\begin{lemma}\label{L:QuasisimpleComponents}
If $\L$ is quasisimple, then $\Comp(\L)=\{\L\}$ and $F^*(\L)=\L$.
\end{lemma}

\begin{proof}
Clearly, we have $\L\in\Comp(\L)$. Let $\K$ be a component of $\L$ and assume $\K\neq\L$. As $\K$ is subnormal in $\L$, there exists then $\N\unlhd\L$ with $\N\neq\L$ and $\K\subseteq\N$. Lemma~\ref{L:Quasisimple} implies now $\K\subseteq\N\leq Z(\L)\leq S$ contradicting $\{\One\}\neq \K=O^p(\K)$.  

\smallskip

If $F^*(\L)\neq \L$, then Lemma~\ref{L:Quasisimple} implies $F^*(\L)\leq Z(\L)$. As $F^*(\L)$ is by Theorem~\ref{T:GeneralizedFitting} centric, it follows then from Corollary~\ref{C:RegularNReplete} and Theorem~\ref{T:RegularPartialNormal}(e) that $\L=C_\L(F^*(\L))=F^*(\L)^\perp\subseteq F^*(\L)$, a contradiction.
\end{proof}

\begin{lemma}\label{L:ComponentIntersectNormal}
Let $\K$ be a component of $\L$ and $\N\unlhd\L$. Then $\K\cap\N\leq Z(\K)$ or $\K\subseteq\N$. In particular, if $\K\cap S\subseteq\N$, then $\K\subseteq\N$.
\end{lemma}

\begin{proof}
As $\N\cap\K\unlhd\K$, this follows from Lemma~\ref{L:Quasisimple}. 
\end{proof}

The following lemma is similar to \cite[Lemma~8.4(b)]{ChermakIII}.

\begin{lemma}\label{L:ComponentCentralizesOp}
If $\K$ is a component of $\L$, then $\K\subseteq C_\L(O_p(\L))=O_p(\L)^\perp$.
\end{lemma}

\begin{proof}
As $\K=O^p(\K)\neq\{\One\}$, it follows from Lemma~\ref{L:ResidueEquivalences} that there exists $P\in\delta(\F_{S\cap\K}(\K))$ such that $D:=O^p(N_\K(P))\neq 1$. It follows from Corollary~\ref{C:RegularSubnormal} that $Q:=PC_S(\K)O_p(\L)\in\Delta$. For every $g\in\L$ and $x\in O_p(\L)$, we have $x\in O_p(\L)\leq S_g\cap S_{g^{-1}}$ and thus $(x^{-1},g^{-1},x,g)\in\D$ via $S_{g^{-1}}=S_g^g$. In particular, $D$ acts on $O_p(\L)$ via conjugation and 
\begin{equation}\label{E:commutator}
[x,g]=x^{-1}x^g=\Pi(x^{-1},g^{-1},x,g)=(g^{-1})^xg\mbox{ for all }g\in D,\;x\in O_p(\L).
\end{equation}
Let  $\K=\N_0\unlhd\N_1\unlhd\cdots\unlhd\N_k=\L$ be a subnormal series for $\K$. By induction on $i$ we will now verify that
\begin{equation}\label{E:commutatorInduction}
[O_p(\L),D]\subseteq \N_{k-i}\mbox{ for all }i=0,1,\dots,k.
\end{equation}
Namely, clearly $[O_p(\L),D]\subseteq \L=\N_k$. Moreover, if $[O_p(\L),D]\leq \N_{k-i}$ for some $0\leq i<k$, then using a property of coprime action, \eqref{E:commutator} and $D\subseteq\K\subseteq\N_{k-i-1}\unlhd\N_{k-i}$, it follows $[O_p(\L),D]=[O_p(\L),D,D]\subseteq [O_p(\L)\cap \N_{k-i},D]\subseteq \N_{k-i-1}$. So \eqref{E:commutatorInduction} holds. In particular, 
\[[O_p(\L),D]\subseteq \N_0=\K.\]
This yields $[O_p(\L),D]\subseteq\K\cap O_p(\L)\leq Z(\K)$, where the last equality uses  Lemma~\ref{L:ComponentIntersectNormal}. As $[Z(\K),D]=1$, it follows $[O_p(\L),D]=[O_p(\L),D,D]=1$. So $D\subseteq C_\L(O_p(\L))$. In particular, $\K\cap C_\L(O_p(\L))\not\leq Z(\K)=O_p(\K)$. By Lemma~\ref{L:Opperp} or by Theorem~\ref{T:RegularPartialNormal}(e), we have $C_\L(O_p(\L))=O_p(\L)^\perp\unlhd\L$. Hence, Lemma~\ref{L:ComponentIntersectNormal} allows us to conclude $\K\subseteq C_\L(O_p(\L))=O_p(\L)^\perp$.
\end{proof}

\begin{lemma}\label{L:ComponentsCharp}
The following are equivalent:
\begin{itemize}
\item [(i)] $\F$ is constrained;
\item [(ii)] $\L$ is a group of characteristic $p$;
\item [(iii)] $F^*(\L)=O_p(\L)$;
\item [(iv)] $\Comp(\L)=\emptyset$.
\end{itemize}
\end{lemma}

\begin{proof}
If $\L$ is a group of characteristic $p$, then $O_p(\L)\unlhd\F$ is centric in $\F$ and thus $\F$ is constrained. On the other hand, if $\F$ is constrained, then $\{1\}\in \F^s$ and thus $\{1\}\in\delta(\F)=\Delta$ by definition of these sets. Hence, $\L=N_\L(\{1\})$ is a group of characteristic $p$ by definition of a linking locality. This shows that (i) and (ii) are equivalent. 

\smallskip

If $\L$ is a group of characteristic $p$, then $O_p(\L)^\perp=C_\L(O_p(\L))\leq O_p(\L)$, i.e. $O_p(\L)$ is centric radical and thus equal to $F^*(\L)$. Hence (ii) implies (iii). 

\smallskip

Suppose now that (iii) holds, i.e. $F^*(\L)=O_p(\L)$. Then Theorem~\ref{T:GeneralizedFitting} implies $O_p(\L)^\perp\subseteq O_p(\L)$. Thus, by Lemma~\ref{L:ComponentCentralizesOp}, every component of $\L$ is contained in $O_p(\L)$. Hence, it follows from Lemma~\ref{L:Quasisimple} that $\Comp(\L)=\emptyset$. This shows that (iii) implies (iv).

\smallskip

We prove now that (iv) implies (ii) by contraposition. Suppose that $\L$ is not a group of characteristic $p$. By \cite[Proposition~5]{Henke:2015}, we have $O_p(\L)=O_p(\F)$, which implies $O_p(\L)\in\F^r$ (cf. Definition~\ref{D:Fcr}). Hence, if $C_\L(O_p(\L))\leq O_p(\L)$, then $O_p(\L)\in \F^{cr}$ and so $O_p(\L)\in\Delta$. Thus, by definition of a linking locality, $\L=N_\L(O_p(\L))$ is in this case a group of characteristic $p$ contradicting our assumption. So we have shown that $C_\L(O_p(\L))\not\leq O_p(\L)$. By Theorem~\ref{T:RegularPartialNormal}, we have $C_\L(O_p(\L))=O_p(\L)^\perp\unlhd\L$. So $C_\L(O_p(\L))$ cannot be a $p$-group by Lemma~\ref{L:OpL}. Hence, the set $\mathbb{X}$ of partial subnormal subgroups of $C_\L(O_p(\L))$ which are not $p$-groups is non-empty (since $C_\L(O_p(\L))\in\mathbb{X}$). Pick $\K\in\mathbb{X}$ minimal with respect to inclusion. Then $\K=O^p(\K)$ and, by the partial subgroup correspondence \cite[Proposition~4.7]{Chermak:2015},  $\K/O_p(\K)$ is simple. As $\K\subn C_\L(O_p(\L))\unlhd\L$, we have $\K\subn\L$. In particular, $O_p(\K)\leq O_p(\L)$ by Corollary~\ref{C:RegularSubnormal}. This implies $\K\subseteq C_\L(O_p(\L))\subseteq C_\L(O_p(\K))$ and thus $O_p(\K)=Z(\K)$. So $\K$ is a component of $\L$, which proves that $\Comp(\L)\neq\emptyset$. This proves that (iv) implies (ii). Hence the proof is complete. 
\end{proof}

If $\mathfrak{C}$ is a non-empty set of subgroups of $\L$, then we say that $\prod_{\K\in\mathfrak{C}}\K$ is well-defined if $\K_1\K_2\cdots \K_n=\K_{1\sigma}\K_{2\sigma}\cdots\K_{n\sigma}$ whenever $\K_1,\dots,\K_n$ are the pairwise distinct elements of $\mathfrak{C}$ and $\sigma\in\Sigma_n$; if so, then we set $\prod_{\K\in\mathfrak{C}}\K:=\K_1\K_2\cdots\K_n$. We define moreover $\prod_{\K\in\emptyset}\K:=\{\One\}$. Part of the following Proposition is equivalent to part of the statement in \cite[Theorem~8.5]{ChermakIII}.

\begin{prop}\label{P:FstarLComponents1}
If $\K$ is a component of $\L$, then $\K\unlhd F^*(\L)$. In particular, for every subset $\mathfrak{C}\subseteq\Comp(\L)$, the product $\prod_{\K\in\mathfrak{C}}\K$ is well-defined and a partial normal subgroup of $F^*(\L)$. Moreover,    
\[F^*(\L)=\left(\prod_{\K\in\Comp(\L)}\K\right)O_p(\L)=O_p(\L)\left(\prod_{\K\in\Comp(\L)}\K\right).\]
\end{prop}

\begin{proof}
Recall that $F^*(\L)$ is a regular locality. Hence, if every component of $\L$ is normal in $F^*(\L)$, it follows from \cite[Theorem~2]{Henke:2015a} that, for every $\mathfrak{C}\subseteq\Comp(\L)$, the product $\Pi_{\K\in\mathfrak{C}}\K$ is well-defined and a partial normal subgroup of $F^*(\L)$; moreover $\left(\prod_{\K\in\Comp(\L)}\K\right)O_p(\L)=O_p(\L)\left(\prod_{\K\in\Comp(\L)}\K\right)$. Thus, it remains to prove that every component of $\L$ is a partial normal subgroup of $F^*(\L)$ and that $F^*(\L)=\left(\prod_{\K\in\Comp(\L)}\K\right)O_p(\L)$. Let $\L$ be a minimal counterexample to that assertion. If $\Comp(\L)=\emptyset$, then $F^*(\L)=O_p(\L)$ by Lemma~\ref{L:ComponentsCharp} and $\L$ is thus not a counterexample. Hence, we have $\Comp(\L)\neq\emptyset$. Fix $\K\in\Comp(\L)$ and  choose $\N\unlhd\L$ minimal with $\K\subseteq\N$. 

\smallskip

Assume first that $\N=\L$. Then the minimality of $\N$ together with the fact that $\K$ is subnormal in $\L$ yields $\K=\L$. Therefore, Lemma~\ref{L:QuasisimpleComponents} yields $\Comp(\L)=\{\L\}$ and $F^*(\L)=\L$. This contradicts the assumption that $\L$ is a counterexample. We have thus $\N\neq \L$.

\smallskip

We show now first that $\K\unlhd F^*(\L)$. By Remark~\ref{R:SubnormalinSubgroupComponents}(a), $\K$ is component of $\N$. Hence, as $\L$ is a minimal counterexample, we have $\K\unlhd F^*(\N)$.  By Theorem~\ref{T:RegularPartialNormal}(d), we have $F^*(\L)=F^*(\N)F^*(\N^\perp)O_p(\L)$ and $F^*(\N)$, $F^*(\N^\perp)$ are partial normal subgroups of $\L$. Notice that $F^*(\N)$ commutes with $F^*(\N^\perp)$. Hence, $\K\unlhd F^*(\N)F^*(\N^\perp)$ by Lemma~\ref{L:Step4Seperate}. An element $f\in F^*(\L)$ can now be written as a product $f=ms$ with $m\in F^*(\N)F^*(\N^\perp)$ and $s\in O_p(\L)$. By Lemma~\ref{L:NLSbiset}, we have  $S_f=S_{(m,s)}$. If $k\in\K$ such that $(f^{-1},k,f)\in\D$, it follows that $u:=(s^{-1},m^{-1},k,m,s)\in\D$ via $S_{(f^{-1},k,f)}$ and $k^f=\Pi(u)=(k^m)^s$. Here $k^m\in\K$ as $\K\unlhd F^*(\N)F^*(\N^\perp)$. By Lemma~\ref{L:ComponentCentralizesOp} and Lemma~\ref{L:CentralizerHelp}, we have $s\in O_p(\L)\subseteq C_\L(\K)$. Hence, $k^f=k^m\in\K$ proving that $\K\unlhd F^*(\L)$.

\smallskip

As $\K$ was arbitrary, we have shown that every component of $\L$ is a partial normal subgroup of $F^*(\L)$. In particular,  Remark~\ref{R:SubnormalinSubgroupComponents} allows us to conclude that $\Comp(\L)=\Comp(F^*(\L))$. Notice also that $O_p(F^*(\L))=O_p(\L)$ by Theorem~\ref{T:RegularPartialNormal}(c). Hence, as $(\L,\Delta,S)$ is a minimal counterexample, it follows $\L=F^*(\L)$. In particular, $\K\unlhd\L$ and thus $\N=\K$. 

\smallskip

As $\K\neq Z(\K)$, we know that $\K$ is not contained in $\K^\perp=C_\L(\K)$ (cf. Theorem~\ref{T:RegularPartialNormal}(e)). In particular, $\K^\perp\neq\L$. As $\L$ is a minimal counterexample, it follows that $F^*(\K^\perp)$ is a product of the components of $\K^\perp$ and of $O_p(\K^\perp)$. To ease notation set $E(\K^\perp):=\prod_{\C\in\Comp(\K^\perp)}\C$ so that $F^*(\K^\perp)=E(\K^\perp)O_p(\K^\perp)$. By Lemma~\ref{L:Quasisimple}, we have $\K=F^*(\K)$. Hence, using  Theorem~\ref{T:RegularPartialNormal}(d), we conclude that 
\[F^*(\L)=\K F^*(\K^\perp)O_p(\L)=\K(E(\K^\perp)O_p(\K^\perp))O_p(\L).\]
By Theorem~\ref{T:RegularPartialNormal}(c) we have $O_p(\K^\perp)=O_p(\L)\cap\K^\perp\unlhd\L$. Therefore, all the factors in the latter product are partial normal subgroups of $\L=F^*(\L)$. Now \cite[Theorem~2(a)]{Henke:2015a} allows us to ``leave out brackets and put brackets'' as we wish. Hence,
\[\L=F^*(\L)=(\K E(\K^\perp))(O_p(\K^\perp)O_p(\L))=(\K E(\K^\perp))O_p(\L).\]
By Remark~\ref{R:SubnormalinSubgroupComponents}(b), we have $\{\K\}\cup\Comp(\K^\perp)\subseteq\Comp(\L)$. So it follows that  $\L=F^*(\L)\subseteq \left(\prod_{\C\in\Comp(\L)}\C\right)O_p(\L)$. As the other inclusion is trivial, $\L$ is not a counterexample contradicting our assumption.
\end{proof}

\begin{definition}\label{D:EL}
We set $E(\L):=\prod_{\K\in\Comp(\L)}\K$ and call $E(\L)$ the \emph{layer} of $\L$.
\end{definition}

Notice that $E(\L)$ is well-defined by Proposition~\ref{P:FstarLComponents1}. 

\begin{lemma}\label{L:LayerPartialNormalFstar}
The layer $E(\L)$ is a partial normal subgroup of $F^*(\L)$ with $F^*(\L)=E(\L)O_p(\L)=O_p(\L)E(\L)$ and $F^*(\L)\cap S=(E(\L)\cap S)O_p(\L)=O_p(\L)(E(\L)\cap S)$.
\end{lemma}

\begin{proof}
It was shown in Proposition~\ref{P:FstarLComponents1} that $E(\L)$ is a partial normal subgroup of $F^*(\L)$ with $F^*(\L)=E(\L)O_p(\L)=O_p(\L)E(\L)$. Hence, it follows from Theorem~\ref{T:ProductsPartialNormal} applied with $F^*(\L)$ in place of $\L$ that $F^*(\L)\cap S=(E(\L)\cap S)O_p(\L)=O_p(\L)(E(\L)\cap S)$.
\end{proof}

\begin{corollary}\label{C:ELcapSindeltaF}
Every subgroup of $S$ containing $E(\L)\cap S$ is an element of $\Delta=\delta(\F)$.  
\end{corollary}

\begin{proof}
Let $P\leq S$ with $E(\L)\cap S\leq P$. It follows from the last part of Lemma~\ref{L:LayerPartialNormalFstar} that $F^*(\L)\cap S\leq PO_p(\L)$ and thus $PO_p(\L)\in\delta(\F)$ by Corollary~\ref{C:FstarcapSindeltaF}. Therefore Lemma~\ref{L:deltaFtimesNormal} yields $P\in\delta(\F)$. 
\end{proof}

Notice that the statement of the following lemma makes sense by Proposition~\ref{P:FstarLComponents1}.

\begin{lemma}\label{L:OupperpEL}
If $\fC\subseteq\Comp(\L)$ and $\H:=\prod_{\K\in\fC}\K$, then $O^p(\H)=\H=O^{p^\prime}(\H)$. In particular, $O^p(F^*(\L))=O^p(E(\L))=E(\L)$.
\end{lemma}

\begin{proof}
By the definition of components and Lemma~\ref{L:Quasisimple}, we have $\K=O^p(\K)=O^{p^\prime}(\K)$. Hence, Lemma~\ref{L:ResidueSubnormalContainment} yields $\K\subseteq O^p(\H)\cap O^{p^\prime}(\H)$ for every $\K\in\fC$. As $O^p(\H)$ and $O^{p^\prime}(\H)$ are partial subgroups, $\H$ is thus contained in $O^p(\H)$ and $O^{p^\prime}(\H)$ proving $\H=O^p(\H)=O^{p^\prime}(\H)$. In particular, $E(\L)=O^p(E(\L))$. By Lemma~\ref{L:LayerPartialNormalFstar}, $E(\L)$ has $p$-power index in $F^*(\L)$. So $O^p(F^*(\L))\subseteq E(\L)=O^p(E(\L))\subseteq O^p(F^*(\L))$, where we use again Lemma~\ref{L:ResidueSubnormalContainment} for the last inclusion. This implies the assertion. 
\end{proof}

\begin{lemma}\label{L:AutComponents}
If $(\tL,\tDelta,\tS)$ is a regular locality and $\alpha\in\Iso(\L,\tL)$, then $E(\L)\alpha=E(\tL)$ and $\alpha$ induces a bijection 
\[\Comp(\L)\rightarrow\Comp(\tL),\K\mapsto\K\alpha.\] 
In particular, the automorphism group $\Aut(\L)$ acts on the set of components of $\L$ and stabilizes $E(\L)$.
\end{lemma}

\begin{proof}
Fix $\K\in\Comp(\L)$. A subnormal series for $\K$ in $\L$ gets mapped to a subnormal series for $\K\alpha$ in $\tL$, so $\K\alpha$ is subnormal in $\tL$. Notice also that $\alpha$ induces an isomorphism from $\K$ to $\K\alpha$. Hence, by Lemma~\ref{L:AutResidue} and Lemma~\ref{L:ResidueRegular}, we have $\K\alpha=O^p(\K)\alpha=O^p_\K(\K)\alpha=O^p_{\K\alpha}(\K\alpha)=O^p(\K\alpha)$. Observe also that $Z(\K)\alpha=Z(\K\alpha)$. Hence, if $\pi\colon \K\alpha\rightarrow \K\alpha/Z(\K\alpha)$ is the natural projection and $\beta=(\alpha|_\K)\pi$, then $\beta\colon\K\rightarrow \K\alpha/Z(\K\alpha)$ is a projection with kernel $\ker(\beta)=Z(\K)$. Thus, by \cite[Theorem~4.6]{Chermak:2015}, there exists an isomorphism from $\K/Z(\K)$ to $\K\alpha/Z(\K\alpha)$. Hence, $\K\alpha/Z(\K\alpha)$ is simple and $\K\alpha$ is quasisimple. This proves $\K\alpha\in\Comp(\tL)$. So the map $F\colon\Comp(\L)\rightarrow\Comp(\tL),\K\mapsto\K\alpha$ is well-defined. A symmetric argument shows that $G\colon\Comp(\tL)\rightarrow\Comp(\L),\K\mapsto\K\alpha^{-1}$ is well-defined. Now $F$ and $G$ are mutually inverse maps and $F$ must thus be bijective. In particular, $E(\L)\alpha=E(\tL)$ and the assertion follows. 
\end{proof}

\begin{lemma}\label{L:ELnormal}
We have $E(\L)\unlhd \L$ and, more generally, $E(\N)\unlhd\L$ for every $\N\unlhd\L$. Moreover, $N_\L(F^*(\L)\cap S)$ is a finite group inducing automorphisms of $F^*(\L)$ and acting on the set of components of $\L$ by conjugation. 
\end{lemma}

\begin{proof}
By Remark~\ref{R:SubnormalinSubgroupComponents} and Proposition~\ref{P:FstarLComponents1}, we have  $\Comp(\L)=\Comp(F^*(\L))$. Recall that $F^*(\L)$ is centric by Theorem~\ref{T:GeneralizedFitting}, so Lemma~\ref{L:RegularConjAutomorphisms}(c) yields that every element of $N_\L(F^*(\L)\cap S)$ induces an automorphism of $F^*(\L)$ by conjugation. Hence, by Lemma~\ref{L:AutComponents}, $N_\L(F^*(\L)\cap S)$ acts on $\Comp(F^*(\L))=\Comp(\L)$. If $\N\unlhd\L$, then $\Comp(\N)\subseteq\Comp(\L)$, so $E(\N)\unlhd F^*(\L)$ by Proposition~\ref{P:FstarLComponents1}. Notice that $N_\L(F^*(\L)\cap S)$ acts on $\Comp(\N)$ and thus on $E(\N)$ as $\N\unlhd\L$. Hence, $E(\N)\unlhd\L$ by \cite[Corollary~3.13]{Chermak:2015}. In particular $E(\L)\unlhd\L$.
\end{proof}

\begin{lemma}\label{L:CompOupperP}
Let $\N$ be a partial normal subgroup of $\L$ of $p$-power index or of index prime to $p$ in $\L$. Then $\Comp(\L)=\Comp(\N)$ and $E(\L)=E(\N)$.
\end{lemma}

\begin{proof}
Let $\K$ be a component of $\L$. If $\N$ is of index prime to $p$, then $S\cap \K\subseteq\N$ and thus $\K\subseteq\N$ by Lemma~\ref{L:ComponentIntersectNormal}. If $\N$ has $p$-power index, then $\K=O^p(\K)\subseteq O^p(\L)\subseteq \N$ by Lemma~\ref{L:ResidueEquivalences} and Lemma~\ref{L:ResidueSubnormalContainment}. Hence the assertion follows from  Remark~\ref{R:SubnormalinSubgroupComponents}(b).
\end{proof}

\begin{lemma}\label{C:ComponentsCentralProduct}
Let $\N_1$ and $\N_2$ be partial normal subgroups of $\L$. Suppose $\L$ is an internal central product of $\N_1$ and $\N_2$ or suppose that $\L=\N_1\N_2$ and $\N_1$ commutes with $\N_2$. Then $\Comp(\L)$ is a disjoint union of $\Comp(\N_1)$ and $\Comp(\N_2)$. In particular, $E(\L)=E(\N_1)E(\N_2)$.
\end{lemma}

\begin{proof}
By Theorem~\ref{T:RegularPartialNormal}(e), we have $\N_i^\perp=C_\L(\N_i)$ for $i=1,2$. If $\L$ is an internal central product of $\N_1$ and $\N_2$, then  Lemma~\ref{L:CentralProductsFactorsCentralize} gives $\N_1\subseteq C_\L(\N_2)=\N_2^\perp$. So we may just assume that 
\[\L=\N_1\N_2\mbox{ and $\N_1$ commutes with $\N_2$.}\]
Using Lemma~\ref{L:CNTinZT}, one sees now that $\N_1\cap\N_2\leq Z(\N_2)\leq Z(S\cap \N_2)$ does not contain any component. Hence,   $\Comp(\N_1)\cap\Comp(\N_2)=\emptyset$.

\smallskip

It remains to show that $\Comp(\L)=\Comp(\N_1)\cup\Comp(\N_2)$. Combining Remark~\ref{R:SubnormalinSubgroupComponents} with Proposition~\ref{P:FstarLComponents1}, one checks that $\Comp(\L)=\Comp(F^*(\L))$ and $\Comp(\N_i)=\Comp(F^*(\N_i))$ for $i=1,2$. Moreover, by Lemma~\ref{L:Step4Seperate} and Theorem~\ref{T:RegularN1timesN2}(b), $F^*(\N_1)$ and $F^*(\N_2)$ are partial normal subgroups of $\L$ with  $F^*(\L)=F^*(\N_1)F^*(\N_2)$. Notice that $F^*(\N_1)$ and $F^*(\N_2)$ commute. Hence, replacing $(\L,\N_1,\N_2)$ by $(F^*(\L),F^*(\N_1),F^*(\N_2))$, we may assume $\L=F^*(\L)$. 

\smallskip

Fix $\K\in\Comp(\L)$. By Remark~\ref{R:SubnormalinSubgroupComponents} it is enough to show that $\K$ is contained in $\N_1$ or $\N_2$. Assume $\K\not\subseteq \N_2$. Since $\L=F^*(\L)$, we have $\K\unlhd\L$ by Proposition~\ref{P:FstarLComponents1}. By Lemma~\ref{L:ComponentIntersectNormal} and Lemma~\ref{L:RegularNReplete}, we have $\K\cap\N_2\leq Z(\K)=Z^\circ(\K)$. Hence, by Lemma~\ref{L:OpNNPinMperp}, we have $\K=O^p(\K)\subseteq \N_2^\perp$. As stated above, we have $\N_1\subseteq \N_2^\perp=C_\L(\N_2)$. Using $\L=\N_1\N_2$ and the Dedekind Lemma \cite[Lemma~1.10]{Chermak:2015}, we can conclude that  $\N_2^\perp=C_\L(\N_2)=\N_1 Z(\N_2)$. Observe that $Z(\N_2)\leq S$ (e.g. by Lemma~\ref{L:CNTinZT}) and so $\N_1$ is a partial normal subgroup of $\N_2^\perp$ of $p$-power index. Hence, by Lemma~\ref{L:CompOupperP}, we have $\K\in\Comp(\N_2^\perp)=\Comp(\N_1)$. This proves the assertion.
\end{proof}

\begin{lemma}\label{L:ComponentinNorNperp}
Let $\N\unlhd\L$. Then $\Comp(\L)$ is a disjoint union of $\Comp(\N)$ and $\Comp(\N^\perp)$. In particular, $E(\L)=E(\N)E(\N^\perp)$.  
\end{lemma}

\begin{proof}
By Lemma~\ref{L:CentricInherit}(c), $\N\N^\perp O_p(\L)$ is centric radical and contains thus $F^*(\L)$. Hence, $\Comp(\L)=\Comp(\N\N^\perp O_p(\L))$ by Remark~\ref{R:SubnormalinSubgroupComponents} and Proposition~\ref{P:FstarLComponents1}. Moreover, $\N\N^\perp$ is a partial normal subgroup of $\N\N^\perp O_p(\L)$ of $p$-power index. Hence, $\Comp(\L)=\Comp(\N\N^\perp O_p(\L))=\Comp(\N\N^\perp)$ by Lemma~\ref{L:CompOupperP}. Now the assertion follows from Lemma~\ref{C:ComponentsCentralProduct} applied with $(\N\N^\perp,\N^\perp,\N)$ in place of $(\L,\N_1,\N_2)$. 
\end{proof}

\begin{lemma}\label{L:ComponentSubnormal}
 Let $\H$ be a partial subnormal subgroup of $\L$ and let $\K$ be a component of $\L$. Then either $\K$ is a component of $\H$ or the following hold:
\begin{itemize}
 \item $\K\subseteq C_\L(\H)$; and
 \item $\H\K=\K\H$ is a partial subnormal subgroup of $\L$;
 \item $\H$ and $\K$ are partial normal subgroups of $\H\K$; and
 \item $\H\K$ is an internal central product of $\H$ and $\K$.
\end{itemize}
\end{lemma}

\begin{proof}
Let $\L$ be a minimal counterexample. If $\L_0\subn\L$, then $\L_0$ is by Corollary~\ref{C:RegularSubnormal} itself a regular locality. Recall also from Lemma~\ref{L:SubnormalinSubgroup}(b) that every subnormal subgroup of $\L_0$ is subnormal in $\L$. Moreover, by Lemma~\ref{L:SubnormalinSubgroup}(a), $\H\subseteq\L_0$ implies $\H\subn\L_0$ and  $\K\subseteq\L_0$ implies  $\K\in\Comp(\L_0)$. As $\L$ is a minimal counterexample, we have thus the  following property:
\begin{equation}\label{E:Lminimal}
 \mbox{If }\L_0\subn\L\mbox{ with }\H\subseteq\L_0\neq\L,\mbox{ then }\K\not\subseteq\L_0.
\end{equation}
Let $\N\unlhd\L$ be minimal with $\H\subseteq\N$. If $\N=\L$, then because of the subnormality of $\H$ in $\L$, we have $\H=\L$. So $\K$ is in this case a component of $\H$ and we are done. Thus, we may assume $\N\neq\L$. Then $\K\not\in\Comp(\N)$ by \eqref{E:Lminimal}. Thus, Lemma~\ref{L:ComponentinNorNperp} implies that $\K$ is a component of $\N^\perp=C_\L(\N)$. As $\N$ commutes with $\N^\perp$ by Corollary~\ref{C:MperpNNperpM}, it follows now from Lemma~\ref{L:H1H2subN1N2} that $\H\K\subn\N\N^\perp\unlhd\L$. Hence \eqref{E:Lminimal} yields $\L=\H\K=\N\N^\perp$. The Dedekind Lemma \cite[Corollary~3.11]{Chermak:2015} gives thus $\N=\H(\N\cap \K)$ and $\N^\perp=(\H\cap \N^\perp)\K$, where $\N\cap\K$ and $\H\cap\N^\perp$ are both contained in $\N\cap\N^\perp\leq Z(\N)\cap Z(\N^\perp)\cap S$ by  Lemma~\ref{L:ZcircBasic} and Lemma~\ref{L:NcapNperp}. Hence, by Lemma~\ref{L:HZHelp}, we have $\H\unlhd\N$ and $\K\unlhd\N^\perp$. So Lemma~\ref{L:Step4Seperate} gives that $\H$ and $\K$ are partial normal subgroups of $\L=\N\N^\perp$. In particular, $\H=\N$. As $\H$ commutes with $\K\subseteq\N^\perp=\H^\perp=C_\L(\H)$ (and since $\H\K=\L$ is clearly centric in $\L$), it follows now from Theorem~\ref{T:RegularN1timesN2}(c) that $\H\K=\K\H$ is an internal central product of $\H$ and $\K$. This proves the assertion.    
\end{proof}

\begin{theorem}\label{T:FstarCentralProduct}
The following hold:
\begin{itemize}
\item[(a)] The layer $E(\L)$ is a central product of the components of $\L$ and $E(\L)=O^p(E(\L))=O^p(F^*(\L))$.
\item[(b)] The generalized Fitting subgroup $F^*(\L)$ is a central product of $E(\L)$ and $O_p(\L)$. In particular, $F^*(\L)$ is a central product of the components of $\L$ and of $O_p(\L)$. 
\item[(c)] Let $\H\subn\L$ and $\mathfrak{C}\subseteq\Comp(\L)\backslash\Comp(\H)$. Then $\M:=\H\left(\prod_{\K\in\mathfrak{C}}\K\right)\subn\L$ is a central product of the elements of $\{\H\}\cup\mathfrak{C}$, and $\Comp(\M)=\Comp(\H)\cup\mathfrak{C}$.
\item[(d)] If $\emptyset\neq \mathfrak{C}\subseteq\Comp(\L)$, then $\M':=\prod_{\K\in\mathfrak{C}}\K\unlhd F^*(\L)$ is a central product of the elements of $\mathfrak{C}$ and $\Comp(\M')=\mathfrak{C}$. Moreover, $O^p(\M')=\M'=O^{p^\prime}(\M')$.
\item[(e)] If $\K,\K^*\in\Comp(\L)$ with $\K\neq\K^*$, then $\K^*\subseteq\K^\perp_{F^*(\L)}=C_{F^*(\L)}(\K)$. In particular, for all $k\in \K$ and $c\in \K^*$, we have $(k,c)\in\D$, $kc=ck$ and 
\[S_{kc}\cap F^*(\L)=S_{(k,c)}\cap F^*(\L)=S_{(c,k)}\cap F^*(\L)\in\Delta=\delta(\F).\]
\end{itemize}
\end{theorem}

\begin{proof}
As $O_p(\L)^\perp$ is a partial subgroup, it follows from Lemma~\ref{L:ComponentCentralizesOp} that $E(\L)\subseteq O_p(\L)^\perp$, so $E(\L)$ and $O_p(\L)$ are partial normal subgroups which commute with each other and whose product $F^*(\L)$ is centric (cf. Theorem~\ref{T:GeneralizedFitting}, Corollary~\ref{C:MperpNNperpM}, Lemma~\ref{L:LayerPartialNormalFstar} and Lemma~\ref{L:ELnormal}). Hence, by Theorem~\ref{T:RegularN1timesN2}(c), $F^*(\L)$ is an internal central product of $E(\L)$ and $O_p(\L)$. Using  Lemma~\ref{L:CentralProductAssociative}, one observes now that part (b) holds if part (a) holds. Moreover, (a) follows from (d) and Lemma~\ref{L:OupperpEL}, and (d) follows from Proposition~\ref{P:FstarLComponents1}, Lemma~\ref{L:OupperpEL} and (c) applied with $\H=\{\One\}$ \footnote{Alternatively, using Lemma~\ref{L:ComponentinNorNperp}, one could show (a) and (b) by  more direct arguments similarly as in the proof of Proposition~\ref{P:FstarLComponents1}}. Hence, it remains only to prove (c) and (e). 

\smallskip

We show (c) by induction on $|\mathfrak{C}|$. The assertion is trivial for $\mathfrak{C}=\emptyset$. Assume now that $\mathfrak{C}\neq\emptyset$, fix $\C\in\mathfrak{C}$ and set $\mathfrak{D}:=\mathfrak{C}\backslash\{\C\}$. By Lemma~\ref{L:ComponentSubnormal}, $\hat{\H}:=\H\C$ is subnormal in $\L$ and a central product of $\H$ and $\C$. Moreover, $\H$ and $\C$ are partial normal subgroups of $\hat{\H}$. In particular, by Corollary~\ref{C:ComponentsCentralProduct} and Lemma~\ref{L:QuasisimpleComponents}, we have  $\Comp(\hat{\H})=\Comp(\H)\cup\Comp(\C)=\Comp(\H)\cup\{\C\}$ and so $\mathfrak{D}\subseteq\Comp(\L)\backslash\Comp(\hat{\H})$. Thus, by induction we may assume that the assertion is true with $(\hat{\H},\mathfrak{D})$ in place of $(\H,\mathfrak{C})$. Using Lemma~\ref{L:CentralProductAssociative} this allows one to verify the assertion also for $(\H,\mathfrak{C})$. Hence (c) holds.

\smallskip

It remains to prove (e). Recall that $\K,\K^*\unlhd F^*(\L)$ by Proposition~\ref{P:FstarLComponents1}. Moreover, by Theorem~\ref{T:RegularPartialNormal}(a),(e), $F^*(\L)$ is a regular locality and $\K^\perp_{F^*(\L)}=C_{F^*(\L)}(\K)$. Lemma~\ref{L:QuasisimpleComponents} gives $\Comp(\K)=\{\K\}$. Thus, Lemma~\ref{L:ComponentinNorNperp} applied with $(\K^*,\K,F^*(\L))$ in place of $(\K,\N,\L)$ gives $\K^*\subseteq \K^\perp_{F^*(\L)}=C_{F^*(\L)}(\K)$. In particular, for all $k\in\K$ and $c\in\K^*$, we have $(k,c)\in\D$ and  $kc=ck$. As $S_{kc}\in\Delta=\delta(\F)$, it follows from Lemma~\ref{L:deltaFDefinition} that $S_{kc}\cap F^*(\L)\in\Delta$. Moreover, by Theorem~\ref{T:mainNperp}(d) applied with $F^*(\L)$ in place of $\L$, we have $S_{kc}\cap F^*(\L)=S_{(k,c)}\cap F^*(\L)=S_{(c,k)}\cap F^*(\L)$. This proves (e).
\end{proof}

\begin{proof}[Proof of Theorem~\ref{T:mainComponents}]
Part (a) follows from Proposition~\ref{P:FstarLComponents1} and Theorem~\ref{T:FstarCentralProduct}(a),(b), while part (b) was shown in Lemma~\ref{L:ELnormal} and Lemma~\ref{L:ComponentCentralizesOp}. Part (c) is true by Theorem~\ref{T:FstarCentralProduct}(b); see also Lemma~\ref{L:LayerPartialNormalFstar}. For (d) see Lemma~\ref{L:ComponentSubnormal} and for (e) see Lemma~\ref{L:ComponentsCharp}.
\end{proof}

In preparation for the next chapter we prove the following result. 

\begin{lemma}\label{L:IsoFusionSystemsELcapS}
Let $(\tL,\tDelta,\tS)$ be a regular locality over $\tF$ and $\alpha\colon \F\rightarrow \tF$ an isomorphism. Then $\delta(\F)\alpha=\delta(\tF)$ and $(E(\L)\cap S)\alpha=E(\tL)\cap \tS$. 
\end{lemma}

\begin{proof}
Let $(\L^s,\F^s,S)$ and $(\tL^s,\tF^s,S)$ be subcentric localities over $\F$ such that $\L^s|_{\delta(\F)}=\L$ and $\tL^s|_{\delta(\tF)}=\tL$; notice that $\L^s$ and $\tL^s$ exist by Theorem~\ref{T:VaryObjects}(c). Then by \cite[Proposition~3.19]{Henke:2016} (which uses the uniqueness of a linking system associated to given saturated fusion system), $\alpha\colon S\rightarrow\tS$ extends to an isomorphism $\beta\colon\L^s\rightarrow\tL^s$. Hence, it follows from Lemma~\ref{L:IsoSubcentricIsoRegular} that $\delta(\F)\alpha=\delta(\F)\beta=\delta(\tF)$ and $\beta|_\L$ is an isomorphism from $\L$ to $\tL$.  Lemma~\ref{L:AutComponents} gives now $E(\L)\beta=E(\tL)$ and thus $(E(\L)\cap S)\alpha=(E(\L)\cap S)\beta =E(\tL)\cap S$.  
\end{proof}

\chapter{Layers and E-balance}

In this chapter we will first introduce layers of arbitrary linking localities and show some properties of these. Then we will prove the theorem on E-balance (Theorem~\ref{T:EbalanceMain}) stated in the Introduction.

\section{Layers of linking localities}

If $(\L,\Delta,S)$ is a regular locality, we defined the layer $E(\L)$ in the previous chapter and showed in Lemma~\ref{L:ResidueRegular} and Lemma~\ref{L:OupperpEL} that $E(\L)=O^p(F^*(\L))=O^p_\L(F^*(\L))$. We will use the latter property implicitly throughout. It implies that the following definition does not conflict with the earlier introduced definition of $E(\L)$ in the case of regular localities.

\begin{definition}\label{D:ELGeneral}
If $(\L,\Delta,S)$ is a linking locality, then call $E(\L):=O^p_\L(F^*(\L))$ the \emph{layer} of $\L$.
\end{definition}

\begin{lemma}\label{L:TakePlusLayer}
 Let $(\L^+,\Delta^+,S)$ and $(\L,\Delta,S)$ be linking localities over $\F$ with $\Delta\subseteq\Delta^+$ and $\L=\L^+|_\Delta$. Adapt Notation~\ref{N:VaryObjects}. Then $E(\L)^+=E(\L^+)$.
\end{lemma}

\begin{proof}
 This follows from Lemma~\ref{L:OupperpVaryObjects} and Lemma~\ref{L:TakePlusGeneralizedFitting}.
\end{proof}

The following proposition will not be used in the proof of the E-balance theorem. However, it gives some evidence that layers of arbitrary linking localities are natural to consider.

\begin{prop}
 Let $(\L,\Delta,S)$ be a linking locality. Then 
\[F^*(\L)=E(\L)O_p(\L)=O_p(\L)E(\L)\mbox{ and }E(\L)\subseteq O_p(\L)^\perp=C_\L(O_p(\L)).\]  
\end{prop}

\begin{proof}
By Lemma~\ref{L:ProductPartialSubgroup} or by Theorem~\ref{T:ProductsPartialNormal}, we have $E(\L)O_p(\L)=O_p(\L)E(\L)$. Moreover, Lemma~\ref{L:Opperp} gives $O_p(\L)^\perp_\L=C_\L(O_p(\L))$. Hence, it is sufficient to prove $F^*(\L)=O_p(\L)E(\L)$ and $E(\L)\subseteq O_p(\L)^\perp$.

\smallskip

By Theorem~\ref{T:VaryObjects}(a), there exists a subcentric locality $(\L^s,\F^s,S)$ over $\F:=\F_S(\L)$ such that $\L^s|_\Delta=\L$. It follows then from Lemma~\ref{L:deltaFBasic} that $\L_\delta=\L^s|_{\delta(\F)}$ is a regular locality over $\F$. By Lemma~\ref{L:ComponentCentralizesOp} and Lemma~\ref{L:LayerPartialNormalFstar}, we have 
\[F^*(\L_\delta)=O_p(\L_\delta)E(\L_\delta)\mbox{ and }E(\L_\delta)\subseteq O_p(\L_\delta)^\perp_{\L_\delta}.\]
Hence, it is enough to show the following properties for a linking locality $(\L^+,\Delta^+,S)$  over $\F$ with $\L^+|_\Delta=\L$:
\begin{eqnarray}
F^*(\L^+)=O_p(\L^+)E(\L^+) &\Longleftrightarrow& F^*(\L)=O_p(\L)E(\L) \label{E:Equivalence1}\\ 
E(\L^+)\subseteq O_p(\L^+)^\perp_{\L^+} &\Longleftrightarrow& E(\L)\subseteq O_p(\L)^\perp_\L.\label{E:Equivalence2}
\end{eqnarray}
The required result follows then by applying these equivalences twice, once with $(\L_\delta,\L^s)$ and once with $(\L,\L^s)$ in place of $(\L,\L^+)$.

\smallskip

To prove \eqref{E:Equivalence1} and \eqref{E:Equivalence2}, adapt Notation~\ref{N:VaryObjects} and observe first that Lemma~\ref{L:TakePlusInS}, Lemma~\ref{L:TakePlusGeneralizedFitting} and Lemma~\ref{L:TakePlusLayer} give 
\[O_p(\L)=O_p(\L^+)=O_p(\L)^+,\;F^*(\L^+)=F^*(\L)^+\mbox{ and }E(\L^+)=E(\L)^+.\]
The latter property says $E(\L^+)\cap\L=E(\L)$. Hence, by Lemma~\ref{L:RestrictionIntersectL}, we have  
 \[O_p(\L)E(\L^+)\cap \L=O_p(\L)(E(\L^+)\cap\L)=O_p(\L)E(\L).\] Together with $F^*(\L^+)=F^*(\L)^+$, this implies \eqref{E:Equivalence1}. By Lemma~\ref{L:NperpPlus}, we have $(O_p(\L)^\perp_\L)^+=(O_p(\L)^+)^\perp_{\L^+}=O_p(\L^+)^\perp_{\L^+}$. Since the map $\fN(\L^+)\rightarrow \fN(\L),\N^+\mapsto \N^+\cap\L$ and its inverse are inclusion-preserving by Theorem~\ref{T:VaryObjects}(b), this shows \eqref{E:Equivalence2}.
\end{proof}

\section{The setup for the E-balance theorem} We will now start working towards the proof of the E-balance theorem (Theorem~\ref{T:EbalanceMain}). Our proof will be similar to Chermak's proof of \cite[Theorem~9.9]{ChermakIII}; however, we use different arguments to show the preliminary Lemma~\ref{L:ProdCompinbNLD}. Note that formulating any form of an E-balance theorem will require us to look at linking localities over $p$-local subsystems. It seems to us that this is best done in the following setting where, for any $X\leq S$ fully $\F$-normalized, we will be able to consider the im-partial subgroup $\bN_\L(X)$ introduced in Section~\ref{SS:Impartial}. 

\bigskip

\textbf{For the remainder of this chapter let $(\L,\Delta,S)$ be a subcentric locality over a fusion system $\F$. Set
\[\L_\delta:=\L|_{\delta(\F)}.\]}
Observe that $(\L_\delta,\delta(\F),S)$ is a regular locality over $\F$. Thus, according to Definition~\ref{D:EL},  $E(\L_\delta)$ is the product of the components of $\L_\delta$. Moreover, $E(\L_\delta)\unlhd\L_\delta$ by Lemma~\ref{L:ELnormal}. Recall also that $E(\L)$ was defined in Definition~\ref{D:ELGeneral}.

\begin{remark}\label{R:ELSubcentric}
By Lemma~\ref{L:TakePlusLayer}, $E(\L)$ is the unique partial normal subgroup of $\L$ with 
\[E(\L)\cap\L_\delta=E(\L_\delta).\]
This property could be used alternatively as the definition of $E(\L)$. 
\end{remark}

We will need the following lemma.

\begin{lemma}\label{L:ELgenELdelta}
$R_\Delta(E(\L)S)\subseteq \delta(\F)$ and $E(\L)=\<E(\L_\delta)\>_\L$. 
\end{lemma}

\begin{proof}
By Lemma~\ref{L:RegularNReplete}(b) we have $R_\Delta(E(\L)S)\subseteq\delta(\F)$. In particular, using Remark~\ref{R:ELSubcentric},  it follows that $N_{E(\L)}(P)=N_{E(\L)\cap \L_\delta}(P)=N_{E(\L_\delta)}(P)$ for every every $P\in R_\Delta(E(\L)S)$. The assertion follows now from Lemma~\ref{L:PartialNormalAlperin}.  
\end{proof}

\section{Localities attached to normalizers of fully $\F$-normalized subgroups} Recall that $(\L,\Delta,S)$ is a subcentric locality. In Section~\ref{SS:Impartial} we introduced the im-partial subgroup
\[\bN_\L(X)=N_\L(X)|_{N_\F(X)^s}\mbox{ for every fully $\F$-normalized subgroup $X\leq S$}.\]
Moreover, we saw that, for any such $X$, the triple $(\bN_\L(X),N_S(X),N_\F(X)^s)$ is a subcentric locality over $N_\F(X)$. Using the usual procedure of restriction, we can thus find a regular locality inside of $\bN_\L(X)$. More precisely, setting
\[\bN_\L^\delta(X)=\bN_\L(X)|_{\delta(N_\F(X))}\mbox{ for every fully $\F$-normalized subgroup $X\leq S$},\]
the triple $(\bN_\L^\delta(X),N_S(X),\delta(N_\F(X)))$ is a regular locality over $N_\F(X)$ for any such $X$. Using \cite[Lemma~9.7(b)]{Henke:2015}, one can actually see that $\bN_\L^\delta(X)=N_\L(X)|_{\delta(N_\F(X))}$. 

\bigskip

As $\bN_\L^\delta(X)$ can be regarded as a regular locality, we have in particular components of $\bN_\L^\delta(X)$  defined as well as the product $E(\bN_\L^\delta(X))$ of the components of $\bN_\L^\delta(X)$. Observe also that, according to Remark~\ref{R:ELSubcentric}, the layer $E(\bN_\L(X))$ is the unique partial normal subgroup of $\bN_\L(X)$ such that 
\[E(\bN_\L(X))\cap \bN_\L^\delta(X)=E(\bN_\L^\delta(X)).\]
So in particular, $E(\bN_\L^\delta(X))\subseteq E(\bN_\L(X))$. We will now show two lemmas which are needed for the proof of the E-balance Theorem, but seem also interesting enough on their own. 

\bigskip

In the next lemma, we will make the assumption that a partial normal subgroup $\N_\delta$ of $\L_\delta$ is an im-partial subgroup of $\bN_\L(X)$. By this we mean more precisely that, writing $\D_\delta$ for the domain of the product on $\L_\delta$ and setting $\D_0:=\D_\delta\cap \W(\N_\delta)$, the triple $(\N_\delta,\D_0,\Pi|_{\D_0})$ is an im-partial subgroup of $\bN_\L(X)$, where $\Pi$ denotes the product on $\L$. In other words, we regard $\N_\delta$ as a partial group as we would naturally do for any partial subgroup of $\L_\delta$, and we assume then that this is an im-partial subgroup of $\bN_\L(X)$. As one easily observes, this is equivalent to assuming that $\D_0$ is contained in the domain of the product on $\bN_\L(X)$.

\begin{lemma}\label{L:PartialNormalinbNLX}
Let $X\leq S$ be fully $\F$-normalized and $\N_\delta\unlhd\L_\delta$ such that $\N_\delta$ is an im-partial subgroup of $\bN_\L(X)$. Then 
\[\N_\delta\unlhd\bN_\L^\delta(X),\] 
in particular $\N_\delta$ is a partial subgroup of $\bN_\L^\delta(X)$. Moreover, if $\N\unlhd\L$ with $\N\cap\L_\delta=\N_\delta$, then $\N_\delta=\N\cap\bN_\L^\delta(X)$. 
\end{lemma}

\begin{proof}
Set $T:=\N_\delta\cap S$. By Theorem~\ref{T:VaryObjects}(b), we can pick $\N\unlhd\L$ such that $\N\cap \L_\delta=\N_\delta$ and thus $\N\cap S=T$. Then $\F_T(\N_\delta)=\F_T(\N)$ by Lemma~\ref{L:RegularNReplete}(c). Notice that 
\[\N_\delta\subseteq \M:=\N\cap\bN_\L(X)\unlhd \bN_\L(X)\]
and $\N_\delta$ is an im-partial subgroup of $\M$. Moreover, $\M\cap S=T$ and $\F_T(\N_\delta)\subseteq \F_T(\M)\subseteq\F_T(\N)=\F_T(\N_\delta)$, so equality holds and $\F_T(\N_\delta)=\F_T(\N)=\F_T(\M)$. Observe that
\[\M_\delta:=\M\cap \bN_\L^\delta(X)=\N\cap\bN_\L^\delta(X)\unlhd\bN_\L^\delta(X)\]
and, by Lemma~\ref{L:RegularNReplete}(c) applied with $\bN_\L(X)$ in place of $\L$, we have $T=\M_\delta\cap S$ and $\F_T(\M_\delta)=\F_T(\M)=\F_T(\N_\delta)$. Now $\M_\delta$ and $\N_\delta$ are both regular localities over the same fusion system and thus isomorphic by \cite[Theorem~A(a)]{Henke:2015} (cf. Remark~\ref{R:ExistenceUniquenessCLS}). Moreover, by Theorem~\ref{T:RegularPartialNormal}, for every $P\in\delta(\F_T(\M_\delta))=\delta(\F_T(\N_\delta))$, we have $PC_S(\N_\delta)\in\delta(\F)$. Observe that $C_S(\N_\delta)=C_S^\circ(\N_\delta)=C_S(\N)\subseteq C_S(\M_\delta)$, where the first equality uses Corollary~\ref{C:RegularNReplete}, the second equality uses the definition of $C_S^\circ(\N_\delta)$, and the inclusion at the end uses $\M_\delta\subseteq\N$. As $(\M_\delta,\delta(\F_T(\M_\delta)),T)$ is a locality, this yields that $\M_\delta$ is an im-partial subgroup of $\L_\delta$ and hence 
\[\M_\delta\subseteq \L_\delta\cap \N=\N_\delta\]
is an im-partial subgroup of $\N_\delta$. Since $\M_\delta$ and $\N_\delta$ are isomorphic, it follows that $\N_\delta=\M_\delta\unlhd\bN_\L^\delta(X)$. This proves the assertion.
\end{proof}

The following lemma is inspired by \cite[Lemma~9.8]{ChermakIII} and plays a similar role in the proof of E-balance below.

\begin{lemma}\label{L:ProdCompinbNLD}
Let $\mathfrak{C}\subseteq\Comp(\L_\delta)$ such that $\H:=\prod_{\K\in\mathfrak{C}}\K\unlhd\L_\delta$. Suppose there exists $X\leq S$ fully $\F$-normalized such that $\H\subseteq N_\L(X)$ and  $E(\H^\perp_{\L_\delta})\cap S\leq N_S(X)$. Then \[\H\unlhd\bN_\L^\delta(X),\;\mathfrak{C}\subseteq\Comp(\bN^\delta_\L(X))\mbox{ and }\H\subseteq E(\bN_\L^\delta(X)).\] 
\end{lemma}

\begin{proof}
We will write $\H^\perp$ for $\H^\perp_{\L_\delta}$. Set $R:=E(\H^\perp)\cap S$. We use throughout that $O_p(\L_\delta)=O_p(\L)$ by Lemma~\ref{L:TakePlusInS}.

\smallskip

\emph{Step~1:} We show that $PR\in N_\F(X)^s$ for every $P\in\delta(\F_{S\cap\H}(\H))$.  By Lemma~\ref{L:ComponentinNorNperp} and Lemma~\ref{L:LayerPartialNormalFstar}, we have $E(\L_\delta)=E(\H)E(\H^\perp)$ and $F^*(\L_\delta)=E(\L_\delta)O_p(\L)=(E(\H)E(\H^\perp))O_p(\L)$. Thus, \cite[Theorem~2(a)]{Henke:2015a} allows to write 
\[F^*(\L)=E(\H)(E(\H^\perp)O_p(\L)).\]
Lemma~\ref{L:ComponentCentralizesOp} implies that $\H\subseteq O_p(\L)^\perp_{\L_\delta}=C_{\L_\delta}(O_p(\L))$. Hence, Corollary~\ref{C:NinMperp} gives $O_p(\L)\subseteq \H^\perp$. By Lemma~\ref{L:ELnormal}, $E(\H)$ and $E(\H^\perp)$ are partial normal subgroups of $\L$. In particular, by Theorem~\ref{T:ProductsPartialNormal}, we have $E(\H^\perp)O_p(\L)\unlhd\L$ and $E(\H^\perp)O_p(\L)\cap S=RO_p(\L)$. Moreover, $\H$ commutes with $E(\H^\perp)O_p(\L)\subseteq\H^\perp$ by Corollary~\ref{C:MperpNNperpM}. Recall furthermore that $F^*(\L)$ is centric in $\L$ by Theorem~\ref{T:GeneralizedFitting}. Thus, fixing $P\in\delta(\F_{S\cap\H}(\H))$, it follows from  Theorem~\ref{T:RegularN1timesN2}(d) that $PRO_p(\L)\in\delta(\F)$. Hence, by Lemma~\ref{L:deltaFtimesNormal}, we have $PR\in\delta(\F)\subseteq\F^s$. As $PR\leq (\H\cap S)R\leq N_S(X)$ by assumption, it follows from Lemma~\ref{L:NFXs} that $PR\in N_\F(X)^s$. 
 
\smallskip

\emph{Step~2:} We argue that the assertion holds. Corollary~\ref{C:NperpCapSGeneral} gives $\H\subseteq C_{\L_\delta}(R)\subseteq C_\L(R)$. Moreover, by assumption $\H\subseteq N_\L(X)$. Hence, it follows from Step~1, the definition of $\mathbb{N}_\L(X)$ and the fact that $(\H,\delta(\F_{S\cap\H}(\H)),S\cap\H)$ is a locality that $\H$ is an im-partial subgroup of $\bN_\L(X)$. Therefore, Lemma~\ref{L:PartialNormalinbNLX} yields $\H\unlhd\bN_\L^\delta(X)$. So by Remark~\ref{R:SubnormalinSubgroupComponents}, we have  $\mathfrak{C}\subseteq\Comp(\H)\subseteq\Comp(\bN_\L^\delta(X))$ and thus $\H=\prod_{\K\in\mathfrak{C}}\K\subseteq E(\bN_\L^\delta(X))$. 
\end{proof}

\section{The proof of E-Balance} We are now in a position to prove the E-balance Theorem stated in the introduction as Theorem~\ref{T:EbalanceMain}. The following theorem is essentially a restatement of that theorem.

\begin{theorem}[$E$-balance]\label{T:Ebalance}
 Let $X\leq S$ be fully $\F$-normalized. Then 
\[E(\bN_\L^\delta(X))\subseteq E(\bN_\L(X))\subseteq E(\L).\]
\end{theorem}

\begin{proof}
Let $(\L,X)$ be a counterexample such that first $|\L|$ is as small as possible and then $|X|$ is as large as possible. For every $U\leq S$ which is fully $\F$-normalized set 
\[\L_U:=\bN_\L(U)\mbox{ and }\L_U^\delta:=\bN_\L^\delta(U).\]
By Lemma~\ref{L:ELgenELdelta}, we have $E(\L_X)=\<E(\L^\delta_X)\>_{\L_X}$. Hence, as $(\L,X)$ is a counterexample, we have  $E(\L_X^\delta)\not\subseteq E(\L)$. As $E(\L_X^\delta)$ is the product of the components of $\L_X^\delta$ in $\L_X^\delta$ and since $\L_X^\delta$ is an im-partial subgroup of $\L$, it follows that 
\[\mathfrak{C}:=\{\K\in\Comp(\L_X^\delta)\colon \K\not\subseteq E(\L)\}\neq \emptyset.\]
Let $\H$ be the product in $\L_X^\delta$ of the elements of $\mathfrak{C}$ (which is well-defined by Proposition~\ref{P:FstarLComponents1}). Set moreover 
\[T:=E(\L)\cap S,\; A:=N_T(X)\mbox{ and }D:=XA.\] 
 
\smallskip

\emph{Step~1:} We argue that $(\L,Y)$ is a counterexample for every $Y\in X^\F$ which is fully $\F$-normalized. For the proof fix such $Y$. By \cite[Lemma~2.6(c)]{Aschbacher/Kessar/Oliver:2011}, there exists $\alpha\in\Hom_\F(N_S(X),S)$ such that $X\alpha=Y$. Then $N_S(X)\alpha\leq N_S(Y)$ and the assumption that $X$ is fully normalized yields $N_S(X)\alpha=N_S(Y)$. It is then easy to check that $\alpha$ induces an isomorphism from the fusion system $N_\F(X)$ to the fusion system $N_\F(Y)$. Thus, Lemma~\ref{L:IsoFusionSystemsELcapS} gives $(E(\L_X^\delta)\cap S)\alpha=E(\L_Y^\delta)\cap S$. Assume now $E(\L_Y^\delta)\subseteq E(\L)$. Since $T$ is strongly $\F$-closed, it follows then $E(\L_X^\delta)\cap S\leq T=E(\L)\cap S$. Hence, for $\K\in\fC$, we have $\K\cap S\leq E(\L)\cap \K$ and thus $\K\cap E(\L)\not\leq Z(\K)$. Thus, as $\K\cap E(\L)\unlhd\K$, Lemma~\ref{L:Quasisimple} yields $\K=\K\cap E(\L)\subseteq E(\L)$ contradicting the definition of $\fC$. So $E(\L_Y^\delta)\not\subseteq E(\L)$ and thus $E(\L_Y)\not\subseteq E(\L)$.

\smallskip

\emph{Step~2:} We show that we may choose $X$ such that $D$ is fully $\F$-normalized. For the proof let $\beta\in\Hom_\F(N_S(D),S)$ such that $D\beta$ is fully $\F$-normalized. Notice that $N_S(X)\subseteq N_S(D)$ as $T\unlhd S$. Hence, the assumption that $X$ is fully normalized implies also that $X\beta$ is fully normalized and $N_S(X)\beta=N_S(X\beta)$. By Step~1, we have in particular that $E(\L_{X\beta})\not\subseteq E(\L)$. As $T$ is strongly $\F$-closed, we have moreover $A\beta\leq N_T(X\beta)$. Similarly, as $N_S(X)\beta=N_S(X\beta)$, we have $N_T(X\beta)\beta^{-1}\leq N_T(X)=A$ and thus $N_T(X\beta)\leq A\beta$. Hence, $A\beta=N_T(X\beta)$ and $D\beta =(X\beta)N_T(X\beta)$. So replacing $X$ by $X\beta$ we may indeed choose $X$ such that $D$ is fully $\F$-normalized.

\smallskip

From now on, we assume without loss of generality that $D$ is fully $\F$-normalized. As $D\unlhd N_S(X)$, the subgroup $D$ is clearly also fully $N_\F(X)$-normalized. Hence, by Lemma~\ref{L:NinNLXofY}, $X$ is fully $N_\F(D)$-normalized.

\smallskip

\emph{Step~3:} We show $\H\subseteq E(\bN_{\L_X}^\delta(D))\subseteq E(\bN_{\L_X}(D))$. Appealing to Lemma~\ref{L:ProdCompinbNLD} applied with $(\L_X,D)$ in place of $(\L,X)$, it is sufficient to show that
\begin{equation}\label{E:Ebalance0}
\H\unlhd \L_X^\delta,\;\H\subseteq N_{\L_X}(D)\mbox{ and }E(\H^\perp_{\L_X^\delta})\cap S\leq A\leq D. 
\end{equation}
Note that $\M:=E(\L)\cap \L_X^\delta\unlhd \L_X^\delta$ and $\mathfrak{C}$ is precisely the set of components of $\L_X^\delta$ which are not in $\M$ or equivalently not components of $\M$ by Remark~\ref{R:SubnormalinSubgroupComponents}. Hence, Lemma~\ref{L:ComponentinNorNperp} yields $\mathfrak{C}=\Comp(\M^\perp_{\L_X^\delta})$ and $\H=E(\M^\perp_{\L_X^\delta})$. In particular, by Lemma~\ref{L:ELnormal}, we have
\[\H\unlhd \L_X^\delta.\]
By Theorem~\ref{T:RegularPartialNormal}(e), we have $\M^\perp_{\L_X^\delta}=C_{\L_X^\delta}(\M)$. So 
$\H\subseteq \M^\perp_{\L_X^\delta}=C_{\L_X^\delta}(\M)\subseteq C_{\L_X^\delta}(\M\cap S)$. Notice that $\M\cap S=E(\L)\cap N_S(X)=N_T(X)=A$. Since $X$ is normal in $\L_X^\delta$ and $D=AX$, it follows $\H\subseteq N_{\L_X^\delta}(D)\subseteq N_{\L_X}(D)$. Applying Lemma~\ref{L:ComponentinNorNperp} with $(\L_X^\delta,\H)$ in place of $(\L,\N)$, one sees that $E(\H^\perp_{\L_X^\delta})$ is the product of the components of $\L_X^\delta$ which are not components of $\H$. By Theorem~\ref{T:FstarCentralProduct}(d), the components of $\H$ are precisely the elements of $\fC$. Hence, $E(\H^\perp_{\L_X^\delta})$ is the product of the components of $\L_X^\delta$ which are not elements of $\fC$ and so (by definition of $\fC$) contained in $E(\L)$. Hence, $E(\H^\perp_{\L_X^\delta})\cap S\leq E(\L)\cap N_S(X)=A\leq D$. This proves \eqref{E:Ebalance0} and thus Step~3 is complete. 

\smallskip

\emph{Step~4:} We show that $\L\neq \L_D$. Assume otherwise. Then $X\leq D\leq O_p(\L)$. Hence, by Lemma~\ref{L:ComponentCentralizesOp}, we have $T\subseteq E(\L_\delta)\subseteq C_\L(O_p(\L))\subseteq C_\L(X)$, which implies $T=A\leq D$. Now Corollary~\ref{C:ELcapSindeltaF} gives $D\in\delta(\F)=\Delta\subseteq\F^s$. As $\L=\L_D=N_\L(D)$, we have $\F=N_\F(D)$. Hence, it follows from \cite[Lemma~3.1]{Henke:2015} that $\F$ is constrained and thus $N_\F(X)$ is constrained by \cite[Lemma~2.11]{Henke:2015}. Now by Lemma~\ref{L:ComponentsCharp}, we have $E(\L_X^\delta)=\{\One\}$ contradicting $\fC\neq\emptyset$.

\smallskip

\emph{Step~5:} We argue that $X<D$ and thus $E(\L_D)\subseteq E(\L)$. Because of the maximality of $|X|$, it is sufficient to prove $X<D$. If this is not the case, then $N_{TX}(X)=AX=D=X$. Since $TX$ is a $p$-group, it follows $TX=X\geq T=E(\L_\delta)\cap S$. Thus $X\in \delta(\F)\subseteq \F^s$ by Corollary~\ref{C:ELcapSindeltaF}. So $N_\F(X)$ is constrained by \cite[Lemma~3.1]{Henke:2015} and therefore Lemma~\ref{L:ComponentsCharp} gives  $E(\L^\delta_X)=\{\One\}$ contradicting $\fC\neq\emptyset$.

\smallskip

\emph{Step~6:} We derive the final contradiction. By Step~4, we have $\L\neq \L_D$. Thus, since $\L$ is a minimal counterexample, we have
\begin{equation}\label{E:Ebalance}
 E(\bN_{\L_D}(X))\subseteq E(\L_D). 
\end{equation}
We can see now that
\begin{eqnarray*}
\H &\subseteq & E(\bN_{\L_X}(D))\mbox{ (by Step~3)}\\
&=&E(\bN_{\L_D}(X)) \mbox{ (by Lemma~\ref{L:NinNLXofY})}\\
&\subseteq & E(\L_D) \mbox{ (by \eqref{E:Ebalance})}\\
&\subseteq & E(\L) \mbox{ (by Step~5).}
\end{eqnarray*}
However, this contradicts $\mathfrak{C}\neq \emptyset$ and the definition of $\H$. 
\end{proof}

\section{Final remarks} Notice that it might not be true in general that $E(\bN_\L^\delta(X))\subseteq E(\L_\delta)$. At least such a statement does not follow from the E-balance theorem. If  $E(\bN_\L^\delta(X))$ could be viewed as an im-partial subgroup of $E(\L_\delta)$ in a systematical way, then it would probably not be necessary anymore to work with subcentric localities, but one could restrict attention to regular localities instead.

\smallskip

We will see now that the layers of $\bN_\L^\delta(X)$ and $\L_\delta$ coincide if $E(\L_\delta)\subseteq N_\L(X)$. We use E-balance in the proof of Lemma~\ref{L:ELdeltainNLX} below, but a more elementary proof could be given if one first translates the arguments used in Aschbacher~\cite[(10.1)]{Aschbacher:2011} to our locality setting. 

\begin{lemma}\label{L:ELdeltainNLX}
Let $X\in\F^f$ such that $E(\L_\delta)\subseteq N_\L(X)$. Then $\Comp(\L_\delta)=\Comp(\bN_\L^\delta(X))$ and  $E(\L_\delta)=E(\bN_\L^\delta(X))$.
\end{lemma}

\begin{proof}
By Lemma~\ref{L:ComponentinNorNperp}, we have $E(E(\L_\delta)^\perp_{\L_\delta})=1$. Hence it follows from Lemma~\ref{L:ProdCompinbNLD} that $E(\L_\delta)\unlhd \bN_\L^\delta(X)$ and $\Comp(\L_\delta)\subseteq \Comp(\bN_\L^\delta(X))$. Suppose there exists a component $\K$ of $\bN_\L^\delta(X)$ which is not a component of $\L_\delta$. Note that it is a consequence of Remark~\ref{R:SubnormalinSubgroupComponents} that $\Comp(E(\L_\delta))=\Comp(\L_\delta)\subseteq \Comp(\bN_\L^\delta(X))$. Hence, Lemma~\ref{L:ComponentinNorNperp} applied with $\bN_\L^\delta(X)$ in place of $\L$ yields that 
\[\K\subseteq E(\L_\delta)^\perp_{\bN_\L^\delta(X)}=C_{\bN_\L^\delta(X)}(E(\L_\delta))\subseteq C_\L(E(\L_\delta)).\]
By the E-balance theorem \ref{T:Ebalance}, we have $\K\subseteq E(\L)$ and thus $\K\cap S\subseteq E(\L)\cap S=E(\L_\delta)\cap S$. Hence, $\K\cap S$ must be abelian contradicting Lemma~\ref{L:Quasisimple}.
\end{proof}

\bibliographystyle{amsalpha}
\bibliography{repcoh}

\end{document}